%% file: arXiv_manuscript.tex
\documentclass[notitlepage]{report}

\usepackage[]{graphicx}
\usepackage{amsmath,amssymb,amsthm}
\usepackage{mathrsfs,mathtools}
\usepackage{array}
\usepackage{dsfont}
\usepackage{tabularx}
\usepackage{csquotes}
\usepackage{booktabs}

\usepackage{algorithm}
\usepackage[noend]{algpseudocode}

\makeatletter
\def\BState{\State\hskip-\ALG@thistlm}
\makeatother

\usepackage{hyperref}
\hypersetup{colorlinks,urlcolor=blue}
\hypersetup{
    colorlinks,
    linkcolor={blue},
    citecolor={blue},
    urlcolor={blue}
}

\makeatletter
\newcommand{\specialcell}[1]{\ifmeasuring@#1\else\omit$\displaystyle#1$\ignorespaces\fi}
\makeatother

\newcommand\pFq[5]{{_{#1}F_{#2}}\left({#3 \atop #4};#5\right)}
\newcommand\pFqReg[5]{{_{#1}\mathbf F_{#2}}\left({#3 \atop #4};#5\right)}
\newcommand\pFqmod[5]{{_{#1}\mathcal F_{#2}}\left({#3 \atop #4};#5\right)}

\newcommand*\pve[0]{\mathsf{PVE}}

\newcommand*\pr[0]{\mathsf P}

\newcommand*\var[0]{\mathsf{Var}}

\newcommand*\skewness[0]{\mathsf{Skew}}
\newcommand*\ev[0]{\mathsf E}
\newcommand*\cv[0]{\mathsf{CV}}
\newcommand*\arb[0]{\mathsf{ARB}}
\newcommand*\sarb[0]{\mathsf{arb}}
\newcommand*\acv[0]{\mathsf{ACV}}
\newcommand*\sacv[0]{\mathsf{acv}}
\newcommand*\bias[0]{\mathsf{Bias}}
\newcommand*\mse[0]{\mathsf{MSE}}

\mathchardef\mhyphen="2D
\newcommand{\electron}{e\mhyphen}
\newcommand{\nudag}{{\nu^{\smash\dagger}}}
\newcommand{\nudaghat}{{{\hat\nu}^{\smash\dagger}}}

\DeclareMathOperator*{\arginf}{arg\,inf}

\newtheorem{lemma}{Lemma}
\newtheorem{theorem}{Theorem}
\newtheorem{definition}{Definition}
\newtheorem{corollary}{Corollary}
\newtheorem{remark}{Remark}
\newtheorem{relation}{Relation}
\newtheorem{proposition}{Proposition}
\newtheorem{conjecture}{Conjecture}

\newenvironment{subproof}[1][\proofname]{%
  \begin{proof}[#1]%
}{%
  \end{proof}%
}

\newcolumntype{L}{>{$}l<{$}}
\newcolumntype{C}{>{$}c<{$}}
\newcolumntype{R}{>{$}r<{$}}

\makeatletter\@addtoreset{chapter}{part}\makeatother%

\title{A novel approach to photon transfer conversion gain estimation}
\author{Aaron Hendrickson\\ ahendr16@jh.edu}

\begin{document}
\maketitle

\begin{abstract}
Nonuniformities in the imaging characteristics of modern image sensors are a primary factor in the push to develop a pixel-level generalization of the photon transfer characterization method. In this paper, we seek to develop a body of theoretical results leading toward a comprehensive approach for tackling the biggest obstacle in the way of this goal: a means of pixel-level conversion gain estimation. This is accomplished by developing an estimator for the reciprocal-difference of normal variances and then using this to construct a novel estimator of the conversion gain. The first two moments of this estimator are derived and used to construct exact and approximate confidence intervals for its absolute relative bias and absolute coefficient of variation, respectively. A means of approximating and computing optimal sample sizes are also discussed and used to demonstrate the process of pixel-level conversion gain estimation for a real image sensor.
\end{abstract}
\begin{table}[b]\footnotesize\hrule\vspace{1mm}
\end{table}

\tableofcontents


\chapter{Introduction}
\label{chap:introduction}

Photon Transfer ({\sc pt}) is a methodology initially developed back in the 1970s for the design, characterization, and optimization of solid state image sensors. Since its inception, {\sc pt} has evolved to become the standard for image sensor characterization, culminating is its use as the basis for the European Machine Vision Association ({\sc emva}) 1288 standard in 2005 \cite{EMVA_1288}. To fully characterize the performance of an image sensor many types of performance metrics are measured including but not limited to conversion gain, read noise, and dynamic range.

Of all performance metrics prescribed by the {\sc pt} method, the so-called conversion gain, $g$, is fundamental for two reasons. First, $g$ is a conversion constant that facilitates unit conversion of sensor measurements from arbitrary units of digital numbers ({\sc dn}) into units of electrons $(\electron)$. For example, the read noise of a pixel is found according to the formula
\[
\sigma_{\mathrm{READ}}\,(\electron)=\sigma_{\mathrm d}\,(\mathrm{DN})\times g\,(\electron/\mathrm{DN}),
\]
where $\sigma_{\mathrm d}$ is the population standard deviation of the pixel's noise in the absence of illumination, i.e.~darkness\footnote{Throughout this work we will consider the characterization of individual pixels via the {\sc pt} method as not to impose the unnecessary assumption of \emph{sensor uniformity}. If however a sensor is \emph{uniform}, so that it is comprised of an array of identical pixels, we may implement {\sc pt} in such a manner as to characterize the entire array with global estimates of key population parameters (see Section \ref{sec:CG_estimation_KAI0407M}).}. Since units of $\mathrm{DN}$ are physically meaningless it is only after multiplying by $g$ that the measurement of $\sigma_{\mathrm d}$ represent a physical quantity. For this reason, nearly all imaging performance metrics measurable by the {\sc pt} method at some point require multiplying quantities in $\mathrm{DN}$ by $g$. Second, in actual implementation of {\sc pt}, $g$ is an estimated quantity; thus, the precision and accuracy of its measurement fundamentally limits the precision and accuracy of the entire {\sc pt} method.  This is easily seen when looking back at the equation for $\sigma_{\mathrm{READ}}$ above, for even if $\sigma_{\mathrm d}$ is measured with perfect certainty, that is, it's a known constant, we still have
\[
\bias\,\hat\sigma_{\mathrm{READ}}=\sigma_{\mathrm d}\times\bias\, G
\]
and
\[
\var\,\hat\sigma_{\mathrm{READ}}=\sigma^2_{\mathrm d}\times\var\, G,
\]
with $G$ being a random variable representing some estimator of $g$.

In some cases the pixels comprising the sensor will exhibit a linear \emph{transfer function} so that $g$ can be expressed by the formula \cite[Eqs.~$5.1$,$6.1$]{janesick_2007}
\begin{equation}
\label{eq:gain_formula}
g=\frac{\mu_{\mathrm p+\mathrm d}-\mu_{\mathrm d}}{\sigma_{\mathrm p+\mathrm d}^2-\sigma_{\mathrm d}^2}=\frac{\mu_{\mathrm p}}{\sigma_{\mathrm p}^2},
\end{equation}
where $\mu_{\mathrm p+\mathrm d}\,(\mathrm{DN})$ and $\sigma_{\mathrm p+\mathrm d}^2\, (\mathrm{DN}^2)$ are the population mean and variance of a pixel's output when exposed to some amount of incident illumination and $\mu_{\mathrm d}\, (\mathrm{DN})$ and $\sigma_{\mathrm d}^2\, (\mathrm{DN}^2)$ are the corresponding population parameters for the pixel's output under dark conditions.  In this way, $g$ is found from the ratio of the photon induced mean, $\mu_{\mathrm p}\,(\mathrm{DN})$, and photon induced variance $\sigma_{\mathrm p}^2\,(\mathrm{DN}^2)$. Linearity of the pixel's transfer function means that despite $\mu_{\mathrm p+\mathrm d}$ and $\sigma_{\mathrm p+\mathrm d}^2$ increasing with increasing illumination, $g$ remains constant and thus can be measured at any illumination level. While linearity makes $g$ easy to compute it tends to be an over idealized assumption for many modern pixel architectures such as {\sc cmos} active-pixel sensors.  Fortunately, the formula for $g$ in $(\ref{eq:gain_formula})$ can still be used so long as the illumination level is sufficiently low \cite{janesick_2007}.

Assuming all quantities in $(\ref{eq:gain_formula})$ are finite, a natural estimator for $g$ is \cite{Hendrickson:17,janesick_2007}
\begin{equation}
\label{eq:dark_corrected_g_est}
G=\frac{\bar X-\bar Y}{\hat X-\hat Y}=\frac{\bar P}{\hat P},
\end{equation}
where $\bar X=n_1^{-1}\sum_{k=1}^{n_1}X_i$ and $\hat X=(n_1-1)^{-1}\sum_{k=1}^{n_1}(X_i-\bar X)^2$ are the sample mean and variance computed from a sample of $n_1$ digital observations of a single pixel when exposed to illumination and $\bar Y$ and $\hat Y$ are the corresponding sample statistics taken from an independent sample of $n_2$ observations of the same pixel in the dark. It follows that $\bar P=\bar X-\bar Y$ and $\hat P=\hat X-\hat Y$ estimate the unknown photon induced mean and variance, respectively. Apart from its simplicity, what makes this estimator for $g$ attractive is that it is independent of distributional assumptions on the pixel's noise since $T=(\bar X,\bar Y,\hat X,\hat Y)$ is an unbiased estimator of $\theta=(\mu_{\mathrm p+\mathrm d},\mu_{\mathrm d},\sigma_{\mathrm p+\mathrm d}^2,\sigma_{\mathrm d}^2)$; regardless of the underlying distribution \cite[Thm.~$5.2.6$]{casella_2002}. This simplicity and lack of distributional assumptions does however come at the cost of some particularly undesirable statistical properties. 

In virtually any conceivable distributional model for the pixel noise, the density of $\hat P$ will be nonzero at the origin. As a result, $\hat P^{-1}$ and subsequently $G$ fall into the domain of attraction of the Cauchy law and thus have no finite moments \cite{piegorsch_1985}. To avoid the Cauchy-like behavior of the estimator $(\ref{eq:dark_corrected_g_est})$ that results from this lack of well-defined moments experimenters typically measure $g$ under high illumination conditions. By doing this, the probability of $\hat P$ being in the neighborhood of zero is negligible, which results in $G$ being quasi well-behaved. However, in the case of nonlinear pixels one is forced to measure $g$ under low illumination where the behavior of $G$ is the most ill-behaved.  In this scenario, one must capture increasingly large samples to force $\pr (|\hat P|<\epsilon)\approx 0$ and produce a well-behaved estimate of $\hat P^{-1}$. This tension between the need to measure $g$ at low-illumination and the large sample sizes it entails ultimately led to the \emph{low-illumination problem} of conversion gain measurement \cite{hendrickson_2019}.  

While convenient, the lack of distributional assumptions on the estimator $(\ref{eq:dark_corrected_g_est})$ is not all that important. Indeed, many authors have shown that most image sensors produce noise that is accurately modeled as normal \cite{janesick_2007,Beecken:96,Hendrickson:17}. Even in the case where the pixel noise exhibits departures from normality, {\sc pt} typically requires large samples sizes such that the distributions of the sample statistics $(\bar X, \bar Y,\hat X,\hat Y)$ show excellent agreement with what is predicted by a normal model. As such, this paper seeks an improved estimator for $g$ under the normal model of pixel noise which does not have the undesirable properties of the estimator $(\ref{eq:dark_corrected_g_est})$.  Since the normal model implies $(\hat X,\hat Y)$ is independent of $(\bar X, \bar Y)$, this task amounts to deriving an estimator for $(\sigma_{\mathrm p+\mathrm d}^2-\sigma_{\mathrm d}^2)^{-1}$ to replace $\hat P^{-1}$.

With the task at hand, this paper is organized as follows. We will begin by closing out this chapter with sections \ref{sec:mathematical_prelims}-\ref{sec:previous_work}, which introduce some of the most important functions in the subsequent analysis and provide a brief background on the progression of estimators for $g$. Chapter \ref{chap:estimation_of_recip_diff_norm_variances} then tackles the problem of finding an improved estimator, denoted $\mathscr T_\nu$, for the reciprocal difference of variance $\tau=(\sigma_{\mathrm p+\mathrm d}^2-\sigma_{\mathrm d}^2)^{-1}$ under the normal model. Several important results pertaining to $\mathscr T_\nu$ will be established including: (1) a proof of its first moment (Section \ref{sec:main_results}), an asymptotic expansion for large sample sizes (Section \ref{subsec:Tv_asym_expansion_derivation}), its second moment (Section \ref{subsec:variance}), and exact confidence intervals for its absolute relative bias and absolute coefficient of variation (Section \ref{sec:confidence_intervals}).

Chapter \ref{chap:new_g_estimator} will then apply the results of Chapter \ref{chap:estimation_of_recip_diff_norm_variances} to construct a novel estimator, $\mathscr G_\nu$, for the conversion gain based on $\mathscr T_\nu$. Section \ref{sec:g_estimation_demo} will utilize this new estimator as well as the theoretical results of the previous sections in a Monte Carlo experiment to demonstrate the process of estimating $g$ with $\mathscr G_\nu$. Since reducing sample sizes is an important consideration, Section \ref{sec:cv_analysis} investigates the behavior of optimal sample sizes for both $\mathscr T_\nu$ and $\mathscr G_\nu$. In particular, a close look at the behavior of the optimal sample sizes in low illumination will be conducted (Section \ref{subsec:characteristics_at_low_illumination}) followed by a detailed discussion on the approximation and computation of the optimal sample sizes (Section \ref{subsec:comp_of_opt_samp_sizes_for_Tv}). Finally, Section \ref{sec:CG_estimation_KAI0407M} will apply all of the preceding results to present the design and control of experiment for pixel-level conversion gain estimation using a real image sensor; opening the door to a comprehensive approach of pixel-level {\sc pt} characterization.


\section{Mathematical preliminaries}
\label{sec:mathematical_prelims}

This work will make extensive use of gamma, Pochhammer, hypergeometric, and related functions.  The purpose of this section is to introduce some notation and present key properties of these functions. It is recommended the reader also briefly familiarize themselves with the additional list of definitions and relations in Appendix \ref{sec:definitions}.

For $\Re s>0$ the gamma function can be expressed in the form of the integral
\[
\Gamma(s)=\int_0^\infty t^{s-1}e^{-t}\,\mathrm dt
\]
and is defined by analytic continuation of this integral to a meromorphic function in the complex plane via the reflection formula $\Gamma(s)\Gamma(1-s)=\pi\csc\pi s$. Directly tied to the gamma function are three additional functions that are great importance, namely, the beta function
\[
\operatorname B (s,z) \coloneqq \frac{\Gamma(s)\Gamma(z)}{\Gamma(s+z)},
\]
Pochhammer symbol (rising factorial)
\[
(s)_n \coloneqq \frac{\Gamma(s+n)}{\Gamma(s)}=%
\begin{cases}
1, &n=0\\
\prod_{k=0}^{n-1}(s+k), &n\in\Bbb N,
\end{cases}
\]
and factorial power (falling factorial)
\[
(s)^{(n)}\coloneqq \frac{\Gamma(s+1)}{\Gamma(s-n+1)}=%
\begin{cases}
1, &n=0\\
\prod_{k=0}^{n-1}(s-k), &n\in\Bbb N,
\end{cases}
\]
of which the latter two are related by $(s)^{(n)}=(-1)^n(-s)_n$ when $n\in\Bbb Z$. Through application of the gamma reflection formula we can also easily derive for $n\in\Bbb Z$ the following transformations
\[
(s)_n=(-1)^n(1-s-n)_n=\frac{(-1)^n}{(1-s)_{-n}}.
\]
Additionally, the Pochhammer symbol is related to the Stirling numbers by
\[
\begin{aligned}
s^n &=\sum_{k=0}^n (-1)^k{_2{\mathcal S}}_n^{(k)}(-s)_k,\\
(s)_n &=\sum_{k=0}^n (-1)^{n-k}\mathcal S_n^{(k)}s^k,
\end{aligned}
\]
where $\mathcal S_n^{(k)}$ and ${_2{\mathcal S}}_n^{(k)}$ represent Stirling numbers of the $1$st-kind and $2$nd-kinds as well as the binomial coefficient via
\[
\binom{n}{k}=\frac{(-1)^k(-n)_k}{k!}.
\]
Of perhaps greater significance is the use of the Pochhammer symbol in defining generalized hypergeometric functions which may be formally represented through the generalized hypergeometric series
\[
\pFq{p}{q}{a_1,\dots,a_p}{b_1,\dots,b_q}{z}=\sum_{k=0}^\infty\frac{(a_1)_k\cdots (a_p)_k}{(b_1)_k\cdots (b_q)_k}\frac{z^k}{k!}.
\]
If any of the top parameters $a_j$ is a nonpositive integer then this series reduces to a polynomial in $z$ since $n\in\Bbb N_0\implies (-n)_k=0$ for all $k=n+1,n+2,\dots$. For the specific case $p=q+1$ the generalized hypergeometric series has a radius of convergence of one and is defined by analytic continuation in $z$ for $|z|>1$. On the unit disk $|z|=1$ the series representation for ${_{q+1}F_q}(z)$ is absolutely convergent if $\Re\gamma_q>0$, convergent except at $z=1$ if $-1<\Re\gamma_q\leq 0$, and divergent if $\Re\gamma_q\leq -1$ where \cite[$\S 16.2(\mathrm{iii})$]{nist_2010}
\[
\gamma_q=b_1+\cdots+b_q-(a_1+\cdots+a_{q+1}).
\]
For convenience, we will use several differing notations for generalized hypergeometric functions when appropriate. These include ${_pF_q}(\mathbf a; \mathbf b; z)$  as well as the regularized version
\[
\pFqReg{p}{q}{a_1,\dots,a_p}{b_1,\dots,b_q}{z}\coloneqq \frac{1}{\prod_{k=1}^q\Gamma(b_k)}\pFq{p}{q}{a_1,\dots,a_p}{b_1,\dots,b_q}{z},
\]
which is an entire function of the parameters $a_1,\dots,a_p,b_1,\dots,b_q$. Of particular importance here is ``the'' hypergeometric function ${_2F_1}(a,b;c;z)$, which due to its prolific use in the literature is commonly denoted simply as $F(a,b;c;z)$. There are several properties of the hypergeometric function that will be heavily used here including the transformations \cite[Eq.~$15.8.1$]{nist_2010}
\begin{equation}
\label{eq:hyper_2F1_xForms}
\pFqReg{}{}{a,b}{c}{z}=%
\begin{cases}
(1-z)^{-a}\pFqReg{}{}{a,c-b}{c}{\frac{z}{z-1}} &(\mathrm i)\\[1.5ex]
(1-z)^{-b}\pFqReg{}{}{c-a,b}{c}{\frac{z}{z-1}} &(\mathrm{ii})\\[1.5ex]
(1-z)^{c-a-b}\pFqReg{}{}{c-a,c-b}{c}{z} &(\mathrm{iii}),
\end{cases}
\end{equation}
which hold for $|\operatorname{ph}(1-z)|<\pi$ as well as the integral form
\[
\mathbf F(a,b;c;z)=\int_0^1\frac{t^{b-1}(1-t)^{c-b-1}(1-zt)^{-a}}{\Gamma(b)\Gamma(c-b)}\,\mathrm dt,
\]
when $|\operatorname{ph}(1-z)|<\pi$ and $\Re c>\Re b>0$. Additionally, we note the special cases
\[
F(a,b;c;1)=\frac{\Gamma(c)\Gamma(c-a-b)}{\Gamma(c-a)\Gamma(c-b)},\quad\Re\{c-a-b\}>0
\]
and
\[
F(1,b;c;z)=(c-1)z^{1-c}(1-z)^{-(b-c+1)}\operatorname B_z(c-1,b-c+1),
\]
where $\operatorname B_z(\alpha,\beta)$ denotes the incomplete beta function of Definition \ref{def:betaInc_fun}. Lastly, for clarity we will use the following notation to denote various sets of numbers.
\[
\begin{array}{*3{>{\displaystyle}l}}
\Bbb N &=\{1,2,\dots\} &\text{natural numbers}\\
\Bbb N_0 &=\Bbb N\cup\{0\} &\text{nonnegative integers}\\
\Bbb Z &=-\Bbb N\cup\Bbb N_0 &\text{integers}\\
\Bbb R &=(-\infty,\infty) &\text{real numbers}\\
\Bbb R^+ &= (0,\infty) &\text{positive real numbers}\\
\Bbb R^+_0 &=\Bbb R^+\cup\{0\} &\text{nonnegative real numbers}
\end{array}
\]


\section{Previous work}
\label{sec:previous_work}

Statistical analysis of estimators for the conversion gain date back to the work of Beecken and Fossum \cite{Beecken:96}. For sensors that are able to achieve a \emph{shot noise limited response}, illuminating the sensor with a sufficiently high illumination implies $\mu_{\mathrm p+\mathrm d}\gg\mu_{\mathrm d}$ and $\sigma_{\mathrm p+\mathrm d}^2\gg\sigma_{\mathrm d}^2$ such that $g\approx\mu_{\mathrm p+\mathrm d}/\sigma_{\mathrm p+\mathrm d}^2$. Consequently, the gain can be estimated with \cite{Beecken:96}
\begin{equation}
\label{eq:basic_gain_estimator}
G=\frac{\bar X}{\hat X},
\end{equation}
where the normal model dictates $\bar X\sim\mathcal N(\mu_{\mathrm p+\mathrm d},\sigma_{\mathrm p+\mathrm d}^2/n_1)$ and $\hat X\sim\mathcal G(\alpha_1,\beta_1)$ with $\alpha_1=(n_1-1)/2$ and $\beta_1=\alpha_1/\sigma_{\mathrm p+\mathrm d}^2$ are independent normal and gamma random variables, respectively. One of the important conclusions in this work was that under sufficiently high illumination the variance of $G$ is dominated by the variance of $\hat X^{-1}$, that is,
\[
\var G\approx \bar X^2\,\var\hat X^{-1}.
\]
Hence, confidence intervals for $G$ can be approximated by scaling confidence intervals for $\hat X^{-1}$ by $\bar X^2$.

For sensors that cannot achieve a shot noise limited response one cannot ignore the noise produced by the pixel in the absence of illumination. As such, Hendrickson studied the estimator \cite{Hendrickson:17}
\begin{equation}
\label{eq:dark_corrected_G_estimator}
G=\frac{\bar X-\bar Y}{\hat X-\hat Y}=\frac{\bar P}{\hat P},
\end{equation}
where $\bar Y\sim\mathcal N(\mu_{\mathrm d},\sigma_{\mathrm d}^2/n_2)$ and $\hat Y\sim\mathcal G(\alpha_2,\beta_2)$ with $\alpha_2=(n_2-1)/2$ and $\beta_2=\alpha_2/\sigma_{\mathrm d}^2$.
The distribution of this estimator was derived in the form of the centralized inverse-Fano distribution and it was noted that this distribution had no finite moments due the density of $\hat P$ being nonzero at the origin.  In an effort to gain some insight about this estimator, further investigations were carried out in \cite{hendrickson_2019} to derive the first moment in the sense of the Cauchy principal value
\[
\pve G=\lim_{R\to\infty}\int_{-R}^R g\,f_G(g)\,\mathrm dg=(\ev\bar P)(\pve\hat P^{-1}).
\]
An expression for $\pve G$ opened the door to discussing the bias of $G$ in experimental settings and thus was useful from both theoretical and practical viewpoints. In particular, knowing $G$ is an estimator for $g$ one could define its principal-valued absolute relative bias as $\arb G=|\pve G/g-1|$ so that
\[
\arb G=\arb \hat P^{-1}=|\mathsf{RB}\hat P^{-1}-1|,
\]
where
\begin{multline*}
\mathsf{RB}\hat P^{-1}=%
\frac{\alpha_1\left(1-\zeta\right)\left(\frac{\alpha_2}{\alpha_1}\zeta^{-1}\right)^{\alpha_2}\left(1+\frac{\alpha_2}{\alpha_1}\zeta^{-1}\right)^{1-\alpha _1-\alpha _2}}{(\alpha_1+\alpha_2-1) \operatorname{B}(\alpha_1,\alpha_2)}%
\biggl(%
\psi(\alpha_1)-\psi(\alpha_2)\\
-\log\left(\frac{\alpha_1}{\alpha_2}\zeta\right)
-\frac{\alpha_2-1}{\alpha_2}\zeta\,\pFq{3}{2}{1,1,2-\alpha_2}{2,\alpha_1+1}{-\frac{\alpha_1}{\alpha_2}\zeta}\\
+\frac{\alpha_1-1}{\alpha_1} \zeta^{-1}\,\pFq{3}{2}{1,1,2-\alpha_1}{2,\alpha_2+1}{-\frac{\alpha_2}{\alpha_1}\zeta^{-1}}%
\biggr),
\end{multline*}
$\zeta=\sigma_{\mathrm d}^2/\sigma_{\mathrm p+\mathrm d}^2$, $\psi(z)\coloneqq\partial_z\log\Gamma(z)$ is the digamma function, and $\log z$ is the natural logarithm. This shows that the absolute relative bias of $G$ is eqaul to that of $\hat P^{-1}$, which is a function of only the sample sizes and variance ratio $\zeta$. To expound on the findings by Beecken and Fossum it was further demonstrated in a simulation that $\hat P^{-1}/\pve\hat P^{-1}$ converges in distribution to $G/\pve G$ as illumination increases. In other words, the dispersion of $G$ is dominated by that of $\hat P^{-1}$ at sufficiently high illumination. While these results led to increased theoretical understanding, the estimator $(\ref{eq:dark_corrected_G_estimator})$ still presented major challenges for the purpose of low illumination conversion gain measurement and thus motivated further work\footnote{Although it does not fit in with the natural progression of estimators discussed here, recent work into conversion gain estimation for more exotic technologies like deep sub-electron read noise image sensors have also been studied \cite{starkey_2016}.}.



\chapter{Estimation of the Reciprocal Difference of Normal Variances}
\label{chap:estimation_of_recip_diff_norm_variances}

We are now ready to turn to deriving an estimator for $(\sigma_{\mathrm p+\mathrm d}^2-\sigma_{\mathrm d}^2)^{-1}$. Under the normal model of pixel noise recall that $\hat X\sim\mathcal G(\alpha_1,\beta_1)$ and $\hat Y\sim\mathcal G(\alpha_2,\beta_2)$ are independent gamma random variables with shape $\alpha_i=(n_i-1)/2$ and $\beta_1=\alpha_1/\sigma_{\mathrm p+\mathrm d}^2$, $\beta_2=\alpha_2/\sigma_{\mathrm d}^2$. For the sake of brevity, we will modify our notation by considering the independent gamma random variables $Y_1$ and $Y_2$ where $Y_i\sim\mathcal G(\alpha_i,\beta_i)$, $\alpha_i$ is known, and $\beta_i=\alpha_i/\kappa_i$ with the understanding that these random variables have the same distributional form as that of $\hat X$ and $\hat Y$. Our first result establishes a key statistical property of these random variables.

\begin{lemma}
\label{lem:Y1_Y2_complete_statistic}
Let $Y_1\sim\mathcal G(\alpha_1,\beta_1)$ and $Y_2\sim\mathcal G(\alpha_2,\beta_2)$ be independent gamma random variables parameterized in terms of known shape $\alpha_i$ and unknown rate $\beta_i=\alpha_i/\kappa_i$, $i=1,2$. Then, $T(Y_1,Y_2)=(Y_1,Y_2)$ is a complete-sufficient statistic for $\theta=(\kappa_1,\kappa_2)$.
\end{lemma}

\begin{proof}
We only need to establish the proof for a single gamma variable with the extension to two independent gamma variables being trivial. For known $\alpha$, the gamma density $f(y|\beta)=h(y)c(\beta)\exp(w(\beta)T(y))$ with $h(y)=\mathds 1_{(0,\infty)}(y)y^{\alpha-1}$, $c(\beta)=\beta^\alpha/\Gamma(\alpha)$, $w(\beta)=-\beta$, and $T(y)=y$ is a member of the exponential family. It follows that that $T(Y)=Y$ is a sufficient statistic for $\beta$ \cite[Thm.~$6.2.10$]{casella_2002}. Furthermore, the parameter space $\beta\in(0,\infty)$ is an open subset of $\Bbb R$ which implies that $T(Y)$ must also be complete \cite[Thm.~$6.2.25$]{casella_2002}. Since $\kappa=\alpha/\beta$ is a one-to-one function of $\beta$ any complete-sufficient statistic for $\beta$ must also be complete-sufficient for $\kappa$. The proof is now complete.
\end{proof}


\section{The estimator $\mathscr T_n$}
\label{sec:Tn_estimmator}

In light of Lemma \ref{lem:Y1_Y2_complete_statistic}, to find an unbiased estimator $\mathscr T(Y_1,Y_2)$ of the estimand $\tau=(\kappa_1-\kappa_2)^{-1}$ one would need to solve the integral equation
\begin{equation}
\label{eq:double_int_eq}
\int_{\Bbb R^+\times\Bbb R^+}\mathscr T(y_1,y_2)f_{Y_1}(y_1)f_{Y_2}(y_2)\,\mathrm d(y_1,y_2)=\frac{1}{\kappa_1-\kappa_2}.
\end{equation}
Taking into account the form of the gamma densities at hand, namely,
\[
f_{Y_i}(y_i)=\frac{\beta_i^{\alpha_i}}{\Gamma(\alpha_i)}y_i^{\alpha_i-1}e^{-\beta_iy_i},\quad i=1,2,
\]
and noting $\kappa_i=\alpha_i/\beta_i$, the l.h.s.~of $(\ref{eq:double_int_eq})$ can be interpreted as an iterated Laplace transform of two variables to produce the equivalent expression
\begin{equation}
\label{eq:double_Laplace_xform}
\mathcal L\{y_2^{\alpha_2-1}\mathcal L\{y_1^{\alpha_1-1} \mathscr T(y_1,y_2)\}(\beta_1)\}(\beta_2)=\Gamma(\alpha_1)\Gamma(\alpha_2)\frac{\beta_1^{-\alpha_1}\beta_2^{-\alpha_2}}{\alpha_1/\beta_1-\alpha_2/\beta_2}.
\end{equation}
Recovering $\mathscr T$ is subsequently achieved by successively inverting the r.h.s.~of $(\ref{eq:double_Laplace_xform})$. Given the rich theory of Laplace transforms one might hope that published tables of transform pairs will provide the necessary result to invert this equation. Indeed, inversion w.r.t.~$\beta_2$ is achieved via \cite[Eq.~$5.4.9$]{california1954tables} yielding an unbiased estimator for $\tau$ when $\kappa_1$ is known (see Appendix \ref{sec:limiting estimator_nu_to_inf}). However, the final inversion w.r.t.~$\beta_1$ leads to an intractable problem; suggesting the possibility that the estimator $\mathscr T$ does not exist.  

\begin{conjecture}
There is no estimator $\mathscr T$ satisfying $(\ref{eq:double_int_eq})$ for unknown $\kappa_1$ and $\kappa_2$.
\end{conjecture}

Even if $\mathscr T$ does exist there is no guarantee that it will have desirable properties outside of unbiasedness. As the following theorem shows, if $\mathscr T$ exists, it must have infinite variance for at least certain values of the parameters $\kappa_1$ and $\kappa_2$.  

\begin{theorem}
\label{thm:no_finite_variance_estimator}
Let $Y_1$ and $Y_2$ be as in Lemma \ref{lem:Y1_Y2_complete_statistic}. If an estimator $\mathscr T$ satisfying $(\ref{eq:double_int_eq})$ exists, then $\var\mathscr T=\infty$ for at least $1/2<\kappa_2/\kappa_1<2$.
\end{theorem}

\begin{proof}
Read Lemma \ref{lem:T_estimator_discrete} and Lemma \ref{lem:T_est_variance_discrete}. Then see Appendix \ref{sec:limiting estimator_nu_to_inf}.
\end{proof}

Theorem \ref{thm:no_finite_variance_estimator} is significant because even if $\mathscr T$ exists and has finite variance outside $1/2<\kappa_2/\kappa_1<2$ there would be no way to know if any given estimate $\hat\tau=\mathscr T(y_1,y_2)$ has finite variance due to the unknown nature of $\kappa_1$ and $\kappa_2$. To overcome these challenges we appeal to the statistical principle of bias-variance tradeoff and expand our search to include biased estimators of $\tau$. Upon inspection, note that $\tau=\kappa_1^{-1}(1-\kappa_2/\kappa_1)^{-1}$. Letting $\zeta=\kappa_2/\kappa_1$, if we assume $\zeta<1$ it follows that the r.h.s.~of $(\ref{eq:double_int_eq})$ can be approximated by the incomplete geometric series
\[
\tau_n=\kappa_1^{-1}(1+\zeta+\cdots+\zeta^{n-1})=\frac{1-\zeta^n}{\kappa_1-\kappa_2},
\]
where the magnitude of the approximation error $E_n=-\tau\times \zeta^n$ can be made arbitrarily small with increasing $n$. In addition to the ability to achieve arbitrarily small error, this approximation is attractive since it is a finite sum of simple terms; thus, rendering it compatible with term-wise inversion to find a biased estimator $\mathscr T_n$. Here we state a simple but useful result and then proceed with the deriving the estimator $\mathscr T_n$.

\begin{lemma}
\label{lem:gamma_moments}
For $Y\sim\mathcal G(\alpha,\beta)$ and $s\in\Bbb C$ with $\alpha+\Re s>0$, $\ev Y^s=\beta^{-s}(\alpha)_s$.
\end{lemma}

\begin{lemma}
\label{lem:T_estimator_discrete}
Let $Y_1$ and $Y_2$ be as in Lemma \ref{lem:Y1_Y2_complete_statistic}. If $n\in\Bbb N_0:0\leq n<\alpha_1$ then
\[
\mathscr T_n=\frac{1}{\alpha_1 Y_1}\sum_{k=0}^{n-1}\frac{1}{(\alpha_1)_{-k-1}(\alpha_2)_k}\left(\frac{\alpha_2Y_2}{\alpha_1Y_1}\right)^k,
\]
is an unbiased estimator of $\tau_n$ where $\mathscr T_0\coloneqq 0$ is the empty sum.
\end{lemma}

\begin{proof}
Assume $n\in\Bbb N$. We seek an estimator $\mathscr T_n$ with the property
\[
\ev\mathscr T_n=\tau_n=\sum_{k=0}^{n-1}\kappa_2^k\,\kappa_1^{-(k+1)}.
\]
From Lemma \ref{lem:gamma_moments} we know $\alpha_2^k/(\alpha_2)_kY_2^k$ and $\alpha_1^{-k-1}/(\alpha_1)_{-k-1}Y_1^{-k-1}$ are unbiased estimators of $\kappa_2^k$ and $\kappa_1^{-(k+1)}$ with the latter having finite expected value if $n<\alpha_1$. Since $Y_1$ and $Y_2$ are independent, the desired expression for $\mathscr T_n$ immediately follows. Now writing $\mathscr T_n=\sum_{k=0}^{n-1}g(k)$ we have $\mathscr T_n=\mathscr T_{n-1}+g(n-1)$. Substituting $n=1$ into this recurrence formula and noting $\mathscr T_1=g(0)$ implies $\mathscr T_0=0$. The proof is now complete.
\end{proof}

\begin{remark}
\label{rem:discrete_est_bias}
Lemma \ref{lem:T_estimator_discrete} makes no mention of the restriction $\kappa_1>\kappa_2$.  However, if one is to use $\mathscr T_n$ as an estimator for $\tau$ this is required for the absolute bias $|\mathsf{Bias}\mathscr T_n|=\zeta^n|\tau|$ to be less than $|\tau|$.
\end{remark}

The discrete nature of the parameter $n$ leads to a simple derivation of the estimator $\mathscr T_n$, however, this simplicity comes at the cost of flexibility. This is easily seen by noting that for fixed $\alpha_i$ and $\kappa_i$ the moments $\ev\mathscr T_n^m$ can only take on a countable set of values which prohibits arbitrary choices of key statistical properties such as bias. To illustrate why this is problematic, suppose there exists a generalized estimator $\mathscr T_\nu$, continuous in $\nu$, that interpolates $\mathscr T_n$ and its moments, that is, $\mathscr T_\nu=\mathscr T_n$ and $\ev\mathscr T_\nu^m=\ev\mathscr T_n^m$ when $\nu=n$. Then again fixing $\alpha_i$ and $\kappa_i$, if one wishes to minimize mean-squared-error we know it must be the case that $\min_{\nu\in\Bbb R_0^+}\mse\mathscr T_\nu\leq\min_{n\in\Bbb N_0}\mse\mathscr T_n$; rendering $\mathscr T_\nu$ the superior estimator in terms of mean-squared-error. Figure \ref{fig:MSE_plot} plots $\mse\mathscr T_n$ for some sample parameters along with one possible continuous generalization showing a discrepancy between their minimum values. Since the additional degree of freedom afforded by a continuous generalization of $\mathscr T_n$ can improve performance the question arises: How do we go about seeking such a generalization? In a rather remarkable fashion we will see how this can be achieved through the methods of \emph{summability calculus}.

\begin{figure}[htb]
\centering
\includegraphics[scale=1]{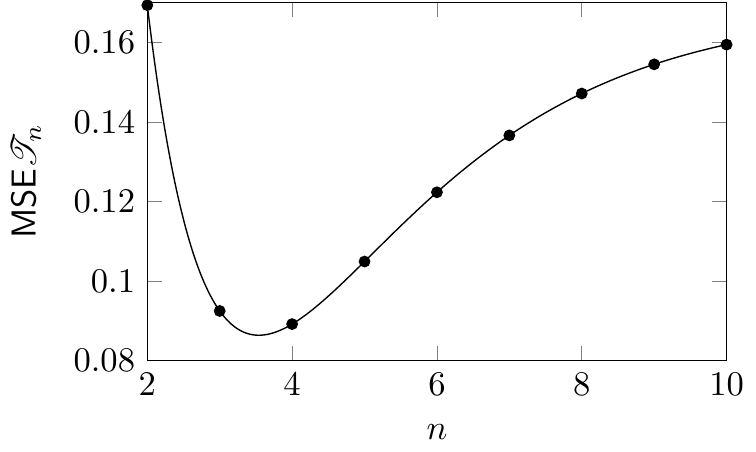}
\caption{Plot of $\mse\mathscr T_n$ versus $n$ (points) with one possible continuous interpolation (line).}
\label{fig:MSE_plot}
\end{figure}


\section{Summability calculus}
\label{sec:frac_sums}

Summability calculus is in essence a theoretical framework for generalizing finite sums $f(n)=\sum_{k=0}^{n-1}s_k g(k,n)$ for a periodic sequence $s_k$ and analytic function $g(k,n)$ to complex-valued $n$ and performing infinitesimal calculus on these generalized sums. A key point of the theory of summability calculus is that $f(n)$ is itself an analytic form leading to the ability to perform infinitesimal calculus on the generalized sum even if the explicit form of the generalized sum is not known. To see how such a generalization is obtained, we will limit the discussion to so-called \emph{simple finite sums} of the form $f(n)=\sum_{k=0}^{n-1}g(k)$. By definition, the simple finite sum $f(n)$ satisfies the recurrence relation $f(n)=f(n-1)+g(n-1)$ and upon substituting $n=1$ we find $f(1)=g(0)+f(0)\implies f(0)=0$. Consequently, one can fully characterize $f(n)$ by the recurrence relation and initial condition
\[
f(n)=g(n-1)+f(n-1),\quad f(0)=0.
\]
With this information we can then construct a generalized \emph{fractional finite sum} $f_G(\nu):\Bbb C\to\Bbb C$ via an iterative polynomial approximation scheme. At the $r$th iteration we approximate generalized sum with
\[
f_{G,r}(\nu)=%
\begin{cases}
p_r(\nu) &\nu\in[0,1]\\
g(\nu-1)+f_{G,r}(\nu-1) &\text{otherwise,}
\end{cases}
\]
where $p_r(\nu)=a_1\nu+\dots+a_r\nu^r$ is a polynomial of degree $r$. If for each iteration we require $f_{G,r}(\nu)$ to be $(r-1)$-times continuously differentiable on $\nu\in(0,2)$ then the coefficients of $p_r(\nu)$ are unique and so is the limiting function $f_G(\nu)=\lim_{r\to\infty}f_{G,r}(\nu)$.

\begin{theorem}[Statement of uniqueness: {\cite[Thm.~$2.1$]{alabdulmohsin_2018}}]
\label{thm:FracSum_uniqueness}
Given a simple finite sum $f(n)=\sum_{k=0}^{n-1}g(k)$ where $g:\Bbb C\to\Bbb C$ is analytic at the origin, let $p_r(\nu)$ be a polynomial in $\nu$ of degree $r$ and define
\[
f_{G,r}(\nu)=%
\begin{cases}
p_r(\nu) &\nu\in[0,1]\\
g(\nu-1)+f_{G,r}(\nu-1) &\text{otherwise.}
\end{cases}
\]
If we require $f_{G,r}(\nu)$ to be $(r-1)$-times differentiable on $\nu\in(0,2)$, then the limiting function $f_G(\nu)=\lim_{r\to\infty}f_{G,r}(\nu)$ is unique, satisfies the recurrence identity $f_G(\nu)=g(\nu-1)+f_G(\nu-1)$, and the initial condition $f_G(0)=0$.
\end{theorem}

Using this iterative polynomial approximation procedure one subsequently obtains the generalized sum $f_G(\nu)$ formally given by the Taylor series
\begin{equation}
\label{eq:f_Gv_Taylor_series}
f_G(\nu)=\sum_{k=1}^\infty \frac{\nu^k}{k!}\sum_{\ell=0}^\infty \frac{B_\ell}{\ell!}g^{(k+\ell-1)}(0),
\end{equation}
where $B_n=\{1,-\tfrac 12,\tfrac 16,0,\dots\}$ are the Bernoulli numbers\footnote{It is interesting to note that if we interchange the order of summation in $(\ref{eq:f_Gv_Taylor_series})$ we find
\[
f_G(\nu)=\sum_{\ell=0}^\infty \frac{B_\ell}{\ell !}\left(g^{(\ell-1)}(\nu)-g^{(\ell-1)}(0)\right),
\]
which is precisey the Euler-Maclaurin formula for $\sum_{k=0}^{\nu-1}g(k)$.}. This formal series expansion for $f_G(\nu)$ may or may not converge and so explicit methods for evaluating fractional sums are desired. The following presents a summability method and theorem that will be used extensively for evaluating fractional finite sums in this work.

\begin{definition}[$\mathfrak T$-summation: {\cite[Def.~$4.1$]{alabdulmohsin_2018}}]
\label{def:T_summation}
Let $h(s)$ be the function whose Taylor series about the origin is given by
\[
h(s)=\sum_{k=0}^\infty g(k)s^k.
\]
If $h(s)$ is analytic on $s\in[0,1]$ then we define the $\mathfrak T$-sum of $\sum_{k=0}^\infty g(k)$ by
\[
\sum_{k=0}^\infty g(k)\overset{\mathfrak T}{=}h(1).
\]
\end{definition}

\begin{lemma}[Properties of $\mathfrak T$ {\cite[Prop.~$4.2$]{alabdulmohsin_2018}}]
\label{lem:Tsummability_properties}
The summability method $\mathfrak T$ is regular, linear, and stable.
\end{lemma}

\begin{theorem}[Evaluating simple finite sums: {\cite[Thm.~$5.5$]{alabdulmohsin_2018}}]
\label{thm:FracSum_form1}
Given a simple finite sum $f(n)=\sum_{k=0}^{n-1}g(k)$, if $\sum_{k=0}^\infty g(k)$ is $\mathfrak T$-summable then the unique generalization of $f(n)$ consistent with Theorem \ref{thm:FracSum_uniqueness} is given by
\[
f_G(\nu)\overset{\mathfrak T}{=}\sum_{k=0}^\infty g(k)-\sum_{k=0}^\infty g(k+\nu).
\]
\end{theorem}

With these results at hand, we will take the next few sections to present three fractional finite sums for later use in deriving the generalized estimator $\mathscr T_\nu$ and its associated properties.


\subsection{Incomplete geometric series}
\label{subsec:incGeomSeries}

The first fractional sum we study is the incomplete geometric series which is foundational to the evaluation of many other fractional sums.

\begin{definition}[Incomplete geometric series]
\label{def:inc_1F0}
The incomplete geometric series ${_1F_0}(1,- ;z)_\nu$ is given by
\[
{_1F_0}(1;-;z)_\nu\coloneqq \nu F(1,1-\nu;2;1-z).
\]
\end{definition}

\begin{lemma}
\label{lem:incGeomSeries_general}
${_1}F_0(1;-;z)_\nu$ is the unique generalization of $f(n)=\sum_{k=0}^{n-1}z^k$ consistent with Theorem \ref{thm:FracSum_uniqueness}.
\end{lemma}

\begin{proof}
Defining $h(s)\coloneqq(1-zs)^{-1}$ we see that $h(s)$ is analytic on $s\in[0,1]$ if $z\notin[1,\infty)$; thus
\[
\sum_{k=0}^\infty z^k\overset{\mathfrak T}{=}h(1)=\frac{1}{1-z},\quad z\in\Bbb C\setminus[1,\infty).
\]
Taking regularity, linearity, and stability as axioms we then extend this results to include all $z\in\Bbb C\setminus\{1\}$. In accordance with Theorem \ref{thm:FracSum_form1} it then follows that the fractional generalization of $f(n)$ is
\[
f_G(\nu)=(1-z^\nu)\sum_{k=0}^\infty z^k\overset{\mathfrak T}{=}\frac{1-z^\nu}{1-z},\quad z\in\Bbb C\setminus\{1\}.
\]
To derive the form in Definition \ref{def:inc_1F0} we use the generalized binomial theorem to obtain the formal expression
\[
(1-z)f_G(\nu)=1-(1-(1-z))^\nu=1-\sum_{k=0}^\infty\binom{\nu}{k}(-1)^k(1-z)^k.
\]
The first term in the series expansion is one, therefore after some algebraic manipulations we arrive at
\[
(1-z)f_G(\nu)=(1-z)\nu\sum_{k=0}^\infty\frac{(1)_k(1-\nu)_k}{(2)_k\,k!}(1-z)^k,
\]
which upon dividing both sides by $(1-z)$ yields the desired result.
\end{proof}


\subsection{Incomplete Lerch Transcendent}
\label{subsec:incLerchPhi}

We now present our second fractional finite sum in the form of the incomplete Lerch Transcendent.  This fractional finite sum will be used in deriving a series expansion for $\ev\mathscr T_\nu^2$ in Section \ref{subsec:variance} as well as an asymptotic expansion of $\mathscr T_\nu$ in Section \ref{subsec:Tv_asym_expansion_derivation}. We first begin with the definition of the (complete) Lerch Transcendent.

\begin{definition}[Lerch transcendent]
\label{def:LerchPhi}
The Lerch transcendent $\Phi(z,s,\omega)$ is defined as the analytic continuation of the series
\[
\Phi(z,s,\omega)\coloneqq\sum_{k=0}^\infty (k+\omega)^{-s}z^k,
\]
with $\Phi(z,0,\omega)\coloneqq{_1F_0}(1;-;z)$.
\end{definition}

\begin{definition}[Incomplete Lerch transcendent]
\label{def:incomplete_LerchPhi}
The incomplete Lerch transcendent $\Phi(z,s,\omega)_\nu$ is given by
\[
\Phi(z,s,\omega)_\nu\coloneqq\Phi(z,s,\omega)-z^\nu\Phi(z,s,\omega+\nu),
\]
with $\Phi(z,0,\omega)_\nu\coloneqq{_1F_0}(1,- ;z)_\nu$.
\end{definition}

\begin{lemma}
\label{lem:inc_Lerch_trans_uniqueness}
$\Phi(z,s,\omega)_\nu$ is the unique generalization of $f(n)=\sum_{k=0}^{n-1}(k+\omega)^{-s}z^k$ consistent with Theorem \ref{thm:FracSum_uniqueness}.
\end{lemma}

\begin{proof}
This follows from the fact that
\[
h(s)\coloneqq\Phi(sz,t,\omega)\overset{\mathfrak T}{=}\sum_{k=0}^\infty(k+\omega)^{-t}(sz)^k
\]
is analytic for $s\in[0,1]$ when $|z|<1$. Therefore, by Theorem \ref{thm:FracSum_form1} the unique generalization is
\[
f_G(\nu)\overset{\mathfrak T}{=}\sum_{k=0}^\infty(k+\omega)^{-s}z^k-\sum_{k=0}^\infty(k+\omega+\nu)^{-s}z^{k+\nu},
\]
which is extended to $z\in\Bbb C$ via analytic continuation of $\Phi(z,s,\omega)$.
\end{proof}

To obtain a deeper understanding of the incomplete Lerch Transcendent we now introduce three differential operators and present many identities relating them. These identities will then allow us to derive properties of $\Phi(z,s,\omega)_\nu$; including its relationship to the incomplete geometric series for use in Section \ref{subsec:gnw_tilde_gnw}.

\begin{definition}[Theta operator]
\label{def:theta_operator}
$\vartheta\coloneqq z\,\partial_z$
\end{definition}

\begin{definition}[Lowering operator]
\label{def:lowering_operator}
$\Lambda_\omega\coloneqq \omega+\vartheta$
\end{definition}

\begin{definition}[Factorial operator]
\label{def:factorial_differential_operator}
Let $\mathcal D$ be a differential operator, then
\[
(\mathcal D)^{(n)}\coloneqq%
\begin{cases}
1, &n=0\\
\prod_{k=0}^{n-1}(\mathcal D-k), &n\in\Bbb N.
\end{cases}
\]
\end{definition}

\begin{lemma}[Operator identities]
\label{lem:diff_op_identities}
For $n\in\Bbb N_0$ and $\omega\in\Bbb Z$,
\[
\begin{array}{*2{>{\displaystyle}l}}
(i) &\textstyle{\Lambda_\omega z^s=z^s\Lambda_{\omega+s}}\\[0.5ex]
(ii) &(\vartheta z)^n=z^n\partial_z^n z^n,\\[0.5ex]
(iii) &(z\vartheta)^{n}=z^{n+1}\partial^n_zz^{n-1},\\[0.5ex]
(iv) &\Lambda_\omega^n=z^{-\omega}\vartheta^nz^\omega,\\[0.5ex]
(v) &(\vartheta)^{(n)}=z^n\partial^n_z,\\[0.5ex]
(vi) &(\Lambda_\omega)^{(n)}=z^{n-\omega}\partial^n_zz^\omega,\\[0.5ex]
(vii) &\textstyle{\vartheta^n=\sum_{k=0}^n{_2\mathcal S}_n^{(k)}z^k\partial_z^k}.
\end{array}
\]
\end{lemma}

\begin{proof}
See Appendix \ref{sec:differential_op_props}.
\end{proof}

\begin{lemma}
\label{lem:LerchPhi_z_derivative}
For $n\in\Bbb N_0$, the Lerch Transcendent satisfies
\[
\Lambda_\omega^n\Phi(z,s,\omega) =\Phi(z,s-n,\omega).
\]
\end{lemma}

We can now establish the properties of the incomplete Lerch transcendent that are needed in the following sections.

\begin{lemma}
\label{lem:IncLerchPhi_z_derivative}
For $n\in\Bbb N_0$, the incomplete Lerch Transcendent satisfies the same differential relation as that of Lemma \ref{lem:LerchPhi_z_derivative}, namely,
\[
\Lambda_\omega^n\Phi(z,s,\omega)_\nu =\Phi(z,s-n,\omega)_\nu.
\]
\end{lemma}

\begin{proof}
Defining $P(n):\Lambda_\omega^n\Phi(z,s,\omega)_\nu =\Phi(z,s-n,\omega)_\nu$ we observe that $P(0)$ trivially holds. Now assuming $P(n)$ and using Lemma \ref{lem:diff_op_identities} $(i)$ and Lemma \ref{lem:LerchPhi_z_derivative} we have
\[
\begin{aligned}
\Lambda_\omega^{n+1}\Phi(z,s,\omega)_\nu
&=\Lambda_\omega\Phi(z,s-n,\omega)_\nu\\
&=\Lambda_\omega\Phi(z,s-n,\omega)-\Lambda_\omega z^\nu\Phi(z,s-n,\omega+\nu)\\
&=\Lambda_\omega\Phi(z,s-n,\omega)-z^\nu\Lambda_{\omega+\nu}\Phi(z,s-n,\omega+\nu)\\
&=\Phi(z,s-(n+1),\omega)-z^\nu\Phi(z,s-(n+1),\omega+\nu)\\
&=\Phi(z,s-(n+1),\omega)_\nu.
\end{aligned}
\]
Thus, $P(n)\implies P(n+1)$ which completes the proof.
\end{proof}

\begin{corollary}
\label{cor:inc_LerchPhi_derivative_integer_n}
Setting $s=0$ in Lemma \ref{lem:IncLerchPhi_z_derivative} yields the formula,
\[
\Phi(z,-n,\omega)_\nu=\Lambda_\omega^n{_1F_0}(1;- ;z)_\nu.
\]
If in addition $\omega=0$ one finds
\[
\Phi(z,-n,0)_\nu=\vartheta^n{_1F_0}(1;- ;z)_\nu.
\]
\end{corollary}

\begin{corollary}
\label{cor:IncPolyLog_HyperSum_form}
For $n\in\Bbb N_0$ and $\lvert\operatorname{ph}z\rvert<\pi$
\[
\Phi(z,-n,0)_\nu=\frac 1z\sum_{k=0}^n\frac{{_2\mathcal S}_n^{(k)}(\nu)^{(k+1)}}{k+1}\pFq{}{}{1+k,1+\nu}{2+k}{1-\frac 1z}.
\]
\end{corollary}

\begin{proof}
With Lemma \ref{lem:diff_op_identities} $(vii)$ and \cite[Eq.~$15.5.2$]{nist_2010}, we have
\[
\Phi(z,-n,0)_\nu=\sum_{k=0}^n\frac{(1)_k}{(2)_k}\,{_2\mathcal S}_n^{(k)}(-1)^k\nu(1-\nu)_kz^k\pFq{}{}{1+k,1-\nu+k}{2+k}{1-z}.
\]
Using $(1)_k/(2)_k=1/(k+1)$, $(-1)^k\nu(1-\nu)_k=(\nu)^{(k+1)}$, and applying the transformation \cite[Eq.~$15.8.1(\mathrm{i})$]{nist_2010} to the hypergeometric term yields the desired result.
\end{proof}

%


\subsection{Sine-modulated incomplete hypergeometric function}
\label{subsec:sine_mod_2F1_inc}

We now introduce our last fractional finite sum: the sine-modulated incomplete hypergeometric function.  This fractional finite sum will serve as the cornerstone for defining the generalized estimator $\mathscr T_\nu$.




\begin{definition}[Sine-modulated incomplete hypergeometric function]
\label{def:Inc2F1}
For $-\gamma\notin\Bbb N_0$, the sine-modulated incomplete hypergeometric functions is given by
\[
\mathcal F(\alpha,\beta;\gamma;z)_\nu\coloneqq\pFq{}{}{\alpha,\beta}{\gamma}{z}
-\frac{(\alpha)_\nu\, (-z)^\nu}{(1)_\nu(1-\beta)_{-\nu}(\gamma)_\nu}\pFq{3}{2}{1,\alpha+\nu,\beta+\nu}{1+\nu,\gamma+\nu}{z}.
\]
\end{definition}

\begin{lemma}
\label{lem:incHyper_unique_generalization}
$\mathcal F(\alpha,\beta;\gamma;-z)_\nu$ is the unique generalization of
\[
f(n)=\sum_{k=0}^{n-1}\frac{\sin(\pi(\beta+k))}{\sin\pi\beta}\frac{(\alpha)_k(\beta)_k}{(\gamma)_k}\frac{z^k}{k!}
\]
consistent with Theorem \ref{thm:FracSum_uniqueness}.
\end{lemma}

\begin{proof}
Again with the help of Theorem \ref{def:T_summation} we define
\[
h(s)\coloneqq\pFq{}{}{\alpha,\beta}{\gamma}{-sz}\overset{\mathfrak T}{=}\sum_{k=0}^\infty\frac{\sin(\pi(\beta+k))}{\sin\pi\beta}\frac{(\alpha)_k(\beta)_k}{(\gamma)_k}\frac{(sz)^k}{k!},
\]
which is analytic on $s\in[0,1]$ for all $|z|<1$; thus, by Theorem \ref{thm:FracSum_form1} the fractional generalization of $f(n)$ is
\[
f_G(\nu)\overset{\mathfrak T}{=}\sum_{k=0}^\infty\frac{(\alpha)_k(\beta)_k}{(\gamma)_k}\frac{(-z)^k}{k!}-\frac{\sin(\pi(\beta+\nu))}{\sin\pi\beta}\sum_{k=0}^\infty(-1)^k\frac{(\alpha)_{k+\nu}(\beta)_{k+\nu}}{(\gamma)_{k+\nu}\Gamma(1+\nu+k)}z^{k+\nu},
\]
for $|z|<1$. Then using the identity $(s)_{z+r}=(s)_z(s+z)_r$ produces
\begin{multline*}
f_G(\nu)\overset{\mathfrak T}{=}\sum_{k=0}^\infty\frac{(\alpha)_k(\beta)_k}{(\gamma)_k}\frac{(-z)^k}{k!}\\
-\frac{\sin(\pi(\beta+\nu))}{\sin\pi\beta}\frac{(\alpha)_\nu(\beta)_\nu}{(1)_\nu(\gamma)_\nu}z^\nu\sum_{k=0}^\infty\frac{(1)_k(\alpha+\nu)_k(\beta+\nu)_k}{(1+\nu)_k(\gamma+\nu)_k}\frac{(-z)^k}{k!},
\end{multline*}
where each series is now identified as being generalized hypergeometric series. Finally, from the gamma reflection formula we have
\[
\frac{\sin(\pi(\beta+\nu))}{\sin\pi\beta}(\beta)_\nu=\frac{1}{(1-\beta)_{-\nu}}.
\]
Substituting this result into $f_G(\nu)$ yields the desired form of $\mathcal F(\alpha,\beta;\gamma;-z)_\nu$, which is extended to $z\in\Bbb C$ via analytic continuation of the generalized hypergeometric function. The proof is now complete.
\end{proof}

\begin{corollary}[Special case for $\alpha=1$]
\label{cor:inc2F1_special_case}
\[
\mathcal F(1,\beta;\gamma;-z)_\nu=\pFq{}{}{1,\beta}{\gamma}{-z}
-\frac{z^\nu}{(1-\beta)_{-\nu}(\gamma)_\nu}\pFq{}{}{1,\beta+\nu}{\gamma+\nu}{-z}
\]
\end{corollary}

\begin{corollary}
\label{cor:Tn_inc_2F1_form}
For $n\in\{n\in\Bbb N_0:0\leq n<\alpha_1\}$
\[
\mathscr T_n=\frac{\alpha_1-1}{\alpha_1Y_1}\pFqmod{}{}{1,2-\alpha_1}{\alpha_2}{-\frac{\alpha_2 Y_2}{\alpha_1 Y_1}}_{\!n}.
\]
\end{corollary}

\begin{proof}
Recalling the expression for $\mathscr T_n$, application of the gamma reflection formula gives
\[
\mathscr T_n=\frac{\alpha_1-1}{\alpha_1 Y_1}\sum_{k=0}^{n-1}\frac{\sin(\pi(\alpha_1-k))}{\sin\pi\alpha_1}\frac{(1)_k(2-\alpha_1)_k}{(\alpha_2)_k k!}\left(\frac{\alpha_2Y_2}{\alpha_1Y_1}\right)^k.
\]
By the properties of the sine function we can write
\[
\frac{\sin(\pi(\alpha_1-k))}{\sin\pi\alpha_1}=\frac{\sin(\pi(2-\alpha_1+k))}{\sin(\pi(2-\alpha_1))},
\]
which upon substitution into the expression for $\mathscr T_n$ yields the form of $f(n)$ in Lemma \ref{lem:incHyper_unique_generalization} for $\alpha=1$, $\beta=2-\alpha_1$, $\gamma=\alpha_2$, and $z=\alpha_2Y_2(\alpha_1Y_1)^{-1}$.
\end{proof}


\section{The generalized estimator $\mathscr T_\nu$}
\label{sec:main_results}

This section introduces the unique fractional generalization of the estimator $\mathscr T_n$ and its properties. In light of Corollary \ref{cor:Tn_inc_2F1_form}, we may now suspect that the sine-modulated incomplete hypergeometric function can be used to extend the domain of $\mathscr T_n$ beyond integer valued $n$ while preserving statistical properties. The following theorem proves this conjecture to be true. We will first present the estimator $\mathscr T_\nu$ with complex parameter $\nu$ in its full generality and then show how \emph{a priori} knowledge of $\kappa_1$ and $\kappa_2$ can be incorporated to produce a useful estimator for $\tau=(\kappa_1-\kappa_2)^{-1}$.

\begin{theorem}
\label{thm:T_estimator}
Let $Y_1$ and $Y_2$ be as in Lemma \ref{lem:Y1_Y2_complete_statistic}. If $\nu\in\Bbb C:-\alpha_2<\Re\,\nu<\alpha_1$, then
\[
\mathscr T_\nu=\frac{\alpha_1-1}{\alpha_1Y_1}\pFqmod{}{}{1,2-\alpha_1}{\alpha_2}{-\frac{\alpha_2 Y_2}{\alpha_1 Y_1}}_{\!\nu},
\]
is an unbiased estimator of $\tau_{\nu}=(1-\zeta^\nu)/(\kappa_1-\kappa_2)$.
\end{theorem}

\begin{proof}
We simply need to show that $\mathscr T_\nu$ yields the correct expected value. Given $Y_i\sim\mathcal G(\alpha_i,\alpha_i/\kappa_i)$ one has $\alpha_i Y_i\sim\kappa_i Y_i^\ast$ where $Y_i^\ast\sim\mathcal G(\alpha_i,1)$. Let $U=Y_1^\ast$, $V=Y_2^\ast/Y_1^\ast$, $\zeta=\kappa_2/\kappa_1$, and $\mathcal F(-\zeta V)_{\nu}=\mathcal F(1,2-\alpha_1;\alpha_2;-\zeta V)_{\nu}$, then
\[
\mathscr T_\nu=\frac{\alpha_1-1}{\kappa_1}U^{-1}\mathcal F(-\zeta V)_{\nu}.
\]
Given $Y_1\perp Y_2$ we use change of variables to write the joint density
\[
f_{UV}(u,v)=\frac{(1+v)^{\alpha_1+\alpha_2}}{\Gamma(\alpha_1+\alpha_2)}u^{\alpha_1+\alpha_2-1}e^{-(1+v)u}\frac{v^{\alpha_2-1}(1+v)^{-\alpha_1-\alpha_2}}{\operatorname{B}(\alpha_1,\alpha_2)},
\]
which upon inspection has the form $f_{UV}(u,v)=f_{U|V}(u)f_V(v)$ where $U|V\sim\mathcal G(\alpha_1+\alpha_2,1+V)$ and $V\sim\beta^\prime(\alpha_2,\alpha_1)$, which is a \emph{beta prime} random variable. Now making use of Lemma \ref{lem:gamma_moments} we write
\[
\ev\mathscr T_\nu=\ev\ev(\mathscr T_\nu|V)=\frac{1}{\kappa_1}\frac{\alpha_1-1}{\alpha_1+\alpha_2-1}\ev\left[ (1+V)\mathcal F(-\zeta V)_\nu\right]=\frac{1}{\kappa_1}\ev{\mathcal F}(-\zeta W)_{\nu},
\]
where $W\sim\beta^\prime(\alpha_2,\alpha_1-1)$\footnote{The last equality can be easily reach upon multiplying the density $f_V$ with $\frac{(\alpha_1-1)(1+v)}{\alpha_1+\alpha_2-1}$.}. Reintroducing the expression for $\mathcal F(-\zeta W)_{\nu}$ we arrive at
\[
\kappa_1\,\ev\mathscr T_\nu=\ev\pFq{}{}{1,2-\alpha_1}{\alpha_2}{-\zeta W}
-\zeta^\nu\frac{\ev W^\nu\pFq{}{}{1,2-\alpha_1+\nu}{\alpha_2+\nu}{-\zeta W}}{(\alpha_1-1)_{-\nu}(\alpha_2)_\nu}.
\]
The expected value of each term can be evaluated via \cite[Eq.~$7.512.10$]{gradshteyn_ryzhik_2014} to find
\begin{multline*}
\kappa_1\zeta^{1-\alpha_1}\ev\mathscr T_\nu=-(1-\alpha_1)\pFq{}{}{1,\alpha_1}{2}{1-\zeta}+(1-\alpha_1+\nu)\pFq{}{}{1,\alpha_1-\nu}{2}{1-\zeta},
\end{multline*}
where the second term is finite if only if $-\alpha_2<\Re\,\nu<\alpha_1$. Now, upon inspection of Definition \ref{def:inc_1F0} we see that this result is equivalent to
\[
\kappa_1\zeta^{1-\alpha_1}\ev\mathscr T_\nu={_1F_0}(1;-;\zeta)_{1-\alpha_1+\nu}-{_1F_0}(1;-;\zeta)_{1-\alpha_1}.
\]
Therefore, given the identity ${_1F_0}(1;-;z)_\nu=(1-z^\nu)(1-z)^{-1}$ we have
\[
\ev\mathscr T_\nu=\frac{1-\zeta^{\alpha_1-1}}{\kappa_1-\kappa_2}-\frac{\zeta^\nu-\zeta^{\alpha_1-1}}{\kappa_1-\kappa_2},
\]
which upon simplifying yields the desired result.
\end{proof}

%

\begin{corollary}[Minimum variance estimation]
\label{cor:Tv_is_an_UMVUE}
$\mathscr T_\nu$ is the uniformly minimum variance unbiased estimator of $\tau_\nu$.
\end{corollary}

\begin{proof}
From Lemma \ref{lem:Y1_Y2_complete_statistic} we know that $T=(Y_1,Y_2)$ is a complete-sufficient statistic of $(\kappa_1,\kappa_2)$. Since $\mathscr T_\nu(Y_1,Y_2)$ is a function of $T$ it follows from the Lehmann-Scheff\' e theorem that $\mathscr T_\nu$ is the unique uniformly minimum variance unbiased estimator of its expected value $\tau_\nu$.
\end{proof}

\begin{corollary}
\label{corollary:T0_distribution}
$\mathscr T_0\sim\delta_0$ is a degenerate random variable with $\ev\mathscr T_0=0$.
\end{corollary}

\begin{proof}
This follows directly from the fundamental property of fractional finite sums: $f_G(0)=0$.
\end{proof}

At this point we have established $\mathscr T_\nu$ to be the unique generalization of $\mathscr T_n$ and presented a couple of its properties. We now proceed to show that $\mathscr T_\nu$ satisfies a reflection formula, which will be used many times in later sections.

\begin{theorem}[Reflection formula]
\label{thm:reflection_formula}
\[
\mathscr T_\nu(Y_1,Y_2,\alpha_1,\alpha_2)=-\mathscr T_{-\nu}(Y_2,Y_1,\alpha_2,\alpha_1)
\]
\end{theorem}

\begin{proof}
We begin with the general expression for $\mathscr T_\nu$, namely,
\begin{multline*}
\mathscr T_\nu(Y_1,Y_2,\alpha_1,\alpha_2)=\frac{\alpha_1-1}{\alpha_1Y_1}\biggl(%
\pFq{}{}{1,2-\alpha_1}{\alpha_2}{-\frac{\alpha_2 Y_2}{\alpha_1 Y_1}}\\
-\frac{\left(\frac{\alpha_2 Y_2}{\alpha_1Y_1}\right)^\nu}{(\alpha_1-1)_{-\nu}(\alpha_2)_\nu}\pFq{}{}{1,2-\alpha_1+\nu}{\alpha_2+\nu}{-\frac{\alpha_2 Y_2}{\alpha_1 Y_1}}%
\biggr).
\end{multline*}
Then, \cite[Eq.~$15.8.2$]{nist_2010} provides the necessary result to derive the transformation formula
\[
\pFq{}{}{1,\beta}{\gamma}{-z}=
\frac{\Gamma(1-\beta)z^{1-\gamma}}{(\gamma)_{-\beta}(1+z)^{\beta-\gamma+1}}
-\frac{\gamma-1}{z(1-\beta)}\pFq{}{}{1,2-\gamma}{2-\beta}{-\frac{1}{z}},
\]
which is subject to the constraint $|\operatorname{ph}z|<\pi$. Applying this transformation to each hypergeometric term in $\mathscr T_\nu$ and simplifying yields
\[
\begin{aligned}
\mathscr T_\nu(Y_1,Y_2,\alpha_1,\alpha_2)%
&=-\frac{\alpha_2-1}{\alpha_2 Y_2}\biggl(%
\pFq{}{}{1,2-\alpha_2}{\alpha_1}{-\frac{\alpha_1 Y_1}{\alpha_2 Y_2}}\\
&\quad\ \ -\frac{\left(\frac{\alpha_1 Y_1}{\alpha_2Y_2}\right)^{-\nu}}{(\alpha_2-1)_\nu(\alpha_1)_{-\nu}}\pFq{}{}{1,2-\alpha_2-\nu}{\alpha_1-\nu}{-\frac{\alpha_1 Y_1}{\alpha_2 Y_2}}%
\biggr)\\
&=-\mathscr T_{-\nu}(Y_2,Y_1,\alpha_2,\alpha_1),
\end{aligned}
\]
which completes the proof.
\end{proof}

\begin{corollary}
\label{cor:moment_reflection_formula}
$(\ev\mathscr T_\nu^n)(\kappa_1,\kappa_2,\alpha_1,\alpha_2)=(-1)^n(\ev\mathscr T_{-\nu}^n)(\kappa_2,\kappa_1,\alpha_2,\alpha_1)$.
\end{corollary}

\begin{proof}
The proof follows immediately from the reflection formula of Corollary \ref{cor:Tv_is_an_UMVUE} and $Y_i\sim \kappa_i/\alpha_i Y_i^\ast$ with $Y_i^\ast\sim\mathcal G(\alpha_i,1)$. In other words, interchanging $Y_1$ and $Y_2$ interchanges $\kappa_1$ and $\kappa_2$ in the expected value.
\end{proof}

\begin{proposition}
\label{prop:T_est_bias}
As an estimator for $\tau=(\kappa_1-\kappa_2)^{-1}$, $\mathscr T_\nu$ is biased with bias equal to
\[
\bias\mathscr T_\nu\coloneqq\ev\mathscr T_\nu-\tau=-\frac{\zeta^\nu}{\kappa_1-\kappa_2}
\]
and absolute relative bias equal to
\[
\arb\mathscr T_\nu\coloneqq\left\lvert\frac{\ev\mathscr T_\nu-\tau}{\tau}\right\rvert=\zeta^\nu.
\]
\end{proposition}

Corollary \ref{corollary:T0_distribution}, Theorem \ref{thm:reflection_formula}, and Proposition \ref{prop:T_est_bias} all hint at the significance of the parameter $\nu$ in controlling the bias, dispersion, and sign of $\mathscr T_\nu$. To gain a better intuition for these relationships it would be informative to plot the density of $\mathscr T_\nu$ for several values $\nu$. Deriving explicit forms of this density is likely impossible so we shall turn to estimating it with Monte Carlo methods. Estimating the density was done by first generating a total of $10^6$ i.i.d.~pseudo-random observations of $Y_1$ and $Y_2$ using the parameters $\alpha_1=20$, $\alpha_2=15$, $\kappa_1=2$, and $\kappa_2=1$, which correspond to $\tau=1$. These pseudo-random observations and parameters were then used to generate equally sized samples of $\mathscr T_\nu(Y_1,Y_2,20,15)$ for the eight selected values of $\nu$ in Table \ref{tbl:special_constants}\footnote{Since $\mathscr T_\nu$ estimates $\tau_\nu$ for any $\nu\in\Bbb C$ along the strip $-\alpha_2<\Re\nu<\alpha_1$, it seemed proper to choose interesting values of $\nu$ for this demonstration.}.
\begin{table}[htb]
\centering
\begin{tabular}{LCR}\toprule
\textbf{Symbol} &\textbf{Name/Expression} &\multicolumn{1}{C}{\textbf{Decimal Expansion}}\\\midrule
\zeta^\prime(2)  &\frac{1}{6}\pi^2(\gamma+\log 2\pi-12\log A) &-0.9375482543\dots  \\[0.5ex]
\delta_1  &\text{Hall--Montgomery constant}  &-0.6569990137\dots  \\[0.5ex]
m_{1,4}  &\text{Meissel--Mertens constant}  &-0.2867420562\dots  \\[0.5ex]
i^i  &\exp(-\pi/2)  &0.2078795764\dots  \\[0.5ex]
K  &\text{Landau--Ramanujan constant}  &0.7642236535\dots  \\[0.5ex]
x_\Gamma  &\text{Minimizer of $\Gamma(x)$ on $\Bbb R^+$}  &1.4616321449\dots  \\[0.5ex]
L  &\text{L\'{e}vy constant,}\ \exp(\pi^2/\log 64)  &10.7310157948\dots\\[0.5ex]
\zeta_0 &\text{First non-trivial zero of $\zeta(s)$} &1/2+14.1347251417i\\\bottomrule
\end{tabular}
\caption{Special constants (see \cite{finch_2003} for details). Here, $\gamma$ is the Euler--Mascheroni constant and $A$ denotes the Glaisher--Kinkelin constant.}
\label{tbl:special_constants}
\end{table}

Kernel density estimates were computed for all eight samples corresponding to each value of $\nu$ as plotted in Figure \ref{fig:bias_variance_tradeoff}. Looking in particular at the plots of the kernel density estimates for the real-valued $\nu$ we see that $\operatorname{sign}(\mathscr T_\nu)=\operatorname{sign}(\nu)$ and that the dispersion of the density increases with increasing $|\nu|$. Notice also that only the estimates generated with values of $\nu>0$ serve as useful estimates for $\tau$; showing that one must incorporate a priori information into $\mathscr T_\nu$ by requiring $\operatorname{sign}(\nu)=\operatorname{sign}(\tau)$ in order to render it a useful estimator of $\tau$.
\begin{figure}[htb]
\centering
\includegraphics[scale=1]{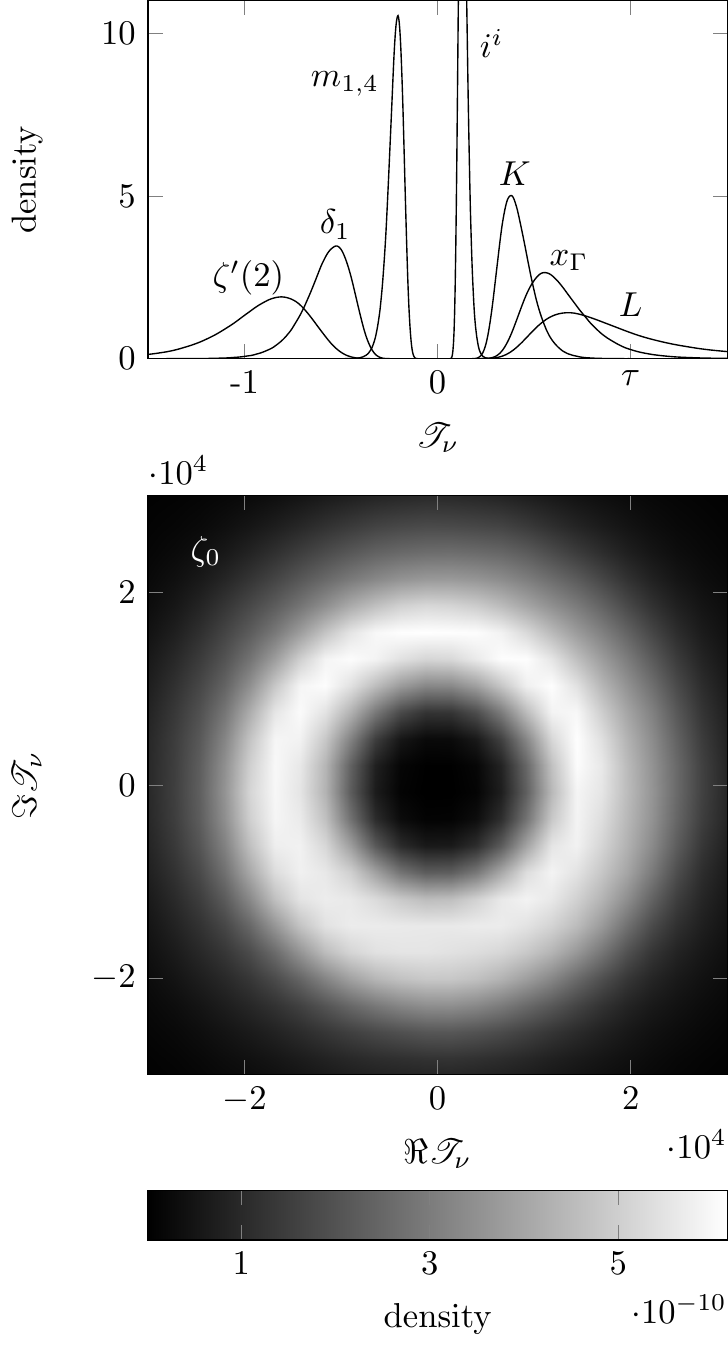}
\caption{Kernel density estimates of the $\mathscr T_\nu$ probability density for select values of $\nu$ given in Table \ref{tbl:special_constants}. Each density estimate is labeled with the value of $\nu$ used to generate it.}
\label{fig:bias_variance_tradeoff}
\end{figure}

\begin{remark}[Incorporating \emph{a priori} knowledge of $\kappa_1$ and $\kappa_2$]
\label{rmk:Tv_as_estimator_for_tauv}
As in Lemma \ref{lem:T_estimator_discrete}, Theorem \ref{thm:T_estimator} makes no mention of a restriction on the relative magnitudes of $\kappa_1$ and $\kappa_2$, i.e.~$\kappa_1>\kappa_2$ or $\kappa_1<\kappa_2$.  By permitting $\nu\in\Bbb R$, $\mathscr T_\nu$ can be used as an estimator for $\tau=(\kappa_1-\kappa_2)^{-1}$ in both cases whereby one imposes the rule $\nu>0$ when $\kappa_1>\kappa_2$ or $\nu<0$ when $\kappa_1<\kappa_2$ to ensure that $\arb\mathscr T_\nu<1$.
\end{remark}


\subsection{Asymptotic expansion for large $\alpha_1$ and $\alpha_2$}
\label{subsec:Tv_asym_expansion_derivation}

As $\alpha_1$ and $\alpha_2$ become large, the estimator $\mathscr T_\nu$ becomes increasingly difficult to evaluate; rendering it incompatible with practical applications involving large shape parameters.  Here we derive a few preliminary results and then proceed to present an asymptotic expansion of $\mathscr T_\nu$ in Theorem \ref{thm:Test_asym_exp}.
\begin{lemma}[\cite{tricomi_1951}]
\label{lem:gamma_ratio_asym_series}
As $z\to\infty$ in the sector $|\operatorname{ph}z|<\pi$
\[
\frac{\Gamma(z+\alpha)}{\Gamma(z+\beta)}\sim z^{\alpha-\beta}\sum_{k=0}^\infty\binom{\alpha-\beta}{k}B_k^{(\alpha-\beta+1)}(\alpha)\frac{1}{z^k},
\]
where $B_n^{(\ell)}(x)$ is the generalized N{\o}rlund polynomial given by Definition \ref{def:GenNorlundB}.
\end{lemma}
\begin{corollary}
\label{cor:power_pochhammer_asym_series}
As $z\to\infty$ in the sector $|\operatorname{ph}z|<\pi$
\[
\frac{z^\alpha}{(z)_\alpha}\sim\sum_{k=0}^\infty\binom{-\alpha}{k}B_k^{(1-\alpha)}\frac{1}{z^k},
\]
where $B_n^{(\ell)}=B_n^{(\ell)}(0)$.
\end{corollary}
\begin{lemma}
\label{lem:power_pochhammer_asym_series_converges}
For $\alpha=n\in\Bbb Z$, the asymptotic series in Corollary \ref{cor:power_pochhammer_asym_series} converges absolutely for all $z>\max\{0,n-1\}$.
\end{lemma}
\begin{proof}
If $n\leq 0$ then
\[
\frac{z^n}{(z)_n}\sim\sum_{k=0}^{-n}\binom{-n}{k}B_k^{(1-n)}\frac{1}{z^k},
\]
which is a sum of a finite number of terms and thus converges absolutely for all $z>0$. We now prove absolute convergence for all remaining $n\geq 1$ by induction. Beginning with the relationship between N{\o}rlund's polynomial and the Stirling number of the second-kind we write \cite[Eq.~3]{carlitz_1960}
\[
\frac{z^n}{(z)_n}\sim\sum_{k=0}^\infty(-1)^k {_2}\mathcal S_{k+n-1}^{(n-1)}\frac{1}{z^k}\eqqcolon S_n.
\]
Substituting $n=1$ gives
\[
S_1=\sum_{k=0}^\infty(-1)^k \delta_k\frac{1}{z^k}=1+0+0+\cdots,
\]
which clearly converges absolutely for all $z>0$. Assuming $S_n$ converges absolutely for all $z>n-1$ we write with the help of \cite[Eq.~$04.15.17.0002.01$]{wolfram_functions}
\[
S_{n+1}=\sum_{k=0}^\infty\sum_{\ell=0}^k(-n)^{k-\ell}(-1)^\ell {_2}\mathcal S_{\ell+n-1}^{(n-1)}\frac{1}{z^k}=S_n\sum_{k=0}^\infty(-n)^k\frac{1}{z^k}.
\]
But $\sum_{k=0}^\infty(-n)^k\frac{1}{z^k}=(1+n/z)^{-1}$ converges absolutely if $z>n$; hence, if $z>n$ then $S_{n+1}$ is the product of absolutely convergent series, which is itself absolutely convergent. The proof is now complete.
\end{proof}

\begin{theorem}
\label{thm:Test_asym_exp}
As $\alpha_1,\alpha_2\to\infty$
\[
\mathscr T_\nu\sim\frac{1}{Y_1}\sum_{k=0}^\infty\sum_{\ell=0}^{2k}p_{k,\ell}(\alpha_1,\alpha_2)\Phi(Y_2/Y_1,-\ell,0)_\nu,
\]
where $p_{k,\ell}(\alpha_1,\alpha_2)=\partial_x^\ell P_{2k}(0)/\ell !$,
\[
P_{2k}(x)=\sum_{m=0}^k\frac{1}{\alpha_1^{m}\alpha_2^{k-m}}\binom{x+1}{m}\binom{-x}{k-m}B_m^{(2+x)}B_{k-m}^{(1-x)},
\]
and $B_n^{(s)}$ is the N{\o}rlund polynomial defined by the generating function in Definitions \ref{def:GenNorlundB}-\ref{def:NorlundB}.
\end{theorem}

\begin{proof}
Beginning with the expression for $\mathscr T_n$ we have
\[
\mathscr T_n=\frac{1}{Y_1}\sum_{k=0}^{n-1}\frac{\alpha_1^{-k-1}\alpha_2^k}{(\alpha_1)_{-k-1}(\alpha_2)_k}\left(\frac{Y_2}{Y_1}\right)^k.
\]
Now consider the result of Corollary \ref{cor:power_pochhammer_asym_series}, which states $z^s/(z)_s\sim\sum_{\ell=0}^\infty G_\ell(s)z^{-\ell}$ as $z\to\infty$ where \cite[Eq.~$5.11.13$]{tricomi_1951,nist_2010}
\[
G_\ell(s)=\binom{-s}{\ell}B_\ell^{(1-s)}.
\]
Substituting in appropriate values, we obtain two separate asymptotic series in $\alpha_1$ and $\alpha_2$ for which the product yields
\begin{equation}
\label{eq:pochhammer_asym_expansion}
\frac{\alpha_1^{-k-1}\alpha_2^k}{(\alpha_1)_{-k-1}(\alpha_2)_k}\sim\sum_{\ell=0}^\infty\sum_{m=0}^\ell G_m(-k-1)G_{\ell-m}(k)\alpha_1^{-m}\alpha_2^{m-\ell}.
\end{equation}
Substituting this asymptotic expansion in $\mathscr T_\nu$ and rearranging the order of summation then gives
\begin{equation}
\label{eq:T_est_asym_form1}
\mathscr T_n\sim\frac{1}{Y_1}\sum_{\ell=0}^\infty\sum_{k=0}^{n-1}\left(\sum_{m=0}^\ell\frac{G_m(-k-1)G_{\ell-m}(k)}{\alpha_1^{m}\alpha_2^{\ell-m}}\right)\left(\frac{Y_2}{Y_1}\right)^k.
\end{equation}
With the help of Definition \ref{def:StirlingS1} and \cite[Thm.~$1$]{Liu_2006} the quantity $G_m(-k-1)G_{\ell-m}(k)$ can be easily shown to be a polynomial in $k$ of degree $2\ell$; thus, the entire sum inside the parentheses of $(\ref{eq:T_est_asym_form1})$ must also be a polynomial in $k$ of degree $2\ell$. Let,
\[
\begin{aligned}
P_{2\ell}(k)%
&\coloneqq\sum_{m=0}^\ell\frac{G_m(-k-1)G_{\ell-m}(k)}{\alpha_1^{m}\alpha_2^{\ell-m}}\\
&=p_{\ell,0}(\alpha_1,\alpha_2)+p_{\ell,1}(\alpha_1,\alpha_2)k+\dots+p_{\ell,2\ell}(\alpha_1,\alpha_2)k^{2\ell}.
\end{aligned}
\]
We can then write
\[
\mathscr T_n\sim\frac{1}{Y_1}\sum_{\ell=0}^\infty\sum_{m=0}^{2\ell}p_{\ell,m}(\alpha_1,\alpha_2)\sum_{k=0}^{n-1}k^m\left(\frac{Y_2}{Y_1}\right)^k.
\]
The interior sum over $k$ can now be expressed by $\Phi(Y_2/Y_1,-m,0)_n$ which is uniquely extended to $n\in\Bbb C$ via Lemma \ref{lem:inc_Lerch_trans_uniqueness}.
\end{proof}

\begin{corollary}[Zeroth order approximation]
Truncating the asymptotic expansion for $\mathscr T_\nu$ at $k=0$ yields
\[
\mathscr T_{\nu,0}=\frac{1-(Y_2/Y_1)^\nu}{Y_1-Y_2}=\tau_\nu(Y_1,Y_2).
\]
\end{corollary}

\begin{proposition}[Asymptotic expansion convergence]
\label{prop:asym_expansion_convergence}
For the special case $\nu=n$ with $n\in\Bbb N$, the asymptotic expansion for $\mathscr T_\nu$ given in Theorem \ref{thm:Test_asym_exp} converges absolutely whenever $\alpha_2>n-2$.
\end{proposition}

\begin{proof}
According to Lemma \ref{cor:power_pochhammer_asym_series}, the asymptotic expansion for $\alpha_1^{-k-1}/(\alpha_1)_{-k-1}$ consists of a finite number of terms while the asymptotic expansion for $\alpha_2^k/(\alpha_2)_k$ converges absolutely for all $\alpha_2>n-2$. Hence, the product of these epansions must also be absolutely convergent, which leads to the desired conclusion.
\end{proof}

From a numerical standpoint, using the form of the asymptotic expansion for $\mathscr T_\nu$ in Theorem \ref{thm:Test_asym_exp} is problematic since $\Phi(z,-n,0)_\nu$ has a removable singularity at $z=1$. As such, to construct an asymptotic approximation for practical applications, Corollary \ref{cor:IncPolyLog_HyperSum_form} leads to the more convenient form
\begin{equation}
\label{eq:Tv_asym_hyper_form}
\mathscr T_{\nu,K}=\frac{1}{Y_2}\sum_{m=0}^{2K}c_{m,K}(\alpha_1,\alpha_2)(\nu)^{(m+1)}\pFq{}{}{m+1,1+\nu}{m+2}{1-\frac{Y_1}{Y_2}},
\end{equation}
where
\[
c_{m,K}(\alpha_1,\alpha_2)=\frac 1{m+1}\sum_{\ell=m}^{2K}\sum_{k=0}^K {_2\mathcal S}_\ell^{(m)}p_{k,\ell}(\alpha_1,\alpha_2).
\]


\section{Measures of dispersion}
\label{sec:measures_of_dispersion}

In this section we derive absolutely convergent series expansions and integrals on the unit square for the variance and absolute coefficient of variation of $\mathscr T_\nu$.  Deriving these expressions directly by integrating powers of $\mathscr T_\nu$ w.r.t. the joint density of $(Y_1,Y_2)$ prove to be intractable. As such, the following derivations rely on deriving the appropriate quantities for the integer-valued estimator $\mathscr T_n$ and then using the uniqueness of the fractional finite sum to generalize to noninteger $\nu$. Central to the following derivations are the functions $g_{n,\omega}(z,\nu)$ and $\tilde g_{n,\omega}(z,\nu)$ presented in section \ref{subsec:gnw_tilde_gnw}. While somewhat extensive, section \ref{subsec:gnw_tilde_gnw} provides many useful results for characterizing the behavior of $\var\mathscr T_\nu$ and $\acv\mathscr T_\nu$ and in particular for establishing the monotonicity of $\acv\mathscr T_\nu$ used in Section \ref{sec:confidence_intervals} on confidence intervals.


\subsection{The functions $g_{n,\omega}$ and $\tilde g_{n,\omega}$}
\label{subsec:gnw_tilde_gnw}

\begin{definition}
\label{def:g_nw_function}
Let $\mathcal N=\{(n,\omega):(n,\omega)\in\Bbb N_0^2\ \land\ \omega\leq n\}$. Then for $(n,\omega)\in\mathcal N$, $z\in\Bbb R^+_0$, and $\nu\in\Bbb R$
\[
g_{n,\omega}(z,\nu)\coloneqq (\Lambda_\omega)^{(n)}{_1F_0}(1;- ;z)_\nu
\]
and
\[
\tilde g_{n,\omega}(z,\nu)\coloneqq \frac{g_{n,\omega}(z,\nu)}{{_1F_0}(1;- ;z)_\nu},
\]
where $\Lambda_\omega$ and $(\mathcal D)^{(n)}$ are the lowering operator of Definition \ref{def:lowering_operator} and factorial operator of Definition \ref{def:factorial_differential_operator}, respectively.
\end{definition}

It is important to note that Definition \ref{def:g_nw_function} explicitly defines $g_{n,\omega}$ and its regularized counterpart $\tilde g_{n,\omega}$ only for natural $n$ and $\omega$ on the set $\mathcal N$. As such, from here on out we will assume $(n,\omega)\in\mathcal N$ whenever discussing these functions unless specified otherwise. In the following Theorem we will derive two explicit forms for $g_{n,\omega}$. The expression given by form $(\mathrm{ii})$ possesses a natural extension to noninteger values of $n$ and $\omega$ via the gamma function and so form $(\mathrm{ii})$ can be viewed as a more general, continuous extension of form $(\mathrm i)$. This property will will come in handy when we get to Lemma \ref{lem:hyper_ratio_monotonicity} where the ability to differentiate these functions w.r.t~$n$ and $\omega$ greatly simplifies the proof. Furthermore, the incomplete Lerch transcendent functions $\Phi(\cdot)_\nu$ in form $(\mathrm i)$ have removable singularities at $z=1$ whereas form $(\mathrm{ii})$ does not and so form $(\mathrm{ii})$ is also more desirable for numerical computation in the neighborhood of $z=1$.

\begin{theorem}
\label{thm:IncLerch_sum}
\[
g_{n,\omega}(z,\nu)=%
\begin{cases}
\sum_{k=0}^n\mathcal S_n^{(k)}\Phi(z,-k,\omega)_\nu &(\mathrm i)\\[0.5em]
n!\, (\omega+\nu)^{(n+1)}z^\nu\mathbf F(1,\omega+\nu+1;n+2;1-z) &(\mathrm{ii})
\end{cases}
\]
\end{theorem}

\begin{proof}
With Relation \ref{rel:rising_falling_factorial} and Definition \ref{def:StirlingS1} we expand the factorial operator yielding
\[
g_{n,\omega}(z,\nu)=\left(\sum_{k=0}^n\mathcal S_n^{(k)}\Lambda_\omega^k\right){_1F_0}(1;- ;z)_\nu.
\]
Upon inspection of Corollary \ref{cor:inc_LerchPhi_derivative_integer_n} we write $\Lambda_\omega^k{_1F_0}(1;- ;z)_\nu=\Phi(z,-k,\omega)_\nu$; hence, form $(\mathrm i)$ is obtained.

To derive form $(\mathrm{ii})$ begin by using Lemma \ref{lem:diff_op_identities} $(vi)$ and the incomplete geometric series in Definition \ref{def:inc_1F0} to write
\[
\begin{aligned}
g_{n,\omega}(z,\nu)
&=\nu z^{n-\omega}\partial^n_zz^\omega F(1,1-\nu;2;1-z)\\
&=\nu z^{n-\omega}\partial^{n-\omega}_zz^{-\omega}\left[(z\partial_zz)^\omega F(1,1-\nu;2;1-z)\right],
\end{aligned}
\]
where the last equality is a consequence of Lemma \ref{lem:diff_op_identities} $(ii)$. Now consider the term in square brackets and substitute $t=1-z$. Noting that $\partial_{1-t}=-\partial_t$ we have
\[
(z\partial_zz)^\omega F(1,1-\nu;2;1-z)\mapsto (-1)^\omega((1-t)\partial_t(1-t))^\omega F(1,1-\nu;2;t).
\]
The resulting differential formula in $t$ is the same form as that of \cite[Eq.~$15.5.7$]{nist_2010} giving
\[
(-1)^\omega((1-t)\partial_t(1-t))^\omega F(1,1-\nu;2;t)=\frac{(\nu+1)_\omega}{\omega+1}(1-t)^\omega F(\omega+1,1-\nu;\omega+2;t).
\]
Reintroducing $z=1-t$ and substituting the result back into $g_{n,\omega}$ then gives
\[
g_{n,\omega}(z,\nu)%
=\nu \frac{(\nu+1)_\omega}{\omega+1}z^{n-\omega}\partial^{n-\omega}_z F(\omega+1,1-\nu;\omega+2;1-z).
\]
Evaluating the remaining derivative with \cite[Eq.~$15.5.2$]{nist_2010} and simplifying we have
\[
g_{n,\omega}(z,\nu)%
=n!\, (\omega+\nu)^{(n+1)}z^{n-\omega}{\mathbf F}(n+1,1-\nu+n-\omega;n+2;1-z),
\]
which upon applying the linear transformation in \cite[Eq.~$15.8.1\,(\mathrm{iii})$]{nist_2010} at last produces form $(\mathrm{ii})$.
\end{proof}

\begin{corollary}[Explicit forms for $\tilde g_{n,\omega}$]
\label{cor:gTilde_hyper_form}
\[
\tilde g_{n,\omega}(z,\nu)=%
\begin{cases}
n!\frac{(\omega+\nu)^{(n+1)}}{\nu}\frac{\mathbf F(1,\omega+\nu+1;n+2;1-z)}{\mathbf F(1,\nu+1;2;1-z)} &(\mathrm i)\\[0.5em]
n!\frac{z^{n-\omega}\operatorname I_{1-z}(n+1,\omega+\nu-n)}{(1-z)^n(1-z^\nu)} &(\mathrm{ii})\\[0.5em]
n!\frac{z^{n-\omega}-z^\nu\sum_{k=0}^n\frac{(\omega+\nu-n)_k}{k!} (1-z)^k}{(1-z)^n(1-z^\nu)} &(\mathrm{iii})
\end{cases}
\]
\end{corollary}

\begin{proof}
The proof for form $(\mathrm i)$ follows from the definition of $\tilde g_{n,\omega}$ and the hypergeometric transformation \cite[Eq.~$15.8.1(\mathrm{iii})$]{nist_2010} to write ${_1F_0}(1;-;z)_\nu=\nu z^\nu\mathbf F(1,\nu+1;2;1-z)$. Form $(\mathrm{ii})$ is derived with the aid of Relation \ref{rel:2F1_BetaInc} and Definition \ref{def:betaIncReg_fun}. Form $(\mathrm{iii})$ is then found using \cite[Eqs.~$06.19.03.0003.01$]{wolfram_functions} on form $(\mathrm{ii})$.
\end{proof}

\begin{lemma}[Recurrence relation]
\label{lem:tilde_gnw_recurrence_relation}
Both $g_{n,\omega}$ and $\tilde g_{n,\omega}$ satisfy
\[
g_{n+1,\omega+1}=g_{n+1,\omega}+(n+1)g_{n,\omega}.
\]
\end{lemma}

\begin{proof}
Denoting $f(z)={_1F_0}(1;- ;z)_\nu$ we have
\[
\begin{aligned}
g_{n+1,\omega+1}%
&=z^{n-\omega}\partial_z^{n+1}z^{\omega+1}f\\
&=z^{n-\omega}\sum_{k=0}^{n+1}\binom{n+1}{k}(\partial_z^k z)(\partial_z^{n+1-k}z^\omega f)\\
&=z^{n+1-\omega}\partial_z^{n+1}z^\omega f+(n+1)z^{n-\omega}\partial_z^nz^\omega f,
\end{aligned}
\]
which is the desired result. Dividing sides of the last equality by $f$ gives the corresponding result for $\tilde g_{n,\omega}$.
\end{proof}

\begin{lemma}[Reflection formula]
\label{lem:tilde_gnw_reflection_formula}
\[
\tilde g_{n,\omega}(z,\nu)=(-1)^n\tilde g_{n,n-\omega}(1/z,-\nu)
\]
\end{lemma}

\begin{proof}
Applying the linear transformation in \cite[Eq.~$15.8.1(\mathrm i)$]{nist_2010} to each hypergeometric term in Corollary \ref{cor:gTilde_hyper_form} and defining $\omega^\prime=n-\omega$ and $\nu^\prime=-\nu$ yields
\[
\begin{aligned}
\tilde g_{n,\omega}(z,\nu)%
&=\frac{n!(n-\omega^\prime-\nu^\prime)^{(n+1)}}{-\nu^\prime}\frac{\mathbf F(1,\omega^\prime+\nu^\prime+1;n+2;1-1/z)}{\mathbf F(1,\nu^\prime+1,2,1-1/z)}\\
&=-\frac{n!(-1)^{n+1}(\omega^\prime+\nu^\prime)^{(n+1)}}{\nu^\prime}\frac{\mathbf F(1,\omega^\prime+\nu^\prime+1;n+2;1-1/z)}{\mathbf F(1,\nu^\prime+1,2,1-1/z)}\\
&=(-1)^n\tilde g_{n,\omega^\prime}(1/z,\nu^\prime),
\end{aligned}
\]
which completes the proof.
\end{proof}

\begin{lemma}
\label{lem:tilde_g_boundary_values}
\[
\begin{aligned}
\lim_{z\to 0}\tilde g_{n,\omega}(z,\nu) &=(\omega+\nu\mathds 1_{\nu<0})^{(n)}\\
\lim_{z\to\infty}\tilde g_{n,\omega}(z,\nu) &=(\omega+\nu\mathds 1_{\nu>0}-1)^{(n)}.
\end{aligned}
\]
\end{lemma}

\begin{proof}
Assume $\nu>0$ and consider the limit approaching zero. As $z\to 0$, $\operatorname I_{1-z}(\alpha,\beta)\sim 1+\mathcal O(z^\beta)$ \cite[Eq.~$06.21.06.0044.01$]{wolfram_functions}. Applying this result to Corollary \ref{cor:gTilde_hyper_form} form $(\mathrm{ii})$ then gives
\begin{equation}
\label{eq:gTilde_asym_form}
\tilde g_{n,\omega}(z,\nu)\sim n!\, (1+\mathcal O(z))(1+\mathcal O(z^\nu))\\
\left(z^{n-\omega}+\mathcal O(z^\nu)\right),
\end{equation}
which upon passing to the limit yields
\[
\lim_{z\to 0}\tilde g_{n,\omega}(z,\nu)=n!\,\mathds 1_{n=\omega}=(\omega)^{(n)}.
\]
Now considering the limit $z\to\infty$ we use the transformation \cite[Eq.~$15.8.1(\mathrm i)$]{nist_2010} to write
\[
\tilde g_{n,\omega}(z,\nu)=\frac{n!(\omega+\nu)^{(n+1)}}{\nu\Gamma(n+2)}\frac{F(1,n-\omega+1-\nu;n+2;1-1/z)}{F(1,1-\nu,2,1-1/z)}.
\]
According to Relation \ref{rel:2F1_zeq1}, the limits of the numerator and denominator are finite if $\omega+\nu>0$ and $\nu>0$, respectively; thus,
\[
\lim_{z\to \infty}\tilde g_{n,\omega}(z,\nu)=\frac{(\omega+\nu)^{(n+1)}}{\nu}\frac{\Gamma(\omega+\nu)\Gamma(\nu+1)}{\Gamma(\omega+\nu+1)\Gamma(\nu)}=(\omega+\nu-1)^{(n)}.
\]
Still assuming $\nu>0$ we may now use the reflection formula in Lemma \ref{lem:tilde_gnw_reflection_formula} to deduce
\[
\begin{aligned}
\lim_{z\to 0}\tilde g_{n,\omega}(z,-\nu) &=(-1)^n\lim_{z\to\infty}\tilde g_{n,n-\omega}(z,\nu)\\
\lim_{z\to\infty}\tilde g_{n,\omega}(z,-\nu) &=(-1)^n\lim_{z\to 0}\tilde g_{n,n-\omega}(z,\nu).
\end{aligned}
\]
But notice that the limit on the r.h.s.~of each equality can be found from our previous results by substituting $\omega\mapsto n-\omega$ and then scaling by $(-1)^n$. It then follows after some simplification that
\[
\begin{aligned}
\lim_{z\to 0}\tilde g_{n,\omega}(z,-\nu) &=(\omega+\nu)^{(n)}\\
\lim_{z\to\infty}\tilde g_{n,\omega}(z,-\nu) &=(\omega-1)^{(n)}.
\end{aligned}
\]
Combining all four result then produces the desired solution.
\end{proof}

The following definition and relation will be used extensively in the rest of this section.

\begin{definition}[Gauss-Hypergeometric distribution]
\label{def:GH_distribution}
For $\alpha,\beta>0$, $\gamma\in\Bbb R$, and $\xi>-1$, the Gauss Hypergeometric random variable $X\sim\mathcal{GH}(\alpha,\beta,\gamma,\xi)$ admits a distribution and density function on $x\in(0,1)$ of the form
\[
\begin{array}{*3{>{\displaystyle}l}}
F_X(x;\alpha,\beta,\gamma,\xi) &=\frac{x^\alpha F_1(\alpha;1-\beta,\gamma;\alpha+1;x,-\xi x)}{\alpha \operatorname B(\alpha,\beta)F(\alpha,\gamma;\alpha+\beta;-\xi)},\\[3ex]
f_X(x;\alpha,\beta,\gamma,\xi) &=\frac{x^{\alpha-1}(1-x)^{\beta-1}(1+\xi x)^{-\gamma}}{\operatorname B(\alpha,\beta)F(\alpha,\gamma;\alpha+\beta;-\xi)},
\end{array}
\]
where $F_1(a;b,b^\prime;c;s,z)$ is Appell's $1$st hypergeometric function of two variables (see Definition \ref{def:appell_F1}).
\end{definition}

\begin{relation}
\label{rel:Hyper2F1_integral}
For $|\operatorname{ph}(1-z)|<\pi$ and $\Re c>\Re d>0$
\[
\mathbf F(a,b;c;z)=\int_0^1F(a,b;d;zt)\frac{t^{d-1}(1-t)^{c-d-1}}{\Gamma(d)\Gamma(c-d)}\,\mathrm dt.
\]
In particular if $\Re c>\Re b>0$ then substituting $d=b$ gives
\[
\mathbf F(a,b;c;z)=\int_0^1(1-zt)^{-a}\frac{t^{b-1}(1-t)^{c-b-1}}{\Gamma(b)\Gamma(c-b)}\,\mathrm dt.
\]
\end{relation}

An important result needed for obtaining a better understanding of $\mathscr \acv\mathscr T_\nu$ are the zeros of $\tilde g_{n,\omega}$. The following lemma will serve as a preliminary result needed to find these zeros.
\begin{lemma}
\label{lem:EVX_equals_one}
Let $\nu\in\Bbb R$, $z\in\Bbb R_0^+$, and $X\sim\mathcal{GH}(1,1,1+\nu,z-1)$. Then $\ev X=1\iff\nu\in\Bbb R_0^+\land z=0$.
\end{lemma}

\begin{proof}
With the help of Definition \ref{def:GH_distribution} we write the density of $X$ as
\[
f_X(x|z)=\frac{(1-(1-z) x)^{-1-\nu}}{F(1,1+\nu;2;1-z)}\mathds 1_{x\in[0,1]}.
\]
Consider the family of density functions $f_z=\{f_X(x|z):z\in\Bbb R_0^+\}$ and let $\delta>0$. We have for the likelihood ratio
\[
L(x)\coloneqq\frac{f_{z+\delta}}{f_z}(x)=C\left(\frac{1-(1-z)x}{1-(1-z-\delta)x}\right)^{1+\nu},
\]
where $C$ is a positive constant. Differentiating $L$ gives
\[
\partial_xL(x)=-\delta C(1+\nu)\frac{(1-(1-z)x)^\nu}{(1-(1-z-\delta)x)^{2+\nu}}.
\]
For fixed $\nu$, $\operatorname{sign}(\partial_xL(x))$ is constant on the interior of $\operatorname{supp}(X)=[0,1]$. Consequently, the family of density functions $f_z$ admits a monotone likelihood ratio and
\begin{equation}
\label{eq:monotonicity_of_EX}
\operatorname{sign}(\partial_z\ev X)=\operatorname{sign}(\partial_xL(x))=%
\begin{cases}
-1, &\nu>-1\\
0, &\nu=-1\\
1, &\nu<-1.
\end{cases}
\end{equation}
We now want to show that $z>0\implies \ev X<1$. Evaluating the expected value yields
\[
\ev X=\frac{1-(1-z)\nu-z^\nu}{(1-\nu)(1-z)(1-z^\nu)}.
\]
Letting $\ev X(\nu)$ denote the expected value as a function of $\nu$ and $\ev X_a(\nu)\coloneqq\lim_{z\to a}\ev X(\nu)$ we find $\ev X_0(\nu)=\frac{1}{1-\nu\mathds 1_{\nu<0}}$, $\ev X_\infty(\nu)=\frac{\nu}{\nu-1}\mathds 1_{\nu<0}$, and
\[
\max\{\ev X_0,\ev X_\infty\}(\nu)=\frac{\mathds 1_{\nu\geq -1}-\nu\mathds 1_{\nu<-1}}{1-\nu\mathds 1_{\nu<0}}.
\]
Utilizing the monotonicity of $\ev X$ in $(\ref{eq:monotonicity_of_EX})$ we have $\ev X(-1)=\max\{\ev X_0,\ev X_\infty\}(-1)=1/2$ and $\ev X(\nu)<\max\{\ev X_0,\ev X_\infty\}(\nu)\leq 1$ for all $\nu\in\Bbb R\setminus\{-1\}$; thus, $z>0\implies \ev X<1$. Furthermore, when $z=0$
\[
\ev X=\ev X_0(\nu)%
\begin{cases}
<1, &\nu<0\\
=1, &\nu\geq 0
\end{cases}
\]
hence, $\ev X=1\iff \nu\in\Bbb R_0^+\land z=0$. The proof is now complete.
\end{proof}

\begin{figure}[htb]
\centering
\includegraphics[scale=1]{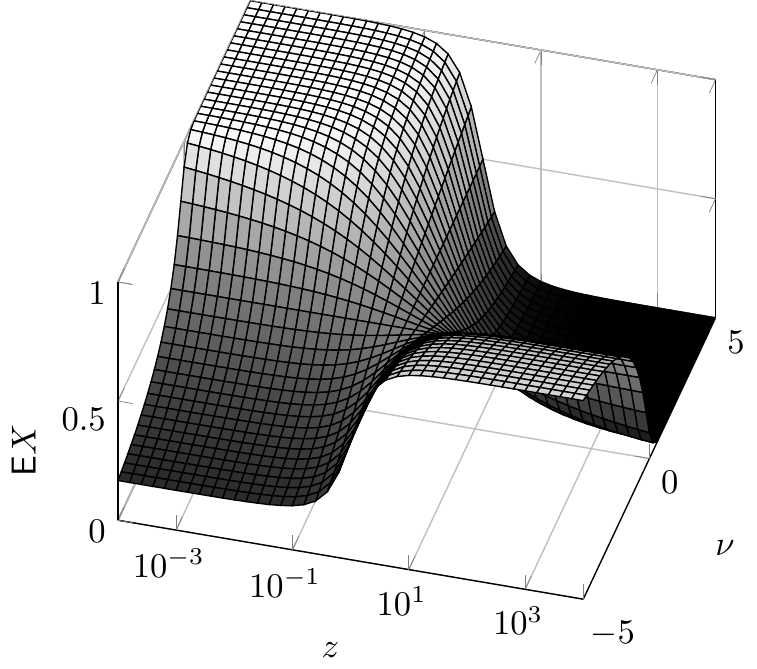}
\caption{Surface plot of $\ev X$ from Lemma \ref{lem:EVX_equals_one}.}
\label{fig:EVX_plot}
\end{figure}

Many of the proofs that follow will utilize a probabilistic interpretation of $\tilde g_{n,\omega}$ to derive some of its most useful properties. The following proposition presents a representation of $\tilde g_{n,\omega}$ in terms of an expected value, which will prove to be immensely helpful. 
\begin{proposition}[$\tilde g_{n,\omega}$ expected value representation]
\label{prop:gTilde_expected_val_form}
Let $C_{n,\omega}(\nu)=\nu^{-1}(\omega+\nu)^{(n+1)}$, $h_{n,\omega}(x,z)=(1-x)^n(1-(1-z)x)^{-\omega}$, and $X\sim\mathcal{GH}(1,1,1+\nu,z-1)$. Then,
\[
\tilde g_{n,\omega}(z,\nu)=C_{n,\omega}(\nu)\ev h_{n,\omega}(X,z).
\]
\end{proposition}

\begin{lemma}[Zeros of $\tilde g_{n,\omega}$]
\label{lem:tilde_gnw_zeros}
For all permissible values of the parameters $n$, $\omega$, $z$, and $\nu$ in Definition \ref{def:g_nw_function}
\[
\tilde g_{n,\omega}(z,\nu)=0\iff (n-\omega-\nu\in\Bbb N_0\land \nu\neq 0)\lor (n\in\Bbb N\land n>\omega\land\nu\in\Bbb R_0^+\land z=0).
\]
\end{lemma}

\begin{proof}
We begin with the expression for $\tilde g_{n,\omega}$ in Proposition \ref{prop:gTilde_expected_val_form}, namely,
\[
\tilde g_{n,\omega}(z,\nu)=C_{n,\omega}(\nu)\, \ev h_{n,\omega}(X,z),
\]
where $C_{n,\omega}(\nu)=\nu^{-1}(\omega+\nu)^{(n+1)}$, $h_{n,\omega}(x,z)=(1-x)^n(1-(1-z)x)^{-\omega}$ and $X\sim\mathcal{GH}(1,1,\nu+1,z-1)$. If both terms are bounded then
\[
\tilde g_{n,\omega}(z,\nu)=0\iff C_{n,\omega}(\nu)=0\lor \ev h_{n,\omega}(X,z)=0.
\]
As such, our approach will be to show that $C_{n,\omega}$ and $\ev h_{n,\omega}$ are bounded and then determine the location of their zeros. Combining the zeros of each term will subsequently give the zeros of $\tilde g_{n,\omega}$.

From Relation \ref{rel:falling_factorial_prodRep} we can easily show that $(\omega+\nu)^{(n+1)}$ is a polynomial in $(\omega+\nu)$ of degree no more than $n+1$. Therefore, the only point of concern for verifying that $C_{n,\omega}(\nu)$ is bounded is at $\nu=0$. With the help of \cite[Eq.~$06.05.06.0026.01$]{wolfram_functions} we have as $\nu\to 0$
\[
C_{n,\omega}(\nu)\sim(-1)^{n-\omega}\Gamma(\omega+\nu+1)(n-\omega)!(1+\mathcal O(\nu)),
\]
which upon passing to the limit yields,
\[
\lim_{\nu\to 0}C_{n,\omega}(\nu)=(-1)^{n-\omega}\omega!(n-\omega)!.
\]
From this result it follows that $|C_{n,\omega}(\nu)|<\infty$ for all permissable values of the parameters $n$, $\omega$, and $\nu$. Now, the zeros of the factorial power term can be easily identified upon writing
\[
(\omega+\nu)^{(n+1)}=\frac{\Gamma(\omega+\nu+1)}{\Gamma(\omega+\nu-n)}=0\iff n-\omega-\nu\in\Bbb N_0.
\]
Noting that $C_{n,\omega}(0)\neq 0$ then gives the result
\[
C_{n,\omega}(\nu)=0\iff n-\omega-\nu\in\Bbb N_0\land \nu\neq 0.
\]

Next we turn our attention to the quantity $\ev h_{n,\omega}(X,z)$. As a function of $x$, $h_{n,\omega}(x,z)=(1-x)^n(1-(1-z)x)^{-\omega}$ is nonnegative on $x\in[0,1]$ for all permissable parameters $n$, $\omega$, and $z$. Since, $\operatorname{supp}(X)=[0,1]$ it follows that $0\leq\ev h_{n,\omega}(X,z)$. Additionally, note that
\[
\partial_xh_{n,\omega}(x,z)=-\underbrace{\frac{(1-x)^{n-1}}{(1-(1-z)x)^{\omega+1}}}_{\geq0}\underbrace{((n-\omega)(1-x)+(n-\omega)zx+\omega z)}_{\geq 0}\leq 0.
\]
By continuity of $h_{n,\omega}$ we have $h_{n,\omega}(x,z)\leq h_{n,\omega}(0,z)=1\implies \ev h_{n,\omega}(X,z)\leq \ev h_{n,\omega}(0,z)=1$ with the last inequality again being a consequence of $\operatorname{supp}(X)=[0,1]$. Bringing both results together we have shown for all permissible parameters that $0\leq\mathsf Eh_{n,\omega}(X,z)\leq 1$. With boundedness establish we turn to locating the zeros of $\mathsf Eh_{n,\omega}(X,z)$. Observe that both $h_{n,\omega}(x,z)$ and the support of $X$ are both nonnegative; thus, the only way for $\mathsf Eh_{n,\omega}(X,z)=0$ is if $X\sim\delta_{x_0}$ where $h_{n,\omega}(x_0,z)=0$ and $\delta_\mu$ is the degenerate density function centered at $\mu$. Combing the facts $0\leq h_{n,\omega}(x,z)$ and $0\geq \partial_xh_{n,\omega}(x,z)$ shows that all zeros of $h_{n,\omega}$ must reside at $x=1$. Upon further inspection we determine
\begin{equation}
\label{eq:zeros_of_hnw}
h_{n,\omega}(1,z)=0\iff n\in\Bbb N\land(z>0\lor(z=0\land n>\omega)).
\end{equation}
Next, we need to determine the parameters for which $X\sim\delta_1$. Since $\operatorname{supp}(X)=[0,1]$ this is equivalent to finding the parameters for which $\ev X=1$, which are given by Lemma \ref{lem:EVX_equals_one}. Combining the conditions given in Lemma \ref{lem:EVX_equals_one} with $(\ref{eq:zeros_of_hnw})$ subsequently shows
\[
\ev h_{n,\omega}(X,z)=0\iff n\in\Bbb N\land n>\omega\land\nu\in\Bbb R_0^+\land z=0.
\]
Combining these zeros with those of $C_{n,\omega}$ then yields the desired result.
\end{proof}

The next two results establish properties about the expected value $\ev h_{n,\omega}(X,z)$. In particular, Lemma \ref{lem:hyper_ratio_monotonicity} establishes the conditions under which this expected value is a monotone function of the parameter $z$. This will then be used in Theorem \ref{thm:gnw_is_monotone} to establish analogous conditions for when $|\tilde g_{n,\omega}|$ is monotone in $z$.
\begin{lemma}
\label{lem:Ehnw_monotonicity_in_n_and_w}
For fized $z$ and $\nu$, $\ev h_{n,\omega}(X,z)$ is decreasing in $n$, increasing in $\omega$ if $z\in[0,1]$, and decreasing in $\omega$ if $z\in[1,\infty)$. Furthermore, $\ev h_{n+t,\omega+t}(X,z)$ is decreasing in $t$.
\end{lemma}

\begin{proof}
Recall from Lemma \ref{lem:tilde_gnw_zeros} that $h_{n,\omega}(x,z)$ is nonnegative for all permissible values of $z$, $\nu$, $n$, and $\omega$ on $\operatorname{supp}(X)=[0,1]$. Furthermore, the following results hold everywhere on $\operatorname{supp}(X)$: $\partial_t h_{t,\omega}\leq 0$, $\partial_t h_{n,t}\geq 0$ if $z\in[0,1]$, $\partial_t h_{n,t}\leq 0$ if $z\in[1,\infty)$, and $\partial_t h_{n+t,\omega+t}\leq 0$. Since, $h_{n,\omega}$ and $\operatorname{supp}(X)$ are nonnegative, the sign of these derivatives will be equal to the sign of their corresponding expected values. The proof is now complete.
\end{proof}

\begin{lemma}
\label{lem:hyper_ratio_monotonicity}
Let $\mathcal N$ be given by Definition \ref{def:g_nw_function}. Then, $\nu>1\implies \ev h_{n,\omega}(X,z)$ is strictly increasing in $z$ on $z\in\Bbb R^+$ for all $(n,\omega)\in\mathcal N\setminus\{(0,0)\}$.
\end{lemma}

\begin{proof}
We begin with the explicit form
\begin{equation}
\label{eq:Ehnw_closed_form}
\ev h_{n,\omega}(X,z)=n! \frac{\mathbf F(1,\omega+\nu+1;n+2;1-z)}{\mathbf F(1,\nu+1;2;1-z)}.
\end{equation}
According to Lemma \ref{lem:tilde_gnw_zeros}, $z\in\Bbb R^+\implies 0<\ev h_{n,\omega}(X,z)$ for all $n$, $\omega$, and $\nu$ permissible by Definition \ref{def:g_nw_function}. As such, we will proceed to establish $\ev h_{n,\omega}(X,z)$ is strictly increasing in $z$ by showing it has a strictly positive logarithmic derivative.  Using Definition \ref{def:GH_distribution} and Relation \ref{rel:Hyper2F1_integral} gives
\[
\partial_z\log\ev h_{n,\omega}(X,z)=r_{0,0}(z,\nu)-r_{n,\omega}(z,\nu),
\]
where
\[
r_{n,\omega}(z,\nu)=\gamma\frac{\mathbf F(2,\gamma+1;\beta+2;1-z)}{\mathbf F(1,\gamma;\beta+1;1-z)}=\frac{\gamma}{z}\mathsf EW,
\]
$\beta=n+1$, $\gamma=\omega+\nu+1$, $W=zX(1-(1-z)X)^{-1}$, and $X\sim\mathcal{GH}(1,\beta,\gamma,z-1)$. Consequently, $\partial_z\log\ev h_{n,\omega}(X,z)$ is strictly positive for all $(n,\omega)\in\mathcal N\setminus\{(0,0)\}$ if $r_{n,\omega}<r_{0,0}$. We will prove this by successively maximizing $r_{n,\omega}$ in $n$ and $\omega$, respectively.


\begin{enumerate}
\item[$(1)$] $r_{n,\omega}(z,\nu)\leq r_{\omega,\omega}(z,\nu)\ \ \forall (n,\omega)\in\mathcal N$.

\begin{subproof}
The random variable $W$ is a monotone increasing transformation of $X$ for all $z>0$; thus,
\[
F_W(w)=\pr(X\leq (z(w^{-1}-1)+1)^{-1}).
\]
After a bit of work we find
\[
F_W(w)=1-\frac{\operatorname B_{(1-z)(1-w)}(\beta,\gamma-\beta)}{\operatorname B_{1-z}(\beta,\gamma-\beta)},
\]
where $\operatorname B_s(\alpha,\beta)$ is the incomplete beta function of definition \ref{def:betaInc_fun} and $\operatorname{supp}(W)=[0,1]$. Evaluating $\partial_wF_W$ then provides the density
\[
f_W(w)=\frac{(1-w)^{\beta-1}(1-(1-z)(1-w))^{\gamma-\beta-1}}{(1-z)^{-\beta}\operatorname B_{1-z}(\beta,\gamma-\beta)}\mathds 1_{w\in[0,1]}.
\]
Consider the family of density functions $f_\beta=\{f_W(w|\beta):\beta\geq 1\}$ and let $\delta>0$. We have for the likelihood ratio
\[
L(w)\coloneqq\frac{f_{\beta+\delta}}{f_\beta}(w)=C\left(\frac{1-w}{1-(1-z)(1-w)}\right)^\delta,
\]
where $C$ is a positive constant. Differentiating $L$ yields
\[
\partial_wL(w)=-\delta C\frac{(1-w)^{\delta-1}}{(1-(1-z)(1-w))^{\delta+1}}.
\]
Observe that $\operatorname{sign}(\partial_wL(w))=-1$ on the interior of $\operatorname{supp}(W)$. Consequently, the family of density functions $f_\beta$ admits a monotone likelihood ratio and
\[
\operatorname{sign}(\partial_\beta\ev W)=\operatorname{sign}(\partial_xL(x))=-1.
\]
But now recall that $\beta=n+1$ which shows that $\partial_n\ev W<0$; hence,
\[
\frac{\gamma}{z}\ev W\leq\frac{\gamma}{z}\ev W|_{n=\omega}\implies r_{n,\omega}(z,\nu)\leq r_{\omega,\omega}(z,\nu).
\]
\end{subproof}


\item[$(2)$] $\nu>1\implies r_{\omega,\omega}(z,\nu)<r_{0,0}(z,\nu)\ \ \forall\omega>0$

\begin{subproof}
We will establish this claim by showing that $r_{\omega,\omega}(z,\nu)$ is strictly decreasing in $\omega$. To aid in the following calculations we introduce the operators
\[
\begin{aligned}
\mathcal A_1^k F(a_1,a_2;a_3;s) &=F(a_1+k,a_2;a_3;s),\\
\mathcal A_2^k F(a_1,a_2;a_3;s) &=F(a_1,a_2+k;a_3;s),\\
\mathcal A_3^k F(a_1,a_2;a_3;s) &=F(a_1,a_2;a_3+k;s),
\end{aligned}
\]
for which $\mathcal A_i^0=\mathcal I$ is the identity. In terms of these operators we have
\[
r_{\omega,\omega}(z,\nu)%
=\frac{\frac{\omega+\nu+1}{\omega+2}\mathcal A_1\mathcal A_2\mathcal A_3F(1,\omega+\nu+1;\omega+2;1-z)}{F(1,\omega+\nu+1;\omega+2;1-z)}.
\]
Now combining the identities \cite[Eqs.~$10$, $13$]{ibrahim_2012}
\begin{gather*}
\mathcal A_1=\mathcal I+\frac{a_2}{a_3}s\mathcal A_1\mathcal A_2\mathcal A_3,\\
\mathcal A_1^{-1}=\frac{a_1(s-1)}{a_1-a_3}\mathcal A_1+\frac{2a_1+(a_2-a_1)s-a_3}{a_1-a_3}\mathcal I,
\end{gather*}
we deduce for $a_1=1$, $a_2=\omega+\nu+1$, $a_3=\omega+2$, and $s=1-z$
\[
\frac{\omega+\nu+1}{\omega+2}\mathcal A_1\mathcal A_2\mathcal A_3=\frac{1}{1-z}\left(\frac{\omega+1}{z}\mathcal A_1^{-1}+\frac{1}{z}((\omega+\nu)(1-z)-\omega)\mathcal I\right),
\]
which permits us to write
\begin{multline}
\label{eq:Psi_ww_form1}
r_{\omega,\omega}(z,\nu)%
=\frac{1}{1-z}\biggl(\frac{\omega+1}{z F(1,\omega+\nu+1;\omega+2;1-z)}\\
+\frac{1}{z}\left((\omega+\nu)(1-z)-\omega\right)\biggr).
\end{multline}
Now using Relation \ref{rel:2F1_BetaInc} we obtain the form
\[
F(1,\omega+\nu+1;\omega+2;1-z)=(\omega+1)(1-z)^{-(\omega+1)}z^{-\nu}\operatorname{B}_{1-z}(\omega+1,\nu).
\]
Substituting this expression into $(\ref{eq:Psi_ww_form1})$ and applying the differential formula \cite[Eq.~$06.19.20.0003.01$]{wolfram_functions}
\[
\partial_\alpha\operatorname B_s(\alpha,\beta)=\operatorname B_s(\alpha,\beta)\log s-\frac{s^\alpha}{\alpha^2}{_3F_2}(1-\beta,\alpha,\alpha;\alpha+1,\alpha+1;s)
\]
gives
\begin{equation}
\label{eq:dwPsi_ww_form1}
\partial_\omega r_{\omega,\omega}(z,\nu)=\frac{1}{1-z}\left(\frac{{_3F_2}(1-\nu,\omega+1,\omega+1;\omega+2,\omega+2;1-z)}{z^{\nu+1}F(1,\omega+\nu+1;\omega+2;1-z)^2}-1\right).
\end{equation}
But in accordance with the integral representation \cite[Eq.~$16.5.2$]{nist_2010}
\[
\pFq{3}{2}{a_0,a_1,a_2}{b_0,b_1}{z}=\int_0^1\frac{t^{a_0-1}(1-t)^{b_0-a_0-1}}{\operatorname{B}(a_0,b_0-a_0)}\pFq{}{}{a_1,a_2}{b_1}{zt}\,\mathrm dt,
\]
which holds for $\Re b_0>\Re a_0>0$, the result in $(\ref{eq:dwPsi_ww_form1})$ equivalent to the expected value
\begin{equation}
\label{eq:dw_r_ww_expectedVal_form}
\partial_\omega r_{\omega,\omega}(z,\nu)=\ev\mathcal H_\omega(X,z,\nu),
\end{equation}
where
\[
\mathcal H_\omega(x,z,\nu)=\left(\frac{h_\omega(x,z,\nu)/h_\omega(1,z,\nu)-1}{1-z}\right),
\]
\[
h_\omega(x,z,\nu)=(1-(1-z)x)F(1,\omega+\nu+1;\omega+2;(1-z)x),
\]
and $X\sim\mathcal{GH}(\omega+1,1,1-\nu,z-1)$. To establish the conditions for when $(\ref{eq:dw_r_ww_expectedVal_form})$ is negative we use \cite[Eq.~$15.5.7$]{nist_2010} to write
\[
\partial_x\mathcal H_\omega(x,z,\nu)=(\nu-1)\frac{\mathbf F(2,\omega+\nu+1;\omega+3;(1-z) x)}{z\mathbf F(1,\omega+\nu+1;\omega+2;1-z)}.
\]
For $z>0$, the ratio of hypergeometric terms is positive resulting in $\operatorname{sign}(\partial_x\mathcal H_\omega(x,z,\nu))=\operatorname{sign}(\nu-1)$. Therefore, if $\nu>1$ then $\mathcal H_\omega(x,z,\nu)<\mathcal H_\omega(1,z,\nu)=0$ on the interior of $\operatorname{supp}(X)=[0,1]$. Furthermore, by an argument very similar to that of Lemma \ref{lem:EVX_equals_one}, we can show that $z>0\implies \ev X<1$; thus, guaranteeing the density of $X$ is not degenerate at $x=1$. Combining these facts leads to the conclusion that $\partial_\omega r_{\omega,\omega}(z,\nu)<0$ and so\footnote{If the density of $X$ was degenerate at the point $x=1$ then we would have $\ev\mathcal H_\omega(X,z,\nu)=\int\mathcal H_\omega(x,z,\nu)\delta(x-1)\,\mathrm dx=\mathcal H_\omega(1,z,\nu)=0.$}
\[
\nu>1\implies r_{\omega,\omega}(z,\nu)<r_{0,0}(z,\nu)\ \ \forall\omega>0,
\]
which is the desired result.
\end{subproof}

\end{enumerate}

Combining claims $(1)$ and $(2)$ we have shown $\nu>1\implies r_{n,\omega}(z,\nu)<r_{0,0}(z,\nu)$ for all $(n,\omega)\in\mathcal N\setminus\{(0,0)\}$ which completes the proof.
\end{proof}

We are finally able to now establish the conditions for which $|\tilde g_{n,\omega}|$ is a monotone function of $z$. This result will serve as the foundation for establishing exact confidence intervals of $\acv\mathscr T_\nu$.
\begin{theorem}[Monotony of $|\tilde g_{n,\omega}|$ in $z$]
\label{thm:gnw_is_monotone}
If $\nu>1$ then
\begin{enumerate}
\item[$(\mathrm i)$]  $|\tilde g_{n,\omega}(z,\nu)|$ is increasing in $z$ for all $(n,\omega)\in\mathcal N$,
\item[$(\mathrm{ii})$] $\exists(n,\omega)\in\mathcal N$ such that $|\tilde g_{n,\omega}(z,\nu)|$ is strictly increasing in $z$.
\end{enumerate}
If instead $\nu<-1$ replace increasing with decreasing.
\end{theorem}

\begin{proof}
First consider the special case $(n,\omega)=(0,0)\implies|\tilde g_{n,\omega}(z,\nu)|=1$ which is trivially increasing. Next assume $\nu>1$ and consider the general case $(n,\omega)\in\mathcal N\setminus\{(0,0)\}$. Using the notation of Lemma \ref{lem:tilde_gnw_zeros} we write
\begin{equation}
\label{eq:absVal_g_tilde}
\partial_z|\tilde g_{n,\omega}(z,\nu)|=|C_{n,\omega}(\nu)|\partial_z\ev h_{n,\omega}(X,z).
\end{equation}
From Lemma \ref{lem:hyper_ratio_monotonicity} we know $\nu>1\implies \partial_z\ev h_{n,\omega}(X,z)>0$ and since $|C_{n,\omega}(\nu)|\geq 0$ we can conclude that $\partial_z|\tilde g_{n,\omega}(z,\nu)|\geq 0$ for all $(n,\omega)\in\mathcal N\setminus\{(0,0)\}$; thus, criteria $(\mathrm i)$ is satisfied. To satisfy criteria $(\mathrm{ii})$ all we need to show is that there is some $(n,\omega)$ for which $C_{n,\omega}(\nu)\neq 0$. To accomplish this, consider the image of $f(n,\omega)=n-\omega-\nu$ under $\mathcal N\setminus\{(0,0)\}$:
\[
f[\mathcal N\setminus\{(0,0)\}]=\{-\nu,1-\nu,2-\nu,\dots\}.
\]
Since $\nu>1$ it follows that $f[\mathcal N\setminus\{(0,0)\}]$ always contains at least one element not in $\Bbb N_0$ which according to Lemma \ref{lem:tilde_gnw_zeros} implies that $C_{n,\omega}(\nu)\neq0$. Hence, $\nu>1\implies\exists(n,\omega)\in\mathcal N\setminus\{(0,0)\}:C_{n,\omega}(\nu)\neq0$ and the proof for $\nu>1$ is complete.

To obtain the proof for $\nu<-1$ observe that if $|\tilde g_{n,\omega}(z,\nu)|$ is increasing (strictly increasing) in $z$ then $|\tilde g_{n,\omega}(1/z,\nu)|$ must be decreasing (strictly decreasing) in $z$. With this observation assume $\nu>1$ and use the reflection formula in Lemma \ref{lem:tilde_gnw_reflection_formula} to write
\begin{equation}
\label{eq:tilde_g_trans}
|\tilde g_{n,\omega}(1/z,\nu)|=|\tilde g_{n,\omega^\prime}(z,-\nu)|,
\end{equation}
where $\omega^\prime=n-\omega$. Since the l.h.s.~side of $(\ref{eq:tilde_g_trans})$ is  decreasing in $z$ and $(n,\omega)\in\mathcal N\implies(n,\omega^\prime)\in\mathcal N$ the proof for $\nu<-1$ immediately follows.
\end{proof}


\subsection{Variance}
\label{subsec:variance}

We are now at last ready to begin deriving expressions for the variance.  To accomplish this task we will derive the second moment for the integer-valued estimator $\mathscr T_n$ and then rely on the uniqueness of the fractional finite sum to obtain the desired result for real-valued $\nu$.

\begin{lemma}
\label{lem:T_est_variance_discrete}
For $n\in\Bbb N_0:0\leq 2n<\alpha_1$,
\[
\ev\mathscr T_n^2=%
\frac{1}{\kappa_1^2}\sum_{k,\ell=0}^{n-1}\pFq{}{}{k+1,\ell+1}{\alpha_1}{1}\pFq{}{}{-k,-\ell}{\alpha_2}{1}\zeta^{k+\ell}.
\]
\end{lemma}

\begin{proof}
We begin by squaring $\mathscr T_n$ to find
\[
\mathscr T_n^2=\sum_{k,\ell=0}^{n-1}\frac{(\alpha_1Y_1)^{-k-\ell-2}(\alpha_2Y_2)^{k+\ell}}{(\alpha_1)_{-k-1}(\alpha_1)_{-\ell-1}(\alpha_2)_k(\alpha_2)_\ell}.
\]
Now using the result of Lemma \ref{lem:gamma_moments} we evaluate the expected value term-wise yielding
\[
\ev\mathscr T_n^2=\frac{1}{\kappa_1^2}\sum_{k,\ell=0}^{n-1}\frac{\Gamma(\alpha_1)\Gamma(\alpha_1-k-\ell-2)\Gamma(\alpha_2)\Gamma(\alpha_2+k+\ell)}{\Gamma(\alpha_1-k-1)\Gamma(\alpha_1-\ell-1)\Gamma(\alpha_2+k)\Gamma(\alpha_2+\ell)}\zeta^{k+\ell},
\]
which requires $2n<\alpha_1$ to guarantee each term is finite. Relation \ref{rel:2F1_zeq1} provides the necessary result to write the summand as the product of hypergeometric functions of unity argument; hence, the desired result is obtained.
\end{proof}

As will be seen in Theorem \ref{thm:T_est_variance_series}, the second moment of the generalized estimator $\mathscr T_\nu$ is rather complicated. To determine the convergence criteria for $\ev\mathscr T_\nu^2$ we will derive an integral representation $\ev\mathscr T_\nu^2=\int_{\mathcal S}f\,\mathrm dA$ on the unit square $\mathcal S=[0,1]^2$ and bound the integrand $f$ with a simpler function, $\Upsilon$, that still captures essential characteristics of $f$ such as its behavior near singular points. Since these singular points ultimately dictate the convergence of the integral representation of $\ev\mathscr T_\nu^2$, studying the convergence of $\int_{\mathcal S}\Upsilon\,\mathrm dA$ will subsequently reveal the parameters which which $\ev\mathscr T_\nu^2<\infty$. However, even the integral of $\Upsilon$ is complicated and in most cases cannot be evaluated in closed-form. As such, Lemmas \ref{lem:double_log_power_integral_1_bound} and \ref{lem:double_log_power_integral_2_bound} study the convergence properties of various double integrals that will used to aid in determining the convergence of $\int_{\mathcal S}\Upsilon\,\mathrm dA$.

\begin{relation}[Integral representation of harmonic numbers]
\label{rel:harmonic_number_integral_rep}
\[
H_z=\int_0^1\frac{1-t^{z-1}}{1-t}\,\mathrm dt,\quad \Re z>-1.
\]
\end{relation}

\begin{lemma}
\label{lem:double_log_power_integral_1_bound}
Let $n\in\Bbb N_0$, $f(x,y)=1-(1-x)(1-y)$, and
\[
I_n=\int_0^1\int_0^1f^{s-2}(x,y)(\log\circ f)^n(x,y)\,\mathrm dx\mathrm dy.
\]
Then, $I_n$ converges for all $s>0$.
\end{lemma}

\begin{proof}
We will first evaluate $I_0$. Noting that $|I_n|=(-1)^nI_n$ we perform the change of variables $(t,v)=((1-x)(1-y),x)$ and integrate over $v$ yielding
\[
|I_n|=\int_0^1(-1)^n\log^n(1-t)(-\log t) (1-t)^{s-2}\,\mathrm dt.
\]
Substituting $n=0$, we again change variables via $x=1-t$ and then integrate by parts with $u=-\log(1-x)$ and $\mathrm dv=x^{s-2}\,\mathrm dx$ to find
\[
I_0=\frac{1}{s-1}\int_0^1\frac{1-x^{s-1}}{1-x}\,\mathrm dx+\log(1-x)\frac{1-x^{s-1}}{s-1}\bigg|_{x=0}^1.
\]
If $s>0$ the limit term vanishes and upon inspection of the integral representation for the harmonic numbers in Relation \ref{rel:harmonic_number_integral_rep}
\[
I_0=\frac{H_{s-1}}{s-1}.
\]
Now consider the general case $I_n$. Without loss of generality assume $n\geq 1$ and perform integration by parts with $u=(-1)^n\log^n(1-t)$ and $\mathrm dv=-\log t(1-t)^{s-2}\,\mathrm dt$. Expanding the logarithm in $\mathrm dv$ as a power series in $(1-t)$ and integrating termwise we find
\[
v=-\sum_{k=0}^\infty\frac{(1-t)^{s+k}}{(s+k)(1+k)}.
\]
In this form, it becomes clear that the limit term $uv|_{t=0}^1$ vanishes if $n\geq 1$ so that $|I_n|=\int_0^1(-v)\,\mathrm du$. Furthermore, we observe for $s>0$:
\[
-v=\frac{1}{s}(1-t)^{s-1}\sum_{k=0}^\infty\frac{s(1-t)^{k+1}}{(s+k)(1+k)}\leq \frac{1}{s}(-\log t)(1-t)^{s-1}.
\]
Hence,
\[
|I_n|\leq\frac{n}{s}\int_0^1(-1)^{n-1}\log^{n-1}(1-t)(-\log t)(1-t)^{s-2}\,\mathrm du=\frac{n}{s}|I_{n-1}|.
\]
Solving the recurrence relation and calling on the result for $I_0$ we find for $s>0$:
\[
|I_n|\leq\frac{n!}{s^n}\frac{H_{s-1}}{s-1}<\infty,
\]
which completes the proof.
\end{proof}

\begin{lemma}
\label{lem:double_log_power_integral_2_bound}
Let $z\in\Bbb R^+$, $m,n\in\Bbb N_0$, $a,b\in(-\infty,2]$,
\[
f(x,y,z) =1-\frac{(1-x)(1-y)}{(1-(1-z)x)(1-(1-z) y)},
\]
\[
F(x,y,z,l,s)=(\log\circ f)^l(x,y,z)f^{s-2}(x,y,z),
\]
and
\[
I=\int_0^1\int_0^1F(x,y,z,n,a)F(x,y,1,m,b)\,\mathrm dx\mathrm dy.
\]
Then $I$ converges when $a+b>2$\footnote{One can easily extend this result to other values of $a$ and $b$ using similar methods to those used here and noting that for $s\geq 2$, $0\leq f^{s-2}(z,y,x)\leq 1$.}.
\end{lemma}

\begin{proof}
For convenience we will introduce the auxillary function
\[
J(x,y,z)=\frac{z^2}{(z+(1-z)x)^2(z+(1-z)y)^2}
\]
as well as the following properties:
\[
\begin{array}{*2{>{\displaystyle}l}}
(\mathrm{FI}) &\text{For $l\in\Bbb N_0$, $z\in\Bbb R^+$, and $(x,y)\in[0,1]^2$: $|F|=(-1)^lF$}.\\[1ex]
(\mathrm{FII}) &\text{$f(\cdot,z)$ is increasing on $z\in\Bbb R^+$ from $f(\cdot,0)=0$ to $f(\cdot,\infty)=1$}.\\[1ex]
(\mathrm{J}) &\text{For all $z\in\Bbb R^+$ and $(x,y)\in[0,1]^2$: $0\leq J(x,y,z)\leq z^{2\operatorname{sign}(z-1)}$}.
\end{array}
\]
Upon inspection of property $(\mathrm{FI})$ we see that $|I|=(-1)^{m+n}$ and so
\[
|I|=\int_0^1\int_0^1|F(x,y,z,n,a)|\, |F(x,y,1,m,b)|\,\mathrm dx\mathrm dy.
\]
Furthermore, since $n\in\Bbb N_0$ and $a\leq 2$ we are able to deduce from property $(\mathrm{FII})$ that $|F(x,y,z,n,a)|$ is nonnegative and a decreasing function of $z$ for all $z\in\Bbb R^+$. Denoting $z^\ast=\min\{1,z\}$, it follows for all $z\in\Bbb R^+$
\[
|I|\leq\int_0^1\int_0^1|F(x,y,z^\ast,m+n,a+b-2)|\,\mathrm dx\mathrm dy.
\]
Performing the change of variables $(1-u,1-v)=((1-x)/(1-(1-z^\ast)x),(1-y),(1-(1-z^\ast)y))$ and then using property $(\mathrm J)$ subsequently gives
\[
|I|\leq\max\{1,z^{-2}\}\int_0^1\int_0^1|F(u,v,1,m+n,a+b-2)|\,\mathrm du\mathrm dv,
\]
which according to Lemma \ref{lem:double_log_power_integral_1_bound} converges if $a+b>2$. The proof is now complete.
\end{proof}

Now that we have the preliminary results needed to study the integral of $\Upsilon$ we derive one last result that will be used to construct $\Upsilon$. The following lemma describes a function $\upsilon(\alpha,\beta,x)$ that bounds a particular case of the hypergeometric function $F(\alpha,\alpha;\beta;x)$ on the interval $x\in[0,1]$. Comparing $\upsilon$ with the asymptotic properties of the hypergeometric function in \cite[$\S 15.4(\mathrm{ii})$]{nist_2010} will reveal that it not only bounds our special case of the hypergeometric function but also has the critical property
\[
F(\alpha,\alpha;\beta;x)\sim\upsilon(\alpha,\beta,x),\quad \text{as } x\searrow 1.
\]

\begin{lemma}
\label{lem:hyper2f1_bounding_functions}
For $\alpha\in\Bbb R$, $\beta\in\Bbb R^+$, and $x\in[0,1]$:
\[
F(\alpha,\alpha;\beta;x)\leq \upsilon(\alpha,\beta,x),
\]
where
\[
\upsilon(\alpha,\beta,x)=%
\begin{cases}
\frac{\Gamma(\beta)\Gamma(\beta-2\alpha)}{\Gamma(\beta-\alpha)\Gamma(\beta-\alpha)}, &\beta>2\alpha\\[1ex]
1-\frac{\Gamma(2\alpha)}{\Gamma(\alpha)\Gamma(\alpha)}\log(1-x), &\beta=2\alpha\\[1ex]
\frac{\Gamma(\beta)\Gamma(2\alpha-\beta)}{\Gamma(\alpha)\Gamma(\alpha)}(1-x)^{\beta-2\alpha}, &\beta<2\alpha.
\end{cases}
\]
\end{lemma}

\begin{proof}
Begin by noting that $\partial_x F(\alpha,\alpha;\beta;x)=(\alpha^2/\beta) F(\alpha+1,\alpha+1;\beta+1;x)$, which can be represented as an absolutely convergent power series in $x$ on $[0,1)$. Since this power series representation also has nonnegative coefficients it follows that $F(\alpha,\alpha;\beta;x)$ must an increasing function of $x$ on this interval. Hence, $F(\alpha,\alpha;\beta;x)\leq F(\alpha,\alpha;\beta;1)$ which upon inspection of Relation \ref{rel:2F1_zeq1} yields the desired result for the case $\beta>2\alpha$. The $\beta=2\alpha$ case in given in \cite[Lem.~$2.3$]{anderson_1995}. For $\beta<2\alpha$ we apply the transformation \cite[Eq.~$15.8.1(\mathrm{iii})$]{nist_2010} yielding $F(\alpha,\alpha;\beta;x)=(1-x)^{\beta-2\alpha}F(\beta-\alpha,\beta-\alpha;\beta;x)\leq (1-x)^{\beta-2\alpha}F(\beta-\alpha,\beta-\alpha;\beta;1)$. The proof is now complete.
\end{proof}

After much effort we may finally derive the second moment of the generalized estimator $\mathscr T_\nu$.
\begin{theorem}[The second moment of $\mathscr T_\nu$]
\label{thm:T_est_variance_series}
Let $\zeta\in\Bbb R_0^+$, $(\alpha_1,\alpha_2)\in\Bbb R^+\times\Bbb R^+:\alpha_1+\alpha_2>2$, and $\nu\in\Bbb R:-\alpha_2<2\nu<\alpha_1$. Then, $\ev\mathscr T_\nu^2<\infty$ where
\[
\ev\mathscr T_\nu^2%
=\frac{1}{\kappa_1^2}\sum_{k=0}^\infty\sum_{\ell=0}^k\frac{g_{k,\ell}^2(\zeta,\nu)}{(\alpha_1)_\ell(\alpha_2)_{k-\ell}\,\ell!\,(k-\ell)!},
\]
and $g_{n,\omega}$ is given by Theorem \ref{thm:IncLerch_sum} $(\mathrm{ii})$. Additionally,
\begin{multline*}
\ev\mathscr T_\nu^2%
=\frac{\nu^2\zeta^{2\nu}}{\kappa_1^2}%
\int_0^1\int_0^1\pFq{}{}{1+\nu,1+\nu}{\alpha_1}{\frac{(1-x)(1-y)}{(1-(1-\zeta)x)(1-(1-\zeta)y)}}\\
\pFq{}{}{1-\nu,1-\nu}{\alpha_2}{(1-x)(1-y)}%
\left((1-(1-\zeta)x)(1-(1-\zeta)y)\right)^{-\nu-1}\,\mathrm dx\mathrm dy.
\end{multline*}
\end{theorem}

\begin{proof}
We begin with the expression for $\ev\mathscr T_n^2$ derived in Lemma \ref{lem:T_est_variance_discrete} which is finite for $n\in\Bbb N_0:0\leq 2n<\alpha_1$. According to \cite[$\S 16.2(\mathrm{iii})$]{nist_2010}, $F(k+1,\ell+1;\alpha_1;1)$ can be expressed as an absolutely convergent series so long as $k+\ell+2<\alpha_1$ which is covered by the already imposed restriction on $\alpha_1$. Furthermore, $F(-k,-\ell;\alpha_2;1)$ is degenerate and truncates after $m=\min\{k,\ell\}+1$ terms due to the presence of nonpositive integers $-k$ and $-\ell$ in the top two arguments. Consequently, the product\footnote{For the purpose of taking the product we can formally express the degenerate hypergeometric term as the infinite series $\sum_{s=0}^\infty\frac{(-k)_s(-\ell)_s}{(\alpha_2)_s\, s!}=\sum_{s=0}^{\min\{k,\ell\}}\frac{(-k)_s(-\ell)_s}{(\alpha_2)_s\, s!}+0+0+\cdots$.}
\begin{multline*}
\pFq{}{}{k+1,\ell+1}{\alpha_1}{1}\pFq{}{}{-k,-\ell}{\alpha_2}{1}\\
=\sum_{r=0}^\infty\sum_{s=0}^r\frac{(k+1)_s(\ell+1)_s(-k)_{r-s}(-\ell)_{r-s}}{(\alpha_1)_s(\alpha_2)_{r-s}\,s!\,(r-s)!}
\end{multline*}
is absolutely convergent and thus permits rearrangement of its terms \cite[Ch.~$13$]{rudin1976}. Substituting the derived product series into the expression for $\ev\mathscr T_n^2$ and rearranging the order of summation produces
\begin{equation}
\label{eq:ETn_fn_form}
\ev\mathscr T_n^2=\frac{1}{\kappa_1^2}\sum_{r=0}^\infty\sum_{s=0}^r\frac{f^2(n)}{(\alpha_1)_s(\alpha_2)_{r-s}\,s!\,(r-s)!},
\end{equation}
with $f(n)=\sum_{k=0}^{n-1}(k+1)_s(-k)_{r-s}\zeta^k$. Focusing on the Pochhammer terms in the summand of $f(n)$ we use Relations \ref{rel:pochhammer_inversion1}--\ref{rel:pochhammer_inversion2} to write
\[
(k+1)_s(-k)_{r-s}=(-1)^{r-s}(k+1)_s(k+1-(r-s))_{r-s}=(-1)^{s}(-k-s)_r,
\]
which upon expanding with Definition \ref{def:StirlingS1} and substituting into $f(n)$ gives
\[
f(n)=(-1)^{r+s}\sum_{\ell=0}^r \mathcal S_r^{(\ell)}\sum_{k=0}^{n-1}(k+s)^\ell\zeta^k.
\]
Identifying the interior sum over $k$ as the incomplete Lerch Transcendent we have by Theorem \ref{thm:IncLerch_sum} the unique generalization of $f(n)$ being given by $f_G(\nu)=(-1)^{r+s}g_{r,s}(\zeta,\nu)$. Substituting $f(n)\mapsto f_G(\nu)$ in $(\ref{eq:ETn_fn_form})$ then gives the desired expression for the series expansion of $\ev\mathscr T_\nu^2$.

Now, to obtain the integral representation we begin by writing $\ev\mathscr T_\nu^2$ in the form
\[
\ev\mathscr T_\nu^2%
=(\ev\mathscr T_\nu)^2\sum_{k=0}^\infty\sum_{\ell=0}^k\frac{\tilde g_{k,\ell}^2(\zeta,\nu)}{(\alpha_1)_\ell(\alpha_2)_{k-\ell}\,\ell!\,(k-\ell)!}.
\]
Upon inspection of Proposition \ref{prop:gTilde_expected_val_form} we see that $\tilde g_{k,\ell}$ can be recast in terms of an expected value to produce
\[
\ev\mathscr T_\nu^2%
=\left(\frac{\nu \zeta^\nu}{\kappa_1}\right)^2\lim_{n\to\infty}\int_0^1\int_0^1f_n(x,y)\,\mathrm dx\mathrm dy,
\]
where
\[
f_n(x,y)=\sum_{k=0}^n\sum_{\ell=0}^k\frac{(\nu^{-1}(\ell+\nu)^{(k+1)})^2}{(\alpha_1)_\ell(\alpha_2)_{k-\ell}\,\ell!\,(k-\ell)!}\frac{((1-x)(1-y))^k}{((1-(1-\zeta)x)(1-(1-\zeta)y))^{\ell+1+\nu}}.
\]
To interchange limits we must find an integrable function $\Upsilon$ that dominates the sequence $|f_n|$ on $(x,y)\in[0,1]^2$. With this goal in mind let's first find an expression for $f=\lim_{n\to\infty}f_n$. Using properties of the Pochhammer symbol we may write $\nu^{-1}(\ell+\nu)^{(k+1)}=(-1)^k(1-\nu)_k(1+\nu)_\ell/(\nu-k)_\ell$ and $[(\alpha_2)_{k-\ell}(k-\ell)!]^{-1}=(-k)_\ell(1-\alpha_2-k)_\ell/(k!(\alpha_2)_k)$, which upon substituting into $f_n$, taking the limit $n\to\infty$, and utilizing Lemma \ref{lem:2F1_Cauchy_product} gives
\[
f(x,y)=t^{-1-\nu}\pFq{}{}{1+\nu,1+\nu}{\alpha_1}{\frac{s}{t}}\pFq{}{}{1-\nu,1-\nu}{\alpha_2}{s},
\]
where $s=(1-x)(1-y)$ and $t=(1-(1-\zeta)x)(1-(1-\zeta)y)$. Now upon closer inspection of $f_n$ we see that it is the sum of nonnegative terms; hence, $0\leq f_n<f$ showing that $f$ dominates $f_n$. Given this observation we could choose $\Upsilon=f$, however, the complexity of $f$ hinders demonstrating its integrability. As such, we seek a simpler function for $\Upsilon$ that bounds $f$ from above.

Looking back at our expression for $f$ we can see that it is the product of three nonnegative components and so we will build $\Upsilon$ by successively bounding each component from above. To bound $t^{-1-\nu}$ we decompose the set $(\zeta,\nu)\in\Bbb R^+\times\Bbb R$ into the four regions: $(0,1]\times(-\infty,-1]$, $(0,1]\times[-1,\infty)$, $[1,\infty)\times(-\infty,-1]$, and $[1,\infty)\times[-1,\infty)$. Combining upper and lower bounds of each case gives
\[
\min\{1,\zeta^{-2-2\nu}\}\leq t^{-1-\nu}\leq\max\{1,\zeta^{-2-2\nu}\}.
\]
Furthermore, observe that $0\leq s\leq 1$ and $0\leq s/t\leq 1$ everywhere on $(x,y)\in[0,1]^2$ so that we may call on Lemma \ref{lem:hyper2f1_bounding_functions} to bound each hypergeometric term in $f$. This leads us to propose
\[
\Upsilon(x,y)=\max\{1,\zeta^{-2-2\nu}\}\,\upsilon(1+\nu,\alpha_1,s/t)\,\upsilon(1-\nu,\alpha_2,s),
\]
with $\upsilon(\alpha,\beta,x)$ given by Lemma \ref{lem:hyper2f1_bounding_functions}. The function $\upsilon(\alpha,\beta,x)$ is piecewise with three possible cases depending on the value of $\operatorname{sign}(\beta-2\alpha)$. Therefore, we must determine the integrability of $\Upsilon$ for nine different cases which we categorize into four regions of the set $(\alpha_1,\alpha_2,\nu)\in\Bbb R^+\times\Bbb R^+\times\Bbb R$. Table \ref{tbl:Upsilon_integrability} presents the four regions which are roughly ordered according to the severity of singularities present in $\Upsilon$ ($\mathrm{IV}$ most severe). Included in the table is the domain of each region, number of cases contained in each region, and any additional constraints determined from Lemma \ref{lem:double_log_power_integral_2_bound} (use $a=\alpha_1+2\nu$ and $b=\alpha_2-2\nu$ in lemma) such that $\Upsilon$ remains integrable. Working with the domains and additional constraints we determine that each of the four regions implicitly requires $\alpha_1+\alpha_2>2$. As such, we may combine the domains and constraints of each region to deduce
\[
\alpha_1+\alpha_2>2\land -\alpha_2<2\nu<\alpha_1\implies \int_0^1\int_0^1\Upsilon(x,y)\,\mathrm dx\mathrm dy<\infty.
\]
Hence if the parameters satisfy these constraints, $\Upsilon$ is integrable and $\ev\mathscr T_\nu^2=\iint f(x,y)\,\mathrm dx\mathrm dy<\infty$. The proof is now complete.
\begin{table}[h]
\centering
\begin{tabular}{LCCC}\toprule
\text{Region} &\text{Domain} &\text{\# Cases} &\text{Constraint(s)}\\\midrule
(\mathrm{I})    &\alpha_1-2\nu\geq2\land\alpha_2+2\nu\geq2 &4 &-                         \\[0.5ex]
(\mathrm{II})   &\alpha_1-2\nu<2\land\alpha_2+2\nu\geq2     &2 &0<\alpha_1-2\nu  \\[0.5ex]
(\mathrm{III})  &\alpha_1-2\nu\geq2\land\alpha_2+2\nu<2     &2 &0<\alpha_2+2\nu \\[0.5ex]
(\mathrm{IV})  &\alpha_1-2\nu<2\land\alpha_2+2\nu<2         &1                              &\substack{0<\alpha_1-2\nu\\[0.25ex] 0<\alpha_2+2\nu\\[0.25ex] 2<\alpha_1+\alpha_2}\\\bottomrule
\end{tabular}
\caption{Regions of the parameter space $(\alpha_1,\alpha_2,\nu)\in\Bbb R^+\times\Bbb R^+\times\Bbb R$ ordered by severity of singularities in $\Upsilon$ along with number of cases contained in each region and additional constraint determined by Lemma \ref{lem:double_log_power_integral_2_bound}.}
\label{tbl:Upsilon_integrability}
\end{table}
\end{proof}

\begin{corollary}
\label{cor:Tv_variance}
For the parameter constraints given by Theorem \ref{thm:T_est_variance_series}
\[
\var\mathscr T_\nu=\ev\mathscr T_\nu^2-(\ev\mathscr T_\nu)^2<\infty.
\]
In particular,
\[
\var\mathscr T_\nu%
=\frac{1}{\kappa_1^2}\sum_{k=1}^\infty\sum_{\ell=0}^k\frac{g_{k,\ell}^2(\zeta,\nu)}{(\alpha_1)_\ell(\alpha_2)_{k-\ell}\,\ell!\,(k-\ell)!},
\]
where $g_{n,\omega}$ is given in Theorem \ref{thm:IncLerch_sum}.
\end{corollary}

\begin{proof}
Using the shorthand $\ev\mathscr T_\nu^2=\sum_{r=0}^\infty a_r$ observe that
\[
a_0=\kappa_1^{-2}g_{0,0}^2(\zeta,\nu)=(\kappa_1^{-1}{_1F_0}(1;-;z)_\nu)^2=(\ev\mathscr T_\nu)^2;
\]
thus, subtracting the $a_0$ term from $\ev\mathscr T_\nu^2$ yields $\var\mathscr T_\nu$ as desired.
\end{proof}

\begin{proposition}
\label{prop:VarTv_1st_term}
Evaluating the $k=1$ term of $\var\mathscr T_\nu$ yields the following asymptotic approximation for large $\alpha_1$ and $\alpha_2$.
\begin{multline*}
\var\mathscr T_\nu=
\frac{1}{\alpha_1}\left(\frac{1-\zeta^\nu (1+\nu (1-\zeta))}{\kappa_1(1-\zeta)^2}\right)^2\\
+\frac{1}{\alpha_2}\left(\frac{\zeta-\zeta^\nu (1+(\nu-1)(1-\zeta))}{\kappa_1(1-\zeta)^2}\right)^2+\cdots
\end{multline*}
\end{proposition}

The following is a generalization of Theorem \ref{thm:T_est_variance_series}, which shows that the mixed partial derivatives of $\ev\mathscr T_\nu^2$ w.r.t. $\alpha_1$ and $\alpha_2$ converge for the same set of parameters as $\ev\mathscr T_\nu^2$ itself. Such a result is a key ingredient to Corollary \ref{cor:convergence_of_ACVTv2_alpha_derivatives} and Section \ref{subsec:comp_of_opt_samp_sizes_for_Tv} on optimal sample sizes.

\begin{theorem}
\label{thm:convergence_of_ETv2_alpha_derivatives}
For all $n,m\in\Bbb N_0$, $\partial_{\alpha_1}^n\partial_{\alpha_2}^m\ev\mathscr T_\nu^2<\infty$ whenever $\ev\mathscr T_\nu^2<\infty$.
\end{theorem}

\begin{proof}[Proof sketch]
The proof amounts to studying the convergence of the double integral representation for $\partial_{\alpha_1}^n\partial_{\alpha_2}^m\ev\mathscr T_\nu^2$. Using the results of Theorem \ref{thm:hyper2f1_c_derivative_bounding_functions} in Appendix \ref{sec:parameter_derivatives_of_2F1} we may construct a function that bounds the magnitude of the differentiated integrand in the form
\[
\Upsilon_{n,m}(x,y)=\max\{1,\zeta^{-2-2\nu}\}\upsilon_n(1+\nu,\alpha_1,s/t)\upsilon_m(1-\nu,\alpha_2,s),
\]
where $s=(1-x)(1-y)$, $t=(1-(1-\zeta) x)(1-(1-\zeta)y)$, and
\[
\upsilon_n(\alpha,\beta,x)=%
\begin{cases}
\frac{n!}{\operatorname B(|\alpha|,|\alpha|)}\frac{1}{(\beta-2\alpha\mathds 1_{\alpha>0})^{n+1}}+\mathds 1_{n=0}, &\beta>2\alpha\\[1ex]
\frac{n!}{\operatorname B(\alpha,\alpha)}(\frac{1}{2\alpha}-\log(1-x))^{n+1}+\mathds 1_{n=0}, &\beta=2\alpha\\[1ex]
\frac{n!}{\operatorname B(\alpha,\alpha)}(1-x)^{\beta-2\alpha}(\frac{1}{\beta}-\log(1-x))^{n+1}+\mathds 1_{n=0}, &\beta<2\alpha.
\end{cases}
\]
Notice that $\Upsilon_{n,m}$ is very similar to the bounding function $\Upsilon$ in Theorem \ref{thm:T_est_variance_series} with the exception of possibly additional constant and higher-order multiplicative logarithmic terms. Using Lemma \ref{lem:double_log_power_integral_2_bound}, we can easily verify that these additional constant and logarithmic terms to not affect convergence of the integral $\iint\Upsilon_{n,m}\,\mathrm dx\mathrm dy$ and so the claim follows.
\end{proof}


\subsection{Coefficient of variation}

With minimal effort we can now derive an expression for the coefficient of variation which gives us a signed measure of relative uncertainty in $\mathscr T_\nu$. Making use of the results for $\tilde g_{n,\omega}$ back in Section \ref{subsec:gnw_tilde_gnw}, we will show that this quantity is a monotone function of the unknown parameter $\zeta$ which will lead directly to obtaining exact confidence intervals of its absolute value in Section \ref{sec:confidence_intervals}. Also important is the absolute coefficient of variation which will be discussed extensively in latter sections and--in particular--used to define optimal sample sizes for $\mathscr T_\nu$ in Section \ref{sec:cv_analysis}.

\begin{corollary}[Corollary of Theorem \ref{thm:T_est_variance_series}]
\label{cor:cvT_nu}
For the parameter constraints given by Theorem \ref{thm:T_est_variance_series}
\[
\begin{aligned}
\cv\mathscr T_\nu &= \operatorname{sign}(\ev\mathscr T_\nu)\sqrt{\var\mathscr T_\nu/(\ev\mathscr T_\nu)^2}\\
\acv\mathscr T_\nu &= |\cv\mathscr T_\nu|,
\end{aligned}
\]
with $\var\mathscr T_\nu$ given in Corollary \ref{cor:Tv_variance} are finite. In particular,
\[
\cv\mathscr T_\nu%
=\operatorname{sign}(\nu)\left(\sum_{k=1}^\infty\sum_{\ell=0}^k\frac{\tilde g_{k,\ell}^2(\zeta,\nu)}{(\alpha_1)_\ell(\alpha_2)_{k-\ell}\,\ell!\,(k-\ell)!}\right)^{1/2},
\]
with $\tilde g_{n,\omega}$ given by Corollary \ref{cor:gTilde_hyper_form}.
\end{corollary}

\begin{proof}
The proof follows from the following facts:
\begin{enumerate}
\item[$(1)$] $\operatorname{sign}(\ev\mathscr T_\nu)=\operatorname{sign}(\nu)$,
\item[$(2)$] $g_{k,\ell}^2(\zeta,\nu)/(\ev\mathscr T_\nu)^2=\kappa_1^2\,\tilde g_{k,\ell}^2(\zeta,\nu)$.
\end{enumerate}
A few simple algebraic manipulations then yield the desired form for $\cv\mathscr T_\nu$.
\end{proof}

The results of the previous section can now be used to show that mixed partial derivatives of $\acv^2\mathscr T_\nu$ w.r.t. $\alpha_1$ and $\alpha_2$ converge. This will be used in Section \ref{subsec:comp_of_opt_samp_sizes_for_Tv} when studying optimal sample sizes.
\begin{corollary}[Corollary of Theorem \ref{thm:convergence_of_ETv2_alpha_derivatives}]
\label{cor:convergence_of_ACVTv2_alpha_derivatives}
Let $\acv^2\mathscr T_\nu(\alpha_1,\alpha_2)$ denote the squared absolute coefficient of variation of $\mathscr T_\nu$ as a function of $\alpha_1$ and $\alpha_2$. Then for all $n\in\Bbb N_0$, $\partial_{\alpha_2}^n\acv^2\mathscr T_\nu(\mathrm A-\alpha_2,\alpha_2)|_{\mathrm A=\alpha_1+\alpha_2}<\infty$ whenever $\ev\mathscr T_\nu^2<\infty$.
\end{corollary}

\begin{proof}
By definition $\acv^2\mathscr T_\nu=\ev\mathscr T_\nu^2/(\ev\mathscr T_\nu)^2-1$ and since $\ev\mathscr T_\nu$ is independent of $\alpha_1$ and $\alpha_2$ we conclude that the convergence of $\partial_{\alpha_2}^n\acv^2\mathscr T_\nu(\mathrm A-\alpha_2,\alpha_2)|_{\mathrm A=\alpha_1+\alpha_2}$ will agree with that of $\partial_{\alpha_2}^n\ev\mathscr T_\nu^2(\mathrm A-\alpha_2,\alpha_2)|_{\mathrm A=\alpha_1+\alpha_2}$. With a straightforward calculation we find
\[
\partial_{\alpha_2}^n\ev\mathscr T_\nu^2(\mathrm A-\alpha_2,\alpha_2)|_{\mathrm A=\alpha_1+\alpha_2}=\sum_{k=0}^n\binom{n}{k}(-1)^k\partial_{\alpha_1}^k\partial_{\alpha_2}^{n-k}\ev\mathscr T_\nu^2(\alpha_1,\alpha_2),
\]
where according to Theorem \ref{thm:convergence_of_ETv2_alpha_derivatives} the magnitude of each term is finite whenever $\ev\mathscr T_\nu^2$ is finite.
\end{proof}

With an expression for $\cv\mathscr T_\nu$ at hand we seek to establish its monotony in the parameter $\zeta$. We will accomplish this by working with the series expansion of $\cv\mathscr T_\nu$ to show that: $(1)$ the series converges at the boundary values $\zeta\to 0$ and $\zeta\to\infty$ and $(2)$ the series terms are montone functions in $\zeta$.  With these results we will call upon Fubini's theorem on differentiation to conclude that $\partial_\zeta\cv\mathscr T_\nu$ is positive. Before proceeding we present an explicit expression for product of hypergeometric series that will be used many times in the rest of the work.

\begin{lemma}[Hypergeometric product, {\cite[Eq.~$07.23.16.0006.01$]{wolfram_functions}}]
\label{lem:2F1_Cauchy_product}
\[
\pFq{}{}{a,b}{c}{gz}\pFq{}{}{\alpha,\beta}{\gamma}{hz}=\sum_{k=0}^\infty c_k z^k,
\]
where
\[
c_k=\frac{g^k(a)_k(b)_k}{k!\,(c)_k}\pFq{4}{3}{-k,1-c-k,\alpha,\beta}{1-a-k,1-b-k,\gamma}{\frac{h}{g}}.
\]
\end{lemma}

\begin{lemma}[Boundary values of $\cv\mathscr T_\nu$]
\label{lem:cvTv_boundary_values}
If the parameters $\alpha_1$, $\alpha_2$, and $\nu$ adhere to the constraints given by Theorem \ref{thm:T_est_variance_series}
\[
\begin{array}{*3{>{\displaystyle}l}}
(\mathrm{i}) &\lim_{\zeta\to 0}\cv\mathscr T_\nu%
&=%
\begin{cases}
-\left(\frac{(\alpha_1-\nu-1)_{-\nu-1}(\alpha_2+\nu)_{\nu}}{(\alpha_1)_{-\nu-1}(\alpha_2)_{\nu}}-1\right)^{1/2}, &\nu<0\\
\frac{1}{\sqrt{\alpha_1-2}}, &\nu>0
\end{cases}\\[2em]
(\mathrm{ii}) &\lim_{\zeta\to\infty}\cv\mathscr T_\nu%
&=%
\begin{cases}
-\frac{1}{\sqrt{\alpha_2-2}}, &\nu<0\\
\left(\frac{(\alpha_1-\nu)_{-\nu}(\alpha_2+\nu-1)_{\nu-1}}{(\alpha_1)_{-\nu}(\alpha_2)_{\nu-1}}-1\right)^{1/2}, &\nu>0.
\end{cases}
\end{array}
\]
\end{lemma}

\begin{proof}
We will only prove $(\mathrm i)$ since the proof for $(\mathrm{ii})$ being nearly identical. Working with the series expansion for $\cv^2\mathscr T_\nu$ we call on Lemma \ref{lem:tilde_g_boundary_values} and Relation \ref{rel:rising_falling_factorial} to deduce for $\nu<0$: $\tilde g_{k,\ell}(0,\nu)=(-1)^k(-\ell-\nu)_k$; hence,
\begin{equation}
\label{eq:cvTv_zto0_form1}
\lim_{\zeta\to 0}\cv^2\mathscr T_\nu=%
\sum_{k=1}^\infty\sum_{\ell=0}^k\frac{(-\ell-\nu)_k^2}{(\alpha_1)_\ell(\alpha_2)_{k-\ell}\,\ell!\,(k-\ell)!}.
\end{equation}
Working with the properties of the Pochhammer symbol, several algebraic manipulations reveal
\[
\frac{(-\ell-\nu)_k^2}{(\alpha_2)_{k-\ell}(k-\ell)!} =\frac{(-\nu)_k^2}{(\alpha_2)_k\,k!}\frac{(-k)_\ell(1-k-\alpha_2)_\ell(\nu+1)_\ell(\nu+1)_\ell}{(\nu+1-k)_\ell(\nu+1-k)_\ell},
\]
which upon substitution into $(\ref{eq:cvTv_zto0_form1})$ and identifying the resulting sum over $\ell$ as a degenerate hypergeometric function yields
\[
\lim_{\zeta\to 0}\cv^2\mathscr T_\nu=%
\sum_{k=1}^\infty\frac{(-\nu)_k^2}{(\alpha_2)_k\, k!}\pFq{4}{3}{-k,1-k-\alpha_2,\nu+1,\nu+1}{\nu+1-k,\nu+1-k,\alpha_1}{1}.
\]
Now calling on Lemma \ref{lem:2F1_Cauchy_product} the remaining series in $k$ is readily identified as a product of hypergeometric series such that
\begin{equation}
\label{eq:cvTv_z_zero_limit}
\lim_{\zeta\to 0}\cv^2\mathscr T_\nu=%
\pFq{}{}{\nu+1,\nu+1}{\alpha_1}{1}\pFq{}{}{-\nu,-\nu}{\alpha_2}{1}-1.
\end{equation}
By Relation \ref{rel:2F1_zeq1} this result can be recast in terms of the ratio of gamma functions.  Simplifying in terms of Pochhammer symbols, taking the square root, and reintroducing the appropriate sign then yields the desired expression for the $\nu<0$ case. Now looking back at Lemma \ref{lem:tilde_g_boundary_values} we see for $\nu>0$: $\tilde g_{k,\ell}(0,\nu)=(-1)^k(-\ell)_k$, which is equivalent to the $\nu<0$ case evaluated at $\nu=0$. Therefore, to find $\lim_{\zeta\to 0}\cv\mathscr T_\nu$ for $\nu>0$ we simply substitute $\nu=0$ into $(\ref{eq:cvTv_z_zero_limit})$ and evaluate the square root yielding
\[
\lim_{\zeta\to 0}\cv\mathscr T_\nu=%
\left(\pFq{}{}{1,1}{\alpha_1}{1}-1\right)^{1/2}=\frac{1}{\sqrt{\alpha_1-2}}.
\]
The proof is now complete.
\end{proof}

\begin{theorem}[Monotony of $\cv\mathscr T_\nu$ in $\zeta$]
\label{thm:CVTv_is_monotone}
Under the conditions of Theorem \ref{thm:T_est_variance_series} if $|\nu|>1$ the coefficient of variation for $\mathscr T_\nu$ is a strictly increasing function in $\zeta$, that is, $\partial_\zeta\cv\mathscr T_\nu>0$ for all $\zeta\in\Bbb R^+$.
\end{theorem}

\begin{proof}
For brevity we write
\[
\cv^2\mathscr T_\nu=\sum_{k=1}^\infty f_k(\zeta,\nu),
\]
where $f_k(\zeta,\nu)=\sum_{\ell=0}^k\tilde g_{k,\ell}^2(\zeta,\nu)/c_{k,\ell}$ and $c_{k\ell}=(\alpha_1)_\ell(\alpha_2)_{k-\ell}\,\ell!\,(k-\ell)!$. Since $c_{k\ell}>0$, if $\nu>1$ it follows from Theorem \ref{thm:gnw_is_monotone} that: $(\mathrm i)$ $f_k(\zeta,\nu)$ is nonnegative and increasing in $\zeta$ for all $k\in\Bbb N$ and $(\mathrm{ii})$ $\exists k\in\Bbb N$ where $f_k(\zeta,\nu)$ is strictly increasing in $\zeta$. Consequently,
\[
\cv^2\mathscr T_\nu<\sum_{k=1}^\infty \lim_{\zeta\to\infty}f_k(\zeta,\nu),
\]
which was shown to converge in Lemma \ref{lem:cvTv_boundary_values}; thus, the series expansion for $\cv^2\mathscr T_\nu$ must converge uniformly for all $\zeta\in\Bbb R^+_0$ \cite[Thm.~$7.10$]{rudin1976}. Since each $f_k$ is increasing and the series expansion for $\cv^2\mathscr T_\nu$ converges uniformly, Fubini's theorem on differentiation asserts
\[
\partial_\zeta\cv^2\mathscr T_\nu=\sum_{k=1}^\infty \partial_\zeta f_k(\zeta,\nu)>0,
\]
almost everywhere on $\zeta\in\Bbb R^+$. But now observe that $\nu>1\implies\cv\mathscr T_\nu>0$, so it follows
\[
\partial_\zeta\cv^2\mathscr T_\nu>0\implies \partial_\zeta\cv\mathscr T_\nu>0,
\]
which establishes $\cv\mathscr T_\nu$ being strictly increasing in $\zeta$ when $\nu>1$.  For the case $\nu<-1$, we can follow the same process to show $\cv^2\mathscr T_\nu$ is positive and strictly decreasing in $\zeta$. However, $\nu<-1\implies\cv\mathscr T_\nu<0$ so again we find $\cv\mathscr T_\nu$ to be strictly increasing in $\zeta$ which completes the proof.
\end{proof}


\section{Confidence intervals}
\label{sec:confidence_intervals}

An important aspect of any estimation procedure is quantifying the error and precision of the estimate itself. In the context of estimating $\tau=(\kappa_1-\kappa_2)^{-1}$ with the estimator $\mathscr T_\nu$, two relevant quantities for providing relative error and precision come to mind, namely, the absolute relative bias $\arb\mathscr T_\nu$ and absolute coefficient of variation $\acv\mathscr T_\nu$. Due to their dependence on the value of $\zeta$, neither of these quantities can ever be known with exact certainty; however, one can resort to confidence sets as a useful alternative. For clarity the following two lemmas state the conditions for which $\arb\mathscr T_\nu$ and $\acv\mathscr T_\nu$ are monotone functions of $\zeta$ which is needed for deriving the corresponding confidence sets in Theorems \ref{thm:ARB_T_v_CI} and \ref{thm:ACV_T_v_CI}.

\begin{lemma}
\label{lem:monotonicity_of_ARB_T_v}
If $-\alpha_2<\nu<\alpha_1$ then $\arb\mathscr T_\nu$ is strictly decreasing in $\zeta$ for $\nu<0$ and strictly increasing in $\zeta$ for $\nu>0$.
\end{lemma}

\begin{proof}
We have $\arb\mathscr T_\nu=|\ev\mathscr T_\nu/\tau-1|=\zeta^\nu$, where $\ev\mathscr T_\nu$ is finite if $-\alpha_2<\nu<\alpha_1$. Since $\zeta>0$, the monotonicity of $\arb\mathscr T_\nu$ immediately follows.
\end{proof}

\begin{lemma}
\label{lem:monotonicity_of_ACV_T_v}
If $\alpha_1+\alpha_2>2$ and $-\alpha_2<2\nu<\alpha_1$ then $\acv\mathscr T_\nu$ is strictly decreasing in $\zeta$ for $\nu<-1$ and strictly increasing in $\zeta$ for $\nu>1$.
\end{lemma}

\begin{proof}
According to Corollary \ref{cor:cvT_nu}, $\acv\mathscr T_\nu$ is finite if $\alpha_1+\alpha_2>2$ and $-\alpha_2<2\nu<\alpha_1$. Furthermore, Theorem \ref{thm:CVTv_is_monotone} states that $\cv\mathscr T_\nu$ is: $(1)$ positive and strictly increasing in $\zeta$ for $\nu>1$ and $(2)$ negative and strictly increasing in $\zeta$ for $\nu<-1$. Given $\acv\mathscr T_\nu=|\cv\mathscr T_\nu|$ the result immediately follows.
\end{proof}

\begin{theorem}
\label{thm:ARB_T_v_CI}
Let $F_{\alpha,d_1,d_2}$ denote the $(1-\alpha)$th quantile for the $F_{d_1,d_2}$ distribution, $V=Y_2/Y_1$, and $Z_\alpha=V F_{\alpha,2\alpha_1,2\alpha_2}$. If $-\alpha_2<\nu<\alpha_1$ then
\[
\operatorname{CI}_\alpha(\arb\mathscr T_\nu)=%
\begin{cases}
(0,Z_{1-\alpha}^\nu], &\nu<0\\
(0,Z_\alpha^\nu], &\nu>0\\
\end{cases}
\]
is a upper bound confidence set for $\arb\mathscr T_\nu$ with coverage probability $(1-\alpha)$.
\end{theorem}

\begin{proof}
Recall the distributional forms for $Y_1$ and $Y_2$ are $Y_i\sim\mathcal G(\alpha_i,\alpha_i/\kappa_i)$ where $\alpha_i$ is known.  Since $Y_1\perp Y_2$ one can define the pivotal quantity
\[
\frac{\zeta}{V}\sim F_{2\alpha_1,2\alpha_2}.
\]
It follows that $\operatorname{CI}_\alpha(\zeta)=[Z_{1-\alpha},\infty)$ is a lower bound confidence set for $\zeta$ with coverage probability $(1-\alpha)$, that is, $\pr(\operatorname{CI}_\alpha(\zeta)\ni\zeta)=1-\alpha$. Furthermore, if $\nu<0$ then $\arb\mathscr T_\nu$ is monotone decreasing in $\zeta$ with $\lim_{\zeta\to\infty}\arb\mathscr T_\nu=0$; thus,
\[
\operatorname{CI}_\alpha(\arb\mathscr T_\nu)=\arb\mathscr T_\nu(\operatorname{CI}_\alpha(\zeta))=(0,Z_{1-\alpha}^\nu]
\]
is a upper bound confidence set for $\arb\mathscr T_\nu$ with coverage probability $(1-\alpha)$. If instead $\nu>0$, $\arb\mathscr T_\nu$ is monotone increasing in $\zeta$ so we use $\operatorname{CI}_\alpha(\zeta)=(0,Z_\alpha]$ as an upper bound confidence set for $\zeta$ to obtain the corresponding set estimator of $\arb\mathscr T_\nu$ for the $\nu>0$ case.
\end{proof}

\begin{theorem}
\label{thm:ACV_T_v_CI}
Let $Z_\alpha$ be as defined in Theorem \ref{thm:ARB_T_v_CI} and $\acv\mathscr T_\nu(\zeta)$ represent the absolute coefficient as a function of $\zeta$. If $\alpha_1+\alpha_2>2$ and $-\alpha_2<2\nu<\alpha_1$ then
\[
\operatorname{CI}_\alpha(\acv\mathscr T_\nu)=%
\begin{cases}
((\alpha_2-2)^{-1/2},\acv\mathscr T_\nu(Z_{1-\alpha})] &\nu<-1\\
((\alpha_1-2)^{-1/2},\acv\mathscr T_\nu(Z_\alpha)] &\nu>1\\
\end{cases}
\]
is a upper bound confidence set for $\acv\mathscr T_\nu$ with coverage probability $(1-\alpha)$.
\end{theorem}

\begin{proof}
The proof follows much along the same line of reasoning given in Theorem \ref{thm:ARB_T_v_CI} and therefore is omitted. Lemma \ref{lem:cvTv_boundary_values} provides the the expressions for the lower boundary of each set estimator. 
\end{proof}


\chapter{Photon Transfer Conversion Gain Estimation}
\label{chap:new_g_estimator}

At this point we have completed the first part of the goal set forth which was to derive an estimator for the reciprocal difference of independent normal variances. Several fundamental results were found for this estimator including: $(1)$ an asymptotic expansion for large shape parameters $\alpha_1$ and $\alpha_2$, $(2)$ the first two moments $\ev\mathscr T_\nu$ and $\ev\mathscr T_\nu^2$ along with the associated quantities $\var\mathscr T_\nu$ and $\cv\mathscr T_\nu$, and $(3)$ exact confidence intervals for the absolute relative bias $\arb\mathscr T_\nu$ and absolute coefficient of variation $\acv\mathscr T_\nu$. With these results at hand, we are finally ready to turn to the tackling the second part of our goal: the problem of conversion gain estimation.

To accomplish this we will first briefly recap and expound upon the discussion and notation surrounding photon transfer theory found in the introduction of Chapter \ref{chap:introduction}. Then we will introduce our newly developed estimator $\mathscr G_\nu$ for the conversion gain $g$ and use the results for $\mathscr T_\nu$ in Chapter \ref{chap:estimation_of_recip_diff_norm_variances} to derive analogous results for $\mathscr G_\nu$. Next, we will rigorously establish a long observed phenomenon about estimators for $g$, that is, their dispersion is dominated at \emph{high illumination} by the dispersion contributed by estimates of variances. Indeed, Section \ref{subsec:characteristics_at_low_illumination} we will take this claim a step further by using the theoretical foundation laid in the previous chapter to show that this dominance can be achieved even at illumination levels near zero if we adopt a notation of \emph{optimal sample sizes}. This observation opens the door to not only justifying the estimation of $g$ and its confidence intervals under low illumination but also provides a pathway for design and control of experiment of $g$-estimation at low illumination; thus, addressing the \emph{low-illumination problem} of conversion gain estimation introduced in \cite{hendrickson_2019}. Finally, this chapter will end by demonstrating the effectiveness of the theoretical results derived herein with an experiment measuring $g$ under low illumination conditions.


\section{A brief review of photon transfer theory}
\label{sec:pt_theory}

\subsection{Bose-Einstein statistics and the uncertain nature of light}

The theory of photon transfer is fundamentally premised on the uncertainty of electro-magnetic radiation emitted by a source. Such uncertainty means that even if the mean photon production of a source per unit time--which we shall simply call the \emph{illumination level}--remains fixed, the observed number of photons emitted over any specific, fixed, time interval will vary; giving rise to the presence of \emph{noise} in the photon stream. From Bose-Einstein statistics the noise in the photon stream of a blackbody source, described by the variance in the number of photons emitted per unit time $\sigma_\gamma^2$ can be related to the mean photon production per unit time $\mu_\gamma$, photon energy $E_\gamma=h\nu$ and thermal energy of the source $E_T=kT$ by \cite{mansuripur_2017}
\[
\sigma_\gamma^2=\mu_\gamma\left[\frac{e^{h\nu/kT}}{e^{h\nu/kT}-1}\right],
\]
where the quantity in square brackets is the so-called \emph{boson factor}. In the regime where the photon energy is significantly larger than the thermal energy $h\nu\gg kT$, which corresponds to photon wavelengths of $0.3-30\,\mu\mathrm m$ and temperatures $T<500\mathrm K$, the boson factor is near unity such that $\sigma_\gamma^2\approx\mu_\gamma$ and the number of photons $k$ emitted from the source per unit time is accurately modeled by the Poisson mass function
\[
P(k)=\frac{\mu_\gamma^k e^{-\mu_\gamma}}{k!},\quad k\in\Bbb N_0.
\]
Light sources with these characteristics are referred to as \emph{Poissonian} and are assumed in the photon transfer method.


\subsection{Photon transfer $\gamma\to\mathrm{DN}$ and conversion gain}

Detection of photons is done via electro-optical image sensors where the pixels comprising the sensor convert, that is, \emph{transfer} photons $(\gamma)$ to electrons $(\electron)$. After exposing the sensor to a light source for some \emph{integration time}, the packets of electrons collected by each pixel are passed through the sensor's circuitry whereby each charge packet is first converted to a voltage $(V)$ and then digitized to produce a \emph{digital number} $(\mathrm{DN})$ representing the intensity of the source at each pixel. As one can see, several conversions, e.g. $\gamma\to\electron$, $\electron\to V$, and $V\to\mathrm{DN}$, take place in the process of image formation. Treating each conversion process as a mapping from one unit to another gives rise to the concept of the system \emph{transfer function} $\mathcal T:\gamma\to\mathrm{DN}$ which represents the aggregate mapping of photons to digital numbers. It is important to note that in general each pixel will have its own unique transfer function and must be characterized individually. For sensors where each pixel's transfer function can be assumed identical we say the sensor is \emph{uniform}. Unless otherwise specified, whenever we discuss the transfer function and measurement thereof, we will assume it is for an individual pixel.

In most cases it is desirable to characterize a pixel in terms of electrons; thus removing the dependence of photon wavelength. Luckily, for Silicon based electro-optical sensors and photons in the visible band $0.4-0.7\,\mu\mathrm m$ the conversion of photons to electrons, described by the quantum yield $\eta$ is unity; therefore rendering the photon-to-digital number transfer function $\mathcal T:\gamma\to\mathrm{DN}$ equivalent to the electron-to-digital number transfer function $\mathcal T:\electron\to\mathrm{DN}$. The rather lofty goal of the photon transfer method is to measure the transfer function solely by observing the pixel's digital output and then use the estimated transfer function to convert the digital numbers back to a physical quantity of electrons produced in the pixel over the integration time. The ability to reverse engineer the transfer function in this way ultimately allows one to characterize the pixel in terms of key imaging performance metrics like read noise and dynamic range all while treating the sensor system as a black box.

Perhaps the simplest transfer function one could have is that represented by the linear equation $\mathcal T(\electron)=\electron/g$ where $g$ is a constant with units $(\electron/\mathrm{DN})$. Pixels that admit such a transfer function are naturally called \emph{linear} and the constant $g$ is referred to as the \emph{conversion gain}. In the special case of linearity, an expression for $g$ is simple to derive. Assuming Poisson photon statistics and unity quantum yield $\eta=1\,(\gamma/\electron)$, the number of electrons collected by a pixel over the integration time is modeled by $N\sim\operatorname{Poisson}(\mu_{\electron})$, where $\mu_{\electron}=\mu_\gamma \operatorname{QE}_\lambda$ is the mean number of electrons collected per integration time, $\mu_\gamma$ is the corresponding mean number of incident photons per integration time, i.e.~the illumination level, and $\operatorname{QE}_\lambda$ is the wavelength dependent quantum efficiency describing the probability that an incident photon with wavelength $\lambda$ produces a free electron. In this context, the transfer function can be viewed as a transformation of the random variable $N$ so that if we let $P=\mathcal T(N)$ denote the random variable representing the \emph{photon induced} digital output then $\mu_{\mathrm p}\,(\mathrm{DN})=\ev P=\mu_{\electron}/g$ and $\sigma_{\mathrm p}^2\,(\mathrm{DN}^2)=\var P=\mu_{\electron}/g^2$; giving rise to the fundamental photon transfer relationship
\begin{equation}
\label{eq:fundamental_g_relationship}
g=\frac{\mu_{\mathrm p}}{\sigma_{\mathrm p}^2}.
\end{equation}

The remarkable aspect of the model that led up to the relationship (\ref{eq:fundamental_g_relationship}) is that it provides a natural and simple way to estimate $g$ based solely on observing the random output of a pixel exposed to incident illumination. If we denote $\mathbf P=\{P_i\}_{i=1}^n$ as a sample of $n$ digital observations generated by a pixel under illumination and $T(\mathbf P)=(\bar P,\hat P)$, where $\bar P=\frac{1}{n}\sum_{i=1}^nP_i$ is the sample mean and $\hat P=\frac{1}{n-1}\sum_{i=1}^n(P_i-\bar P)^2$ as the sample variance, then $T$ is an unbiased estimator of the parameter $\theta=(\mu_{\mathrm p},\sigma_{\mathrm p}^2)$ so that $g$ can be estimated by
\begin{equation}
\label{eq:shot_noise_limited_estimator}
G=\frac{\bar P}{\hat P}.
\end{equation}

While simple, the estimator (\ref{eq:shot_noise_limited_estimator}) has an inherent weakness due to the fact that the relationship $(\ref{eq:fundamental_g_relationship})$ does not account for the background signal and noise produced by real pixels in the absence of photon interaction. In some cases, given a sensor of sufficiently high quality, it can be possible to expose the pixels to high levels of illumination so that this background noise is dominated by photon noise and $g\approx\mu_{\mathrm{p}}/\sigma_{\mathrm{p}}^2$ approximately holds. Such sensors are said to achieve a \emph{shot noise limited} response and the measurement of $g$ is said to be performed in the shot noise limited region of the pixel's dynamic range.

It is of no surprise that many sensors cannot achieve a shot noise-limited response such that the background signal and noise produced by the pixels cannot be ignored. To extend our model to include such cases we first let $D\sim F_D$ represent the digital output of the pixels in the absence of illumination (dark), which is distributed according to some distribution $F_D$ with mean $\mu_{\mathrm d}\,(\mathrm{DN})=\ev D$ and variance $\sigma_{\mathrm d}^2\,(\mathrm{DN}^2)=\var D$. Assuming $D$ is independent of $P$, the population mean and variance of the pixels digital output under illumination becomes
\[
\begin{aligned}
\ev(P+D)\coloneqq\mu_{\mathrm p+\mathrm d} &=\mu_{\mathrm p}+\mu_{\mathrm d}=\mu_{\electron}/g+\mu_{\mathrm d}\\
\var(P+D)\coloneqq\sigma_{\mathrm p+\mathrm d}^2 &=\sigma_{\mathrm p}^2+\sigma_{\mathrm d}^2=\mu_{\electron}/g^2+\sigma_{\mathrm d}^2,
\end{aligned}
\]
respectively, which leads us to the modify the gain relationship $(\ref{eq:fundamental_g_relationship})$ as
\begin{equation}
\label{eq:general_g_relationship}
g=\frac{\mu_{\mathrm p+\mathrm d}-\mu_{\mathrm d}}{\sigma_{\mathrm p+\mathrm d}^2-\sigma_{\mathrm d}^2}.
\end{equation}

Unlike the one-sample, shot noise-limited estimator for $g$ given by (\ref{eq:shot_noise_limited_estimator}), one can see from (\ref{eq:general_g_relationship}) that the measurement of $g$ in a sub shot noise-limited regime will require two separate samples captured under dark and illuminated conditions. Denoting the sample under illumination by $\mathbf X=\{X_i\}_{i=1}^{n_1}$ and the sample in the dark by $\mathbf Y=\{Y_i\}_{i=1}^{n_2}$, the unknown parameter vector $\theta=(\mu_{\mathrm p+\mathrm d},\sigma_{\mathrm p+\mathrm d}^2,\mu_{\mathrm d},\sigma_{\mathrm d}^2)$ can be estimated with the unbiased two-sample statistic $T(\mathbf X,\mathbf Y)=(\bar X,\hat X,\bar Y,\hat Y)$ with $\bar X$ and $\hat X$ representing the sample mean and sample variance of $\mathbf X$ and likewise for the sample $\mathbf Y$. Substituting the components of $T$ directly into $(\ref{eq:general_g_relationship})$ gives an estimator for $g$ of the form
\begin{equation}
\label{eq:sub_shot_noise_g_estimator}
G=\frac{\bar X-\bar Y}{\hat X-\hat Y}=\frac{\bar P}{\hat P},
\end{equation}
where $\bar P$ and $\hat P$ are the new estimators of the photon induced mean $\mu_{\mathrm p}$ and variance $\sigma_{\mathrm p}^2$, respectively.

By accounting for the effects of dark signal and noise with $(\bar Y,\hat Y)$, the estimator (\ref{eq:sub_shot_noise_g_estimator}) does relieve the need for a pixel to achieve a shot noise limited response. However, as was discussed in the introduction of Chapter \ref{chap:introduction}, the introduction of these dark corrections, in particular that of $\hat Y$, lead to (\ref{eq:sub_shot_noise_g_estimator}) exhibiting ill-behaved characteristics in low illumination conditions. To combat this one may measure $g$ with $(\ref{eq:sub_shot_noise_g_estimator})$ under high illumination conditions where the estimator will be most well-behaved. That said, such a procedure becomes invalid if the pixel admits a nonlinear transfer function.

Characterizing pixels with nonlinear transfer functions using only sample statistics of the output signal is a much more challenging task. The inherent complexity of characterizing nonlinear pixels comes from the fact that we can no longer impose the very restrictive assumption that $g$ is a constant and instead must treat it as some unknown function $g(\cdot)$ that varies with illumination level. In years past, two methods for nonlinear characterization have been proposed, namely, the nonlinear compensation ({\sc nlc}) and nonlinear estimation ({\sc nle}) techniques \cite{janesick_2007,pain_2003,bohndiek:2008}. In the {\sc nlc} technique, the approach taken is to not directly measure $g$ but rather measure analogous gains for the first two moments of the photon induced signal $P$. What makes the {\sc nlc} method work is the observation that pixels with nonlinear transfer functions typically exhibit a linear response at low illumination \cite{janesick_2006,janesick_2007}. As such, if one assumes Poisson photon statistics and exposes the pixel to a sufficiently low level of illumination $\mu_{\gamma}^\ast$, the mean number of electrons collected in the pixel can be described by
\[
\mu_{\electron}^\ast=(\mu_{\mathrm p+\mathrm d}^\ast-\mu_{\mathrm d})\times g^\ast,
\]
where $\mu_{\mathrm p+\mathrm d}^\ast$ and $g^\ast$ represent the the corresponding quantities at the illumination level $\mu_\gamma^\ast$. If we then define the relative illumination level $r=\mu_\gamma/\mu_\gamma^\ast$ one can subsequently define the illuminated population parameters and mean electron signal as functions of $r$ whereby $\mu_{\electron}^\ast=\mu_{\electron}(1)$, $\mu_{\mathrm p+\mathrm d}^\ast=\mu_{\mathrm p+\mathrm d}(1)$, $\sigma_{\mathrm p+\mathrm d}^{2\ast}=\sigma_{\mathrm p+\mathrm d}^2(1)$. Using this notation the expected mean number of electrons at relative illumination level $r$ becomes
\[
\mu_{\electron}(r)=\mu_{\electron}^\ast\times r.
\]
Having an explicit expression for this quantity then allows us to define a \emph{signal gain} and \emph{noise gain} in terms of relative illumination level as
\begin{equation}
\label{eq:signal_noise_transfer_functions}
\begin{aligned}
s(r) &=\frac{\mu_{\electron}(r)}{\mu_{\mathrm p+\mathrm d}(r)-\mu_{\mathrm d}}\\
n(r) &=\sqrt{\frac{\mu_{\electron}(r)}{\sigma_{\mathrm p+\mathrm d}^2(r)-\sigma_{\mathrm d}^2}},
\end{aligned}
\end{equation}
where both functions have units of $(\electron/\mathrm{DN})$.

In practice, {\sc nlc} photon transfer characterization is performed by first identifying the linear region of the pixel and measuring $g^\ast$. Then, $s(r)$ and $n(r)$ are measured at several illumination levels. At each illumination level, $r$ is recorded and the transfer functions are estimated by according to the formulas in $(\ref{eq:signal_noise_transfer_functions})$ above. The estimated transfer functions are then fit by some curve to produce signal and noise gains, which allow one to convert sample means and standard deviations in units of digital numbers back into meaningful quantities of electrons.\\



\section{Noise model}
\label{sec:assumed_noise_model}

Dark noise $D\sim F_D$ represents the sum of many different noise sources present in the pixel and downstream circuitry and therefore is justifiably modeled as normal. We further assume that individual observations $D_i$ of dark noise are mutually independent such that the sample $\mathbf D=(D_1,\dots,D_n)^\mathsf T$ is modeled by $\mathbf D\sim\mathcal N(\mu_{\mathrm d}\mathbf 1_n,\sigma_{\mathrm d}^2\mathbf I_n)$ where $\mathbf 1_n\in\Bbb R^{n\times 1}$ denotes a column vector of ones and $\mathbf I_n\in\Bbb R^{n\times n}$ is the $n\times n$ identity matrix. Furthermore, photon induced noise $P$, although Poissonian by assumption, quickly approaches a normal approximation for even small values of $\mu_\gamma$\footnote{A typical rule of thumb is the normal approximation is useful for $\mu_{\gamma}>30$.}. Since the arrival of photons is independent of the dark noise $D$ and individual observations $P_i$ are mutually independent, the observed vector $\mathbf P+\mathbf D=((P+D)_1,\dots,(P+D)_n)^\mathsf T$  is modeled as $(\mathbf P+\mathbf D)\sim\mathcal N(\mu_{\mathrm p+\mathrm d}\mathbf 1_n,\sigma_{\mathrm p+\mathrm d}^2\mathbf I_n)$. With these underlying assumptions we will let $\mathbf X=(X_1,\dots,X_{n_1})^\mathsf T$ be a sequences of $n_1$ digital observations of a pixel under some level of illumination and $\mathbf Y=(Y_1,\dots,Y_{n_2})^\mathsf T$ be a separate sample of $n_2$ observations of the same pixel in the dark so that the joint vector is modeled by
\begin{equation}
\label{eq:pixel_noise_model}
\left(\mathbf X \atop \mathbf Y\right)\sim\mathcal N\left(%
\begin{pmatrix}\mu_{\mathrm p+\mathrm d}\mathbf 1_{n_1}\\ \mu_{\mathrm d}\mathbf 1_{n_2}\end{pmatrix},%
\begin{pmatrix}\sigma_{\mathrm p+\mathrm d}^2\mathbf I_{n_1} &\mathbf 0\\ \mathbf 0 &\sigma_{\mathrm d}^2\mathbf I_{n_2}\end{pmatrix}\right).
\end{equation}

Under the proposed model, the two-sample statistic $T(\mathbf X,\mathbf Y)=(\bar X,\hat X,\bar Y,\hat Y)^\mathsf T$, where
\[
\begin{aligned}
\bar X &=\frac{1}{n_1}\sum_{k=1}^{n_1}X_k, & \hat X &=\frac{1}{n_1-1}\sum_{k=1}^{n_1}(X_k-\bar X)^2,\\
\bar Y &=\frac{1}{n_2}\sum_{k=1}^{n_2}Y_k, & \hat Y &=\frac{1}{n_2-1}\sum_{k=1}^{n_2}(Y_k-\bar Y)^2,
\end{aligned}
\]
forms a vector of mutually independent components and constitutes a complete-sufficient statistic for the unknown parameter vector $\theta=(\mu_{\mathrm p+\mathrm d},\sigma_{\mathrm p+\mathrm d}^2,\mu_{\mathrm d},\sigma_{\mathrm d}^2)$. Normal sampling theory gives the distributional results
\[
\begin{aligned}
\bar X &\sim\mathcal N(\mu_{\mathrm p+\mathrm d},\sigma_{\mathrm p+\mathrm d}^2/n_1), & \hat X &\sim\mathcal G(\alpha_1,\alpha_1/\sigma_{\mathrm p+\mathrm d}^2),\\
\bar Y &\sim\mathcal N(\mu_{\mathrm d},\sigma_{\mathrm d}^2/n_2), & \hat Y &\sim\mathcal G(\alpha_2,\alpha_2/\sigma_{\mathrm d}^2),\\
\end{aligned}
\]
where $\alpha_i=(n_i-1)/2$. Since the shapes parameters $\alpha_i$ are directly related to the sample sizes $n_i$ in this manner, we will use the term \emph{sample size} to refer to both quantities when it is expedient to do so. Further denoting the estimator for the photon induced mean as $\bar P=\bar X-\bar Y$ also gives the sometimes useful result
\[
\bar P\sim\mathcal N(\mu_{\mathrm p},\sigma_{\mathrm p+\mathrm d}^2/n_1+\sigma_{\mathrm d}^2/n_2).
\]

We will further assume that the observed data $(\mathbf X^\mathsf T\ \mathbf Y^\mathsf T)^\mathsf T$ is produced by a pixel with a linear transfer function or equivalently is produced in the linear region of a pixel with nonlinear transfer function. This assumption implies $\mu_{\mathrm p+\mathrm d}=\mu_{\mathrm d}+\mu_{\electron}/g$ and $\sigma_{\mathrm p+\mathrm d}^2=\sigma_{\mathrm d}^2+\mu_{\electron}/g^2$ where $g$ is the conversion gain given by $(\ref{eq:general_g_relationship})$. Since $\mu_{\electron}\geq 0$ it follows that $\mu_{\mathrm p+\mathrm d}\geq \mu_{\mathrm d}>0$ and $\sigma_{\mathrm p+\mathrm d}^2\geq\sigma_{\mathrm d}^2>0$ with equality reached at zero illumination. Throughout the remaining sections, we will also use the notation of Chapter \ref{chap:estimation_of_recip_diff_norm_variances} to denote the ratio of dark and illuminated variances by $\zeta=\sigma_{\mathrm d}^2/\sigma_{\mathrm p+\mathrm d}^2$ where $\zeta\in[0,1]$, $\zeta=1$ is achieved at zero illumination, and $\zeta=0$ is achieved at infinite illumination which is the shot noise limit\footnote{$\zeta=0$ can be treated as the mathematical definition of the shot noise limit. In practice, the term  \emph{shot noise limited} is a subjective term meant to describe a pixel that can achieve a $\zeta$-value near zero before saturating.}.


\section{The estimator $\mathscr G_\nu$}
\label{sec:the_est_Gv}

We are finally ready to present our new estimator for the photon transfer conversion gain. Our first result shows that under the proposed noise model, if an unbiased estimator for $g$ exists, it must have infinite variance on at least a portion of the $\zeta$-domain.

\begin{theorem}[Corollary of Theorem \ref{thm:no_finite_variance_estimator}]
Under the normal model of pixel noise in $(\ref{eq:pixel_noise_model})$, if an unbiased estimator of $g$ exists it must have infinite variance for at least $\zeta\in(1/2,1]$.
\end{theorem}

\begin{proof}
Assuming the model $(\ref{eq:pixel_noise_model})$, $T(\mathbf X,\mathbf Y)=(\bar P,\hat X,\hat Y)^\mathsf T$ is a complete-sufficient statistic for the parameter $\theta=(\mu_{\mathrm p},\sigma_{\mathrm p+\mathrm d}^2,\sigma_{\mathrm d}^2)^\mathsf T$. Since the components of $T$ are mutually independent and $\ev\bar P=\mu_{\mathrm p}$, it follows that if unbiased estimator of $g$ exists it will be of the form $\mathscr G(T)=\bar P\times\mathscr T(\hat X,\hat Y)$ with $\mathscr T(\hat X,\hat Y)$ denoting an unbiased estimator of $(\sigma_{\mathrm p+\mathrm d}^2-\sigma_{\mathrm d}^2)^{-1}$. Again making use of the independence of $\bar P$ and $\mathscr T$ we then write
\[
\var\mathscr G=(\ev\bar P^2)\var\mathscr T+(\ev\mathscr T)^2\var\bar P.
\]
But according to Theorem \ref{thm:no_finite_variance_estimator}, if the estimator $\mathscr T$ exists it must have infinite variance on at least $\zeta\in(1/2,1]$. Since $\mathscr G$ is an {\sc umvue} for $g$ and $\var\mathscr T<\var\mathscr G$ the desired result immediately follows.
\end{proof}

Given that there is no unbiased estimator for $g$ that can achieve finite variance on the entire $\zeta$-domain we again turn to biased estimation. The following theorem presents a biased estimator based on the estimator $\mathscr T_\nu$ derived in Chapter \ref{chap:estimation_of_recip_diff_norm_variances}.

\begin{proposition}
\label{prop:parameter_relations}
Under the assumed model we have $\sigma_{\mathrm p+\mathrm d}^2=\sigma_{\mathrm d}^2/\zeta$, $\sigma_{\mathrm p}^2=\sigma_{\mathrm d}^2\zeta^{-1}(1-\zeta)$, and $\mu_{\mathrm p}=\sigma_{\mathrm d}^2\zeta^{-1}(1-\zeta)g$.
\end{proposition}

\begin{proof}
The first two relationships are derived from $\sigma_{\mathrm p+\mathrm d}^2=\sigma_{\mathrm p}^2+\sigma_{\mathrm d}^2$ and $\zeta=\sigma_{\mathrm d}^2/\sigma_{\mathrm p+\mathrm d}^2$. To derive the relationship for $\mu_{\mathrm p}$ we use $\mu_{\mathrm p}=\mu_{\electron}/g$ and $\sigma_{\mathrm p}^2=\mu_{\electron}/g^2$ to write $\mu_{\mathrm p}=\sigma_{\mathrm p}^2g$.
\end{proof}

\begin{theorem}
\label{thm:new_g_estimator}
Let $T=(\bar P,\hat X,\hat Y)^\mathsf T$ and
\[
\mathscr G_\nu(T)=\bar P\times\mathscr T_\nu(\hat X,\hat Y),\quad\nu>0.
\]
Then, $\mathscr G_\nu$ is the unique uniformly minimum variance unbiased estimator of
\[
\ev \mathscr G_\nu=(1-\zeta^\nu)\, g,
\]
with
\[
\var\mathscr G_\nu=\frac{\sigma_{\mathrm d}^2}{\zeta}\left(\frac{1}{n_1}+\frac{\zeta}{n_2}+\frac{(\sigma_{\mathrm d}g)^2}{\zeta}(1-\zeta)^2\right)\ev\mathscr T_\nu^2-(1-\zeta^\nu)^2g^2
\]
and $\ev\mathscr T_\nu^2$ is given by Theorem \ref{thm:T_est_variance_series} for $\kappa_1=\sigma_{\mathrm d}^2/\zeta$ and $\alpha_i=(n_i-1)/2$.
\end{theorem}

\begin{proof}
First, note that $\sigma_{\mathrm p+\mathrm d}^2>\sigma_{\mathrm d}^2$ for any nonzero illumination level; thus, by Remark \ref{rmk:Tv_as_estimator_for_tauv} we must require $\nu>0$ to guarantee $\arb\mathscr T_\nu<1$. Next, for the proposed normal model $T(\mathbf X,\mathbf Y)=(\bar X,\bar Y,\hat X,\hat Y)$ is a complete-sufficient statistic for the parameter $\theta=(\mu_{\mathrm p+\mathrm d},\mu_{\mathrm d},\sigma_{\mathrm p+\mathrm d}^2,\sigma_{\mathrm d}^2)$; thus by the Lehmann-Scheff\' e theorem $\mathscr G_\nu$ is the unique {\sc umvue} of its expected value. To evaluate the expected value, we use the independence of $\bar P$ and $\mathscr T_\nu(\hat X,\hat Y)$ to write $\ev \mathscr G_\nu=(\ev \bar P)(\ev\mathscr T_\nu)$. We know $\ev\bar P=\mu_{\mathrm p}$ and $\ev\mathscr T_\nu$ is given in Theorem \ref{thm:T_estimator}. Using the formula for $g$ in $(\ref{eq:general_g_relationship})$ then gives the desired result for $\ev \mathscr G_\nu$. To obtain the variance we again use independence of $\bar P$ and $\mathscr T_\nu$ to write
\[
\var\mathscr  G_\nu=(\var\bar P+(\ev\bar P)^2)(\ev\mathscr T_\nu^2)-(1-\zeta^\nu)^2g^2.
\]
Recalling $\bar P\sim\mathcal N(\mu_{\mathrm p},\sigma_{\mathrm p+\mathrm d}^2/n_1+\sigma_{\mathrm d}^2/n_2)$ and the relations of Proposition \ref{prop:parameter_relations} then leads to the desired result.
\end{proof}

In the following analysis we will derive many results involving the coefficient of variation of $\mathscr G_\nu$. Given the restrictions $\nu>0$ and $\zeta\in(0,1)$ we know that $\ev\mathscr G_\nu>0$ so that $\cv\mathscr G_\nu=\acv\mathscr G_\nu$.  As such, to maintain some level of congruence between the following results and those of Chapter \ref{chap:estimation_of_recip_diff_norm_variances} we will use these quantities interchangabley when its convenient to do so. We now show that $\arb\mathscr G_\nu$ and $\acv\mathscr G_\nu$ are strictly increasing functions of $\zeta$.

\begin{proposition}
\label{prop:ARBGv_ACVGv_monotonicity}
$\arb\mathscr G_\nu$ is a strictly increasing function of $\zeta$. Furthermore, if $\nu>1$ then $\acv\mathscr G_\nu$ is also a strictly increasing function of $\zeta$.
\end{proposition}

\begin{proof}
Given the definition of the absolute relative bias we have
\[
\arb\mathscr G_\nu=\left\lvert\frac{(1-\zeta^\nu)g-g}{g}\right\rvert=\zeta^\nu.
\]
Since $\nu>0$ it immediately follows that $\partial_\zeta\arb\mathscr G_\nu>0$ on $\zeta\in(0,1)$. Next, we use the independence of $\bar P$ and $\mathscr T_\nu$ to write
\begin{equation}
\label{eq:CV2Gv_expanded}
\acv^2\mathscr G_\nu=\cv^2\mathscr T_\nu+(\cv^2\mathscr T_\nu)(\cv^2\bar P)+\cv^2\bar P
\end{equation}
so that
\begin{equation}
\label{eq:CV2Gv_zeta_derivative_expanded}
\partial_\zeta\acv^2\mathscr G_\nu=(\cv^2\bar P+1)(\partial_\zeta\cv^2\mathscr T_\nu)+(\cv^2\mathscr T_\nu+1)(\partial_\zeta\cv^2\bar P).
\end{equation}
Using the relations in Proposition \ref{prop:parameter_relations} we have
\begin{equation}
\label{eq:CV2P_zeta_form}
\cv^2\bar P=\frac{1}{(\sigma_{\mathrm d}g)^2}\frac{\zeta}{(1-\zeta)^2}\left(\frac 1{n_1}+\frac{\zeta}{n_2}\right),
\end{equation}
which is easily shown to be decreasing in $\zeta$ on $\zeta\in(0,1)$\footnote{In $(\ref{eq:CV2P_zeta_form})$ we see the quantity $\sigma_{\mathrm d}g$ which is the sensor dark noise in units of electrons. This quantity represents the sum of read noise and dark current noise present in the pixel and is a function of integration time.}. Likewise, for $\nu>1$ we know from Lemma \ref{lem:monotonicity_of_ACV_T_v} that $\partial_\zeta\acv^2\mathscr T_\nu>0$ which implies all terms in $(\ref{eq:CV2Gv_zeta_derivative_expanded})$ are positive. Noting that $\acv\mathscr G_\nu$ is positive completes the proof.
\end{proof}

In the following Lemma we establish a property of $\mathscr G_\nu$ that agrees with a long standing observation about estimators of $g$ in the literature, that is, the dispersion of $\mathscr G_\nu$ is dominated by the dispersion of the estimator for $(\sigma_{\mathrm p+\mathrm d}^2-\sigma_{\mathrm d}^2)^{-1}$ as the illumination increases. To accomplish this we will consider the quantity
\[
\mathscr E=\frac{\cv^2\mathscr T_\nu}{\cv^2\mathscr G_\nu}.
\]
This ratio is useful for studying the dominance of $\var\mathscr G_\nu$ by the dispersion of $\mathscr T_\nu$ for a couple reasons. First, using the definition of the coefficient of variation and independence of $\bar P$ and $\mathscr T_\nu(\hat X,\hat Y)$ we are able to write
\[
\cv^2\mathscr G_\nu=\cv^2\mathscr T_\nu\left(1+(1+\cv^{-2}\mathscr T_\nu)\cv^2\bar P\right),
\]
and thus
\begin{equation}
\label{eq:E_ratio_expression}
\mathscr E=\left[1+(1+\cv^{-2}\mathscr T_\nu)\cv^2\bar P\right]^{-1}.
\end{equation}
Since the quantity inside the brackets is bounded below by one it follows that $0<\mathscr E<1$ and so $\mathscr E$ gives an intuitive measure of dominance. Second, if we define the random variable $\mathscr G_\nu^\ast=\mu_{\mathrm p}\mathscr T_\nu(\hat X,\hat Y)$ which estimates $g$ when $\mu_{\mathrm p}$ is known, i.e.~has zero variance, then $\mathscr E$ is equivalent to
\[
\mathscr E=\frac{\var\mathscr G_\nu^\ast}{\var\mathscr G_\nu}.
\]
So in this context we see that $\mathscr E\approx 1$ corresponds to $\bar P$ behaving like a fixed constant in comparison to $\mathscr T_\nu$; hence, $\var\mathscr G_\nu$ is completely dominated by the dispersion of $\mathscr T_\nu$.

\begin{lemma}
\label{lem:convergence_of_coeff_of_var}
As the illumination level increases $\mathscr E\searrow 1$.
\end{lemma}

\begin{proof}
We have already shown that $\mathscr E<1$; thus, if it approaches one then it must do so from below. All that is left is to show that $\mathscr E=1$ in the limit of infinite illumination. Our starting point is the expression for $\mathscr E$ in $(\ref{eq:E_ratio_expression})$, namely,
\[
\mathscr E=\left(1+(1+\cv^{-2}\mathscr T_\nu)\cv^2\bar P\right)^{-1}.
\]
As the illumination level increases without bound, $\zeta\to 0$ and from Lemma \ref{lem:cvTv_boundary_values}
\[
\cv^2\mathscr T_\nu\to%
\begin{cases}
\frac{2}{n_1-5}, &n_1>5\\
\infty, &n_1\leq 5.
\end{cases}
\]
Hence, $1+\cv^{-2}\mathscr T_\nu$ approaches a positive and finite constant in the limit. Next we turn our attention to the $\cv^2\bar P$ term. Using $(\ref{eq:CV2P_zeta_form})$ it is immediately obvious that $\cv^2\bar P\to 0$ as $\zeta\to 0$. Thus, as the illumination increases without bound $(1+\cv^{-2}\mathscr T_\nu)\cv^2\bar P\to 0$ which completes the proof.
\end{proof}

\begin{theorem}
\label{thm:acvTv_converges_to_acvGv}
For the estimator $\mathscr G_\nu$ in Theorem \ref{thm:new_g_estimator}
\[
\arb\mathscr G_\nu=\arb\mathscr T_\nu
\]
and as illumination increases
\[
\acv\mathscr G_\nu\sim\acv \mathscr T_\nu.
\]
In particular, if $\alpha_1+\alpha_2>2$, $\alpha_1>2\nu$, and $\nu>1$ then as illumination increases, $\zeta\nearrow 0$ and
\[
\acv\mathscr G_\nu\sim\acv \mathscr T_\nu\left(1+\frac{n_1-3}{4n_1}\frac{\zeta}{(\sigma_{\mathrm d}g)^2}+\mathcal O(\zeta^2)\right).
\]
\end{theorem}

\begin{proof}
The first two claims follow directly from Proposition \ref{prop:ARBGv_ACVGv_monotonicity} and Lemma \ref{lem:convergence_of_coeff_of_var}, respectively. For the last claim we assume $\nu>1$ and use $(\ref{eq:gTilde_asym_form})$ to deduce
\[
\tilde g^2_{k,\ell}(\zeta,\nu)\sim%
\begin{cases}
(k!)^2+\mathcal O(\zeta), &k=\ell\\
\mathcal O(\zeta^2), &k>\ell .
\end{cases}
\]
If in addition $\alpha_1+\alpha_2>2$ and $\alpha_1>2\nu$ then $\cv\mathscr T_\nu$ is finite and it follows that as $\zeta\nearrow 0$
\[
\cv^2\mathscr T_\nu\sim\sum_{k=1}^\infty \frac{k!}{(\alpha_1)_k}+\mathcal O(\zeta)=\frac{1}{\alpha_1-2}+\mathcal O(\zeta)
\]
and
\[
1+\cv^{-2}\mathscr T_\nu\sim \alpha_1-1+\mathcal O(\zeta)=\frac{n_1-3}{2}+\mathcal O(\zeta).
\]
Multiplying this result by
\[
\cv^2\bar P\sim\frac{1}{(\sigma_{\mathrm d}g)^2}\frac{\zeta}{n_1}+\mathcal O(\zeta^2)
\]
and using $\sqrt{1+f(x)}\sim 1+\frac{1}{2}f(x)+\mathcal O(f^2(x))$ as $f(x)\to 0$ then gives
\[
\mathscr E^{-1/2}=\frac{\acv\mathscr G_\nu}{\acv\mathscr T_\nu}\sim 1+\frac{n_1-3}{4n_1}\frac{\zeta}{(\sigma_{\mathrm d}g)^2}+\mathcal O(\zeta^2),
\]
which is the desired result.
\end{proof}

The main conclusion of Theorem \ref{thm:acvTv_converges_to_acvGv} is that confidence intervals for $\arb\mathscr T_\nu$ are equal to those of $\arb\mathscr G_\nu$ and that for sufficiently high illumination confidence intervals for $\acv\mathscr T_\nu$ may be used to approximate those of $\acv\mathscr G_\nu$. That said, these results do not give any indication as to how high the illumination level must be to achieve a good approximation or how $\acv\mathscr G_\nu$ compares to $\acv\mathscr T_\nu$ under low illumination when taking into account other parameters like the dark noise $\sigma_{\mathrm d}$, conversion gain $g$, and sample sizes. We will return to this problem later and for now will be content with knowing that given a sufficiently high illumination level $\acv\mathscr T_\nu\approx\acv\mathscr G_\nu$.  


\section{A demonstration of $g$-estimation with confidence intervals}
\label{sec:g_estimation_demo}

With Theorem \ref{thm:acvTv_converges_to_acvGv} at hand, we now bring together several of the results derived thus far and demonstrate the process of estimating $g$ in the context of the photon transfer method. Using the parameter values in Table \ref{tbl:simulation_parameters}, $N=10^6$ pseudo-random observations of $T=(\bar P,\hat X,\hat Y)$ were generated. The parameter values chosen represent what one might typically see in the photon transfer method. Since the sample sizes used are large we use $(\ref{eq:Tv_asym_hyper_form})$ to derive the \emph{first order} approximation
\[
\mathscr G_{\nu,1}(T)=\bar P\times\mathscr T_{\nu,1}(\hat X,\hat Y),
\]
where
\begin{multline*}
\mathscr T_{\nu,1}(\hat X,\hat Y)=\frac{1}{\hat Y}\Biggl(
\left(1-\frac{1}{\alpha_1}\right)\nu\pFq{}{}{1,1+\nu}{2}{1-\frac{\hat X}{\hat Y}}\\
-\frac{1}{\alpha_1}(\nu)^{(2)}\pFq{}{}{2,1+\nu}{3}{1-\frac{\hat X}{\hat Y}}%
-\frac{\alpha_1+\alpha_2}{6\alpha_1\alpha_2}(\nu)^{(3)}\pFq{}{}{3,1+\nu}{4}{1-\frac{\hat X}{\hat Y}}%
\Biggr)
\end{multline*}
and $\alpha_i=(n_i-1)/2$. Note that the presence of integers in the top and bottom parameters of the hypergeometric terms mean they reduce to elementary functions of the argument $1-\hat X/\hat Y$.

Computation of the confidence intervals was done by first computing $Z_\alpha=\hat Y/\hat X\, F_{\alpha,2\alpha_1,2\alpha_2}$ for each pair of $(\hat X,\hat Y)$ and confidence level $\alpha=0.05$. The $95\%$ confidence intervals for $\arb\mathscr G_\nu$ were then computed by substituting the values for $Z_\alpha$ into Theorem \ref{thm:ARB_T_v_CI}. As for the confidence intervals of $\acv\mathscr G_\nu$, a quick computation of $\mathscr E$ using the parameters in Table \ref{tbl:simulation_parameters} shows that despite the relatively low level of illumination used in the simulation $\mathscr E=0.998685\dots$ so that we may approximate these intervals with those for $\mathscr T_\nu$ in Theorem \ref{thm:ACV_T_v_CI}. The relatively large values used for $\alpha_1$ and $\alpha_2$ means computation of these interval estimates will have  to be done with the series expansion for $\acv\mathscr T_\nu$ which introduces the practical problem of truncation error. If we define
\[
\acv_n\mathscr T_\nu=\left(\sum_{k=1}^n\sum_{\ell=0}^k\frac{\tilde g_{k,\ell}^2(z,\nu)}{(\alpha_1)_\ell(\alpha_2)_{k-\ell}\,\ell!\,(k-\ell)!}\right)^{1/2},
\]
then the absolute relative error incurred in using this truncated expansion is given by
\[
R_{n,1}=\left\lvert\frac{\acv_n\mathscr T_\nu}{\acv\mathscr T_\nu}-1\right\rvert.
\]
Corollary \ref{cor:rel_truncation_error_bound} (see Appendix \ref{sec:trunc_error_estimate}) provides an upper bound, $R_{n,m,1}^\star$, for $R_{n,1}$ so that one may compute the number of terms $n=K^\star$ that guarantees $R_n<\epsilon$ according to
\[
K^\star=\min\{K|R_{K,m,1}^\star\leq\epsilon\land K\in\Bbb N\}.
\]
Note that the quantity $R_{n,m,1}^\star$ in Corollary \ref{cor:rel_truncation_error_bound} provides a single upper bound for the entire set of $Z_\alpha$ values so that $K^\star$ is the appropriate number of terms needed in computing all $10^6$ interval estimates within the specified error tolerance. The approximate interval estimates for $\acv\mathscr G_\nu$ may then be computed by substituting $\acv_{K^\star}\mathscr T_\nu$ and $Z_\alpha$ into Theorem \ref{thm:ACV_T_v_CI}.

Table \ref{tbl:simulation_results} presents the results of the simulation. Using the simulated values for $Z_\alpha$ and choosing $m=0$ and an absolute relative truncation error tolerance $\epsilon=5\times 10^{-4}$ led to only needing the first $K^\star=3$ terms in the series expansion for $\acv\mathscr T_\nu$. Substituting $K^\star$ into $R_{n,0,1}^\star$ subsequently showed that using this number of terms guaranteed the relative truncation error was less than $2.7\times 10^{-4}$ for all $10^6$ values of $Z_\alpha$. Additionally, note the relative difference between the exact and estimated values for $\ev\mathscr G_\nu$ and $\var\mathscr G_\nu$ are
\[
R_{\ev\mathscr G_\nu}=\left\lvert\frac{\widehat{\ev}\mathscr G_{\nu,1}}{\ev\mathscr G_\nu}-1\right\rvert\times 100\%=0.001\%
\]
and
\[
R_{\var\mathscr G_\nu}=\left\lvert\frac{\widehat{\var}\mathscr G_{\nu,1}}{\var\mathscr G_\nu}-1\right\rvert\times 100\%=0.24\%,
\]
which indicates no significant error was incurred from using the asymptotic approximation $\mathscr G_{\nu,1}$ in place of $\mathscr G_\nu$. Lastly, we see that the estimated coverage probability for the $95\%$ confidence intervals for $\arb\mathscr G_\nu$ and $\acv\mathscr G_\nu$ agree with the target values, the latter of which is due to the fact that $\mathscr E\approx 1$ for the chosen parameters.

This exercise raises additional questions as no instruction is given on the appropriate values for $\nu$, $n_1$ and $n_2$ to achieve desired values of $\arb\mathscr G_\nu$ and $\acv\mathscr G_\nu$. In particular, for real experiments, gathering large numbers of observations takes time and so knowledge of \emph{optimal} sample sizes would be useful in reducing the time required to gather data.  We will take a closer look at this problem in the following section.

\newcolumntype{Z}{>{$}c<{$}}
\begin{table}[p]
\centering
\renewcommand{\arraystretch}{1.5}
\begin{tabular}{ZZZZ}
\hline
\text{\bf Parameter} &\text{\bf Equation} &\text{\bf Value} &\text{\bf Unit}\\ \hline
\alpha &- &0.05 &-\\
n_1 &- &3001 &-\\
n_2 &- &1501 &-\\
g &- &5 &\electron/\mathrm{DN}\\
\nu &- &\frac{1}{2}e^\pi &-\\
\mu_{\electron} &- &150 &\electron\\
\mu_{\mathrm d} &- &10 &\mathrm{DN}\\
\sigma_{\mathrm d}^2 &- &16 &\mathrm{DN}^2\\
\mu_{\mathrm p+\mathrm d} &\mu_{\mathrm d}+\mu_{\electron}/g &40 &\mathrm{DN}\\
\sigma_{\mathrm p+\mathrm d}^2 &\sigma_{\mathrm d}^2+\mu_{\electron}/g^2 &22 &\mathrm{DN}^2\\
\zeta &\sigma_{\mathrm d}^2/\sigma_{\mathrm p+\mathrm d}^2 &\frac{8}{11} &-\\
\mu_{\bar P} &\mu_{\mathrm p+\mathrm d}-\mu_{\mathrm d} &30 &\mathrm{DN}\\
\sigma_{\bar P}^2 &\sigma_{\mathrm p+\mathrm d}^2/n_1+\sigma_{\mathrm d}^2/n_2 &\frac{81038}{4504501} &\mathrm{DN}^2\\
\alpha_1 &(n_1-1)/2 &1500 &-\\
\alpha_2 &(n_2-1)/2 &750 &-\\
\beta_1 &\alpha_1/\sigma_{\mathrm p+\mathrm d}^2 &\frac{750}{11} &\mathrm{DN}^{-2}\\
\beta_2 &\alpha_2/\sigma_{\mathrm d}^2 &\frac{375}{8} &\mathrm{DN}^{-2}\\[0.5em]
\hline\hline
\end{tabular}
\renewcommand{\arraystretch}{1}
\caption{Simulation parameters.}
\label{tbl:simulation_parameters}
\end{table}

\newcolumntype{Z}{>{$}c<{$}}
\begin{table}[p]
\centering
\renewcommand{\arraystretch}{1.5}
\begin{tabular}{ZZZZ}
\hline
\text{\bf Quantity} &\text{\bf Exact Value} &\text{\bf Estimated Value} &\text{\bf Unit}\\ \hline
K^\star &3 &- &-\\
R_{K^\star,0,1}^\star &0.02671 &- &\%\\
\ev\mathscr G_\nu &4.87447 &4.87440 &\electron/\mathrm{DN}\\
\var\mathscr G_\nu &0.36671 &0.36759 &(\electron/\mathrm{DN})^2\\
\pr(\operatorname{CI}_\alpha(\arb\mathscr T_\nu)\ni\arb\mathscr G_\nu) &0.95 &0.95008 &-\\
\pr(\operatorname{CI}_\alpha(\acv\mathscr T_\nu)\ni\acv\mathscr G_\nu) &\sim 0.95 &0.94945 &-\\[0.5em]
\hline\hline
\end{tabular}
\renewcommand{\arraystretch}{1}
\caption{Simulation results based on $N=10^6$ pseudo-random observations.}
\label{tbl:simulation_results}
\end{table}


\section{Optimal sample sizes}
\label{sec:cv_analysis}

When discussing the notion of optimal sample sizes we must first begin with a definition of what is optimal. Typically, experimenters wish to achieve the smallest possible uncertainty with the fewest number of total observations. This goal is manifested in the following definition.
\begin{definition}[Optimal sample sizes]
\label{def:optimal_sample_sizes}
Let $\mathbf X_j=(X_{i,j},\dots,X_{n_j,j})$ with $X_{i,j}\overset{\mathrm{iid}}{\sim}F_j(\theta_j)$ denote a random sample of size $n_j$ drawn from the distribution $F_j$ with parameters $\theta_j$. Furthermore, let $T(\mathbf X)$ denote an estimator based on $M$ random samples $\mathbf X=(\mathbf X_1,\dots,\mathbf X_M)$ and $\acv T(n|\theta_T)$ denote its absolute coefficient of variation as a function of the sample sizes $n=(n_1,\dots, n_M)$ with parameters $\theta_T=\cup_j\theta_j$. Then the optimal sample sizes are defined as the vector-valued function $n^{\normalfont\text{opt}}(\mathrm N,\theta_T)=(n_1^{\normalfont \text{opt}},\dots,n_M^{\normalfont \text{opt}})(\mathrm N,\theta_T)$ that minimizes $\acv T(n|\theta_T)$ subject to $\mathrm N=n_1+\cdots+n_M$ for some fixed $\mathrm N>0$.
\end{definition}

In the case where there are two samples, so that $n=(n_1,n_2)$, this definition of optimality requires solving the single-variable optimization problem
\begin{subequations}
\label{eq:optimal_SOEs_parameterized_in_N}
\begin{align}
n_2^{\text{opt}}(\mathrm N,\theta_T) &=\arginf_{n_2\in[0,\mathrm N]}\acv^2 T(\mathrm N-n_2,n_2|\theta_T) \label{eq:optim_eq1_Nparam}\\
n_1^{\text{opt}}(\mathrm N,\theta_T) &=\mathrm N-n_2^{\text{opt}}(\mathrm N,\theta_T) \label{eq:optim_eq2_Nparam},
\end{align}
\end{subequations}
for some total number of samples $\mathrm N$ and parameters $\theta_T$. Note that while this formulation does give the optimal sample sizes according to definition \ref{def:optimal_sample_sizes}, it does not allow one to specify a predetermined value (desired outcome) for $\acv T$. Since it may be desirable to specify the outcome of $\acv T$ instead of $\mathrm N$ we can introduce the constraint $\acv T(n_1^\text{opt},n_2^\text{opt}|\theta_T)=\sacv_0$ and reformulate (\ref{eq:optimal_SOEs_parameterized_in_N}) into the equivalent two-variable problem
\begin{subequations}
\label{eq:optimal_SOEs}
\begin{align}
\inf_{n_2\in[0,\mathrm N]} &\acv^2 T(\mathrm N-n_2,n_2|\theta_T)\big|_{(\mathrm N,n_2)=(n_1^\text{opt}+n_2^\text{opt},n_2^\text{opt})} \label{eq:optim_eq1}\\
&\acv^2 T(n_1^\text{opt},n_2^\text{opt}|\theta_T)-\sacv_0^2=0 \label{eq:optim_eq2},
\end{align}
\end{subequations}
which can then be solved for $(n_1^\text{opt},n_2^\text{opt})$ in terms of $\sacv_0$ and $\theta_T$.

Before proceeding, let's consider what this system of equations says and how we solve it. Three cases must be considered which we shall denote as: \emph{nondegenerate} ($n_1^\text{opt},n_2^\text{opt}\neq 0$), \emph{weakly degenerate} ($n_1^\text{opt}=0$ or $n_2^\text{opt}=0$ but not both), and \emph{degenerate} ($n_1^\text{opt},n_2^\text{opt}=0$). For the nondegenerate case we begin with equation (\ref{eq:optim_eq1}) by fixing the total number of samples to $\mathrm N$ and then minimizing $\acv^2T$ in $n_2$ subject to this constraint. This process gives the optimal sample size $n_2^\text{opt}$ in terms of $\mathrm N$ and $\theta_T$\footnote{In fact, if we were to stop here equation (\ref{eq:optim_eq1}) would be identical to (\ref{eq:optim_eq1_Nparam})}, which upon substituting $\mathrm N=n_1^\text{opt}+n_2^\text{opt}$ yields an equation representing the optimal relationship, a.k.a.~the \emph{optimality relation}, between the two sample sizes. Solving the optimality relation for either $n_1^\text{opt}$ or $n_2^\text{opt}$ and then substituting into (\ref{eq:optim_eq2}) we then obtain one of the optimal sample sizes parameterized in terms of $\sacv_0$ and $\theta_T$. This solution may then be used in conjunction with the optimality relation to obtain the solution for the remaining optimal sample size. In the case of weak degeneracy, the infimum of $\acv^2T(\mathrm N-n_2,n_2|\theta_T)$ occurs at the boundary $n_2=0,\mathrm N$ and so our two-sample problem reduces to that of finding a single optimal sample size via (\ref{eq:optim_eq2}). Lastly, the case of degeneracy results when $\acv T\to 0$ and so zero samples are needed to achieve the desired outcome $\acv T=\sacv_0$. As will be seen in the latter analysis, this situation can occur at boundary points of the parameter space and thus be treated as a limit of the nondegenerate case.


\subsection{Asymptotic properties at low illumination}
\label{subsec:characteristics_at_low_illumination}

One of the unanswered curiosities of Section \ref{sec:g_estimation_demo} was how the quantity $\mathscr E$ behaves at low illumination w.r.t.~the parameters $(\sigma_{\mathrm d},g,n_1,n_2)$. If we fix these parameters and consider what only happens as the illumination decreases, i.e.~as $\zeta\searrow 1$, we see that $\cv\bar P\to\infty$ while $\cv\mathscr T_\nu$ approaches a constant and thus $\mathscr E\to 0$. From this observation one might conclude that the only way to achieve $\mathscr E$-values near unity is to perform measurements at high illumination; thus, extinguishing the possibility of approximating confidence intervals and optimal sample sizes for $\mathscr G_\nu$ with those of $\mathscr T_\nu$ under low illumination conditions. However, if we instead consider $\mathscr E$ as a function of the optimal samples sizes for $\mathscr T_\nu$ or $\mathscr G_\nu$, which vary with the illumination level, the limiting behavior at low illumination changes drastically. The following analysis takes a look at the behavior of the optimal sample sizes for $\bar P$, $\mathscr T_\nu$, and $\mathscr G_\nu$ under low illumination conditions. The main results are found in Lemma \ref{lem:ACVbarP_opt_sample_sizes}, Theorem \ref{thm:Tv_opt_n_limit}, and Corollary \ref{cor:Gv_opt_n_limit}, which show for each estimator that the optimal samples sizes are asymptotically equal and proportional to $(1-\zeta)^{-2}$ as $\zeta\searrow 1$. We will then conclude this section by using these asymptotic results in Corollary \ref{cor:E_with_opt_samp_sizes} to study how $\mathscr E$ behaves at low illumination when subject to optimal sample sizes.

\begin{lemma}
\label{lem:ACVbarP_opt_sample_sizes}
Let $(n_1^{{\normalfont\text{opt}}},n_2^{{\normalfont\text{opt}}})$ denote the optimal sample sizes for $\bar P$ that also satisfy $\acv\bar P(n_1^{{\normalfont\text{opt}}},n_2^{{\normalfont\text{opt}}}|\zeta,\sigma_{\mathrm d},g)=\sacv_0$ for some fixed $\sacv_0\in\Bbb R^+$. Then as illumination decreases, $\zeta\searrow 1$, $n_2^{{\normalfont\text{opt}}}/n_1^{{\normalfont\text{opt}}}\to 1$, and $n_i^{{\normalfont\text{opt}}}\sim C_{\bar P}(1-\zeta)^{-2}$ where $C_{\bar P}=2((\sigma_{\mathrm d}g)^2\sacv_0^2)^{-1}$.
\end{lemma}

\begin{proof}
For the sake of brevity, we begin by letting $a=\frac{1}{(\sigma_{\mathrm d}g)^2}\frac{\zeta}{(1-\zeta)^2}$ and writing
\[
\acv^2\bar P(n_1,n_2|\zeta,\sigma_{\mathrm d},g)=a\left(\frac{1}{n_1}+\frac{\zeta}{n_2}\right).
\]
If $\zeta\in(0,1)$, then $a\neq 0$ and $\acv^2\bar P(\mathrm N-n_2,n_2|\theta_{\bar P})$ is smooth and strictly convex in $n_2$ on $(0,\mathrm N)$. Additionally, the limiting value of $\acv^2\bar P(\mathrm N-n_2,n_2|\theta_{\bar P})$ at the endpoints $n_2=0,\mathrm N$ is infinite and so we know our optimization problem is nondegenerate. These observations lead us to conclude that $(\ref{eq:optimal_SOEs})$ can be uniquely solved by
\[
\begin{aligned}
\partial_{n_2}\acv^2 \bar P(\mathrm N-n_2,n_2|\theta_{\bar P})\Big|_{(\mathrm N,n_2)=(n_1^\text{opt}+n_2^\text{opt},n_2^\text{opt})} &=0\\
\acv^2 \bar P(n_1^\text{opt},n_2^\text{opt}|\theta_{\bar P})-\sacv_0^2 &=0,
\end{aligned}
\]
which upon substituting appropriate values produces the system of equations
\[
\begin{aligned}
&(1) &a\left(\frac{1}{(n_1^\text{opt})^2}-\frac{\zeta}{(n_2^\text{opt})^2}\right) &=0\\
&(2) &a\left(\frac{1}{n_1^\text{opt}}+\frac{\zeta}{n_2^\text{opt}}\right)-\sacv_0^2 &=0.
\end{aligned}
\]
Upon inspection, $(1)$ gives us the optimality relation $n_2^\text{opt}/n_1^\text{opt}=\sqrt\zeta$ and so as $\zeta\searrow 1$ we have $n_2^\text{opt}/n_1^\text{opt}\to 1$ which is the first claim. Substituting the optimality relation into $(2)$ and solving for $n_2^\text{opt}$ further yields $n_2^\text{opt}=a/\sacv_0^2(\sqrt\zeta+\zeta)$. Making use of this result and again calling on the optimality relation we have after reintroducing $a$:
\[
\begin{aligned}
n_1^{\text{opt}} &=\frac{1}{(\sigma_{\mathrm d}g)^2\sacv_0^2}\frac{\zeta(1+\sqrt\zeta)}{(1-\zeta)^2}\\
n_2^{\text{opt}} &=\frac{1}{(\sigma_{\mathrm d}g)^2\sacv_0^2}\frac{\zeta(\sqrt\zeta+\zeta)}{(1-\zeta)^2}.
\end{aligned}
\]
Now as $\zeta\searrow 1$, $\zeta(1+\sqrt\zeta)\sim\zeta(\sqrt\zeta+\zeta)\sim 2+\mathcal O(1-\zeta)$; hence,
\[
n_i^{\text{opt}}\sim\frac{2}{(\sigma_{\mathrm d}g)^2\sacv_0^2}(1-\zeta)^{-2}+\mathcal O((1-\zeta)^{-1}),\quad i=1,2
\]
which proves the second claim. Recall that we assumed $\zeta\in(0,1)$ so as to render our optimization problem nondegenerate. However, one can obtain the optimal sample sample sizes at the endpoints $\zeta=1$ and $\zeta=0$ by considering the appropriate limits, the latter of which gives us the degenerate case $(n_1^\text{opt},n_2^\text{opt})=(0,0)$. The proof is now complete.
\end{proof}

In Lemma \ref{lem:ACVbarP_opt_sample_sizes} we studied the optimal sample sizes for $\bar P$ that satisfied Definition \ref{def:optimal_sample_sizes} and the constraint $\acv\bar P=\sacv_0$. Unlike $\bar P$, the estimator $\mathscr T_\nu$ is biased so we will investigate its optimal sample sizes w.r.t.~a desired outcome for both $\acv\mathscr T_\nu$ and $\arb\mathscr T_\nu$. We could also choose $\arb\mathscr T_\nu=\sarb_0$ to be constant but this may not always be desirable as will be seen in later sections. In particular, notice that for any $\nu>0$, $\arb\mathscr T_\nu=\zeta^\nu\to 0$ in the shot noise limit $\zeta\nearrow 0$. As such, forcing a bias on $\mathscr T_\nu$ in the shot noise limit  by requiring $\arb\mathscr T_\nu=\sarb_0$ could be viewed as counterproductive in certain contexts. These observations steer us in the direction of specifying a constraint on $\arb\mathscr T_\nu$ that varies with illumination level and vanishes in the shot noise limit. To help aid in our discussion we introduce the notion of the \emph{bias profile}.

\begin{definition}[Bias profile]
\label{def:bias_profile}
Let $T\sim F_T(\theta_T)$ be an estimator of $g(\theta_T)\neq 0$ for which $\arb T=\rho(\theta_T)$. Then we say $\rho(\theta_T)$ is the bias profile for $T$.
\end{definition}

Given the simplicity of the expression for $\arb\mathscr T_\nu$ we can choose just about any bias profile we desire by setting $\nu=\log_\zeta \rho$. With so many available options we will aim for simplicity and proceed with the following.

\begin{definition}
\label{def:nudag}
For $\zeta\in[0,1]$, $\sarb_0\in(0,1)$, and $b\in\Bbb R_0^+$, $\nudag\coloneqq\log_\zeta\sarb_0+b$.
\end{definition}

\begin{proposition}
\label{prop:nu_as_func_of_zeta}
Let $\nudag$ be as defined in Definition \ref{def:nudag}. Then, the bias profile of $\mathscr T_\nudag$ is $\rho=\sarb_0\zeta^b$.
\end{proposition}

This choice of bias profile is useful since it is simple and guarantees $\arb\mathscr T_\nudag\leq\sarb_0$ everywhere on $\zeta\in[0,1]$. Furthermore, the parameter $b$ provides the flexibility to choose a constant bias profile $(b=0)$ or a profile that vanishes to zero in the shot noise limit $(b>0)$. With this matter out of the way we present Lemmas \ref{lem:gnw_constant_bias_limit}-\ref{lem:CV2Tv_limit_double_integral}, which provide us with the necessary results for studying the behavior of the optimal sample sizes of $\mathscr T_\nudag$ under low-illumination conditions in Theorem \ref{thm:Tv_opt_n_limit}.

\begin{definition}[Lower-incomplete gamma function]
\label{def:gaminc_lower}
For $\Re s>0$
\[
\gamma(s,z) \coloneqq \int_0^z t^{s-1}e^{-t}\,\mathrm dt.
\]
\end{definition}

\begin{lemma}
\label{lem:gnw_constant_bias_limit}
Let $\nudag$ be as define in Definition \ref{def:nudag}. Then as $\zeta\searrow 1$
\[
\tilde g_{n,\omega}(\zeta,\nudag)\sim\frac{\gamma(n+1,-\log\sarb_0)}{1-\sarb_0}(1-\zeta)^{-n}+\mathcal O((1-\zeta)^{1-n}),
\]
where $\gamma(s,z)$ is the lower-incomplete gamma function of Definition \ref{def:gaminc_lower}.
\end{lemma}

\begin{proof}
From Corollary \ref{cor:gTilde_hyper_form} form $(\mathrm{iii})$ we have
\[
(1-\zeta)^n\tilde g_{n,\omega}(\zeta,\nudag)=n!\frac{\zeta^{n-\omega}-\sarb_0\zeta^b\sum_{k=0}^n\frac{(\log_\zeta\sarb_0+b-n+\omega)_k}{k!}(1-\zeta)^k}{1-\sarb_0\zeta^b}.
\]
As $\zeta\searrow 1$, one may use the generalized binomial theorem to deduce $\zeta^x\sim 1+\mathcal O(1-\zeta)$. Furthermore, we also determine as $\zeta\searrow 1$: $\log_\zeta\sarb_0\to\infty$ so that $(\log_\zeta\sarb_0+b-n+\omega)_k\sim\log_\zeta^k\sarb_0$ and $[(1-\zeta)\log_\zeta\sarb_0]^k\sim(-\log\sarb_0)^k(1+\mathcal O(1-\zeta))$. After making the appropriate substitutions and collecting terms we find
\[
(1-\zeta)^n\tilde g_{n,\omega}(\zeta,\nudag)\sim\Gamma(n+1)\frac{1-\sarb_0\sum_{k=0}^n\frac{(-\log\sarb_0)^k}{k!}}{1-\sarb_0}+\mathcal O(1-\zeta),
\]
which according to \cite[Eq.~$06.06.03.0009.01$]{wolfram_functions} can be written in terms of the lower incomplete gamma function. Solving for $\tilde g_{n,\omega}$ completes the proof.
\end{proof}

\begin{definition}[Polygamma function]
\label{def:polygamma_func}
For $n\in\Bbb N_0$,
\[
\psi^{(n)}(z)\coloneqq\partial_z^{n+1}\log\Gamma(z).
\]
\end{definition}

\begin{lemma}
\label{lem:CV2Tv_ell_sum_optimization}
Let $(\alpha_1,\alpha_2)\in\Bbb R^+\times\Bbb R^+$ and
\[
f_k(\alpha_1,\alpha_2)=\sum_{\ell=0}^k\frac{1}{(\alpha_1)_\ell(\alpha_2)_{k-\ell} \ell!(k-\ell)!},\quad k\in\Bbb N.
\]
Then for all $k$ and some constant $\mathrm A>0$, the optimal pair $(\alpha_1^{\normalfont\text{opt}},\alpha_2^{\normalfont\text{opt}})$ that minimize $f_k$ subject to $\mathrm A=\alpha_1+\alpha_2$ is
\[
(\alpha_1^{\normalfont\text{opt}},\alpha_2^{\normalfont\text{opt}})=(\tfrac{1}{2}\mathrm A,\tfrac{1}{2}\mathrm A).
\]
\end{lemma}

\begin{proof}
Working with the properties of the Pochhammer symbol and Relation \ref{rel:2F1_zeq1} we may write $f_k$ in the equivalent form
\[
f_k(\alpha_1,\alpha_2)=\frac{1}{k!}\frac{(\alpha_1+\alpha_2-1+k)_k}{(\alpha_1)_k(\alpha_2)_k}.
\]
Substituting $\mathrm A=\alpha_1+\alpha_2$ and differentiating w.r.t.~$\alpha_2$ yields
\[
\partial_{\alpha_2}f_k(\mathrm A-\alpha_2,\alpha_2)=f_k(\mathrm A-\alpha_2,\alpha_2)\Psi_k^{(0)}(\mathrm A-\alpha_2,\alpha_2),
\]
where
\[
\Psi_k^{(0)}(\alpha_1,\alpha_2)=\psi^{(0)}(\alpha_1+k)-\psi^{(0)}(\alpha_1)-\psi^{(0)}(\alpha_2+k)+\psi^{(0)}(\alpha_2)
\]
and $\psi^{(n)}(z)$ is the polygamma function of Definition \ref{def:polygamma_func}. By repeated application of the recurrence relation $\psi^{(n)}(z+1)=\psi^{(n)}(z)+(-1)^nn!z^{-n-1}$ we deduce
\[
\Psi_k^{(0)}(\alpha_2)=\sum_{m=0}^{k-1}\left(\frac{1}{\mathrm A-\alpha_2+m}-\frac{1}{\alpha_2+m}\right)%
\begin{cases}
<0, &0<\alpha_2<\frac{1}{2}\mathrm A\\
=0, &\alpha_2=\frac{1}{2}\mathrm A\\
>0, &\frac{1}{2}\mathrm A<\alpha_2<\mathrm A.
\end{cases}
\]
Further noting that $f_k(\mathrm A-\alpha_2,\alpha_2)$ is positive proves it has a unique global minimum at $\alpha_2=\frac{1}{2}\mathrm A$, which completes the proof.
\end{proof}

\begin{definition}[Exponential integral function]
\label{def:Ei_function}
\[
\operatorname{Ei}(z)\coloneqq\int_0^z\frac{e^t-1}{t}\,\mathrm dt+\frac{1}{2}\left(\log z-\log\left(\tfrac{1}{z}\right)\right)+\gamma,
\]
where $\gamma=0.577216\dots$ is the Euler-Mascheroni constant.
\end{definition}

\begin{lemma}
\label{lem:CV2Tv_limit_double_integral}
For $c,x\in\Bbb R$
\[
\int_{[0,x]^2}e^{-u+cuv-v}\,\mathrm du\mathrm dv=\tfrac{1}{c}e^{-1/c}\left(\operatorname{Ei}\left(\tfrac{1}{c}\right)-2\operatorname{Ei}\left(\tfrac{1}{c}(1-cx)\right)+\operatorname{Ei}\left(\tfrac{1}{c}(1-cx)^2\right)\right),
\]
where $\operatorname{Ei}(z)$ is the exponential integral function of Definition \ref{def:Ei_function}.
\end{lemma}

\begin{proof}
Let $I$ denote the integral in question. Integrating w.r.t.~$u$ and then substituting $t=(cv-1)x$ yields
\[
I=\tfrac{1}{c}e^{-1/c}\int_{-x}^{x(cx-1)}e^{-t/(cx)}\frac{e^t-1}{t}\,\mathrm dt.
\]
With a bit of algebra we may further write
\[
I=\tfrac{1}{c}e^{-1/c}\left(\int_{-x}^{x(cx-1)}\frac{e^{(1-1/(cx))t}-1}{t}\,\mathrm dt-\int_{-x}^{x(cx-1)}\frac{e^{-1/(cx)t}-1}{t}\,\mathrm dt\right).
\]
Now substituting $s=(1-1/(cx))t$ and $s=-1/(cx) t$ into the first and second integral, respectively, we have after some simplification
\[
I=\tfrac{1}{c}e^{-1/c}\left(\int_0^{1/c}-2\int_0^{(1-cx)/c}+\int_0^{(1-cx)^2/c}\right)\frac{e^s-1}{s}\,\mathrm ds.
\]
From Definition \ref{def:Ei_function}, if $z\in\Bbb R$ then $\int_0^z\frac{1}{t}(e^t-1)\,\mathrm dt=\operatorname{Ei}(z)-\log|z|-\gamma$. Substituting this result into $I$, all of the logarithmic and constant terms cancel leaving us with the desired result.
\end{proof}

\begin{theorem}
\label{thm:Tv_opt_n_limit}
Let $(n_1^{{\normalfont\text{opt}}},n_2^{{\normalfont\text{opt}}})$ denote the optimal sample sizes for $\mathscr T_\nudag$ that also satisfy $\acv\mathscr T_\nudag(n_1^{{\normalfont\text{opt}}},n_2^{{\normalfont\text{opt}}}|\zeta,\nudag)=\sacv_0$ and $\arb\mathscr T_\nudag=\sarb_0\zeta^b$ for $\sacv_0\in\Bbb R^+$, $\sarb_0\in(0,1)$, and $b\in\Bbb R_0^+$ fixed. Then as illumination decreases, $\zeta\searrow 1$, $n_2^{{\normalfont\text{opt}}}/n_1^{{\normalfont\text{opt}}}\to 1$, and $n_i^{{\normalfont\text{opt}}}\sim C_{\mathscr T_\nudag}(1-\zeta)^{-2}$ where $C_{\mathscr T_\nudag}$ is the solution to
\[
\overline{\acv^2\mathscr T_\nudag}(C_{\mathscr T_\nudag},\sarb_0)-\sacv_0^2=0
\]
with
\begin{multline*}
\overline{\acv^2\mathscr T_\nudag}=\frac{(C_{\mathscr T_\nudag}/4)e^{-C_{\mathscr T_\nudag}/4}}{(1-\sarb_0)^2}\biggl(\operatorname{Ei}\left(\tfrac{C_{\mathscr T_\nudag}}{4}\right)-2\operatorname{Ei}\left(\log\sarb_0+\tfrac{C_{\mathscr T_\nudag}}{4}\right)\\%
+\operatorname{Ei}\left(\tfrac{(C_{\mathscr T_\nudag}/4+\log\sarb_0)^2}{C_{\mathscr T_\nudag}/4}\right)\biggr)-1
\end{multline*}
and $\operatorname{Ei}(z)$ is the exponential integral function of Definition \ref{def:Ei_function}.
\end{theorem}

\begin{proof}
First note that $\mathscr T_\nudag$ satisfies the constraint $\arb\mathscr T_\nudag=\sarb_0\zeta^b$ independent of $n_1$ and $n_2$. Using the result of Lemma \ref{lem:gnw_constant_bias_limit} and the relationship $\alpha_i=(n_i-1)/2$ we deduce as $\zeta\searrow 1$
\[
\acv^2\mathscr T_\nudag(n_1,n_2|\zeta,\sarb_0,b)\sim\sum_{k=1}^\infty \frac{f_k\left(\tfrac{n_1-1}{2},\tfrac{n_2-1}{2}\right)}{(1-\zeta)^{2k}}(a_k^2(\sarb_0)+\mathcal O(1-\zeta)),
\]
where $f_k(\alpha_1,\alpha_2)$ is as defined in Lemma \ref{lem:CV2Tv_ell_sum_optimization} and $a_k(\sarb_0)=\gamma(k+1,-\log\sarb_0)/(1-\sarb_0)$. By Lemma \ref{lem:CV2Tv_ell_sum_optimization} we know that the optimal sample sizes $(n_1^{\text{opt}},n_2^{\text{opt}})$ that minimizes $f_k$ with $\mathrm N=n_1+n_2$ fixed occurs when $n_1=n_2$; thus, $n_2^{{\normalfont\text{opt}}}/n_1^{{\normalfont\text{opt}}}\to 1$ as $\zeta\searrow 1$ which is our first claim. Letting $n_1^{{\normalfont\text{opt}}}=n_2^{{\normalfont\text{opt}}}=n^{{\normalfont\text{opt}}}$ we have
\[
\acv^2\mathscr T_\nudag(n_1^{\text{opt}},n_2^{\text{opt}}|\theta_{\mathscr T_\nudag})\sim\sum_{k=1}^\infty \frac{f_k\left(\tfrac{n^{\text{opt}}-1}{2},\tfrac{n^{\text{opt}}-1}{2}\right)}{(1-\zeta)^{2k}}(a_k^2(\sarb_0)+\mathcal O(1-\zeta)),
\]
Since $(1-\zeta)^{-2k}$ becomes arbitrarily large in the neighborhood of $\zeta=1$ the function $f_k$ must simultaneously become arbitrarily small as to maintain a finite value of $\acv^2\mathscr T_\nudag$ in the limit. But $f_k((n^{\text{opt}}-1)/2,(n^{\text{opt}}-1)/2)$ is a decreasing function of $n^{\text{opt}}$ so it must be that $n^{\text{opt}}\to\infty$ as $\zeta\searrow 1$. Making use of the asymptotic relation $(s)_k\sim s^k(1+\mathcal O(s^{-1}))$ as $s\to\infty$ and binomial theorem we deduce for $n^{\text{opt}}\to\infty$:
\[
f_k\left(\tfrac{n^{\text{opt}}-1}{2},\tfrac{n^{\text{opt}}-1}{2}\right)\sim\frac{1}{k!}\left(\frac{4}{n^{\text{opt}}}\right)^k+\mathcal O((n^{\text{opt}})^{-k-1}),
\]
which upon substitution into our expression for $\acv^2\mathscr T_\nudag$ gives
\begin{multline*}
\acv^2\mathscr T_\nudag(n_1^{\text{opt}},n_2^{\text{opt}}|\theta_{\mathscr T_\nudag})\sim\sum_{k=1}^\infty \left(\frac{1}{k!}\left(\frac{4}{(1-\zeta)^2n^{\text{opt}}}\right)^k+\frac{\mathcal O((n^{\text{opt}})^{-k-1})}{(1-\zeta)^{2k}}\right)\\
\cdots\times(a_k^2(\sarb_0)+\mathcal O(1-\zeta)).
\end{multline*}
Given that we require $\acv\mathscr T_\nudag=\sacv_0<\infty$ in the limit it follows that we must have $n^{\text{opt}}\sim C_{\mathscr T_\nudag} (1-\zeta)^{-2}+\mathcal O((1-\zeta)^{-1})$ for some positive constant $C_{\mathscr T_\nudag}$ which proves our second claim. Substituting this asymptotic form for $n^{\text{opt}}$ into our expression for $\acv^2\mathscr T_\nudag$ we see that all error terms are $\mathcal O(1-\zeta)$ so that upon passing to the limit we obtain
\[
\lim_{\substack{\zeta\searrow 1\\ \acv\mathscr T_\nudag=\sacv_0\\ \arb\mathscr T_\nudag=\sarb_0\zeta^b}}\acv^2\mathscr T_\nudag(n_1^{\text{opt}},n_2^{\text{opt}}|\theta_{\mathscr T_\nudag})=\sum_{k=1}^\infty \frac{\gamma^2(k+1,-\log\sarb_0)}{(1-\sarb_0)^2}\frac{\left(4/C_{\mathscr T_\nudag}\right)^k}{k!},
\]
where $C_{\mathscr T_\nudag}$ is the positive constant that solves
\begin{equation}
\label{eq:CT_solution_series_form}
\sum_{k=1}^\infty \frac{\gamma^2(k+1,-\log\sarb_0)}{(1-\sarb_0)^2}\frac{\left(4/C_{\mathscr T_\nudag}\right)^k}{k!}=\sacv_0^2.
\end{equation}
Letting $\overline{\acv^2\mathscr T_\nudag}$ denote the series on the l.h.s.~of (\ref{eq:CT_solution_series_form}) we use the integral representation of the lower incomplete gamma function in Definition \ref{def:gaminc_lower} to write
\[
\overline{\acv^2\mathscr T_\nudag}=\lim_{n\to\infty}\int_{[0,-\log\sarb_0]^2}\frac{e^{-(u+v)}}{(1-\sarb_0)^2}\sum_{k=1}^n\frac{\left (4uv/C_{\mathscr T_\nudag}\right)^k}{k!}\,\mathrm du\mathrm dv.
\]
The integrand is the sum of nonnegative terms and is bounded above by its limiting form as $n\to\infty$ which in turn is integrable on the domain of integration; thus, by argument of dominated convergence we have
\[
\overline{\acv^2\mathscr T_\nudag}=\int_{[0,-\log\sarb_0]^2}\frac{e^{-u+4uv/C_{\mathscr T_\nudag}-v}}{(1-\sarb_0)^2}\,\mathrm du\mathrm dv-1,
\]
which is evaluated in terms of the exponential integral function with Lemma \ref{lem:CV2Tv_limit_double_integral}. The proof is now complete.
\end{proof}

Figure \ref{fig:CT_surf_plot} plots $C_{\mathscr T_\nudag}$ on the unit square. These values for $C_{\mathscr T_\nudag}$ were computed in {\sc mathematica} using a Newton-Raphson iteration with starting point
\[
C_{\mathscr T_\nudag}^\ast=\frac{2\gamma^2(2,-\log\sarb_0)}{\sacv_0^2(1-\sarb_0)^2}\left(1+\left(1+2\sacv_0^2(1-\sarb_0)^2\frac{\gamma^2(3,-\log\sarb_0)}{\gamma^4(2,-\log\sarb_0)}\right)^{1/2}\right),
\]
which is the solution to $(\ref{eq:CT_solution_series_form})$ using only the $k=1,2$ terms.
\begin{figure}[htb]
\centering
\includegraphics[scale=1]{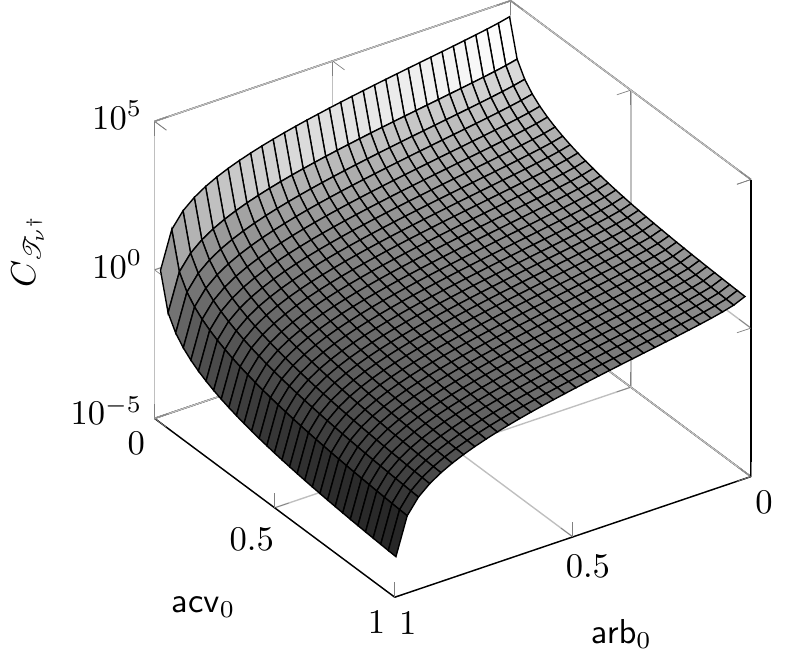}
\caption{Plot of $C_{\mathscr T_\nudag} (\sarb_0,\sacv_0)$ on the unit square.}
\label{fig:CT_surf_plot}
\end{figure}

Since $\arb\mathscr G_\nu=\arb\mathscr T_\nu$ we may combine the results of Lemma \ref{lem:ACVbarP_opt_sample_sizes} and Theorem \ref{thm:Tv_opt_n_limit} to make the following claims about the optimal sample sizes for $\mathscr G_\nudag$ at low illumination.

\begin{corollary}
\label{cor:Gv_opt_n_limit}
Let $(n_1^{{\normalfont\text{opt}}},n_2^{{\normalfont\text{opt}}})$ denote the optimal sample sizes for $\mathscr G_\nudag$ that also satisfy $\acv\mathscr G_\nudag(n_1^{{\normalfont\text{opt}}},n_2^{{\normalfont\text{opt}}}|\zeta,\nudag,\sigma_{\mathrm d},g)=\sacv_0$ and $\arb\mathscr G_\nudag=\sarb_0\zeta^b$ for $\sacv_0\in\Bbb R^+$, $\sarb_0\in(0,1)$, and $b\in\Bbb R_0^+$ fixed. Then as illumination decreases, $\zeta\searrow 1$, $n_2^{{\normalfont\text{opt}}}/n_1^{{\normalfont\text{opt}}}\to 1$, and $n_i^{{\normalfont\text{opt}}}\sim C_{\mathscr G_\nudag}(1-\zeta)^{-2}$ where $C_{\mathscr G_\nudag}$ is the solution to
\[
\overline{\acv^2\mathscr G_\nudag}(C_{\mathscr G_\nudag},\sarb_0,\sigma_{\mathrm d},g)-\sacv_0^2=0
\]
with
\[
\overline{\acv^2\mathscr G_\nudag}=\overline{\acv^2\mathscr T_\nudag}(C_{\mathscr G_\nudag},\sarb_0)+2\frac{\overline{\acv^2\mathscr T_\nudag}(C_{\mathscr G_\nudag},\sarb_0)}{(\sigma_{\mathrm d}g)^2C_{\mathscr G_\nudag}}+\frac{2}{(\sigma_{\mathrm d}g)^2C_{\mathscr G_\nudag}}
\]
and $\overline{\acv^2\mathscr T_\nudag}$ given in Theorem \ref{thm:Tv_opt_n_limit}.
\end{corollary}

\begin{proof}
Begin by writing
\[
\acv^2\mathscr G_\nudag=\acv^2\mathscr T_\nudag+(\acv^2\mathscr T_\nudag)(\acv^2\bar P)+\acv^2\bar P.
\]
As was the case in Theorem \ref{thm:Tv_opt_n_limit}, $\mathscr G_\nudag$ satisfies the constraint on $\arb\mathscr G_\nu$ independent of $n_1$ and $n_2$. Furthermore, from Lemma \ref{lem:ACVbarP_opt_sample_sizes} and Theorem \ref{thm:Tv_opt_n_limit} we know in the limit $\zeta\searrow 1$ that the optimal sample sizes for $\bar P$ and $\mathscr T_\nudag$ are asymptotically equal and proportional to $(1-\zeta)^{-2}$. Substituting $n_1^{\text{opt}}=n_2^{\text{opt}}\sim C_{\mathscr G_\nudag}(1-\zeta)^{-2}$ and passing to the limit gives the desired expression for $\overline{\acv^2\mathscr G_\nudag}$.
\end{proof}

We are finally able to return to the questions surrounding the behavior of $\mathscr E$ as a function of the optimal sample sizes at low illumination.

\begin{corollary}
\label{cor:E_with_opt_samp_sizes}
Let $T$ stand in place for $\mathscr T_\nudag$ or $\mathscr G_\nudag$ and $(n_1^\text{\normalfont opt},n_2^\text{\normalfont opt})$ denote the optimal sample sizes for $T$ such that $\acv T=\sacv_0$ and $\arb T=\sarb_0\zeta^b$ with $\sacv_0\in\Bbb R^+$, $\sarb_0\in(0,1)$, and $b\in\Bbb R_0^+$ fixed. Furthermore, define $\mathscr E_T$ to be the quantity $\mathscr E$ as a function of $(n_1^\text{\normalfont opt},n_2^\text{\normalfont opt})$ and $\overline{\mathscr E_T}=\lim_{\zeta\searrow 1}\mathscr E_T$. Then,
\[
\overline{\mathscr E_T}=\left(1+(1+\overline{\acv^{-2}\mathscr T_\nudag}(C_T,\sarb_0))\frac{2}{(\sigma_{\mathrm d}g)^2 C_T}\right)^{-1},
\]
where $\overline{\acv^2\mathscr T_\nudag}$ is given in Theorem \ref{thm:Tv_opt_n_limit}. In particular, if $T=\mathscr T_\nudag$:
\[
\mathscr E_{\mathscr T_\nudag}=\left(1+\frac{1+\sacv_0^{-2}}{(\sigma_{\mathrm d}g)^2}\frac{\zeta}{(1-\zeta)^2}\left(\frac{1}{n_1^{\text{\normalfont opt}}}+\frac{\zeta}{n_2^{\text{\normalfont opt}}}\right)\right)^{-1}
\]
and
\[
\overline{\mathscr E_{\mathscr T_\nudag}}=\left(1+\frac{2(1+\sacv_0^{-2})}{(\sigma_{\mathrm d}g)^2C_{\mathscr T_\nudag}}\right)^{-1}.
\]
\end{corollary}

Corollary \ref{cor:E_with_opt_samp_sizes} shows us that when one substitutes the optimal sample sizes for $\mathscr T_\nudag$ or $\mathscr G_\nudag$ into $\mathscr E$, the limit $\lim_{\zeta\searrow 1}\mathscr E_{\mathscr T_\nudag}$ is nonzero. Additionally, for the special case where the optimal sample sizes for $\mathscr T_\nudag$ are chosen, both $\mathscr E_{\mathscr T_\nudag}$ and $\overline{\mathscr E_{\mathscr T_\nudag}}$ reduce to very simple and easy to compute expressions. To get a sense for how $\overline{\mathscr E_{\mathscr T_\nudag}}$ behaves w.r.t.~its parameters, Table \ref{tbl:bar_Etv_valuesTab2} presents numerical values of this quantity for select values of $\sigma_{\mathrm d}g$, $\sarb_0$, and $\sacv_0$. One can see for the chosen parameter values that the dark noise, denoted $\sigma_{\mathrm d}g$, plays a significant role in the proximity of $\overline{\mathscr E_{\mathscr T_\nudag}}$ to one. Indeed, for $\sigma_{\mathrm d}g$ greater than about $5\,\electron$ we observe $\overline{\mathscr E_{\mathscr T_\nudag}}\approx 1$, which implies that the optimal sample sizes for $\mathscr T_\nudag$ and $\mathscr G_\nudag$ are nearly identical even at near zero illumination.

Using the optimal sample sizes of $\mathscr T_\nudag$ in place of those for $\mathscr G_\nudag$ is not just convenient for producing a compact expression for $\mathscr E$.  When measuring $g$ in an actual experiment, one cannot know the optimal sample sizes with certainty because they are based on unknown parameters. For $\mathscr T_\nudag$, the optimal sample sizes are a function of $\zeta=\sigma_{\mathrm d}^2/\sigma_{\mathrm p+\mathrm d}^2$, which is itself a function of two unknown parameters that must be estimated. This problem is compounded when considering optimal sample sizes for $\mathscr G_\nudag$ which are a function of $(\sigma_{\mathrm d}^2,g,\zeta)$ and thus require estimating four parameters, i.e.~$(\mu_{\mathrm d},\mu_{\mathrm p+\mathrm d},\sigma_{\mathrm d},\sigma_{\mathrm p+\mathrm d})$. The two additional degrees of freedom render estimating the optimal sample sizes for $\mathscr G_\nudag$ in a real experiment impractical due to the amount of uncertainty introduced in the process of estimation. Instead, it is far better to measure $g$ in a region where $\mathscr E\approx1$ so that the optimal sample sizes for $\mathscr T_\nudag$ can be substituted. As has been demonstrated here, this situation is possible even when the illumination level is near zero small provided a sufficiently large dark noise.
  
\begin{table}[htb]
\centering
\begin{tabular}{LCCCCCC}\toprule
	& \multicolumn{3}{C}{\sigma_{\mathrm d}g=1\,\electron} & \multicolumn{3}{C}{\sigma_{\mathrm d}g=2\,\electron}
	\\\cmidrule(r){2-4}\cmidrule(r){5-7}   
	\sarb_0/\sacv_0 & 0.01 & 0.02 & 0.05 & 0.01 & 0.02 & 0.05\\\midrule
	0.01  & 0.645 & 0.645 & 0.646 & 0.879 & 0.879 & 0.879\\
	0.02  & 0.629 & 0.629 & 0.629 & 0.871 & 0.871 & 0.872\\
	0.05  & 0.587 & 0.587 & 0.587 & 0.850 & 0.850 & 0.850\\\midrule
	\multicolumn{7}{C}{}\\
	& \multicolumn{3}{C}{\sigma_{\mathrm d}g=5\,\electron} & \multicolumn{3}{C}{\sigma_{\mathrm d}g=10\,\electron}
	\\\cmidrule(r){2-4}\cmidrule(r){5-7}   
	\sarb_0/\sacv_0 & 0.01 & 0.02 & 0.05 & 0.01 & 0.02 & 0.05\\\midrule
	0.01  & 0.978 & 0.978 & 0.979 & 0.995 & 0.995 & 0.995\\
	0.02  & 0.977 & 0.977 & 0.977 & 0.994 & 0.994 & 0.994\\
	0.05  & 0.973 & 0.973 & 0.973 & 0.993 & 0.993 & 0.993\\\bottomrule
\end{tabular}
\caption{$\overline{\mathscr E_{\mathscr T_\nudag}}$ for select values of $\sigma_{\mathrm d}g$, $\sarb_0$, and $\sacv_0$.}
\label{tbl:bar_Etv_valuesTab2}
\end{table}


\subsection{Computation of optimal sample sizes for $\mathscr T_\nudag$}
\label{subsec:comp_of_opt_samp_sizes_for_Tv}

Given that the optimal samples sizes for $\mathscr T_\nudag$ are most desirable for experimentation, this section is dedicated to further studying their characteristics as well as computing them. Our approach will be organized as follows. We will begin by using the analysis in Section \ref{subsec:gnw_tilde_gnw} to derive a few results pertaining to zeros and special values of $\tilde g_{n,\omega}(z,\nu)$ evaluated at $(z,\nu)=(\zeta,\nudag)$. These results will then allow us to write the system of equations needed to solve for the optimal sample sizes $(n_1^\text{opt},n_2^\text{opt})$ as well as provide exact solutions at two points of interest. From here we will derive explicit approximations for the optimal sample sizes and use our previous findings to determine the usefulness of these approximations. Finally, the derived approximations will be used as the starting point in numerical methods to compute $(n_1^\text{opt},n_2^\text{opt})$, which will be demonstrated for some sample values.

The following proposition defines some notation needed to discuss the zeros of $\tilde g_{n,\omega}(\zeta,\nudag)$ in Lemma \ref{lem:zeros_of_gTilde_z_nudag}.
\begin{proposition}
\label{prop:facts_about_nudag}
Let $\nudag(\zeta)=\log_\zeta\sarb_0+b$ be a function of $\zeta$ as defined in Definition \ref{def:nudag}, $\Bbb N_b\coloneqq\{n\in\Bbb N_0:n\geq b\}$, $\zeta_{b,n}\coloneqq\sqrt[n-b]{\sarb_0}\mathds 1_{b<n}$, and $Z_b\coloneqq\{\zeta_{b,n}:n\in\Bbb N_b\}$. Then,
\begin{itemize}
\item[$(1)$] $\nudag(\zeta)$ is strictly increasing from $[0,1]$ onto $[b,\infty)$,
\item[$(2)$] $\nudag[Z_b]=\Bbb N_b$.
\end{itemize}
\end{proposition}

\begin{lemma}[Zeros of $\tilde g_{n,\omega}(\zeta,\nudag)$]
\label{lem:zeros_of_gTilde_z_nudag}
\[
\tilde g_{n,\omega}(\zeta,\nudag)=0\iff(n,n-\omega)\in\Bbb N^2\land(\zeta=0\lor(\zeta\in Z_b\setminus\{0\}\land n-\omega\geq\nudag)).
\]
\end{lemma}

\begin{proof}
The proof follows from Lemma \ref{lem:tilde_gnw_zeros} which states
\[
\tilde g_{n,\omega}(\zeta,\nudag)=0\iff A\lor B,
\]
with $A=(n\in\Bbb N\land n>\omega\land\nudag\in\Bbb R_0^+\land z=0)$ and $B=(n-\omega-\nudag\in\Bbb N_0\land \nudag\neq 0)$. Starting with $A$ notice that $\nudag\in\Bbb R_0^+$ always holds so that we may equivalently write $A=((n,n-\omega)\in\Bbb N^2\land\zeta=0)$. Next, consider $B$ and observe that a necessary condition for $\nudag=0$ is $\zeta=0$. Since the condition $\zeta=0$ is already covered by $A$ we may write $B=(n-\omega-\nudag\in\Bbb N_0\land \zeta\neq 0)$. Furthermore, $n-\omega-\nudag\in\Bbb N_0$ requires $\nudag$ to be an integer and $n-\omega\geq\nudag$. But $\nudag$ can only be an integer if $\zeta\in Z_b$ so that $B=(n-\omega\geq\nudag\land \zeta\in Z_b\setminus\{0\})$. But the conditions specified by $B$ implicitly require $n-\omega\geq 1\implies(n,n-\omega)\in\Bbb N^2$ so that this requirement can be factored out of $A$ and $B$ leaving us with the desired result.
\end{proof}

\begin{corollary}
\label{cor:special_zeros_of_gTilde_z_nudag}
For all $(\zeta,\nudag)\in[0,1]\times[b,\infty)$
\[
n=\omega\implies \tilde g_{n,\omega}(\zeta,\nudag)\neq 0.
\]
Furthermore,
\[
\tilde g_{n,\omega}(\zeta,\nudag)=0,\ \forall (n,n-\omega)\in\Bbb N^2\iff\zeta\in\{0,\zeta_{b,1}\}.
\]
\end{corollary}

\begin{proof}
The first claim immediately follows from Lemma \ref{lem:zeros_of_gTilde_z_nudag} since $n=\omega\implies n-\omega=0\notin\Bbb N$. Furthermore, the conditions of Lemma \ref{lem:zeros_of_gTilde_z_nudag} are always satisfied if $(n,n-\omega)\in\Bbb N^2$ and $\zeta=0\lor\nudag=1$, where the latter holds if and only if $\zeta\in\{0,\zeta_{b,1}\}$.
\end{proof}

Corollary \ref{cor:special_zeros_of_gTilde_z_nudag} will turn out to be very important in determining the behavior of optimal sample sizes of $\mathscr T_\nudag$. In particular, the results of this corollary indicate that $\zeta=0$ and $\zeta=\zeta_{b,1}$ are the only values of $\zeta$ for which $\tilde g_{n,\omega}(\zeta,\nudag)=0$ for all $n>\omega$. We shall see that these two points of interest corresponds to the only $\zeta$-values that render the optimization problem for the optimal sample sizes of $\mathscr T_\nudag$ weakly degenerate. Before proceeding, we need one more result, which gives us the expression for $\tilde g_{n,n}(\zeta,\nudag)$ at $\zeta=0,\zeta_{b,1}$.

\begin{lemma}
\label{lem:gTilde_vzb_limit}
Let $L\in\{0,\zeta_{b,1}\}$. Then $\lim_{\zeta\nearrow L}\tilde g_{n,n}(\zeta,\nudag)=n!$ for all $n\in\Bbb N_0$.
\end{lemma}

\begin{proof}
From Corollary \ref{cor:gTilde_hyper_form} form $(\mathrm{iii})$ we have for all $n\in\Bbb N_0$:
\[
\tilde g_{n,n}(\zeta,\nudag)=n!\frac{1-\sarb_0\zeta^b\sum_{k=0}^n\frac{(\nudag)_k}{k!} (1-\zeta)^k}{(1-\zeta)^n(1-\sarb_0\zeta^b)}.
\]
Each case is now examined separately.


\begin{enumerate}
\item[$(1)$] $\zeta\nearrow 0$.

\begin{subproof}
If $b>0$ then $\zeta\nearrow 0\implies\zeta^b\nearrow 0$ and so $\tilde g_{n,n}(\zeta,\nudag_{b>0})\to n!$. Likewise, if $b=0$ then $\zeta^b=1$ while $(\nudag)_k\to\mathds 1_{k=0}$ and so again we find $\tilde g_{n,n}(\zeta,\nudag_{b=0})\to n!$ as expected.
\end{subproof}


\item[$(2)$] $\zeta\nearrow \zeta_{b,1}$.

\begin{subproof}
If $b\geq 1$ then $\zeta_{b,1}=0$, which was already covered in part $(1)$. Now suppose $0\leq b\leq 1$. As $\zeta\nearrow\zeta_{b,1}$: $\nudag\to 1$ and $\sarb_0\zeta^b\to\zeta_{b,1}$; hence
\[
\lim_{\zeta\nearrow\zeta_{b,1}}\tilde g_{n,n}(\zeta,\nudag_{b\in[0,1]})=n!\frac{1-\zeta_{b,1}\sum_{k=0}^n(1-\zeta_{b,1})^k}{(1-\zeta_{b,1})^{n+1}}=n!,
\]
which completes the proof.
\end{subproof}
\end{enumerate}
\end{proof}

Our next task is to derive the system of equations needed for computing the optimal sample sizes. Since $\acv\mathscr T_\nudag$ is naturally parameterized in terms of the shape parameters $\alpha_1$ and $\alpha_2$ and not the sample sizes $n_1$ and $n_2$ our discussion will focus on finding the optimal shape parameters according to Definition \ref{def:optimal_sample_sizes}, knowing that we may convert these optimal shape parameters to their corresponding sample sizes through the relation $n_i^\text{opt}=2\alpha_i^\text{opt}+1$. As was done in Lemma \ref{lem:ACVbarP_opt_sample_sizes} we will accomplish this task by letting $\mathrm A=\alpha_1+\alpha_2$ be constant and then showing that $\acv^2\mathscr T_\nu(\mathrm A-\alpha_2,\alpha_2)$ is: $(1)$ twice continuously differentiable and strictly convex on $\alpha_2\in(0,\mathrm A)$ and $(2)$ is infinite at the endpoints $\alpha_2=0,\mathrm A$ with the exception of two points of interest. These findings will then be used to show that the minimization in $(\ref{eq:optimal_SOEs})$ can be uniquely solved by equating a derivative with zero. The following result presents and intermediate step in demonstrating convexity of $\acv^2\mathscr T_\nu(\mathrm A-\alpha_2,\alpha_2)$.

\begin{lemma}
\label{lem:ACVTnudag_interior_sum_is_convex}
Let $k\in\Bbb N$, $\mathrm A>0$, $\zeta\in[0,1]$, and $\nudag$ be as defined in Definition \ref{def:nudag}. Then,
\[
f_k(\alpha_2|\mathrm A,\zeta,\nudag):=\sum_{\ell=0}^k\frac{\tilde g_{k,\ell}^2(\zeta,\nudag)}{(\mathrm A-\alpha_2)_\ell(\alpha_2)_{k-\ell}\ell !(k-\ell)!}
\]
is positive and strictly convex on $\alpha_2\in(0,\mathrm A)$. Furthermore,
\[
\lim_{\alpha_2\nearrow 0}f_k(\alpha_2|\cdot)=%
\begin{cases}
k!/(\mathrm A)_k, &\zeta\in\{0,\zeta_{b,1}\}\\
\infty, &\text{otherwise}
\end{cases}
\]
and
\[
\lim_{\alpha_2\searrow \mathrm A}f_k(\alpha_2|\cdot)=\infty.
\]
\end{lemma}

\begin{proof}
The claim of positivity comes from the fact that the summand of $f_k$ is nonnegative and from Corollary \ref{cor:special_zeros_of_gTilde_z_nudag} which implies $\tilde g_{n,\omega}(\zeta,\nudag)\neq 0$ for at least one $\ell\in\{0,\dots,k\}$. For the claim of convexity we first write
\[
\partial_{\alpha_2}^2\frac{1}{(\mathrm A-\alpha_2)_\ell(\alpha_2)_{k-\ell}}=\frac{[\Psi^{(0)}_{k,\ell}(\mathrm A-\alpha_2,\alpha_2)]^2+\Psi^{(1)}_{k,\ell}(\mathrm A-\alpha_2,\alpha_2)}{(\mathrm A-\alpha_2)_\ell(\alpha_2)_{k-\ell}},
\]
where
\[
\begin{aligned}
\Psi^{(0)}_{k,\ell}(\alpha_1,\alpha_2) &=\psi^{(0)}(\alpha_1)-\psi^{(0)}(\alpha_1+\ell)-\psi^{(0)}(\alpha_2)+\psi^{(0)}(\alpha_2+k-\ell)\\
\Psi^{(1)}_{k,\ell}(\alpha_1,\alpha_2) &=\psi^{(1)}(\alpha_1)-\psi^{(1)}(\alpha_1+\ell)+\psi^{(1)}(\alpha_2)-\psi^{(1)}(\alpha_2+k-\ell).
\end{aligned}
\]
Certainly, $[\Psi^{(0)}_{k,\ell}(\mathrm A-\alpha_2,\alpha_2)]^2\geq 0$ for all $k$, $\ell$, and $\alpha_2\in(0,\mathrm A)$. Furthermore,
\[
\Psi^{(1)}_{k,\ell}(\mathrm A-\alpha_2,\alpha_2)=\sum_{j=0}^{\ell-1}\frac{1}{(\mathrm A-\alpha_2+j)^2}+\sum_{j=0}^{k-\ell-1}\frac{1}{(\alpha_2+j)^2}.
\]
Since $(k,\ell)\in(\Bbb N,\Bbb N_0)$ with $\ell\leq k$, it follows that at least one of these sums will be nonzero such that $\Psi^{(1)}_{k,\ell}(\mathrm A-\alpha_2,\alpha_2)>0$ and $[(\mathrm A-\alpha_2)_\ell(\alpha_2)_{k-\ell}]^{-1}$ is a strictly convex function of $\alpha_2$ on $(0,\mathrm A)$. Combining this observation with the fact that there is always at least one $\ell\in\{0,\dots,k\}$ for which $\tilde g_{k,\ell}(\zeta,\nudag)\neq 0$ further implies that $f_k$ is a finite sum of at least one positive strictly convex function with the rest of the terms being either being zero or also positive and strictly convex; hence, $f_k$ is positive and strictly convex on $\alpha_2\in(0,\mathrm A)$.

Now turning to the limits notice that if $\zeta\in\{0,\zeta_{b,1}\}$ then by Corollary \ref{cor:special_zeros_of_gTilde_z_nudag} and Lemma \ref{lem:gTilde_vzb_limit} we have $f_k(\alpha_2|\cdot)=\tilde g_{k,k}^2(\zeta,\nudag)/((\mathrm A-\alpha_2)_k k!)\to k!/(\mathrm A)_k$ as $\alpha_2\nearrow 0$. For the complementary case $\zeta\in[0,1]\setminus\{0,\zeta_{b,1}\}$ we may use the the asymptotic expansion $\frac{1}{(x)_m}\sim\frac{1}{\Gamma(m)}(\frac{1}{x}-H_{m-1}+\mathcal O(x))$ for $x\to 0$ and Lemma \ref{lem:gTilde_vzb_limit} to further deduce
\[
f_k(\alpha_2|\cdot)\sim\frac{k!}{(\mathrm A)_k}+\mathcal O(\alpha_2^{-1})\to\infty
\]
as $\alpha_2\nearrow 0$. For the other limit $\alpha_2\searrow\mathrm A$, Corollary \ref{cor:special_zeros_of_gTilde_z_nudag} tells us that the $\ell=0$ term in $f_k$ is always positive; hence, as $\alpha_2\searrow\mathrm A$
\[
f_k(\alpha_2|\cdot)\sim\mathcal O((\mathrm A-\alpha_2)^{-1})\to\infty,
\]
which completes the proof.
\end{proof}

\begin{theorem}[Optimal sample sizes of $\mathscr T_\nudag$]
\label{thm:TvNudag_optimal_SOE}
Let $(n_1^{{\normalfont\text{opt}}},n_2^{{\normalfont\text{opt}}})(\theta_{\mathscr T_\nudag})$ with $\theta_{\mathscr T_\nudag}=(\zeta,\sacv_0,\sarb_0,b)$ denote the optimal sample sizes for $\mathscr T_\nudag$ as a function of $\theta_{\mathscr T_\nudag}$ that also satisfy $\acv\mathscr T_\nudag(n_1^{{\normalfont\text{opt}}},n_2^{{\normalfont\text{opt}}}|\theta_{\mathscr T_\nudag})=\sacv_0$ and $\arb\mathscr T_\nudag=\sarb_0\zeta^b$ for $\sacv_0\in\Bbb R^+$, $\sarb_0\in(0,1)$, and $b\in\Bbb R_0^+$ fixed. Then, for $\zeta\in\{0,\zeta_{b,1}\}$:
\[
(n_1^{{\normalfont\text{opt}}},n_2^{{\normalfont\text{opt}}})(\theta_{\mathscr T_\nudag})=(2\sacv_0^{-2}+5,1).
\]
Likewise, for all remaining $\zeta\in[0,1]\setminus\{0,\zeta_{b,1}\}$:
\[
(n_1^{{\normalfont\text{opt}}},n_2^{{\normalfont\text{opt}}})(\theta_{\mathscr T_\nudag})=(2\alpha_1^{{\normalfont\text{opt}}}+1,2\alpha_2^{{\normalfont\text{opt}}}+1)(\theta_{\mathscr T_\nudag}),
\]
where $\alpha_1^{{\normalfont\text{opt}}}$ and $\alpha_2^{{\normalfont\text{opt}}}$ satisfy
\begin{equation}
\label{eq:optimal_alpha_SOE}
\begin{aligned}
\partial_{\alpha_2}\acv^2 \mathscr T_\nudag(\mathrm A-\alpha_2,\alpha_2|\theta_{\mathscr T_\nudag})\Big|_{(\mathrm A,\alpha_2)=(\alpha_1^{\normalfont\text{opt}}+\alpha_2^{\normalfont\text{opt}},\alpha_2^{\normalfont\text{opt}})} &=0\\
\acv^2 \mathscr T_\nudag(\alpha_1^{\normalfont\text{opt}},\alpha_2^{\normalfont\text{opt}}|\theta_{\mathscr T_\nudag})-\sacv_0^2 &=0.
\end{aligned}
\end{equation}
\end{theorem}

\begin{proof}
We begin with the special case $\zeta\in\{0,\zeta_{b,1}\}$ and call upon Corollary \ref{cor:special_zeros_of_gTilde_z_nudag} and Lemma \ref{lem:gTilde_vzb_limit} to find
\[
\acv^2 \mathscr T_\nudag(\alpha_1,\alpha_2|\theta_{\mathscr T_\nudag})=\sum_{k=1}^\infty\frac{(1)_k(1)_k}{(\alpha_1)_k k!}=\frac{1}{\alpha_1-2}.
\]
Since $\acv^2 \mathscr T_\nudag$ is independent of $\alpha_2$ it follows that the optimization problem (\ref{eq:optimal_alpha_SOE}) is weakly degenerate so that $\alpha_2^\text{opt}=0$ and $\alpha_1^\text{opt}$ satisfies
\[
\frac{1}{\alpha_1^\text{opt}-2}=\sacv_0^2.
\]
Solving for $\alpha_1^\text{opt}$ and then using the relation $n_i^\text{opt}=2n_i^\text{opt}+1$ then gives the desired result for the special case $\zeta\in\{0,\zeta_{b,1}\}$.

Now consider the complementary case $\zeta\in[0,1]\setminus\{0,\zeta_{b,1}\}$. Upon inspection of the double integral representation of $\ev\mathscr T_\nu^2$ in Theorem \ref{thm:T_est_variance_series} and assuming appropriate values of $(\zeta,\nudag)$ to guarantee convergence, we see that $\ev\mathscr T_\nudag^2(\mathrm A-\alpha_2,\alpha_2|\theta_{\mathscr T_\nudag})$ is a smooth function of $\alpha_2$ and hence so is $\acv^2\mathscr T_\nudag(\mathrm A-\alpha_2,\alpha_2|\theta_{\mathscr T_\nudag})$. Additionally, upon writing
\[
\acv^2\mathscr T_\nu(\mathrm A-\alpha_2,\alpha_2|\theta_{\mathscr T_\nudag})=\sum_{k=1}^\infty f_k(\alpha_2|\mathrm A,\theta_{\mathscr T_\nudag}),
\]
we conclude from Lemma \ref{lem:ACVTnudag_interior_sum_is_convex} that $\acv^2\mathscr T_\nu(\mathrm A-\alpha_2,\alpha_2|\theta_{\mathscr T_\nudag})$ is: $1)$ positive and strictly convex on $\alpha_2\in(0,\mathrm A)$ and $2)$ infinite at the endpoints $\alpha_2=0,\mathrm A$. As such, it follows that for every set of parameters $\theta_{\mathscr T_\nudag}$ with $\zeta\notin\{0,\zeta_{b,1}\}$ that $\alpha_2^\text{opt}(\mathrm A,\theta_{\mathscr T_\nudag})=\arginf_{\alpha_2\in(0,\mathrm A)}\acv^2\mathscr T_\nu(\mathrm A-\alpha_2,\alpha_2|\theta_{\mathscr T_\nudag})$ is unique and corresponds to
\[
\partial_{\alpha_2}\acv^2\mathscr T_\nu(\mathrm A-\alpha_2,\alpha_2|\theta_{\mathscr T_\nudag})_{\alpha_2=\alpha_2^\text{opt}}=0.
\]
This is turn implies the optimization problem for finding $(\alpha_1^\text{opt},\alpha_2^\text{opt})$ is nondegenerate and is given by (\ref{eq:optimal_alpha_SOE}). The proof is now complete.
\end{proof}

Taking a small detour, we now relate the optimal sample sizes for $\mathscr T_\nudag$ in Theorem \ref{thm:TvNudag_optimal_SOE} back to a classical result given by James Janesick in his book \emph{Photon Transfer}.

\begin{remark}[Optimal sample sizes in the shot noise limit]
\label{rem:minimum_sample_sizes}
Using the expression for $\mathscr E_{\mathscr T_\nudag}$ in Corollary \ref{cor:E_with_opt_samp_sizes} as well as the results from Theorem \ref{thm:TvNudag_optimal_SOE} it is straightforward to deduce
\[
\lim_{\zeta\nearrow 0}\mathscr E_{\mathscr T_\nudag}=1,
\]
which shows that the optimal sample sizes for $\mathscr T_\nudag$ equal those for $\mathscr G_\nudag$ in the shot noise limit. If one wants to obtain a $1\%$ relative uncertainty for the measurement of $g$ in the shot noise limit, i.e.~$\lim_{\zeta\nearrow 0}\acv\mathscr G_\nudag=0.01$, then according to Theorem \ref{thm:TvNudag_optimal_SOE} the optimal total number of samples needed is $\mathrm N^{\normalfont\text{opt}}=20,006$. This exact result agrees with Janesick's estimate $\mathrm N^{\normalfont\text{opt}}\approx 20,000$ of {\normalfont \cite[pgs.~79--81]{janesick_2007}}.
\end{remark}

Up to this point we have made a significant amount of progress in studying the optimal sample sizes of $\mathscr T_\nudag$. In particular, we have established uniqueness of these optimal sample sizes, derived their exact expressions at the two points of interest $\zeta\in\{0,\zeta_{b,1}\}$, determined a system of equations for evaluating them at all remaining $\zeta\in[0,1]\setminus\{0,\zeta_{b,1}\}$, studied their asymptotic behavior in the low illumination limit $\zeta\searrow 1$, and derived the expression $\mathscr E_{\mathscr T_\nudag}$, which allow us to determine when they give a good approximation to the optimal sample sizes of $\mathscr G_\nudag$. With so much theoretical understanding of these optimal sample sizes we're in a good place to begin discussing how we actually compute them.

It is of no surprise that the system of equations for evaluating the optimal sample sizes given in (\ref{eq:optimal_alpha_SOE}) lacks a closed-from solution and so we must turn to numerical methods to solve it. This of course introduces the practical problem of implementing root finding algorithms. We already know the bounds of the parameters $\theta_{\mathscr T_\nudag}$ and so the only thing left to implement (\ref{eq:optimal_alpha_SOE}) is a good approximation for $(n_1^\text{opt},n_2^\text{opt})$ to use as a starting point. Because $\acv\mathscr T_\nudag$ is naturally parameterized in terms of $(\alpha_1,\alpha_2)$ we will seek an approximate solution for the optimal pair $(\alpha_1^\text{opt},\alpha_2^\text{opt})$ and then transform it into $(n_1^\text{opt},n_2^\text{opt})$ via the relation $n_i^\text{opt}=2\alpha_i^\text{opt}+1$.

To begin deriving our approximation we call on the series expansion given in Corollary \ref{cor:cvT_nu} to write $\acv^2\mathscr T_\nudag(\alpha_1,\alpha_2|\theta_{\mathscr T_\nudag})\approx \acv^2\mathscr T_\nudag^\ast(\alpha_1,\alpha_2|\theta_{\mathscr T_\nudag})$, where
\begin{equation}
\label{eq:acvTv_approximation_for_approx_sol}
\acv^2\mathscr T_\nudag^\ast(\alpha_1,\alpha_2|\theta_{\mathscr T_\nudag})=\frac{1}{\alpha_1}\tilde g_{1,1}^2(\zeta,\nudag)+\frac{1}{\alpha_1\alpha_2}\tilde g_{2,1}^2(\zeta,\nudag)+\frac{1}{\alpha_2}\tilde g_{1,0}^2(\zeta,\nudag)
\end{equation}
and $\tilde g_{k,\ell}(z,\nu)$ given by Corollary \ref{cor:gTilde_hyper_form} $(\mathrm i)$. Upon substituting $\alpha_1=\mathrm A-\alpha_2$ into this approximation we see that it possesses the same useful properties of the exact form of $\acv^2\mathscr T_\nudag(\mathrm A-\alpha_2,\alpha_2|\theta_{\mathscr T_\nudag})$, namely, positivity, smoothness, strictly convexity on $\alpha_2\in(0,\mathrm A)$, as well as the same limiting properties at the endpoints $\alpha_2=0,\mathrm A$. For the sake of brevity we let $a=\tilde g_{1,1}(\zeta,\nudag)$, $b=\tilde g_{2,1}(\zeta,\nudag)$, $c=\tilde g_{1,0}(\zeta,\nudag)$, and $d=\sacv_0$ and then substitute $\acv^2\mathscr T_\nudag^\ast$ into (\ref{eq:optimal_alpha_SOE}) yielding the system of equations
\[
\begin{aligned}
&(1) &\frac{a^2}{\alpha_1^2}+\frac{b^2}{\alpha_1^2\alpha_2}-\frac{b^2}{\alpha_1\alpha_2^2}-\frac{c^2}{\alpha_2^2} &=0\\
&(2) &\frac{a^2}{\alpha_1}+\frac{b^2}{\alpha_1\alpha_2}+\frac{c^2}{\alpha_2}-d^2 &=0.
\end{aligned}
\]
To put this system into a more useful form we perform the transformations $(1)^\prime=(2)/\alpha_1-(1)$ and $(2)^\prime=(2)/\alpha_2+(1)$ to get the equivalent set of equations
\[
\begin{aligned}
&(1)^\prime &\alpha_1 &=\frac{1}{c^2}\left(d^2\alpha_2^2-c^2\alpha_2-b^2\right)\\
&(2)^\prime &\alpha_2 &=\frac{1}{a^2}\left(d^2\alpha_1^2-a^2\alpha_1-b^2\right).
\end{aligned}
\]
Now substituting the r.h.s.~of $(1)^\prime$ into $(2)^\prime$ and the r.h.s.~of $(2)^\prime$ into $(1)^\prime$ separates the variables $\alpha_1$ and $\alpha_2$ into two quartic equations, each of which yields four roots for a total of sixteen possible solutions to our problem. Substituting all sixteen potential solutions back into our original system we find only two work, namely,
\begin{equation}
\label{eq:aprx_system_potential_solutions}
(\alpha_1,\alpha_2)=\left(\frac{a^2\pm\sqrt{a^2c^2+b^2d^2}}{d^2},\frac{c^2\pm\sqrt{a^2c^2+b^2d^2}}{d^2}\right).
\end{equation}
To obtain the final solution we must determine which sign in front of the root to take. Taking the positive sign clearly yields a nonnegative solution for $\alpha_1$ and $\alpha_2$ so to verify this is the correct choice we must show that choosing the negative sign yields a nonsensical solution. The following two lemmas show that if we choose the negative sign then the solution for $\alpha_2$ is nonpositive and thus is not the desired choice. Since the choice of sign for both $\alpha_1$ and $\alpha_2$ must be the same it follows that the positive sign is indeed the correct choice.

\begin{lemma}
\label{lem:gTilde_10_lower_bound}
Let $\nudag$ be as defined in Definition \ref{def:nudag}. Then for all $\zeta\in(0,1)$: $\tilde g_{1,0}(\zeta,\nudag)\geq-\frac{1}{2}$.
\end{lemma}

\begin{proof}
We begin with the explicit form
\[
\tilde g_{1,0}(\zeta,\nudag)=\frac{\zeta}{1-\zeta}-\frac{(\log_\zeta\sarb_0+b)\sarb_0 \zeta^b}{1-\sarb_0 \zeta^b}.
\]
Our approach will be to put a lower bound on $\tilde g_{1,0}(\zeta,\nudag)$ by successively minimizing it w.r.t.~each of its variables. Differentiating w.r.t.~$b$ we find
\[
\partial_b \tilde g_{1,0}(\zeta,\nudag)=\frac{x(x-\log x-1)}{(1-x)^2}\Big|_{x=\sarb_0 \zeta^b}.
\]
For the specified parameter restriction we have $0<\sarb_0 \zeta^b<\sarb_0<1\implies x\in(0,1)$; thus, it is straightforward to show that the above expression in $x$ is positive. It follows that
\[
\tilde g_{1,0}(\zeta,\nudag)\geq\tilde g_{1,0}(\zeta,\nudag_{b=0})=\frac{\zeta}{1-\zeta}-\frac{(\log_\zeta\sarb_0)\sarb_0}{1-\sarb_0}.
\]
Now differentiating w.r.t.~$\sarb_0$ we have
\[
\partial_{\sarb_0}\tilde g_{1,0}(\zeta,\nudag_{b=0})=\frac{1}{\log\zeta}\frac{\sarb_0-\log\sarb_0-1}{(1-\sarb_0)^2}.
\]
The quantity involving $\sarb_0$ is positive on $\sarb_0\in(0,1)$ while $\log\zeta<0$ on $\zeta\in(0,1)$. Therefore, $\partial_{\sarb_0}\tilde g_{1,0}(\zeta,\nudag_{b=0})<0$ and
\[
\tilde g_{1,0}(\zeta,\nudag_{b=0})\geq\lim_{\sarb_0\searrow 1}\tilde g_{1,0}(\zeta,\nudag_{b=0})=\frac{\zeta}{1-\zeta}+\frac{1}{\log\zeta}.
\]
The resulting function of $\zeta$ is strictly decreasing in $\zeta$. Taking the limit $\zeta\searrow 1$ subsequently gives
\[
\tilde g_{1,0}(\zeta,\nudag_{b=0})\geq\lim_{\zeta\searrow 1}\frac{\zeta}{1-\zeta}+\frac{1}{\log\zeta}=-\frac{1}{2}.
\]
The proof is now complete.
\end{proof}

\begin{lemma}
\label{lem:}
Choosing the negative sign in $(\ref{eq:aprx_system_potential_solutions})$ leads to $\alpha_2\leq 0$.
\end{lemma}

\begin{proof}
For brevity we will use $\tilde g_{n,\omega}$ to denote $\tilde g_{n,\omega}(\zeta,\nudag)$. We need to show
\[
\tilde g_{1,0}^2-\sqrt{\tilde g_{1,0}^2\tilde g_{1,1}^2+\tilde g_{2,0}^2\sacv_0^2}\leq 0.
\]
Since all quantities under the radical are positive it follows that
\[
\tilde g_{1,0}^2-\sqrt{\tilde g_{1,0}^2\tilde g_{1,1}^2+\tilde g_{2,0}^2\sacv_0^2}\leq \tilde g_{1,0}^2-\sqrt{\tilde g_{1,0}^2\tilde g_{1,1}^2}.
\]
From Lemma \ref{lem:tilde_gnw_recurrence_relation} we have $\tilde g_{1,1}=\tilde g_{1,0}+1$; hence,
\[
\tilde g_{1,0}^2-\sqrt{\tilde g_{1,0}^2\tilde g_{1,1}^2+\tilde g_{2,0}^2\sacv_0^2}\leq \tilde g_{1,0}^2-\sqrt{\tilde g_{1,0}^4+\tilde g_{1,0}^2(2\tilde g_{1,0}+1)}.
\]
But now recall from Lemma \ref{lem:gTilde_10_lower_bound} that $\tilde g_{1,0}\geq-\frac{1}{2}\implies \tilde g_{1,0}^2(2\tilde g_{1,0}+1)\geq 0$. So it follows that
\[
\tilde g_{1,0}^2-\sqrt{\tilde g_{1,0}^2\tilde g_{1,1}^2+\tilde g_{2,0}^2\sacv_0^2}\leq \tilde g_{1,0}^2-\sqrt{\tilde g_{1,0}^4}=0,
\]
which completes the proof.
\end{proof}

Now knowing that we must choose the positive sign in $(\ref{eq:aprx_system_potential_solutions})$ we have after reintroducing $a$, $b$, $c$, $d$, and $n_i=2\alpha_i+1$:
\[
\begin{aligned}
n_1 &=\frac{2}{\sacv_0^2}\left(\tilde g_{1,1}^2+\sqrt{\tilde g_{1,1}^2\tilde g_{1,0}^2+\tilde g_{2,1}^2\sacv_0^2}\right)+1\\
n_2 &=\frac{2}{\sacv_0^2}\left(\tilde g_{1,0}^2+\sqrt{\tilde g_{1,1}^2\tilde g_{1,0}^2+\tilde g_{2,1}^2\sacv_0^2}\right)+1,
\end{aligned}
\]
which is the desired solution for the approximate optimal sample sizes of $\mathscr T_\nudag$. So how good are these approximations? First note that with the results of Corollary \ref{cor:special_zeros_of_gTilde_z_nudag} and Lemma \ref{lem:gTilde_vzb_limit} we have
\[
\lim_{\zeta\nearrow L}(n_1,n_2)(\theta_{\mathscr T_\nudag})=(2\sacv_0^{-2}+1,1),\quad L\in\{0,\zeta_{b,1}\}.
\]
Comparing this with Theorem \ref{thm:TvNudag_optimal_SOE} we see that our approximate solutions can be made exact at $\zeta=\{0,\zeta_{b,1}\}$ by instead using the slightly modified solution
\begin{equation}
\label{eq:approx_opt_n1_n2}
\begin{aligned}
n_1^\ast &=\frac{2}{\sacv_0^2}\left(\tilde g_{1,1}^2+\sqrt{\tilde g_{1,1}^2\tilde g_{1,0}^2+\tilde g_{2,1}^2\sacv_0^2}\right)+5\\
n_2^\ast &=\frac{2}{\sacv_0^2}\left(\tilde g_{1,0}^2+\sqrt{\tilde g_{1,1}^2\tilde g_{1,0}^2+\tilde g_{2,1}^2\sacv_0^2}\right)+1.
\end{aligned}
\end{equation}
At the opposite end of the $\zeta$-domain, we may use Lemma \ref{lem:gnw_constant_bias_limit} to deduce as $\zeta\searrow 1$: $n_2^\ast/n_1^\ast\to 1$ and $n_i^\ast\sim C_{\mathscr T_\nudag}^\ast (1-\zeta)^{-2}$, where
\[
C_{\mathscr T_\nudag}^\ast=\frac{2\gamma^2(2,-\log\sarb_0)}{\sacv_0^2(1-\sarb_0)^2}\left(1+\left(1+\sacv_0^2(1-\sarb_0)^2\frac{\gamma^2(3,-\log\sarb_0)}{\gamma^4(2,-\log\sarb_0)}\right)^{1/2}\right).
\]
Table \ref{tbl:CT_ratio} tabulates the ratio $C_{\mathscr T_\nudag}^\ast/C_{\mathscr T_\nudag}$ for a variety of values for $\sarb_0$ and $\sacv_0$. One can see that the ratio is near unity for these values which indicates our approximate solutions for $(n_1^\text{opt},n_2^\text{opt})$ given by (\ref{eq:approx_opt_n1_n2}) are not only exact at $\zeta=\{0,\zeta_{b,1}\}$ but are also exceptionally accurate for small $\sarb_0$ and $\sacv_0$ in the neighborhood of $\zeta=1$.
\begin{table}[htb]
\centering
\begin{tabular}{LCCCCCC}\toprule
	\sarb_0/\sacv_0 & 0.01 & 0.02 & 0.05 & 0.10 & 0.20\\\midrule
	0.01  &0.9999&0.9999&0.9999&0.9999&0.9999\\
	0.02  &0.9997&0.9997&0.9997&0.9997&0.9998\\
	0.05  &0.9978&0.9980&0.9982&0.9984&0.9985\\
	0.10  &0.9913&0.9919&0.9927&0.9934&0.9942\\
	0.20  &0.9655&0.9679&0.9714&0.9743&0.9771\\\bottomrule
\end{tabular}
\caption{$C_{\mathscr T_\nudag}^\ast/C_{\mathscr T_\nudag}$ for various values of $\sarb_0$ and $\sacv_0$.}
\label{tbl:CT_ratio}
\end{table}

With a suitable approximation to the optimal sample sizes we may finally turn to implementing a numerical routine for computing them. Numerically solving the system (\ref{eq:optimal_alpha_SOE}) ultimately requires a method for computing $\acv^2\mathscr T_\nudag$ and its partial derivative w.r.t.~$\alpha_2$. The fact that the approximation for $\acv^2\mathscr T_\nudag$ in (\ref{eq:acvTv_approximation_for_approx_sol}) lead to excellent approximations for the optimal sample sizes across the entire $\zeta$-domain tells us that we may compute the optimal sample sizes by substituting
\begin{equation}
\label{eq:ACV2Tv_n_truncated}
\acv_n^2\mathscr T_\nudag=\sum_{k=1}^n\sum_{\ell=0}^k\frac{\tilde g_{k,\ell}^2(\zeta,\nudag)}{(\alpha_1)_\ell(\alpha_2)_{k-\ell} \ell!(k-\ell)!}.
\end{equation}
into (\ref{eq:optimal_alpha_SOE}) for some appropriate value of $n$. To determine an appropriate value for $n$ we consider the absolute relative error
\[
R_{n,2}=\left\lvert\frac{\acv_n^2\mathscr T_\nudag}{\acv^2\mathscr T_\nudag}-1\right\rvert,
\]
which according to Corollary \ref{cor:rel_truncation_error_bound} is bounded above by $R_{n,2}\leq R_{n,m,2}^\ast$ with
\[
R_{n,m,2}^\ast=\left\lvert\left(1+\frac{E_{n,m}^\ast}{\acv_n^2\mathscr T_\nudag}\right)^{-1}-1\right\rvert
\]
and $E_{n,m}^\ast$ given in Theorem \ref{thm:ACVTv_tail_error_bound}. Before proceeding, the following result gives us an exact value for $R_{n,2}$ at $\zeta\in\{0,\zeta_{b,1}\}$.
\begin{lemma}
\label{lem:Rn2_exact_value}
For $\zeta\in\{0,\zeta_{b,1}\}$
\[
R_{n,2}|_{(n_1,n_2)=(n_1^{\normalfont\text{opt}},n_2^{\normalfont\text{opt}})}=\frac{(n+1)!}{(\sacv_0^{-2}+2)_n}.
\]
\end{lemma}

\begin{proof}
The proof follows from noting that at $\zeta=0,\zeta_{b,1}$:
\[
\acv^2\mathscr T_\nudag=\frac{1}{\alpha_1-2},
\]
\[
\acv_n^2\mathscr T_\nudag=\frac{1}{\alpha_1-2}\left(1-\frac{(n+1)!}{(\alpha_1)_n}\right),
\]
and $\alpha_1^\text{opt}=\sacv_0^{-2}+2$.
\end{proof}

The sharpness of the upper bound $R_{n,m,2}^\ast$ improves with increasing $m$ but the rate at which this happens significantly decreases for $m>n+1$; thus, it seems reasonable to use $R_{n,n+1}^\ast$ as the upper bound in our numerical routine. To get a sense of how $R_{n,n+1,2}^\ast$ behaves, we evaluated it at $(n_1,n_2)=(n_1^\ast,n_2^\ast)$ for $\sarb_0=0.5$, $\sacv_0=0.05$, and various values of $b$ and $n$ as plotted in Figure \ref{fig:Rmn_plots}. We observe that as $n$ increases, the upper bound decreases rapidly except for values of $\zeta$ in the neighborhood of $\zeta=0,\zeta_{b,1},1$. However, notice that for the exact value of $R_{n,2}$ given by Lemma \ref{lem:Rn2_exact_value}: $R_{1,2}=4.98\times 10^{-3}$, $R_{3,2}=3.67\times 10^{-7}$, $R_{5,2}=6.69\times 10^{-11}$, $R_{7,2}=2.26\times 10^{-14}$. Since the approximations $n_1^\ast$ and $n_2^\ast$ are exact as $\zeta\to L$, $L\in\{0,\zeta_{b,1}\}$, we can conclude that the high upper bound for $R_{n,2}$ near these values of $\zeta$ is the result of a deficiency of $R_{n,n+1,2}^\ast$ and not that the relative error isn't decrease rapidly with $n$. As such, estimates for the appropriate value of $n$ in (\ref{eq:ACV2Tv_n_truncated}) determined using $R_{n,n+1,2}^\ast$ will be, in some cases, significantly overestimated near these points.
\begin{figure}[htb]
\centering
\includegraphics[scale=1]{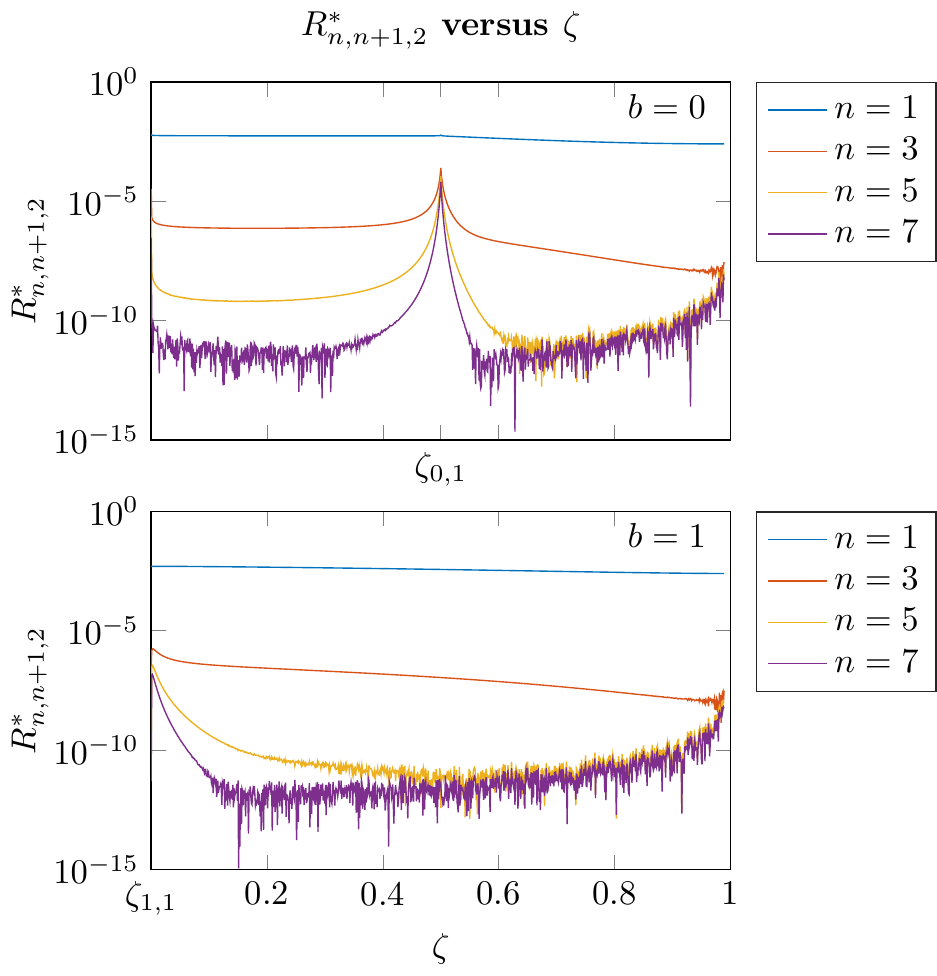}
\caption{Absolute relative error bound $R_{n,n+1,2}^\ast$ versus $\zeta$ for $\sarb_0=0.5$, $\sacv_0=0.05$, $b=0,1$, and $n=1,3,5,7$.}
\label{fig:Rmn_plots}
\end{figure}

Given these observations, Algorithm \ref{alg:optimal_sample_sizes} presents a simple procedure for computing the optimal sample sizes for $\mathscr T_\nudag$.
\begin{algorithm}[htb]
\caption{Optimal sample size pseudo-code.}
\label{alg:optimal_sample_sizes}
\begin{algorithmic}[1]
\Procedure{OptimalSamples}{$\zeta$, $\sarb_0$, $b$, $\sacv_0$, $\epsilon$, $n_{\max}$}
\If{$\zeta\in\{0,\zeta_{b,1}\}$}
\State $(n_1^\text{opt},n_2^\text{opt})=(2\sacv_0^{-2}+5,1)$;
\State \Return $(n_1^\text{opt},n_2^\text{opt})$;
\Else
\State $n=1$;
\While{$R_{n,n+1,2}^\ast(\zeta,n_1^\ast,n_2^\ast)>\epsilon$ \textbf{and} $n<n_{\max}$}
\State $n=n+1$;
\EndWhile
\State Substitute $\acv_n^2\mathscr T_\nudag$ in (\ref{eq:optimal_alpha_SOE});\label{line:make_SOE}
\State Solve for $(n_1^\text{opt},n_2^\text{opt})$ using starting point $(n_1^\ast,n_2^\ast)$;\label{line:solve_n1n2}
\If{$R_{n,n+1,2}^\ast(\zeta,n_1^\text{opt},n_2^\text{opt})\leq\epsilon$ \textbf{or} $n=n_{\max}$}
\State \Return $(n_1^\text{opt},n_2^\text{opt})$;
\Else
\State $n=n+1$;
\State \textbf{goto} line \ref{line:make_SOE};
\EndIf
\EndIf
\EndProcedure
\end{algorithmic}
\end{algorithm}

To demonstrate its effectiveness, Algorithm \ref{alg:optimal_sample_sizes} was implemented in {\sc mathematica} using the {\ttfamily{FindRoot[]}} function for Line \ref{line:solve_n1n2} to numerically solve the system (\ref{eq:optimal_alpha_SOE}). The procedure was executed for $\zeta\in[0,1)$, $\sarb_0=0.5$, $\sacv_0=0.05$, and four different values of $b$ ranging from zero to one. Figure \ref{fig:optimal_n1_n2_vs_b} plots the computed sample sizes along with $\mathscr E_{\mathscr T_\nudag}$ evaluated at a dark noise of $\sigma_{\mathrm d}g=1\electron$. The values chosen for $\sarb_0$ and $\sigma_{\mathrm d}g$, while not typical, were chosen to exaggerate the features of the optimal sample size curves discussed in the preceding analysis.

Focusing first on the curve for $n_2^\text{opt}$ we observe for values of $b<1$ that the curve has a cusp at $\zeta=0$ and $\zeta=\zeta_{b,1}$ where it takes on a value of one. As, $\zeta\searrow 1$ this curve approaches the $n_1^\text{opt}$ curve; both of which approach $C_{\mathscr T_\nudag}(1-\zeta)^{-2}$ in the limit. As for the $\mathscr E_{\mathscr T_\nudag}$ curve, we see that it equals one at $\zeta=0$ and $\overline{\mathscr E_{\mathscr T_\nudag}}=0.15847\dots$ at $\zeta=1$. Furthermore, the cusp present in the $n_2^\text{opt}$ curve is also present in the curve for $\mathscr E_{\mathscr T_\nudag}$, which highlights the reason why choosing values of $b$ less than one may yield undesirable results, in particular, unacceptably low values of $\mathscr E_{\mathscr T_\nudag}$.

As for the computational details of the routine, it was decided to use $\epsilon=5\times 10^{-7}$ so that $\acv^2\mathscr T_\nudag$ was approximated to at least seven significant digits and $n_{\max}=20$ to avoid issues of numerical stability. Figure \ref{fig:terms_and_error_plots} plots the number of terms, $n$, and upper bound on the absolute relative error, $R_{n,n+1,2}^\ast$, as a function of $\zeta$ obtained from the routine for the case $b=0$. From the figure we observe that $R_{n,n+1,2}^\ast$ exhibits a sharp increase near the points $\zeta\in\{0,\zeta_{b,1}\}$; resulting in larger values for $n$. Only a few $\zeta$-values near $\zeta_{b,1}$ hit the bound $n=20$, which is the same behavior observed for all other values of $b$ in Figure \ref{fig:optimal_n1_n2_vs_b}. To check the quality of the solution at these problematic points we note that in the neighborhood of $\zeta=\zeta_{b,1}$ that $\nudag$ and $n_2^\text{opt}$ are quite small while $n_1^\text{opt}$ is large. Recalling the double integral representation of $\ev\mathscr T_\nudag^2$ we may then approximate $\acv^2\mathscr T_\nudag$ with
\begin{multline}
\label{eq:acv2_integral_approx}
\acv_n^2\mathscr T_\nudag=%
\int_0^1\int_0^1\sum_{k=0}^n\frac{(1+\nudag)_k^2}{(\alpha_1)_k k!}\frac{\left((1-x)(1-y)\right)^k}{\left((1-(1-\zeta)x)(1-(1-\zeta)y)\right)^k}\\
\times\pFq{}{}{1-\nudag,1-\nudag}{\alpha_2}{(1-x)(1-y)} f_{XY}(x,y)\,\mathrm dx\mathrm dy-1,
\end{multline}
where
\[
f_{XY}(x,y)=\frac{\left((1-(1-\zeta)x)(1-(1-\zeta)y)\right)^{-\nudag-1}}{F(1,1+\nudag;2;1-\zeta)^2}
\]
and the error $E_n$ such that $\acv^2\mathscr T_\nudag=\acv_n^2\mathscr T_\nudag+E_n$ is positive and bounded above by
\begin{multline*}
E_n^\ast=%
\pFq{}{}{1+\nudag,1+\nudag}{\alpha_1}{1}_{\!\! n+1}^{\! c}
\int_0^1\int_0^1\frac{\left((1-x)(1-y)\right)^{n+1}}{\left((1-(1-\zeta)x)(1-(1-\zeta)y)\right)^{n+1}}\\
\times\pFq{}{}{1-\nudag,1-\nudag}{\alpha_2}{(1-x)(1-y)}f_{XY}(x,y)%
\,\mathrm dx\mathrm dy
\end{multline*}
with $F(\alpha,\beta;\gamma;z)_n^c$ given in Proposition \ref{prop:comp_inc_2F1_explicit_forms}. Looking back as the numerical data we observe that one of the problematic points in Figure \ref{fig:terms_and_error_plots} occurs at $\zeta=0.499443$. Using this value for $\zeta$ and the optimal sample sizes computed from our routine we find
\[
\left\lvert\frac{\acv_{20}^2\mathscr T_\nudag}{\acv^2\mathscr T_\nudag}-1\right\rvert\leq6.21\times 10^{-15},
\]
which shows that the approximation $\acv_{20}^2\mathscr T_\nudag$ given by (\ref{eq:acv2_integral_approx}) approximates the exact value $\acv^2\mathscr T_\nudag$ to approximately fourteen significant digits at this point of interest. Substituting $\acv_{20}^2\mathscr T_\nudag$ into (\ref{eq:optimal_alpha_SOE}) we then find
\[
\begin{aligned}
\partial_{\alpha_2}\acv_{20}^2 \mathscr T_\nudag(\mathrm A-\alpha_2,\alpha_2|\theta_{\mathscr T_\nudag})\Big|_{(\mathrm A,\alpha_2)=(\alpha_1^{\normalfont\text{opt}}+\alpha_2^{\normalfont\text{opt}},\alpha_2^{\normalfont\text{opt}})} &=-6.84\times 10^{-9}\\
\acv_{20}^2 \mathscr T_\nudag(\alpha_1^{\normalfont\text{opt}},\alpha_2^{\normalfont\text{opt}}|\theta_{\mathscr T_\nudag})-\sacv_0^2 &=9.48\times 10^{-10},
\end{aligned}
\]
which shows no significant deviation from the desired solution $(0,0)$. This confirms that the solution to the optimal sample sizes obtained in our routine near $\zeta=\zeta_{b,1}$ are not in significant error as the upper bound $R_{n,n+1,2}^\ast$ might suggest.

\begin{figure}[htb]
\centering
\includegraphics[scale=1]{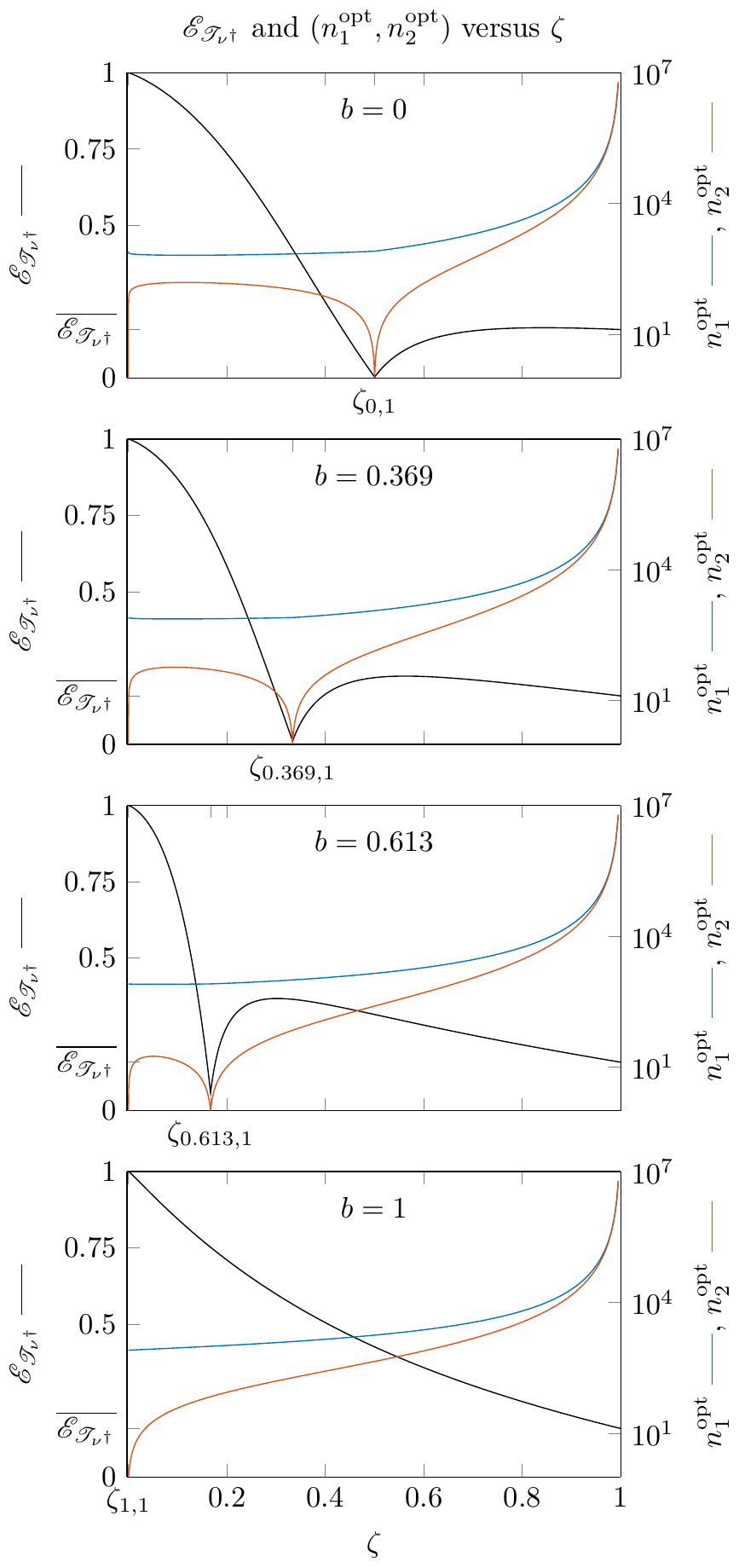}
\caption{$\mathscr E_{\mathscr T_\nudag}$ and $(n_1^\text{opt},n_2^\text{opt})$ versus $\zeta$ for $\sarb_0=0.5$, $\sacv_0=0.05$, $b=\{0,0.369,0.613,1\}$, and $\sigma_{\mathrm d}g=1\electron$. Values of $\zeta_{b,1}$ and $\overline{\mathscr E_{\mathscr T_\nudag}}$ are also indicated on each plot.}
\label{fig:optimal_n1_n2_vs_b}
\end{figure}

\begin{figure}[htb]
\centering
\includegraphics[scale=1]{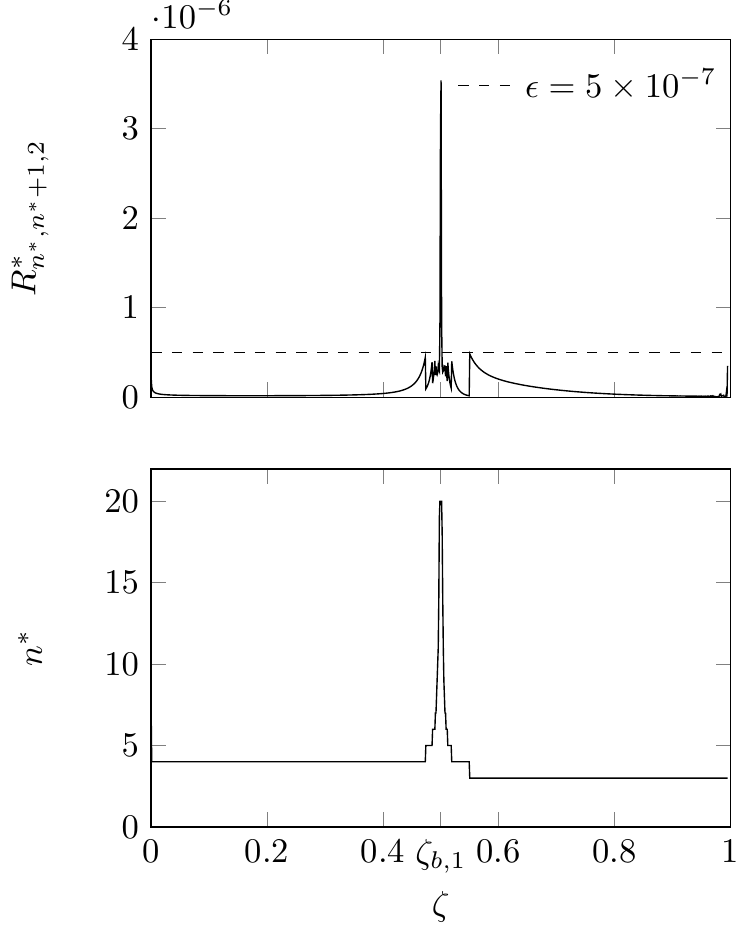}
\caption{$R_{n^\ast,n^\ast+1,2}^\ast$ (top) and $n^\ast$ (bottom) versus $\zeta$ for $\sarb_0=0.5$, $b=0$, and $\sacv_0=0.05$.}
\label{fig:terms_and_error_plots}
\end{figure}

%
%
%
%

\section{Pixel-level conversion gain estimation: An introduction}
\label{sec:CG_estimation_KAI0407M}

Estimating conversion gain with the {\sc pt} method is typically performed under the assumption that each pixel in the sensor array exhibits identical characteristics. When this assumption holds we say the sensor is \emph{uniform} and treat each pixel as a source of i.i.d.~random noise. Such an assumption is very convenient for performing {\sc pt} characterization as we can leverage the large number of pixels in a sensor array to obtain low-variance estimates of critical population values. Take for example the estimate of population dark variance $\sigma_{\mathrm d}^2$. When the assumption of uniformity holds, one can estimate this quantity from just two images captured under dark conditions. Denoting $(D_1,D_2)$ as these two dark images each with $R$ rows of pixels and $C$ columns of pixels we first compute the pixelwise difference $\Delta D=D_1-D_2$. Since this differencing operation effectively doubles the population variance we then estimate $\sigma_{\mathrm d}^2$ by computing half of the sample variance of the difference-frame
\[
\hat\sigma_{\mathrm d}^2=\frac{1}{2}\frac{1}{RC-1}\sum_{i=1}^R\sum_{j=1}^C(\Delta D_{ij}-\overline{\Delta D})^2,
\]
where $\Delta D_{ij}$ denotes the value of the difference-frame $\Delta D$ at row $i$ and column $j$ and $\overline{\Delta D}=\frac{1}{RC}\sum_{i=1}^R\sum_{j=1}^C\Delta D_{ij}$ denotes the sample mean of the difference-frame. For a sensor with a very modest one megapixel resolution $RC=10^6$ and so from just two images we can obtain an estimate of $\sigma_{\mathrm d}^2$ based on an effective sample size of one-million i.i.d.~observations.

In actual sensors, the assumption of uniformity is rarely a good model and is only useful for a small number of pixel architectures. Indeed, in Janesick's book \emph{Photon Transfer} he points out that nonuniformity is a key characteristic of the \emph{Complementary Metal-Oxide Semiconductor} ({\sc cmos}) sensor architecture, which is by far the most prevalent architecture in modern imaging systems \cite[pg.~82]{janesick_2007}. When the uniformity assumption is violated we no can longer provide global estimates of the necessary population parameters and must turn to estimates for each individual pixel. Going back to the example of estimating $\sigma_{\mathrm d}^2$, pixel-level estimation is performed by capturing a sequence of images, i.e.~an \emph{image stack}, and then computing the pixelwise temporal variance of the stack. Supposing we capture a stack of $N$ dark images $(D_1,\dots,D_N)$ to form a $R\times C\times N$ dark image stack, the estimate of $\sigma_{\mathrm d}^2$ for the pixel in row $i$ and column $j$ becomes
\begin{equation}
\label{eq:temporal_dark_variance}
(\hat\sigma_{\mathrm d}^2)_{ij}=\frac{1}{N-1}\sum_{k=1}^N(D_{ijk}-\overline{D}_{ij})^2,
\end{equation}
where $D_{ijk}$ is the value of pixel $(i,j)$ in the $k$th image and $\overline{D}_{ij}=\frac{1}{N}\sum_{k=1}^ND_{ijk}$ is the temporal sample mean of that same pixel. From this, we can immediately see the challenge introduced in pixelwise estimation, for in order to estimate $\sigma_{\mathrm d}^2$ for each pixel with the same level of uncertainty as in the uniform case a total of $N=10^6$ images must be captured! Setting aside the shear ammount of data this entails we also run into the practical problem of simply capturing the images in a short enough period of time so that drift does not corrupt the population parameters we are attempting to measure. So while the extension to pixelwise estimation is conceptually straightforward, implementation can be quite challenging.

Given that uniformity is generally the exception and not the rule, a generalized pixel-level approach to {\sc pt} characterization is in order.  Our goal here is to take the theoretical results presented in this work and apply them to develop a first attempt at the process of pixel-level conversion gain estimation. Before this is done we must discuss the topic of \emph{pixel grouping}.


\subsection{Pixel grouping}
\label{subsec:pixel_grouping}

When the assumption of uniformity is removed and pixel-level measurement is performed, the characterization of each individual pixel effectively becomes its own experiment with its own unique optimal sample sizes. However, since we can only simultaneously sample all pixels by capturing full-frame images, the conclusion of any pixel-level characterization procedure will be a single set of sample sizes that must work for (nearly) all pixels in the sensor array\footnote{We say nearly because typical sensors always contain some minority of defect pixels which are either non-functioning or exhibit statistical behavior too extreme to be properly characterized.}. Knowing that estimates for each pixel will ultimately be computed with the same number of samples, individually estimating optimal sample sizes is nonsensical and thus we should group pixels with similar characteristics to reduce the variance in estimates of the optimal sample sizes. This point is made evident by first recalling the asymptotic behavior of the optimal sample sizes in the low-illumination limit $\zeta\searrow 1$
\[
(n_1^\text{opt},n_2^\text{opt})(\zeta)\sim (C_{\mathscr T_\nudag}(1-\zeta)^{-2},C_{\mathscr T_\nudag}(1-\zeta)^{-2}).
\]
Letting $\hat Y=\hat\sigma_{\mathrm d}^2$ be an estimate of $\sigma_{\mathrm d}^2$ according to (\ref{eq:temporal_dark_variance}) and $\hat X=\hat\sigma_{\mathrm p+\mathrm d}^2$ be the analogous estimate of $\sigma_{\mathrm p+\mathrm d}^2$ for a single pixel based on dark and illuminated image stacks of $n_2$ and $n_1$ images, respectively, we can estimate $\zeta$ with
\[
Z=\frac{n_1-3}{n_1-1}\frac{\hat Y}{\hat X}.
\]
Under the assumed normal model, $Z$ is the {\sc umvue} for $\zeta$ with $\ev Z=\zeta$ and
\[
\var Z=\left(\frac{(n_1-3)(n_2+1)}{(n_1-5)(n_2-1)}-1\right)\zeta^2.
\]
While $Z$ possesses good characteristics as an estimator for $\zeta$ its density is nonzero in the neighborhood of $\zeta=1$ and so the moments of the optimal sample size estimates $(\hat n_1^\text{opt},\hat n_2^\text{opt})=(n_1^\text{opt},n_2^\text{opt})(Z)$ are strictly speaking undefined. However, if instead of measuring $Z$ for each pixel individually we identify a group of $m$ pixels with similar population values for $\sigma_{\mathrm d}^2$ and $\sigma_{\mathrm p+\mathrm d}^2$ we may estimate $\zeta$ for the entire group with
\[
Z=\frac{mn_1-3}{mn_1-1}\frac{\sum_{\text{group}}\hat Y_\ell}{\sum_{\text{group}}\hat X_\ell},
\]
where $\hat Y_\ell$ and $\hat X_\ell$ denotes the estimates for the $\ell$th pixel in the group. Grouping in this manner does not alter the unbiasedness of $Z$ and effectively increases the dark and illuminated sample sizes by a factor of $m$. Hence, given sufficiently large $n_1$, $n_2$, and group size $m$, one may drive down the asymptotic variance of the optimal sample size estimates
\[
\var\, \hat n_i^\text{opt}\sim(\partial_\zeta n_i^\text{opt}(\zeta))^2\var Z,\quad i=1,2
\]
to the point where they becomes useful.

Another important reason why grouping pixels in this manner is not only important but practically necessary is in estimating the quantity $\nudag(\zeta)=\log_\zeta(\sarb_0)+b$. In a similar manner to the optimal sample sizes we see as $\zeta\searrow 1$
\[
\nudag(\zeta)\sim-\log\sarb_0(1-\zeta)^{-1}.
\]
The pole of order one at $\zeta=1$ again implies that the estimate $\nudaghat=\nudag(Z)$ has undefined moments although we may assign it an asymptotic variance
\[
\var\,\nudaghat\sim\frac{\log^2_\zeta\sarb_0}{\zeta^2\log^2\zeta}\var Z.
\]
Since we ultimately want an estimate of $\nudag$ with nearly zero variance\footnote{Recall that the derivation of the mean and variance of the estimator $\mathscr G_\nu$ assumed $\nu$ is a known constant.}, we need to turn to pixel grouping to leverage the asymptotics of large sample sizes to produce low-variance estimates of $\nudag$ for groups of similar pixels.

Identifying appropriate pixel groups for pixel-level characterization is largely determined by the architecture of the sensor under test. In what follows, we will demonstrate how to identify pixels groups through example as we step through the process of pixel-level conversion gain estimation for a real image sensor.


\subsection{Experimental setup}
\label{subsec:experimental_setup}

For this experiment, the ON Semiconductor KAI-0407M monochrome interline transfer \emph{Charge-Coupled Deivce} ({\sc ccd}) sensor was chosen as an initial test case. This particular sensor was chosen for two main reasons. First, preliminary tests show that the noise it produces closely adheres to the normal model assumed in the development of the estimator $\mathscr G_\nu$. Secondly, interline transfer {\sc ccd}s typically exhibit only minor nonuniformities. Since we expect significant nonuniformities to complicate the experimental procedure, this particular sensor provides a gentle introduction into pixel-level characterization.

The experimental setup began with a $650\,\mathrm{nm}$ superluminescent light emitting diode ({\sc sled}), which was collimated and attached to a \emph{Variable Optical Attenuator} ({\sc voa}) to facilitate control over total output power of the light source. The intensity of the light exiting the {\sc voa} could be adjusted from $0-100\%$ power by rotating a half-wave plate inside the device. To illuminate the sensor with a uniform field, the {\sc sled} beam exiting the {\sc voa} was directed into a $12\,\mathrm{in}.$ diameter integrating sphere and the sensor was placed at the output port of the sphere where uniformity is highest. The sensor was operated at its full bit-depth of $14$-bits to minimize quantization error and image data was read off the sensor using a single readout register at its slowest setting ($40\,\mathrm{mhz}$) as not to introduce additional nonuniformities or smear. For the sake of reducing the total amount of data captured in the experiment, the source was turned on and a live-stream of the sensor output was examined to select the most uniformly illuminated $512\times 512\,\mathrm{px}$ subregion of the sensor array. All subsequent images captured in the experiment were cropped to this subregion prior to saving. Since image data needed to be captured under both dark and illuminated conditions, a motorized mirror was placed next to the path of the {\sc sled} beam. Moving the mirror into the beam path effectively redirected the light away from the integrating sphere and thus provided a dark environment for the sensor.


\subsection{Design of experiment}
\label{subsec:design_of_experiment}

\emph{Design of Experiment} ({\sc doe}) for pixel-level conversion gain measurement at its core involves five major steps: (1) choosing values for $\sarb_0$, $b$ (bias profile), and $\sacv_0$, (2) estimating a global lower bound on the dark noise $\sigma_{\mathrm d}g$, (3) computing $n_1^\text{opt}$, $n_2^\text{opt}$, and $\mathscr E_{\mathscr T_\nudag}$ on $\zeta\in[0,1)$, (4) choosing an illumination level, and (5) identifying pixel groups.

Beginning with the first step it was decided to use $\sarb_0=0.01$, $b=0$ (constant bias profile), and $\sacv_0=0.05$ as these choices reflect typical values one might use. Then, to obtain a global lower bound estimate of the dark noise, two $512\times 512\,\mathrm{px}$ dark images $(D_1,D_2)$ were captured. The pixelwise difference $\Delta D=D_1-D_2$ of these two images was computed and $\sigma_{\mathrm d}^2$ was estimated by half of the sample variance
\[
\hat\sigma_{\mathrm d}^2=\frac{1}{2(512^2-1)}\sum_{i=1}^{512}\sum_{j=1}^{512}(\Delta D_{ij}-\overline{\Delta D})^2=39.94\,\mathrm{DN}^2.
\]
Since the value of the conversion gain is unknown, we make the safe assumption $g\geq 1$ and then obtain a single lower bound estimate of the dark noise for each pixel by
\[
\widehat{\sigma_{\mathrm d}g}=\hat\sigma_{\mathrm d}\,(\mathrm{DN})\times 1 (\electron/\mathrm{DN})=\sqrt{39.94}\,\electron.
\]
For the next step we compute the optimal sample size curves according to Algorithm \ref{alg:optimal_sample_sizes}. Figure \ref{fig:optimal_n1_n2} plots the computed curves for $\mathrm N^\text{opt}=n_1^\text{opt}+n_2^\text{opt}$ and $\hat{\mathscr E}_{\mathscr T_\nudag}$ given by
\[
\hat{\mathscr E}_{\mathscr T_\nudag}=\left(1+\frac{1+\sacv_0^{-2}}{(\widehat{\sigma_{\mathrm d}g})^2}\frac{\zeta}{(1-\zeta)^2}\left(\frac{1}{n_1^{\text{\normalfont opt}}}+\frac{\zeta}{n_2^{\text{\normalfont opt}}}\right)\right)^{-1}
\]
along with the curves for $n_1^\text{opt}$ and $n_2^\text{opt}$ as functions of $\zeta$.
\begin{figure}[htb]
\centering
\includegraphics[scale=1]{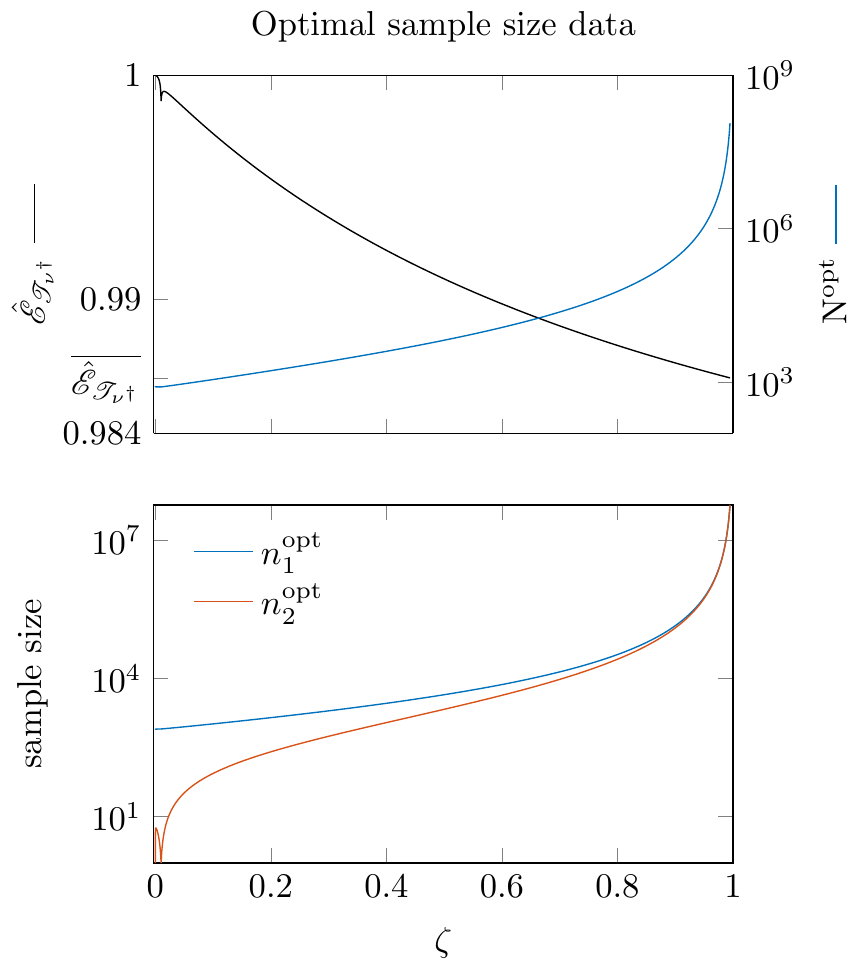}
\caption{$\mathscr E_{\mathscr T_\nudag}$ and $\mathrm N^\text{opt}$ versus $\zeta$ (top) with $n_1^\text{opt}$ and $n_2^\text{opt}$ versus $\zeta$ (bottom) for $\sarb_0=0.01$, $\sacv_0=0.05$, $b=0$, and $\sigma_{\mathrm d}g=\sqrt{39.94}\electron$.}
\label{fig:optimal_n1_n2}
\end{figure}
Due to the sufficiently large value of $\widehat{\sigma_{\mathrm d}g}$ we see that $\hat{\mathscr E}_{\mathscr T_\nudag}$ is near unity for virtually any illumination level so that the requirement $\mathscr E_{\mathscr T_\nudag}\approx 1$ will not restrict what illumination levels we can choose for the experiment. To select an appropriate illumination level we first note that this sensor can record images at $\approx 5\,\mathrm{fps}$ for the chosen readout rate of $40\,\mathrm{mhz}$. Looking back at Figure \ref{fig:optimal_n1_n2} we observe that the illumination level corresponding to $\zeta\approx 0.4$ is paired with an optimal total sample size of $\mathrm N^\text{opt}\approx 4000$. At a recording rate of $5\,\mathrm{fps}$ this number of images will take $\approx 13\,\mathrm{min.}$ to capture, which is short enough to avoid any significant drift in the sensor or source.

Now equipped with a desired value for $\zeta$, the variable optical attenuator was adjusted until the illumination level at the sensor plane resulted in a mean pixel output of about $1\%$ of the sensors dynamic range. Using the same process to estimate the dark noise, two illuminated frames $(I_1,I_2)$ were captured and their difference $\Delta I=I_1-I_2$ was used to obtain the estimate
\[
\hat\sigma_{\mathrm p+\mathrm d}^2=\frac{1}{2(512^2-1)}\sum_{i=1}^{512}\sum_{j=1}^{512}(\Delta I_{ij}-\overline{\Delta I})^2=112.89\,\mathrm{DN}^2,
\]
which then gave an estimate $\hat\zeta=Z$ equal to
\[
Z=\frac{512^2-3}{512^2-1}\frac{\hat\sigma_{\mathrm d}^2}{\hat\sigma_{\mathrm p+\mathrm d}^2}=0.354.
\]
This illumination was sufficiently close to the target $\zeta=0.4$ and corresponded to $\hat{\mathscr E}_{\mathscr T_\nudag}=0.9928$ and $\mathrm N^\text{opt}=3237$ images, which needs only $\approx 11\,\mathrm{min}.$ to capture.

The last step in the {\sc doe} process is to determine appropriate pixel groupings for the data collection algorithm. As discussed in Section \ref{subsec:pixel_grouping}, identifying these pixel groups is largely aided by knowledge of the sensor architecture under test. In this example, we know we are working with an interline transfer {\sc ccd} architecture \cite{KAI04070_datasheet}. These devices work by transferring the packets of charge collected by each pixel into columnwise vertical shift registers that subsequently facilitate the transfer of the charge packets off the sensor to downstream readout circuity. Since variations in the columnwise circuitry supporting these vertical shift registers is expected, nonuniformities in interline transfer {\sc ccd}s is typically observed between its columns. To see this in action, a stack of dark images were captured with the KAI-0407M {\sc ccd} and the pixelwise temporal average of the stack was computed as seen in Figure \ref{fig:barY_map}. From the figure we see clear columnwise variations in the estimated population dark mean $\mu_{\mathrm d}$, which supports our conclusions about the sensor nonuniformities. As such, we will group pixels according to their column number for the data collection procedure.
\begin{figure}[htb]
\centering
\includegraphics[scale=1]{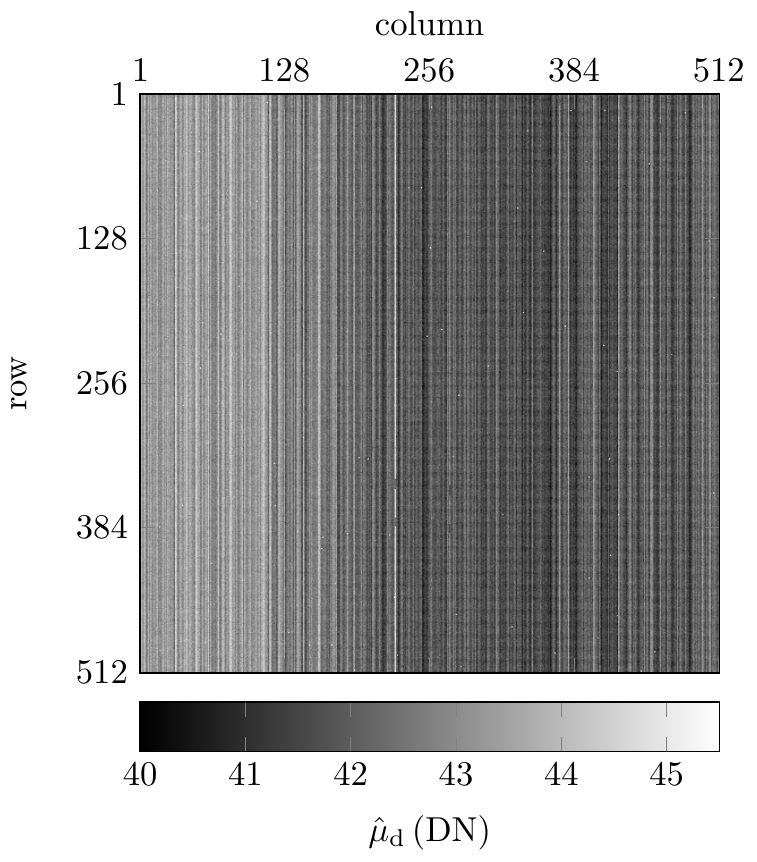}
\caption{Average of dark image stack generated from the KAI-0407M interline transfer {\sc ccd} showing columnwise nonuniformities in mean dark signal.}
\label{fig:barY_map}
\end{figure}


\subsection{Data collection}
\label{subsec:data_collection}

The procedure for the data collection portion of this experiment centers around successively capturing images under dark and illuminated conditions, using these images to update sample statistics for each pixel, and determining when the appropriate number of dark and illuminated images have been captured based on these sample statistics. We will initialize the procedure by capturing some predetermined number of dark images (i.e.~a $Y$-image stack) and illuminated images (i.e.~a $X$-image stack) and then compute four \emph{master frames}, denoted $\bar Y$, $\bar X$, $\hat Y$, and $\hat X$, from these stacks. Using this nomenclature, the $\bar Y$ and $\bar X$ master frames denote arrays computed by averaging the $Y$- and $X$-image stacks in the temporal dimension while the $\hat Y$ and $\hat X$ master frames denote arrays computed by evaluating the variance of the $Y$- and $X$-image stacks in the temporal dimension. As an example, Figure \ref{fig:variance_operation} depicts the process of computing a master $\hat X$-frame from a stack of $n$ $X$-images.
\begin{figure}[htb]
\centering
\includegraphics[scale=1]{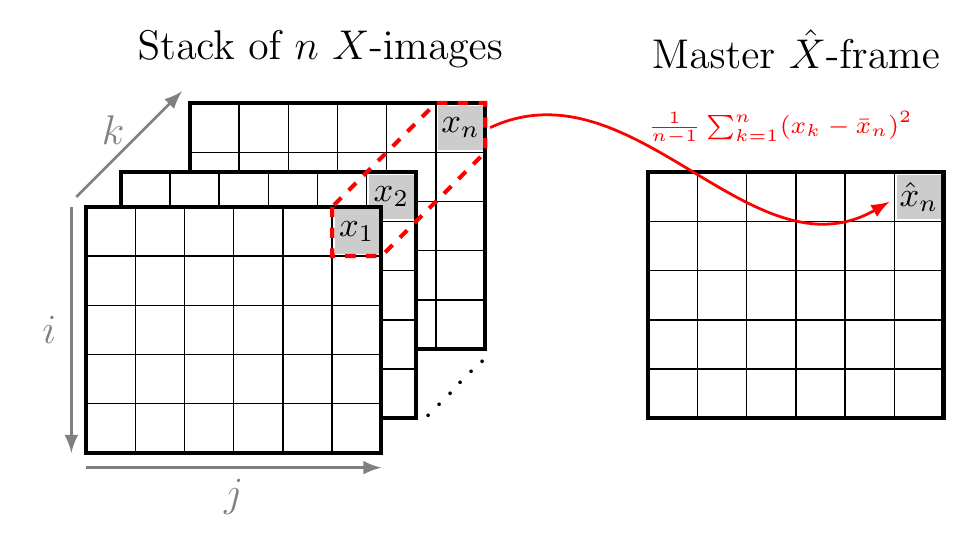}
\caption{Master $\hat X$-frame computed from a stack of $n$ $X$-images.}
\label{fig:variance_operation}
\end{figure}

With master frames computed from the two initial image stacks, we will then capture individual $Y$- and $X$-images and use these to continuously update the master frames until a specified criteria is met to halt data capture. Due to the large number of images that will be captured in this experiment, updating the master frames by recomputing the sample statistics from the full image stacks is too computationally expensive to work in practice. As such we will use Welford's algorithm to update the master frames when new image data becomes available. Algorithm \ref{alg:Welford_online} presents the procedure {\sc UpdateStats}(), which is an adaptation of Welford's algorithm to two-dimensional arrays \cite{welford_1962}. This procedure accepts as arguments a new two-dimensional array of data, {\ttfamily newdata}, the current master sum of squares array, $M_{2,n}$, the current master average frame $\bar x_n$, and the number of samples, $n$, the current master frames are computed from. If no master frames have been initialized we may use the syntax {\sc UpdateStats}({\ttfamily newdata}, $0$) to create them. For all subsequent $n\geq 1$, the syntax {\sc UpdateStats}({\ttfamily newdata}, $n$, $M_{2,n}$, $\bar x_n$) takes in the current master arrays and updates them with {\ttfamily newdata} according to Welford's algorithm. We note that all arithmetic operations in this procedure are performed elementwise, where, in particular, the operator $\odot$ denotes the Hadamard product (elementwise multiplication of matrices).
\begin{algorithm}[htb]
\caption{Welford's online algorithm for two-dimensional arrays.}
\label{alg:Welford_online}
\begin{algorithmic}[1]
\Procedure{UpdateStats}{{\ttfamily{newdata}}, $n$, $M_{2,n}$, $\bar x_n$}
\State $n=n+1$;
\If{$n=1$}
\State $\bar x_n=\text{\ttfamily{newdata}}$;
\State $M_{2,n}=\text{zeros(size(\ttfamily{newdata}))}$;
\State $s_n^2=\text{zeros(size(\ttfamily{newdata}))}$;
\Else
\State $\bar x_n=\bar x_{n-1}+\frac{1}{n}(\text{\ttfamily{newdata}}-\bar x_{n-1})$;
\State $M_{2,n}=M_{2,n-1}+(\text{\ttfamily{newdata}}-\bar x_{n-1})\odot(\text{\ttfamily{newdata}}-\bar x_n)$;
\State $s_n^2=\frac{1}{n-1}M_{2,n}$;
\EndIf
\State \Return $n$, $\bar x_n$, $M_{2,n}$, $s_n^2$
\EndProcedure
\end{algorithmic}
\end{algorithm}

The initial dark and illuminated sample sizes for the data collection algorithm are determined from looking back at the plots for $n_1^\text{opt}$ and $n_2^\text{opt}$ in Figure \ref{fig:optimal_n1_n2}. Regardless of the parameters chosen, we know that an optimal dark sample size of $n_2=1$ occurs for at least one $\zeta$-value and so we set the default initial number of $Y$-images to be one. This, however, cannot be said for the optimal illuminated sample size $n_1$, which always takes on values greater than one. While the optimal illuminated sample size curve is not always strictly increasing, we will choose its shot noise-limited value $2\sacv_0^{-2}+5$ as a good approximation to the minimum possible value and make this as our required initial number of initial $X$-images.

Algorithm \ref{alg:capture_data} presents the generic {\sc CollectImageData}({\ttfamily n1Min}, {\ttfamily Yrule}, {\ttfamily Xrule}, {\ttfamily HaltRule}) procedure used for collecting and processing real-time image data in the pixel-level conversion gain estimation experiment. The procedure accepts an initial illuminated sample size, {\ttfamily n1Min}, which will be $2\sacv_0^{-2}+5$ for this experiment, as well as three rules for determining when to halt $Y$-image collection ({\ttfamily Yrule}), $X$-image collection ({\ttfamily Xrule}), and the entire routine ({\ttfamily HaltRule}). From Algorithm \ref{alg:capture_data} we observe that the {\sc CollectImageData}() procedure begins by initializing the master frames and updating the master frames until the initial minimum sample sizes are captured. The procedure then enters a loop whereby the master frames are first updated if their corresponding rule is false. Once the master frames are updated, the estimates of $\zeta$, $\nudag$, $n_1^\text{opt}$, and $n_2^\text{opt}$ are updated for each pixel group. Lastly, the {\ttfamily Yrule}, {\ttfamily Xrule}, and {\ttfamily HaltRule} are reevaluated to determine if more images are needed and if the procedure can stop. Upon completion, the procedure then returns the final sample sizes $n_1$ and $n_2$, a vector $V$ containing the final group estimates of $\nudag$, another vector $Z$ containing the final group estimates of $\zeta$, and the four master frames $\bar Y$, $\bar X$, $\hat Y$, and $\hat X$.
\begin{algorithm}[htb]
\caption{Data collection procedure.}
\label{alg:capture_data}
\begin{algorithmic}[1]
\Procedure{CollectImageData}{{\ttfamily n1Min}, {\ttfamily Yrule}, {\ttfamily Xrule}, {\ttfamily HaltRule}}
\State $n_2=0$;
\State Capture one $Y$-image;
\State [$n_2$, $\bar Y$, $\hat Y$, $M_2^Y$]={\sc UpdateStats}($Y$, $n_2$);
\State $n_1=0$;
\State Capture one $X$-image;
\State [$n_1$, $\bar X$, $\hat X$, $M_2^X$]={\sc UpdateStats}($X$, $n_1$);
\While{$n_1<\text{{\ttfamily n1Min}}$}
\State Capture one $X$-image;
\State [$n_1$, $\bar X$, $\hat X$, $M_2^X$]={\sc UpdateStats}($X$, $n_1$, $M_2^X$, $\bar X$);
\EndWhile
\State
\State Yflag=false;
\State Xflag=false;
\State Stop=false;
\State $N=\text{\# of groups}$;
\While{Stop=false}
\If{Yflag=false}
\State Capture one $Y$-image;
\State [$n_2$, $\bar Y$, $\hat Y$, $M_2^Y$]={\sc UpdateStats}($Y$, $n_2$, $M_2^Y$, $\bar Y$);
\EndIf
\If{Xflag=false}
\State Capture one $X$-image;
\State [$n_1$, $\bar X$, $\hat X$, $M_2^X$]={\sc UpdateStats}($X$, $n_1$, $M_2^X$, $\bar X$);
\EndIf
\For{$i=1:N$}
\State $m_i=\text{\# of pixels in group }i$;
\State $Z_i=\frac{m_in_1-3}{m_in_1-1}(\sum_{\text{group }i}\hat Y)/(\sum_{\text{group }i}\hat X)$;
\State $V_i=\nudag(Z_i)$;
\State $(\hat n_1^\text{opt},\hat n_2^\text{opt})_i=(n_1^\text{opt},n_2^\text{opt})(Z_i)$;
\EndFor
\State $\text{Yflag}$={\ttfamily Yrule};
\State $\text{Xflag}$={\ttfamily Xrule};
\State Stop={\ttfamily HaltRule};
\EndWhile
\State\Return $n_2$, $n_1$, $V$, $\bar Y$, $\bar X$, $\hat Y$, $\hat X$
\EndProcedure
\end{algorithmic}
\end{algorithm}

For the KAI-0407M pixel-level conversion gain estimation experiment the following three simple rules were chosen:
\[
\text{\ttfamily Yrule}=%
\begin{cases}
\text{true}, &\text{if}\ \hat n_2^\text{opt}\geq n_2\ \text{for all groups}\\
\text{false}, &\text{otherwise},
\end{cases}
\]
\[
\text{\ttfamily Xrule}=%
\begin{cases}
\text{true}, &\text{if}\ \hat n_1^\text{opt}\geq n_1\ \text{for all groups}\\
\text{false}, &\text{otherwise},
\end{cases}
\]
\[
\text{\ttfamily HaltRule}=%
\begin{cases}
\text{true}, &\text{if}\ (\hat n_1^\text{opt}\geq n_1\land\hat n_2^\text{opt}\geq n_2)\ \text{for}\ \geq 95\%\ \text{of all groups}\\
\text{false}, &\text{otherwise}.
\end{cases}
\]
We see that {\ttfamily Yrule} and {\ttfamily Xrule} direct the algorithm to keep collecting images if the estimated optimal sample sizes are less than the current sample sizes for any of the column groups. Furthermore, {\ttfamily HaltRule} halted the procedure when the current sample sizes exceeded the estimates for at least $95\%$ of the column groups.


\subsection{Summary of results}
\label{subsec:summary_of_results}

The data collection algorithm was executed on the KAI-0407M {\sc ccd} sensor resulting in dark and illuminated sample sizes of $n_2=862$ and $n_1=2469$ images, respectively. Figure \ref{fig:nu_z_vs_iteration} plots the group estimates of $\zeta$ and $\nudag$ versus iteration number of the algorithm for fifteen randomly selected columns. We can see that these estimates converge to different values, which confirms our choice of pixel grouping is appropriate. Furthermore, we observe that the variance of these estimates in the final iterations of the algorithm is practically zero showing that the size (i.e.~number of pixels) of the chosen pixel groups was also sufficiently large.
\begin{figure}[htb]
\centering
\includegraphics[scale=1]{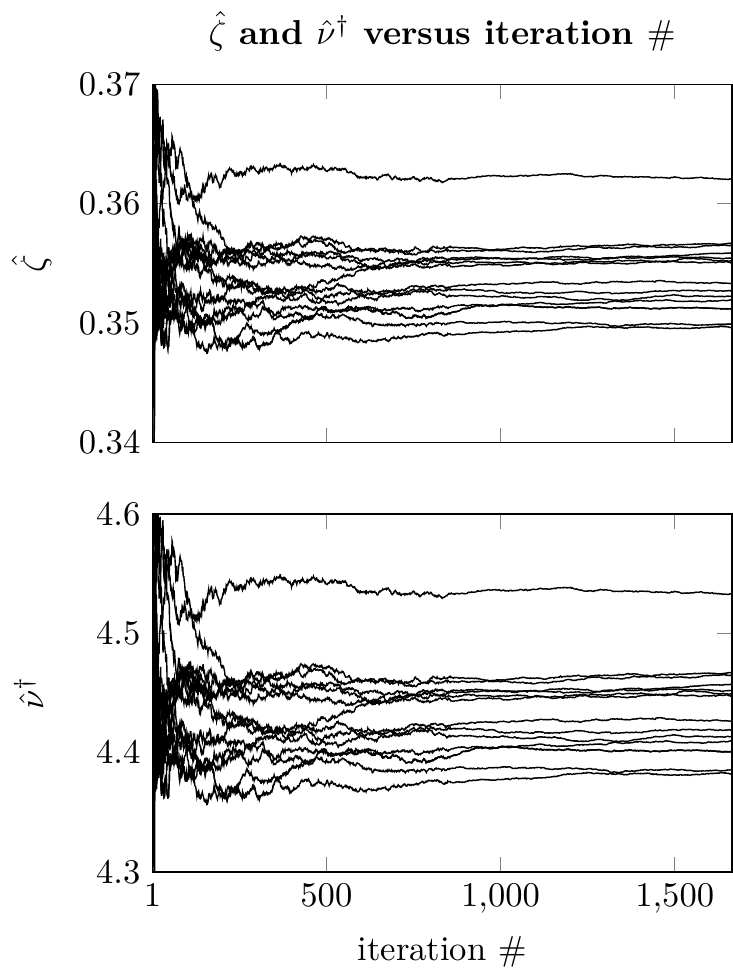}
\caption{Estimates of $\zeta$ (top) and $\nudag$ (bottom) versus iteration number for fifteen randomly selected columns.}
\label{fig:nu_z_vs_iteration}
\end{figure}

To compute the pixel-level conversion gain array, a.k.a.~the \emph{$g$-map}, we first created a pixel-level $\nudag$ array via
\[
V^\dagger=\mathbf 1_{512\times 512}\times\operatorname{diag}(V),
\]
where $\mathbf 1_{512\times 512}$ is a $512\times 512$ array of ones, $V$ is the $512\times 1$ vector of group estimates for $\nudag$, and $\operatorname{diag}(V)$ is a $512\times 512$ diagonal matrix with diagonal elements equal to elements of $V$. The master frames $\bar Y$, $\bar X$, $\hat Y$, $\hat X$, along with $V^\dagger$, and the final sample size values $n_1$ and $n_2$ were then imported into a {\sc Mathematica}. For each set of pixel coordinates $(i,j)$, $1\leq i,j\leq 512$, the conversion gain was estimated with
\[
(\mathscr G_{\nudag,K})_{ij}=(\bar X_{ij}-\bar Y_{ij})\times\mathscr T_{V_{ij}^\dagger,K}(\hat X_{ij},\hat Y_{ij},\alpha_1,\alpha_2),
\]
where $\mathscr T_{\nudag,K}$ is the $K$th order asymptotic approximation of $\mathscr T_\nudag$ given in (\ref{eq:Tv_asym_hyper_form}) and $\alpha_i=(n_i-1)/2$, $i=1,2$. Comparing the histograms and sample statistics of $\mathscr G_{\nudag,1}$ and $\mathscr G_{\nudag,2}$ showed negligible difference indicating that a $K=2$ order approximation was sufficient to accurately compute $\mathscr G_\nudag$ for each pixel.

Figure \ref{fig:G_map} displays the final $g$-map for the KAI-0407M {\sc ccd} along with its histogram. First comparing the histogram to the provided normal reference we see it exhibits a positive skewness, which was estimated to be
\[
\widehat{\skewness}\,\mathscr G_\nudag=0.2426.
\]
As for the $g$-map itself, we observe what appears to be purely random noise with no noteworthy features or patterns. This behavior is expected since we deliberately inhibited gain nonuniformity by transferring image data from the sensor pixels through a single readout register. Since the fluctuations in the $g$-map for this particular case should be almost entirely due to statistical noise, we can see how effective the data collection algorithm was by comparing the sample absolute coefficient of variation to the target value of $\sacv_0=0.05$. Computing the sample mean $\hat\ev\mathscr G_\nudag$ and sample variance $\widehat{\var}\mathscr G_\nudag$ of the $g$-map data we found for the sample absolute coefficient of variation
\[
\widehat{\acv}\mathscr G_\nudag=\frac{\sqrt{\widehat{\var}\mathscr G_\nudag}}{\hat\ev\mathscr G_\nudag}=0.0497,
\]
which differs from the target value by only $0.7\%$. This small discrepancy indicates the data collection algorithm was able to adequately control the experiment and halt data capture at the appropriate time.
\begin{figure}[htb]
\centering
\includegraphics[scale=1]{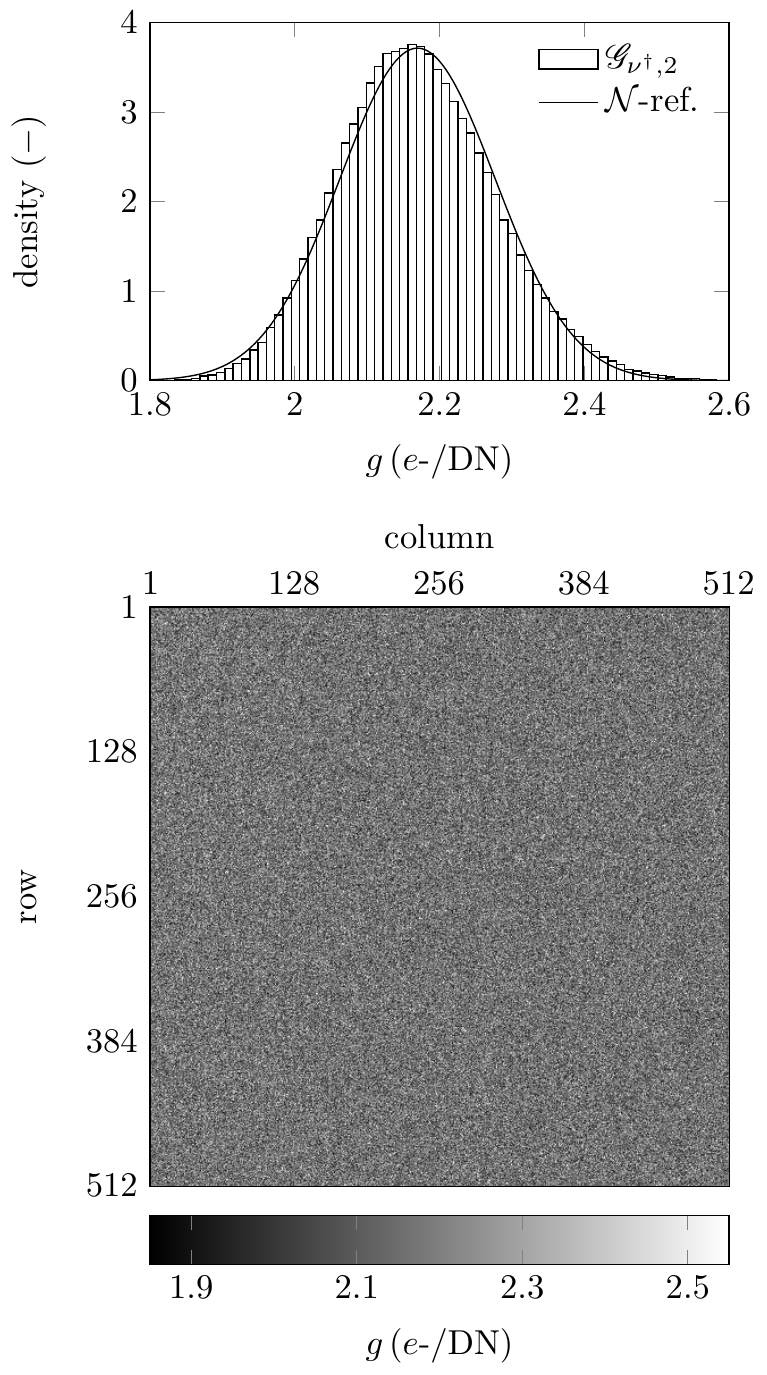}
\caption{Histogram of $g$-map data with best normal approximation (top) and image plot of the final $g$-map (bottom).}
\label{fig:G_map}
\end{figure}

To see how $\mathscr G_\nudag$ compares to the traditional conversion gain estimator we also computed the traditional $g$-map with elements
\[
G_{ij}=(\bar X_{ij}-\bar Y_{ij})(\hat X_{ij}-\hat Y_{ij})^{-1}.
\]
Table \ref{tbl:G_Gv_comparison} presents tabulated values for select sample statistics of both maps. While none of these sample statistics are able to compare the bias of each estimator we know that $\arb\mathscr G_\nudag\approx 0.01$. Furthermore, a quick comparison of the histogram for $\hat X-\hat Y$ against its normal fit shows excellent agreement so that $\arb G$ can be approximated by \cite[c.f.~Corollary 3.2]{hendrickson_2019}
\begin{equation}
\label{eq:arbG_normal_approx}
\arb G=\left\lvert\frac{2(1-\zeta)}{\sqrt{2(\frac{1}{\alpha_1}+\frac{1}{\alpha_2}\zeta^2)}}\mathcal D\left(\frac{1-\zeta}{\sqrt{2(\frac{1}{\alpha_1}+\frac{1}{\alpha_2}\zeta^2)}}\right)-1\right\rvert.
\end{equation}
Substituting the global estimate
\[
Z=\left(\frac{512^2 n_1-3}{512^2 n_1-1}\right)\frac{\sum_{i=1}^{512}\sum_{y=1}^{512} \hat Y_{ij}}{\sum_{i=1}^{512}\sum_{y=1}^{512} \hat X_{ij}}
\]
for $\zeta$ in (\ref{eq:arbG_normal_approx}) gives $\arb G\approx 0.0027$ and so we conclude $G$ incurs less relative bias than $\mathscr G_\nudag$\footnote{Technically speaking $\arb G$ is undefined; however, the expression in (\ref{eq:arbG_normal_approx}) does give a useful measure of relative bias for $G$ whenever $\pr(\hat X-\hat Y>0)\approx 1$. Inspection of the $\hat X$ and $\hat Y$ data for the KAI-0407M {\sc ccd} reveals all $512^2$ values of $\hat X-\hat Y$ are positive; thus, the comparison of $\arb\mathscr G_\nudag$ and $\arb G$ is informative.}. However, note that this small increase in relative bias affords the estimator $\mathscr G_\nudag$ a decrease in sample variance of about $11.4\%$ when compared to that of $G$.
\begin{table}[htb]
\centering
\begin{tabular}{LCCCC}\toprule
	 &\mathscr G_\nudag &G &\text{Unit}\\\midrule
	\hat\ev(\cdot)  &2.1697&2.1975 &\electron/\mathrm{DN}\\
	\widehat{\var}(\cdot)  &0.0116&0.0131&(\electron/\mathrm{DN})^2\\
	\widehat{\acv}(\cdot)  &0.0497&0.0521&-\\\bottomrule
\end{tabular}
\caption{Comparison of sample statistics for the $\mathscr G_\nudag$- and $G$-maps.}
\label{tbl:G_Gv_comparison}
\end{table}


\chapter{Conclusions}
\label{chap:conclusions}

In this work we covered a lot of ground in understanding the estimation of the reciprocal difference of normal variances, $\tau=(\sigma_1^2-\sigma_2^2)^{-1}$, and how this is applicable to the photon transfer method of image sensor characterization.

We began in Theorem \ref{thm:no_finite_variance_estimator} by showing that no unbiased, finite-variance estimator of $\tau$ existed under the normal model. Appealing to the principle of bias-variance tradeoff a biased yet finite-variance estimator, $\mathscr T_n$, was produced, which estimated the first $n$ terms of the Taylor series for $\tau$. Working with the methods of \emph{Summability Calculus} the domain of $\mathscr T_n$ was then extended to include complex-valued $n$ resulting in the generalized estimator $\mathscr T_\nu$. Many properties of this generalized estimator were discovered including a reflection formula as well as its first two moments. The absolute relative bias $\arb\mathscr T_\nu$ and absolute coefficient of variation $\acv\mathscr T_\nu$ were also derived along with their exact confidence intervals. An asymptotic expansion of $\mathscr T_\nu$ for large sample sizes was then given, which played a critical role in applications presented in latter sections.

Equipped with a substantial theoretical foundation, Chapter \ref{chap:new_g_estimator} tackled the problem of applying the results of the preceding analysis to construct a novel estimator, $\mathscr G_\nu$, of the photon transfer conversion gain measurement. As a corollary to Theorem \ref{thm:no_finite_variance_estimator} it was proven that no unbiased, finite-variance estimator of the conversion gain, $g$, existed under the normal model of pixel noise. Mirroring the analysis of $\mathscr T_\nu$, the first two moments of $\mathscr G_\nu$ were derived and used to construct expressions for $\arb\mathscr G_\nu$ and $\acv\mathscr G_\nu$. Then, using long standing observations from the literature as a clue, the function $\mathscr E$ was introduced as a sort of normalized metric for determining when the dispersion of $\mathscr G_\nu$ was dominated by the dispersion of $\mathscr T_\nu$. It was shown that $\mathscr E$ approaches unity in the shot noise limit $\zeta\searrow 1$, which supported the observations in the literature and showed that $\acv\mathscr G_\nu$ could be very well approximated by $\acv\mathscr T_\nu$ given the level of illumination was sufficiently large. These observations were subsequently utilized in a Monte Carlo simulation to demonstrate the process of conversion gain estimation with $\mathscr G_\nu$ as well as in the computation confidence intervals for $\arb\mathscr G_\nu$ and $\acv\mathscr G_\nu$. A short discussion followed the Monte Carlo experiment, which introduced the notion of manipulating the parameter $\nu$ and sample sizes as to achieved desired values for absolute relative bias and absolute coefficient of variation.

The questions following the Monte Carlo experiment in Section \ref{sec:g_estimation_demo} served as a springboard into the study of bias control and optimal sample sizes. To control bias, two new estimators $\mathscr T_\nudag$ and $\mathscr G_\nudag$ were defined, which varied $\nudag$ in such a manner as to force estimator bias to follow a prescribed profile. A definition of optimal sample sizes for was then introduced followed by a study of the optimal sample sizes for $\mathscr T_\nudag$ and $\mathscr G_\nudag$ in the low illumination limit $\zeta\nearrow 1$. It was demonstrated that as the illumination level decreased, the optimal sample sizes, $(n_1^\text{opt},n_2^\text{opt})$, for both estimators were asymptotically equal and proportional to $(1-\zeta)^{-2}$. Perhaps the most significant finding of this investigation at low illumination came in Corollary \ref{cor:E_with_opt_samp_sizes}, which showed that substituting the optimal sample sizes for $\mathscr T_\nudag$ into the metric $\mathscr E$ produced nonzero values in the limit $\zeta\nearrow 1$. Furthermore, for sufficiently large magnitudes of dark noise, $\sigma_{\mathrm d}g$, this limiting value of $\mathscr E$ could be near unity meaning that the optimal sample sizes of $\mathscr G_\nudag$ can be very closely approximated to those of $\mathscr T_\nudag$ even if the illumination level was near zero.

The fact that the optimal sample sizes of $\mathscr T_\nudag$ served as such good approximations to those of $\mathscr G_\nudag$ provided the justification needed to perform a detailed investigation of their properties. The investigation began by showing exact solutions for the optimal sample sizes of $\mathscr T_\nudag$ could be derived at at two special point of interest in the $\zeta$-domain. For all remaining values of $\zeta$ explicit approximations for the optimal sample sizes were found and these approximations were shown to perform very well for a wide range of parameters. To compute the optimal sample sizes, a numerical routine was implemented, which using the explicit approximations as a starting point. A brief analysis of this numerical routine was conducted as to highlight potential weaknesses that could be improved in future work.

With a means for computing optimal sample sizes, Section \ref{sec:CG_estimation_KAI0407M} concluded this work with an introduction to \emph{pixel-level} photon transfer conversion gain estimation. The concept of sensor nonuniformity was introduced as a motivation for a pixel-level approach to photon transfer characterization and the challenges of such an approach were discussed. Using a real image sensor a first attempt at pixel-level conversion gain estimation was presented. In particular, a focus was given on how to design a pixel-level conversion gain estimation experiment using the optimal sample size curves and the quantity $\mathscr E$ as tools for selecting an appropriate illumination level for collecting data. An algorithm for collecting and halting data capture was discussed and then executed on the chosen image sensor. A summary of the experimental result then ensued, which presented the pixel-level $g$-map as the primary data product of the experiment. Details of how the $g$-map was computed as well as analysis of how the data collection algorithm performed were conducted; revealing the algorithms were successful in controlling the experiment and halting data capture at the appropriate time.

The successfulness of the theoretical results presented herein and their application to pixel-level conversion gain estimation subsequently open the door to developing a much needed comprehensive approach to pixel-level photon transfer characterization.


\chapter{Appendices}
\label{chap:appendicies}
\renewcommand{\thesection}{\Alph{section}}

\input{definitions.tex}
\input{limiting_estimator.tex}
\input{differential_operators.tex}
\input{acv_trunc_error_estimate.tex}
\input{parameter_derivatives_of_2F1_function.tex}

\bibliographystyle{plain}
\bibliography{mybibfile}

\end{document}

%% file: definitions.tex

\section{Definitions and relations}
\label{sec:definitions}

\begin{definition}[Sign Function]
\label{def:sign_function}
\[
\operatorname{sign}(x)\coloneqq%
\begin{cases}
x/|x|, &x\neq 0\\
0, &x=0.
\end{cases}
\]
\end{definition}


\begin{definition}[Indicator Function]
\label{def:sign_function}
\[
\mathds 1_A\coloneqq%
\begin{cases}
1, &A\text{ is true}\\
0, &A\text{ is false}.
\end{cases}
\]
\end{definition}


\begin{definition}[Gamma Function]
\label{def:poch_sym}
\[
\Gamma(z) \coloneqq \int_0^\infty t^{z-1}e^{-t}\,\mathrm dt,\quad \Re z>0
\]
\end{definition}


\begin{definition}[Beta Function]
\label{def:beta_fun}
\[
\operatorname B(s,z) \coloneqq \frac{\Gamma(s)\Gamma(z)}{\Gamma(s+z)}
\]
\end{definition}


\begin{definition}[Pochhammer Symbol (rising factorial)]
\label{def:pochhammer_symbol}
\[
(s)_z \coloneqq \frac{\Gamma(s+z)}{\Gamma(s)}
\]
\end{definition}


\begin{definition}[Factorial Power (falling factorial)]
\label{def:FactorialPower}
\[
(s)^{(z)}\coloneqq \frac{\Gamma(s+1)}{\Gamma(s-z+1)}
\]
\end{definition}


\begin{relation}[Rising and falling factorical connection formula]
\label{rel:rising_falling_factorial}
\[
(s)^{(n)}=(-1)^n(-s)_n,\quad n\in\Bbb Z
\]
\end{relation}


\begin{relation}[Falling factorical product representation]
\label{rel:rising_factorial_prodRep}
\[
(s)_n=%
\begin{cases}
1, &n=0\\
\prod_{k=0}^{n-1}(s+k) &n\in\Bbb N
\end{cases}
\]
\end{relation}


\begin{relation}[Falling factorical product representation]
\label{rel:falling_factorial_prodRep}
\[
(s)^{(n)}=%
\begin{cases}
1, &n=0\\
\prod_{k=0}^{n-1}(s-k) &n\in\Bbb N
\end{cases}
\]
\end{relation}


\begin{definition}[Generating function of Stirling numbers of the $1$st-kind $\mathcal S_n^{(k)}$]
\label{def:StirlingS1}
\[
(s)_n\coloneqq\sum_{k=0}^n (-1)^{n-k}\mathcal S_n^{(k)}s^k
\]
\end{definition}


\begin{definition}[Generating function of generalized N{\o}rlund polynomial $B_k^{(\ell)}(z)$]
\label{def:GenNorlundB}
\[
\left(\frac{t}{e^t-1}\right)^\ell e^{zt}\coloneqq\sum_{k=0}^\infty B_k^{(\ell)}(z)\frac{t^k}{k!}
\]
\end{definition}


\begin{definition}[N{\o}rlund polynomial]
\label{def:NorlundB}
\[
B_k^{(\ell)}\coloneqq B_k^{(\ell)}(0)
\]
\end{definition}


\begin{definition}[Bernoulli polynomial]
\label{def:BernoulliPoly}
\[
B_k(z)\coloneqq B_k^{(1)}(z)
\]
\end{definition}

%
%

\begin{definition}[Lerch's Transcendent]
\label{def:LerchPhi}
\[
\Phi(z,s,\omega)\coloneqq\sum_{k=0}^\infty\frac{z^k}{(k+\omega)^s}
\]
\end{definition}






\begin{definition}[Generalized Hypergeometric Series]
\label{def:pFq_function}
\[
\pFq{p}{q}{a_1,\dots,a_p}{b_1,\dots,b_q}{z}\coloneqq\sum_{k=0}^\infty\frac{(a_1)_k\cdots (a_p)_k}{(b_1)_k\cdots (b_q)_k}\frac{z^k}{k!}
\]
\end{definition}


\begin{definition}[Regularized Generalized Hypergeometric Function]
\label{def:pFqReg_function}
\[
\pFqReg{p}{q}{a_1,\dots,a_p}{b_1,\dots,b_q}{z}\coloneqq \frac{1}{\prod_{k=1}^q\Gamma(b_k)}\pFq{p}{q}{a_1,\dots,a_p}{b_1,\dots,b_q}{z}
\]
\end{definition}




\begin{definition}[Appell $F_1$ Hypergeometric Series]
\label{def:appell_F1}
\[
F_1(a;b,b^\prime;c;s,z)\coloneqq\sum_{k,\ell=0}^\infty\frac{(a)_{k+\ell}(b)_k(b^\prime)_\ell}{(c)_{k+\ell}\, k!\,\ell!}s^k z^\ell,\quad\max\{|s|,|z|\}<1
\]
\end{definition}


\begin{definition}[Appell $F_2$ Hypergeometric Series]
\label{def:appell_F2}
\[
F_2(a;b,b^\prime;c,c^\prime;s,z)\coloneqq\sum_{k,\ell=0}^\infty\frac{(a)_{k+\ell}(b)_k(b^\prime)_\ell}{(c)_k(c^\prime)_\ell\, k!\,\ell!}s^k z^\ell,\quad |s|+|z|<1
\]
\end{definition}


\begin{definition}[Incomplete Beta Function]
\label{def:betaInc_fun}
\[
\operatorname B_z(\alpha,\beta) \coloneqq \Gamma(\alpha)z^\alpha\pFqReg{}{}{\alpha,1-\beta}{\alpha+1}{z},\quad -\alpha\notin\Bbb N_0
\]
\end{definition}


\begin{definition}[Regularized Incomplete Beta Function]
\label{def:betaIncReg_fun}
\[
\operatorname I_z(\alpha,\beta) \coloneqq \frac{\operatorname B_z(\alpha,\beta)}{\operatorname B(\alpha,\beta)}
\]
\end{definition}


\begin{relation}{\normalfont \cite[Eq.~$07.23.03.0122.01$]{wolfram_functions}}.
\label{rel:2F1_BetaInc}
\[
F(1,\beta;\gamma;s)=(\gamma-1)s^{1-\gamma}(1-s)^{-(\beta-\gamma+1)}\operatorname B_s(\gamma-1,\beta-\gamma+1).
\]
\end{relation}




\begin{relation}[Gamma reflection formula]
\label{rel:gamma_reflection}
\[
\Gamma(z)\Gamma(1-z)=\pi\csc\pi z,\quad z\notin\Bbb Z
\]
\end{relation}


\begin{relation}
\label{rel:pochhammer_inversion1}
\[
(1-z-n)_n=(-1)^n(z)_n,\quad n\in\Bbb Z
\]
\end{relation}

\begin{proof}
Using the gamma reflection formula in Relation \ref{rel:gamma_reflection} one writes
\[
(1-z-n)_n=\frac{\Gamma(1-z)}{\Gamma(1-z-n)}=\frac{\csc\pi z}{\csc\pi (z+n)}\frac{\Gamma(z+n)}{\Gamma(z)}=(-1)^n(z)_n
\]
\end{proof}


\begin{relation}
\label{rel:pochhammer_inversion2}
\[
\frac{1}{(1-z)_{-n}}=(-1)^n(z)_n,\quad n\in\Bbb Z
\]
\end{relation}

\begin{proof}
Using Relation \ref{rel:pochhammer_inversion1} one writes
\[
\frac{1}{(1-z)_{-n}}=\frac{\Gamma(1-z)}{\Gamma(1-z-n)}=\frac{\Gamma(1-z-n+n)}{\Gamma(1-z-n)}=(1-z-n)_n=(-1)^n(z)_n
\]
\end{proof}


\begin{relation}
\label{rel:pochhammer_product}
\[
(z)_n=(z)_m(z+m)_{n-m}
\]
\end{relation}

\begin{proof}
\[
(z)_n=\frac{\Gamma(z+n)}{\Gamma(z)}=\frac{\Gamma(z+m)}{\Gamma(z+m)}\frac{\Gamma(z+m+n-m)}{\Gamma(z)}=(z)_m(z+m)_{n-m}
\]
\end{proof}


\begin{relation}
\label{rel:binomial_pochhammer}
\[
\binom{n}{k}=\frac{(-1)^k(-n)_k}{k!}
\]
\end{relation}








\begin{relation}{\normalfont \cite[Eq.~$07.23.03.0002.01$]{wolfram_functions}}.
\label{rel:2F1_zeq1}
\[
\pFq{}{}{a,b}{c}{1}=\frac{\Gamma(c)\Gamma(c-a-b)}{\Gamma(c-a)\Gamma(c-b)},\quad\Re\{c-a-b\}>0
\]
\end{relation}




\begin{relation}
\label{rel:pochhammer_2k_inversion}
For $2k=0,2,4,\dots$
\[
(1-a-2k)_{2k}=2^{2k}\left(\tfrac{a}{2}\right)_k\left(\tfrac{a+1}{2}\right)_k.
\]
\end{relation}

\begin{proof}
Using the gamma reflection formula in Relation \ref{rel:gamma_reflection} one writes
\[
(1-a-2k)_{2k}=\frac{\Gamma(1-a)}{\Gamma(1-a-2k)}=\frac{\csc\pi a}{\csc\pi (a+2k)}\frac{\Gamma(a+2k)}{\Gamma(a)}=\frac{\Gamma(a+2k)}{\Gamma(a)}.
\]
Then applying the gamma duplication formula in Relation \ref{rel:gamma_duplication} yields
\[
\frac{\Gamma(a+2k)}{\Gamma(a)}=\frac{\Gamma(\frac{a}{2}+k)\Gamma(\frac{a}{2}+k+\frac{1}{2})2^{2k}}{\Gamma(\frac{a}{2})\Gamma(\frac{a}{2}+\frac{1}{2})}=2^{2k}\left(\tfrac{a}{2}\right)_k\left(\tfrac{a+1}{2}\right)_k.
\]
\end{proof}


\begin{relation}
\label{rel:pochhammer_product}
\[
(a)_{n+k}=(a)_n(a+n)_k.
\]
\end{relation}

\begin{proof}
\[
(a)_{n+k}=\frac{\Gamma(a+n+k)}{\Gamma(a)}=\frac{\Gamma(a+n)}{\Gamma(a)}\frac{\Gamma(a+n+k)}{\Gamma(a+n)}=(a)_n(a+n)_k.
\]
\end{proof}




%% file: limiting_estimator.tex
\section{Limiting properties of $\mathscr T_\nu$ as $|\nu|\to\infty$}
\label{sec:limiting estimator_nu_to_inf}

One curiosity that remains is what happens to $\mathscr T_\nu$ and its moments as $|\nu|\to\infty$.  The following theorem presents these results.

\begin{theorem}
\label{thm:T_est_nlim}
Let $Y_1\sim\mathcal G(\alpha_1,\beta_1)$ and $Y_2\sim\mathcal G(\alpha_2,\beta_2)$ be gamma random variables parameterized in terms of a known shape $\alpha_i$ and unknown rate of the form $\beta_i=\alpha_i/\kappa_i$. Then,
\[
\begin{aligned}
\mathscr U &=\hphantom{-}\frac{1}{\kappa_1}\pFq{1}{1}{1}{\alpha_2}{\frac{\alpha_2 Y_2}{\kappa_1}}, &\text{with}\ \kappa_1>\kappa_2\ \text{and}\ \kappa_1\ \text{known}\\
\mathscr V &=-\frac{1}{\kappa_2}\pFq{1}{1}{1}{\alpha_1}{\frac{\alpha_1 Y_1}{\kappa_2}}, &\text{with}\ \kappa_1<\kappa_2\ \text{and}\ \kappa_2\ \text{known}
\end{aligned}
\]
are unbiased estimators of $\tau=(\kappa_1-\kappa_2)^{-1}$.
\end{theorem}

\begin{proof}
Without loss of generality we consider the discrete estimator of Lemma \ref{lem:T_estimator_discrete} expressed by
\[
\mathscr T_n=\frac{1}{Y_1}\sum_{k=0}^{n-1}\frac{\alpha_1^{-k-1}}{(\alpha_1)_{-k-1}(\alpha_2)_k}\left(\frac{\alpha_2Y_2}{Y_1}\right)^k,
\]
where we recall that $Y_1\sim\mathcal G(\alpha_1,\alpha_1/\kappa_1)$ and $Y_2\sim\mathcal G(\alpha_2,\alpha_2/\kappa_2)$ are independent. From Theorem \ref{thm:T_estimator} we know that $\ev\mathscr T_n<\infty$ if and only if $n<\alpha_1$; thus; in order for $\ev(\lim_{n\to\infty}\mathscr T_n)<\infty$ we require $\alpha_1\to\infty$. To evalute the limit in $\alpha_1$ note that if $\alpha_1\in\Bbb N$ then in terms of distribution
\[
Y_1=\frac{1}{\alpha_1}\sum_{k=1}^{\alpha_1}Y_k^\prime,\quad Y_k^\prime\sim\mathcal G(1,1/\kappa_1)
\]
such that by the strong law of large numbers $Y_1\overset{\mathrm{a.s.}}{\to}\kappa_1$ as $\alpha_1\to\infty$. Furthermore, by \cite[Eq.~$5.11.13$]{nist_2010} as $\alpha_1\to\infty$ we have $\alpha_1^{-k-1}/(\alpha_1)_{-k-1}\sim 1+\mathcal O\{1/\alpha_1\}$; thus, passing to the limit $\alpha_1\to\infty$:
\[
\mathscr T_n\overset{\mathrm{a.s.}}{\to}\mathscr U_n\ \text{where}\ \mathscr U_n=\frac{1}{\kappa_1}\sum_{k=0}^{n-1}\frac{1}{(\alpha_2)_k}\left(\frac{\alpha_2Y_2}{\kappa_1}\right)^k.
\]
As a consequence of taking $\alpha_1\to\infty$ there is no longer any restriction on how large $n$ can be.  Taking the limit $n\to\infty$ in the previous result subsequently yields
\[
\lim_{n\to\infty}\mathscr U_n=\frac{1}{\kappa_1}\sum_{k=0}^\infty\frac{(1)_k}{(\alpha_2)_k\,k!}\left(\frac{\alpha_2Y_2}{\kappa_1}\right)^k,
\]
which is the estimator $\mathscr U$. In addition to $\kappa_1$ being known if we assume $\kappa_1>\kappa_2$ then one can easily confirm
\[
\ev\mathscr U=\lim_{n\to\infty}\ev\mathscr T_n=\frac{1}{\kappa_1-\kappa_2},
\]
which completes the proof for the estimator $\mathscr U$. To obtain the corresponding proof for the estimator $\mathscr V$ one can use the reflection formula in Theorem \ref{thm:reflection_formula} to write
\[
\lim_{n\to-\infty}\mathscr T_n(Y_1,Y_2,\alpha_1,\alpha_2)=-\lim_{n\to\infty}\mathscr T_n(Y_2,Y_1,\alpha_2,\alpha_1),
\]
where we require $\alpha_2\to\infty$, i.e.~$\kappa_2$ known, and $\kappa_1<\kappa_2$ in order for the limiting expected value to converge to the desired quantity.
\end{proof}

\begin{remark}
The estimator $\mathscr U$ is the solution to the integral equation $(\ref{eq:double_int_eq})$ when $\kappa_1$ is known, that is, it satisfies
\[
\mathcal L\{y_2^{\alpha_2-1}\mathscr U(y_2)\}(\beta_2)=\frac{\Gamma(\alpha_2)\beta_2^{-\alpha_2}}{\kappa_1-\alpha_2/\beta_2}.
\]
Likewise,
\[
\mathcal L\{y_1^{\alpha_1-1}\mathscr V(y_1)\}(\beta_1)=\frac{\Gamma(\alpha_1)\beta_1^{-\alpha_1}}{\alpha_1/\beta_1-\kappa_2}.
\]
These results agree with \cite[Eq.~$5.4.9$]{california1954tables}.
\end{remark}

\begin{lemma}
\label{lem:limiting variance}
\[
\var\mathscr U =%
\frac{1}{(\kappa_1-\kappa_2)^2}\left(%
\pFq{}{}{1,1}{\alpha_2}{\frac{\kappa_2^2}{(\kappa_1-\kappa_2)^2}}-1
\right)%
\]
if $\kappa_1>2\kappa_2$ and infinite otherwise. Likewise,
\[
\var\mathscr V =%
\frac{1}{(\kappa_1-\kappa_2)^2}\left(%
\pFq{}{}{1,1}{\alpha_1}{\frac{\kappa_1^2}{(\kappa_1-\kappa_2)^2}}-1%
\right)
\]
if $\kappa_2>2\kappa_1$ and infinite otherwise.
\end{lemma}

\begin{proof}
We will present the proof for $\var\mathscr U$ with the proof for $\var\mathscr V$ being essentially the same. We begin by evaluating $\lim_{n,\alpha_1\to\infty}\ev\mathscr T_n^2$ from the expression in Lemma \ref{lem:T_est_variance_discrete}. According to \cite[Eq.~$15.12.2$]{nist_2010} we have the asymptotic relation
\[
\pFq{}{}{k+1,\ell+1}{\alpha_1}{1}\sim 1+\mathcal O\{1/\alpha_1\},\quad \alpha_1\to\infty.
\]
Thus, upon taking the appropriate limits and writing ${_2}F_1(-k,-\ell;\alpha_1;1)$ in terms of Pochhammer symbols one has
\begin{equation}
\label{eq:EUsqrd_double_series}
\ev\mathscr U^2=\frac{1}{\kappa_1^2}\sum_{k,\ell=0}^\infty\frac{(\alpha_2)_{k+\ell}(1)_k(1)_\ell}{(\alpha_2)_k(\alpha_2)_\ell\,k!\,\ell!}\,\zeta^k\,\zeta^\ell.
\end{equation}
where we again use the shorthand $\zeta=\kappa_2/\kappa_1$. According to Definition \ref{def:appell_F2}, this double series can be expressed in terms of Appell's second hypergeometric function and is absolutely convergent if $\zeta<1/2\implies \kappa_1>2\kappa_2$. To simplify this result we apply the reduction formula in \cite[Eq.~$16.16.3$]{nist_2010} to find
\[
\begin{aligned}
\ev\mathscr U^2%
&=\frac{1}{\kappa_1^2}F_2(\alpha_2;1,1;\alpha_2,\alpha_2;\zeta,\zeta)\\
&=\frac{1}{\kappa_1^2}(1-\zeta)^{-1}F_1(1;\alpha_2-1,1;\alpha_2;\zeta,\zeta(1-\zeta)^{-1}),
\end{aligned}
\]
where $F_1(\cdot)$ is Appell's first hypergeometric function as defined in Definition \ref{def:appell_F1}. Taking advantage of the symmetry $F_1(a ;b,b^\prime ;c ;s,z)=F_1(a;b^\prime,b;c;z,s)$ and applying the reduction formula \cite[Eq.~$15.12.2$]{nist_2010} then yields
\[
\ev\mathscr U^2=\frac{1}{\kappa_1^2}(1-\zeta)^{-2}\pFq{}{}{1,1}{\alpha_2}{\frac{\zeta^2}{(1-\zeta)^2}}.
\]
Subtracting $(\ev\mathscr U)^2$ and simplifying yields the desired result.
\end{proof}

\begin{remark}
The divergence of $\var\mathscr U$ for $\kappa_1<2\kappa_2$ is not just a consequence of the double series representation $(\ref{eq:EUsqrd_double_series})$. Indeed, using the integral definition of $\ev\mathscr U^2$ and $\zeta=\kappa_2/\kappa_1$ one may deduce with a simple substitution
\[
\ev\mathscr U^2\propto\int_0^\infty [{_1F_1}(1;\alpha_2;t)]^2t^{\alpha_2-1}e^{-t/\zeta}\,\mathrm dt.
\]
But according to \cite[Eq.~$13.7.1$]{nist_2010}, as $t\to\infty$
\[
[{_1F_1}(1;\alpha_2;t)]^2t^{\alpha_2-1}e^{-t/\zeta}\sim \Gamma^2(\alpha_2)t^{1-\alpha_2}e^{(2-1/\zeta)t},
\]
which diverges for $\zeta^{-1}<2$, i.e.~$\kappa_1<2\kappa_2$. At the boundary $\kappa_1=2\kappa_2$ we use Relation \ref{rel:2F1_zeq1} to further note that
\[
\ev\mathscr U^2=
\begin{cases}
4\kappa_1^{-2}(\alpha_2-1)/(\alpha_2-2), &\alpha_2>2\\
\infty, &\alpha_2\leq 2.
\end{cases}
\]
\end{remark}

\begin{proof}[Proof of Theorem \ref{thm:no_finite_variance_estimator}]
Assume $\mathscr T$ exists and let $\zeta=\kappa_2/\kappa_1$ with $\zeta<1$. Then by the law of total variance we have
\[
\mathsf{Var}\mathscr T(Y_1,Y_2)=\mathsf E(\mathsf{Var}(\mathscr T(Y_1,Y_2)|Y_2))+\mathsf{Var}g(\kappa_1,Y_2),
\]
where $g(\kappa_1,Y_2)=\mathsf E(\mathscr T(Y_1,Y_2)|Y_2)$ and by assumption $\ev g(\kappa_1,Y_2)=(\kappa_1-\kappa_2)^{-1}$. Now, from Lemma \ref{lem:Y1_Y2_complete_statistic} we know $Y_2$ is a complete-sufficient statistic for $\kappa_2$ such the Lehmann-Scheff\' e theorem asserts $g(\kappa_1,Y_2)$ is the unique {\sc umvue} of its expected value when $\kappa_1$ is known. But if $g(\kappa_1,Y_2)$ is unique and $\zeta<1$ then Theorem \ref{thm:T_est_nlim} proves $g(\kappa_1,Y_2)=\mathscr U$. Given $\mathsf E(\mathsf{Var}(\mathscr T(Y_1,Y_2)|Y_2))\geq 0$ it follows that $\var\mathscr U\leq\var\mathscr T$. Furthermore, Lemma \ref{lem:limiting variance} tells us that $\var\mathscr U=\infty$ if $\zeta>1/2$; thus, it must be that $\var\mathscr T=\infty$ on $1/2<\zeta<1$. By a similar argument if $\zeta>1$ then $\var\mathscr V\leq\mathscr T$ with $\mathscr V$ also given in Theorem \ref{thm:T_est_nlim}. Since $\var\mathscr V=\infty$ for $\zeta<2$ we can combine the previous result to conclude $\var\mathscr T=\infty$ on $1/2<\zeta<2$ which completes the proof.
\end{proof}

%% file: differential_operators.tex

\section{Proofs of differential operator identities}
\label{sec:differential_op_props}

Here present the proofs associated with the differential operator identities in Lemma \ref{lem:diff_op_identities}.

\begin{proof}[Proof of Lemma \ref{lem:diff_op_identities} $(i)$]
Begin by using the Definition \ref{def:lowering_operator} to write
\[
\Lambda_\omega z^s=\omega z^s+z\partial_z z^s.
\]
Then by the chain rule we find
\[
\Lambda_\omega z^s=\omega z^s+z(sz^{s-1}+z^s\partial_z)=z^s(\omega+s+z\partial_z)=z^s\Lambda_{\omega+s},
\]
which completes the proof.
\end{proof}

\begin{proof}[Proof of Lemma \ref{lem:diff_op_identities} $(ii)$]
See \cite{fleury_1994}.
\end{proof}

\begin{proof}[Proof of Lemma \ref{lem:diff_op_identities} $(iii)$]
Let $P(n):(z\vartheta)^n=z(\vartheta z)^nz^{-1}$. It is trivial to show that $P(0)$ holds; thus, assuming $P(n)$ we have for $P(n+1)$
\[
(z\vartheta)^{n+1}=(z\vartheta) z(\vartheta z)^nz^{-1}=z(\vartheta z)(\vartheta z)^nz^{-1}=z(\vartheta z)^{n+1}z^{-1},
\]
thus, $P(n)\implies P(n+1)$. Substituting the result in $(ii)$ for $(\vartheta z)^n$ subsequently produces the desired result.
\end{proof}

\begin{proof}[Proof of Lemma \ref{lem:diff_op_identities} $(iv)$]
Let $P(n):\Lambda_\omega^n=z^{-\omega}\vartheta^nz^\omega$ and note that $P(0)$ trivially holds.  Assuming $P(n)$ we have for $P(n+1)$
\[
\Lambda_\omega^{n+1}=(\omega+z\partial_z) \Lambda_\omega^n=(\omega+z\partial_z) z^{-\omega}\vartheta^nz^\omega=z^{-\omega}(z\partial_z)\vartheta^nz^\omega=z^{-\omega}\vartheta^{n+1}z^\omega,
\]
where the second to last equality is due to result $(i)$. Thus, $P(n)\implies P(n+1)$ which completes the proof.
\end{proof}

\begin{proof}[Proof of Lemma \ref{lem:diff_op_identities} $(v)$]
Let $P(n):(\vartheta)^{(n)}=\frac{1}{z}(z\vartheta)^nz^{1-n}$. By Definition \ref{def:factorial_differential_operator}, it immediately follows that $P(0)$ holds.  Assuming $P(n)$ we have for $P(n+1)$
\[
(\vartheta)^{(n+1)}=(\vartheta)^{(n)}(\vartheta-n)=\frac{1}{z}(z\vartheta)^nz^{1-n}\Lambda_{-n}.
\]
Then, making use of $(iv)$ we find
\[
(\vartheta)^{(n+1)}=\frac{1}{z}(z\vartheta)^nz^{1-n}z^n\vartheta z^{-n}=\frac{1}{z}(z\vartheta)^{n+1} z^{1-(n+1)}.
\]
Therefore, $P(n)\implies P(n+1)$. Substituting the result of $(iii)$ for $(z\vartheta)^n$ then completes the proof.
\end{proof}

\begin{proof}[Proof of Lemma \ref{lem:diff_op_identities} $(vi)$]
Let $P(n):(\Lambda_\omega)^{(n)}=z^{-\omega}(\vartheta)^{(n)}z^\omega$. By Definition \ref{def:factorial_differential_operator}, it immediately follows that $P(0)$ holds. Now use Definition \ref{def:factorial_differential_operator} to write
\[
(\Lambda_\omega)^{(n+1)}=(\omega-n+\vartheta)(\Lambda_\omega)^{(n)}.
\]
Assuming $P(n)$ we have for $P(n+1)$
\[
(\Lambda_\omega)^{(n+1)}=(\omega-n+\vartheta)z^{-\omega}(\vartheta)^{(n)}z^\omega=z^{-\omega}(-n+\vartheta)(\vartheta)^{(n)}z^\omega=z^{-\omega}(\vartheta)^{(n+1)}z^\omega,
\]
where the second to last equality is due to result $(i)$. Thus, $P(n)\implies P(n+1)$ and upon substituting the result of $(v)$ for $(\vartheta)^{(n)}$ the desired result is obtained.
\end{proof}

\begin{proof}[Proof of Lemma \ref{lem:diff_op_identities} $(vii)$]
From \cite[Eq.~$26.8.10$]{nist_2010} we have
\[
\vartheta^n=\sum_{k=0}^n{_2\mathcal S}_n^{(k)}(\vartheta)^{(k)}.
\]
Substituting the result of $(v)$ for $(\vartheta)^{(k)}$ produces the desired result.
\end{proof}

%% file: acv_trunc_error_estimate.tex
\section{Estimate of truncation error for the $\acv\mathscr T_\nu$ series expansion}
\label{sec:trunc_error_estimate}

An important ingredient in the computation of confidence intervals in Section \ref{sec:g_estimation_demo} and optimal sample sizes in Section \ref{subsec:comp_of_opt_samp_sizes_for_Tv} is the ability to evaluate $\acv^2\mathscr T_\nu$ via its series expansion. To evaluate the series expansion we write
\[
\acv^2\mathscr T_\nu=\acv_n^2\mathscr T_\nu+E_n,
\]
where
\[
\acv_n^2\mathscr T_\nu=\sum_{k=1}^n\sum_{\ell=0}^k\frac{\tilde g_{k,\ell}^2(z,\nu)}{(\alpha_1)_\ell(\alpha_2)_{k-\ell}\,\ell!\,(k-\ell)!}
\]
and
\[
E_n=\sum_{k=n+1}^\infty\sum_{\ell=0}^k\frac{\tilde g_{k,\ell}^2(z,\nu)}{(\alpha_1)_\ell(\alpha_2)_{k-\ell}\,\ell!\,(k-\ell)!}
\]
and then determine the number $n$ such that the error incurred in the approximation $\acv^2\mathscr T_\nu\approx \acv_n^2\mathscr T_\nu$ falls within some specified tolerance. Of course, we do not know a closed form for $E_n$ and so whatever procedure used to find $n$ will ultimately require finding a useful upper bound.

\begin{definition}[Complementary incomplete Hypergeometric function]
\label{def:comp_inc_2F1}
For $\nu\in\Bbb C$
\[
F(\alpha,\beta;\gamma;z)_\nu^c\coloneqq F(\alpha,\beta;\gamma;z)-F(\alpha,\beta;\gamma;z)_\nu,
\]
where $F(\alpha,\beta;\gamma;z)_\nu$ is the incomplete hypergeometric function.
\end{definition}

\begin{proposition}
\label{prop:comp_inc_2F1_explicit_forms}
For $\nu\in\Bbb C$
\[
F(\alpha,\beta;\gamma;z)_\nu^c=\frac{(\alpha)_\nu(\beta)_\nu}{(\gamma)_\nu}\frac{z^\nu}{\Gamma(1+\nu)}\pFq{3}{2}{1,\alpha+\nu,\beta+\nu}{1+\nu,\gamma+\nu}{z}.
\]
Likewise, if $\nu=n\in\Bbb N_0$
\[
F(\alpha,\beta;\gamma;z)_n^c=F(\alpha,\beta;\gamma;z)-\frac{(\alpha)_{n-1}(\beta)_{n-1}}{(\gamma)_{n-1}}\frac{z^{n-1}}{\Gamma(n)}\pFq{3}{2}{1,1-n,2-n-\gamma}{2-n-\alpha,2-n-\beta}{\frac{1}{z}}.
\]
\end{proposition}

\begin{theorem}
\label{thm:ACVTv_tail_error_bound}
Let $E_n$ denote the truncation error incurred in approximating $\acv^2\mathscr T_\nu$ with $\acv_n^2\mathscr T_\nu$. Then, $E_n\leq E_{n,m}^\ast$, where
\begin{multline*}
E_{n,m}^\ast=\sum_{j=0}^{m-1}(\ev h_{k^\prime+j,k^\prime})^2b_j\pFq{}{}{1+\nu,1+\nu}{\alpha_1}{1}_{k^\prime}^c\\
+(\ev h_{k^{\prime\prime}+m,k^{\prime\prime}\mathds 1_{z\in[0,1]}})^2%
\Biggl(%
\pFq{}{}{1+\nu,1+\nu}{\alpha_1}{1}\pFq{}{}{1-\nu,1-\nu}{\alpha_2}{1}_m^c\\%
-b_m\sum_{k=0}^{k^{\prime\prime}-1}a_k\pFq{5}{4}{-k,1-k-\alpha_1,1,1-\nu+m,1-\nu+m}{1+m,m+\alpha_2,-k-\nu,-k-\nu}{1}%
\Biggr),
\end{multline*}
$k^\prime=\max(0,n+1-j)$, $k^{\prime\prime}=\max(0,n+1-m)$, $a_j=(1+\nu)_j^2/((\alpha_1)_j j!)$, $b_j=(1-\nu)_j^2/((\alpha_2)_j j!)$, $F(\alpha,\beta;\gamma;z)_\nu^c$ is the complementary incomplete hypergeometric function of Definition \ref{def:comp_inc_2F1}, and $\ev h_{n,\omega}$ is given in (\ref{eq:Ehnw_closed_form}). Furthermore, $E_{n,m}^\ast<\infty$ if $\alpha_1>2(1+\nu)$ and $\alpha_2>2(1-\nu)$ and infinite otherwise.
\end{theorem}

\begin{proof}
We begin by writing
\[
E_n=\sum_{k=n+1}^\infty\sum_{\ell=0}^k a_\ell b_{k-\ell}(\ev h_{k,\ell})^2,
\]
where $a_j=(1+\nu)_j^2/((\alpha_1)_j j!)$ and $b_j=(1-\nu)_j^2/((\alpha_2)_j j!)$. The series $E_n$ is absolutely convergent and so we may rearrange its terms as
\[
E_n=\sum_{j=0}^{m-1}b_j\sum_{k=k^\prime}^\infty a_k(\ev h_{k+j,k})^2+\sum_{k=k^{\prime\prime}}^\infty\sum_{\ell=0}^k a_\ell b_{k-\ell+m}(\ev h_{k+m,\ell})^2,
\]
where $k^\prime=\max(0,n+1-j)$ and $k^{\prime\prime}=\max(0,n+1-m)$. Since all terms are nonnegative we may call on Lemma \ref{lem:Ehnw_monotonicity_in_n_and_w} to obtain the upper bound $E_n\leq E_{n,m}^\ast$, where
\[
E_{n,m}^\ast=\sum_{j=0}^{m-1}b_j(\ev h_{k^\prime+j,k^\prime})^2\sum_{k=k^\prime}^\infty a_k+(\ev h_{k^{\prime\prime}+m,k^{\prime\prime}\mathds 1_{z\in[0,1]}})^2\sum_{k=k^{\prime\prime}}^\infty\sum_{\ell=0}^k a_\ell b_{k-\ell+m}.
\]
Upon inspection,
\[
\sum_{k=k^\prime}^\infty a_k=\pFq{}{}{1+\nu,1+\nu}{\alpha_1}{1}_{k^\prime}^c,
\]
which converges when $\alpha_1>2(1+\nu)$. Furthermore,
\[
\sum_{k=k^{\prime\prime}}^\infty\sum_{\ell=0}^k a_\ell b_{k-\ell+m}=\sum_{k=0}^\infty a_k\sum_{k=m}^\infty b_k-\sum_{k=0}^{k^{\prime\prime}-1}\sum_{\ell=0}^k a_\ell b_{k-\ell+m},
\]
where
\[
\sum_{k=0}^\infty a_k=\pFq{}{}{1+\nu,1+\nu}{\alpha_1}{1},
\]
which converges for $\alpha_1>2(1+\nu)$ and
\[
\sum_{k=m}^\infty b_k=\pFq{}{}{1-\nu,1-\nu}{\alpha_2}{1}_m^c,
\]
which converges for $\alpha_2>2(1-\nu)$. Lastly, with a bit of algebra and working with the properties of the Pochhammer symbol we may use \cite[Eq.~$16.2.4$]{nist_2010} to write
\[
\sum_{\ell=0}^ka_\ell b_{k-\ell+m}=a_kb_m\pFq{5}{4}{-k,1-k-\alpha_1,1,1-\nu+m,1-\nu+m}{1+m,m+\alpha_2,-k-\nu,-k-\nu}{1}.
\]
Bringing all results together gives the desired form for $E_{n,m}^\ast$. The proof is now complete.
\end{proof}

\begin{corollary}
\label{cor:rel_truncation_error_bound}
Let, $p>0$ and
\[
R_{n,p}=\left\lvert\frac{\acv_n^p\mathscr T_\nu}{\acv^p\mathscr T_\nu}-1\right\rvert
\]
denote the absolute relative error incurred when approximating $\acv^p\mathscr T_\nu$ with $\acv_n^p\mathscr T_\nu$. Then, $R_{n,p}\leq R_{n,m,p}^\ast$ where
\[
R_{n,m,p}^\ast=\left\lvert\left(1+\frac{E_{n,m}^\ast}{\acv_n^2\mathscr T_\nu}\right)^{-p/2}-1\right\rvert
\]
and $E_{n,m}^\ast$ is given in Theorem \ref{thm:ACVTv_tail_error_bound}. Furthermore, if $\nu>1$ is constant then for all $z\in Z\subset\Bbb R^+$, $R_{n,p}^\ast\leq R_{n,m,p}^\star$ where
\[
R_{n,m,p}^\star=\left\lvert\left(1+\frac{E_{n,m}^\ast|_{z=\sup Z}}{\acv_n^2\mathscr T_\nu|_{z=\inf Z}}\right)^{-p/2}-1\right\rvert\times 100\%.
\]
If instead $\nu<-1$ interchange $\inf Z$ and $\sup Z$.
\end{corollary}

\begin{proof}
The first result follows from writing $\acv^p\mathscr T_\nu= \left(\acv_n^2\mathscr T_\nu+E_n\right)^{p/2}$ and then substituting $E_n\mapsto E_{n,m}^\ast$ to obtain an upper bound on $R_{n,p}$. For the second result we may use Lemma \ref{lem:hyper_ratio_monotonicity} to claim that $\ev h_{k^\prime+j,k^\prime}$, $\ev h_{k^{\prime\prime}+m,k^{\prime\prime}}$, and $\ev h_{k^{\prime\prime}+m,0}$ are all increasing functions of $z$ for some constant $\nu>1$ on $z\in\Bbb R^+$. Furthermore, since $\ev h_{k^{\prime\prime}+m,k^{\prime\prime}}=\ev h_{k^{\prime\prime}+m,0}=(k^{\prime\prime}+m+1)^{-1}$ at $z=1$ it follows that $\ev h_{k^{\prime\prime}+m,k^{\prime\prime}\mathds 1_{z\in[0,1]}}$ must also be increasing in $z$ on $z\in\Bbb R^+$. Combinbing these observations with the fact that $E_{n,m}^\ast$ and $\acv_n^2\mathscr T_\nu$ are sums of nonnegative terms implies that both of these functions must also be increasing in $z$ on $z\in\Bbb R^+$; hence, for all $z\in Z$: $E_n^\ast\leq E_n^\ast|_{z=\sup Z}$ and $\acv_n^{-2}\mathscr T_\nu\leq\acv_n^{-2}\mathscr T_\nu|_{z=\inf Z}$, which gives the upper bound on $R_{n,m,p}^\ast$. Noting the monotonicity of these functions is reversed if $\nu<-1$ gives the complementary bound.
\end{proof}

%% file: parameter_derivatives_of_2F1_function.tex
\section{Bounding functions for $|\partial_\beta^n F(\alpha,\alpha;\beta;x)|$}
\label{sec:parameter_derivatives_of_2F1}

In this section we derive several results that will be used to determine bounding functions for $|\partial_\beta^n F(\alpha,\alpha;\beta;x)|$ on $x\in[0,1]$ when $\alpha\in\Bbb R$ and $\beta\in\Bbb R^+$. The main results are found in Theorems \ref{thm:2F1_c_derivative_order_n_bound} and \ref{thm:hyper2f1_c_derivative_bounding_functions}.

Before we start deriving the bounding functions we must find a suitable expression for higher order derivatives of the hypergeometric function w.r.t.~its bottom parameter. The following Lemma turns out to be key for doing just that.

\begin{lemma}
\label{lem:nth_order_reciprocal_Pochhammer_derivative}
For $a,b>0$, $n\in\Bbb N_0$, and $X\sim\operatorname{Beta}(a,b)$
\[
(a)_b\partial_a^n(a)_b^{-1}=\ev\log^nX.
\]
\end{lemma}

\begin{proof}
Denoting $P(n):(a)_b\partial_a^n(a)_b^{-1}=\ev\log^nX$ we immediately conclude that $P(0)$ holds. Assuming $P(n)$ we find
\[
\begin{aligned}
P(n)\implies\partial_a^{n+1}(a)_b^{-1}%
&=\partial_a(a)_b^{-1}\ev\log^nX\\
&=\partial_a\int_0^1(\log^nt)\,\frac{t^{a-1}(1-t)^{b-1}}{\Gamma(b)}\,\mathrm dt\\
&=\int_0^1(\log^{n+1}t)\,\frac{t^{a-1}(1-t)^{b-1}}{\Gamma(b)}\,\mathrm dt\\
&=(a)_b^{-1}\ev\log^{n+1}X.
\end{aligned}
\]
Thus, $P(n)\implies P(n+1)$ and the proof is complete.
\end{proof}

Using this probabilistic interpretation for derivatives of the Pochhammer symbol now facilitates a simple integral representation for $\partial_\gamma^nF(\alpha,\beta;\gamma;x)$ as seen in the following result.

\begin{lemma}
\label{lem:2F1_c_derivative_order_n_integral_rep}
For $\gamma>0$, $|x|<1$, and $n\in\Bbb N_0$
\[
\partial_\gamma^nF(\alpha,\beta;\gamma;x)=\int_0^1\log^n(1-t)\,(1-t)^{\gamma-1}\partial_t F(\alpha,\beta;1;xt)\,\mathrm dt+\mathds 1_{n=0}
\]
\end{lemma}

\begin{proof}
From Lemma \ref{lem:nth_order_reciprocal_Pochhammer_derivative} we write
\[
\partial_\gamma^nF(\alpha,\beta;\gamma;x)=\sum_{k=1}^\infty\frac{(\alpha)_k(\beta)_k}{(\gamma)_k}\frac{x^k}{k!}\ev\log^nX+\mathds 1_{n=0},\quad X\sim\operatorname{Beta}(\gamma,k).
\]
Now writing the expected value in integral form subsequently gives
\[
\partial_\gamma^nF(\alpha,\beta;\gamma;x)=\lim_{m\to\infty}\int_0^1(\log^nt) t^{\gamma-1}\sum_{k=1}^m\frac{(\alpha)_k(\beta)_k}{\Gamma(k)}\frac{x^k}{k!}(1-t)^{k-1}\,\mathrm dt+\mathds 1_{n=0}.
\]
Since $|x|<1$ the magnitude of the integrand is bounded above by $C(-\log t)^n t^{\gamma-1}$ for some $C>0$, which is integrable for all $n\in\Bbb N_0$ and $\gamma>0$. Consequently, we have by argument of dominated convergence
\[
\partial_\gamma^nF(\alpha,\beta;\gamma;x)=-\int_0^1(\log^nt) t^{\gamma-1}\partial_tF(\alpha,\beta;1;x(1-t))\,\mathrm dt+\mathds 1_{n=0},
\]
which upon substituting $t\mapsto 1-t$ gives the desired form. The proof is now complete.
\end{proof}

Upon inspection of the integral representation in Lemma \ref{lem:nth_order_reciprocal_Pochhammer_derivative}, one bounding function immediately stands out.

\begin{theorem}
\label{thm:2F1_c_derivative_order_n_bound}
Let $\alpha\in\Bbb R$, $\beta>0$, $x\in[0,1]$, and $n\in\Bbb N_0$. Then for any $d>1$
\[
|\partial_\beta^nF(\alpha,\alpha;\beta;x)|<e^{-n}\left(\frac{d n}{\beta}\right)^n(F(\alpha,\alpha;(1-d^{-1})\beta;x)-1)+\mathds 1_{n=0},
\]
where for the case $n=0$ we define $0^0\coloneqq 1$.
\end{theorem}

\begin{proof}
For brevity denote $\beta^\prime=(1-d^{-1})\beta$. Then by Lemma \ref{lem:2F1_c_derivative_order_n_integral_rep} we have
\[
|\partial_\beta^nF(\alpha,\alpha;\beta;x)|=\int_0^1(1-t)^{\beta/d}(-\log(1-t))^n\,(1-t)^{\beta^\prime-1}\partial_t F(\alpha,\alpha;1;xt)\,\mathrm dt+\mathds 1_{n=0}.
\]
Now, $0\leq (1-t)^{\beta/d}(-\log(1-t))^n\leq e^{-n}(d n/\beta)^n$; thus, with the help of \cite[Eq.~$7.512.12$]{gradshteyn_ryzhik_2014} if $d>0$ then $\beta^\prime>0$ and we have
\[
|\partial_\beta^nF(\alpha,\alpha;\beta;x)|<e^{-n}\left(\frac{d n}{\beta}\right)^n\frac{\alpha^2x}{\beta^\prime}\pFq{3}{2}{1,\alpha+1,\alpha+1}{2,\beta^\prime+1}{x}+\mathds 1_{n=0}.
\]
The contigious relation in \cite[Eq.~$07.27.17.0012.01$]{wolfram_functions} then reduces this result to its final form. The proof is now complete.
\end{proof}

Our ultimate goal in this effort is to find bounding function for the purposes of studying the convergence of higher order derivatives of $\acv^2\mathscr T_\nu$ w.r.t.~$\alpha_1$ and $\alpha_2$. However, observe that a key feature of the bounding function given in Theorem \ref{thm:2F1_c_derivative_order_n_bound} is a reduction of the bottom parameter by a multiplicative factor of $(1-d^{-1})$, $d>1$. Using this bounding function for studying higher order derivatives of $\acv^2\mathscr T_\nu$ would effectively reduce $\alpha_1$ and $\alpha_2$ by this same factor and thus would alter the convergence properties of the integral representing the higher order derivatives of $\acv^2\mathscr T_\nu$. As a consequence, we would not determine the full range of possible values for $\alpha_1$ and $\alpha_2$ and must turn to finding a more suitable bounding function. We proceed with two more preliminary results and then onto Theorem \ref{thm:hyper2f1_c_derivative_bounding_functions} which gives us the result were looking for.

\begin{lemma}
\label{lem:indefinite_hypergeom_integral}
\[
\int (1-t)^{s-1}(1-xt)^{-s}\,\mathrm dt=-\frac{1}{s}(1-t)^s(1-xt)^{-s}\pFq{}{}{1,s}{1+s}{x\frac{1-t}{1-xt}}+C.
\]
\end{lemma}

\begin{proof}
Let $I$ denote the integral in question. Substituting $s=(1-t)/(1-xt)$ yields
\[
I=-\int(1-xs)^{-1} s^{c-1}\,\mathrm ds=-\int_0^s(1-xz)^{-1} z^{c-1}\,\mathrm dz.
\]
Another substitution of $z=sw$ subsequently allows us to write the resulting integral in the form of the integral representation for the hypergeometric function in Relation \ref{rel:Hyper2F1_integral}. The proof is now complete.
\end{proof}


\begin{lemma}[{\cite[Lem.~$2.3$]{anderson_1995}}]
\label{lem:balanced_hypergeometric_bound}
For $\alpha,\beta\in\Bbb R^+$ and $x\in[0,1]$
\[
F(\alpha,\beta;\alpha+\beta;x)\leq 1-\frac{1}{\operatorname{B}(\alpha,\beta)}\log(1-x).
\]
\end{lemma}

We now present the 

\begin{theorem}
\label{thm:hyper2f1_c_derivative_bounding_functions}
For $\alpha\in\Bbb R$, $\beta\in\Bbb R^+$, $x\in[0,1]$, and $n\in\Bbb N_0$:
\[
|\partial_\beta^nF(\alpha,\alpha;\beta;x)|\leq \upsilon_n(\alpha,\beta,x),
\]
where
\[
\upsilon_n(\alpha,\beta,x)=%
\begin{cases}
\frac{n!}{\operatorname B(|\alpha|,|\alpha|)}\frac{1}{(\beta-2\alpha\mathds 1_{\alpha>0})^{n+1}}+\mathds 1_{n=0}, &\beta>2\alpha\\[1ex]
\frac{n!}{\operatorname B(\alpha,\alpha)}(\frac{1}{2\alpha}-\log(1-x))^{n+1}+\mathds 1_{n=0}, &\beta=2\alpha\\[1ex]
\frac{n!}{\operatorname B(\alpha,\alpha)}(1-x)^{\beta-2\alpha}(\frac{1}{\beta}-\log(1-x))^{n+1}+\mathds 1_{n=0}, &\beta<2\alpha.
\end{cases}
\]
\end{theorem}

\begin{proof}
We begin with the integral representation of Lemma \ref{lem:2F1_c_derivative_order_n_integral_rep} and write
\[
|\partial_\beta^nF(\alpha,\alpha;\beta;x)|\leq\alpha^2(-1)^n\int_0^1\log^n(1-t)\,(1-t)^{\beta-1} F(\alpha+1,\alpha+1;2;xt)\,\mathrm dt+\mathds 1_{n=0}.
\]
Each case will now be proven from this integral representation.

\begin{enumerate}
\item[$(1)$] $\beta>2\alpha$.

\begin{subproof}
The case $\beta>2\alpha$ must itself be broken down into three subcases: $\alpha<0$, $\alpha=0$, and $\alpha>0$. If $\alpha<0$ then we have from Lemma \ref{lem:hyper2f1_bounding_functions} that
\[
F(\alpha+1,\alpha+1;2;xt)\leq\frac{1}{\alpha^2\operatorname B(-\alpha,-\alpha)};
\]
hence,
\[
\begin{aligned}
|\partial_\beta^nF(\alpha,\alpha;\beta;x)|%
&\leq\frac{(-1)^n}{\operatorname B(-\alpha,-\alpha)}\int_0^1\log^n(1-t)\,(1-t)^{\beta-1}\,\mathrm dt+\mathds 1_{n=0}\\
&=\frac{n!}{\operatorname B(-\alpha,-\alpha)}\frac{1}{\beta^{n+1}}+\mathds 1_{n=0}.
\end{aligned}
\]
For $\alpha=0$ we have the trivial calculation
\[
|\partial_\beta^nF(0,0;\beta;x)|=\mathds 1_{n=0}.
\]
Now if $\alpha>0$ we again call on Lemma \ref{lem:hyper2f1_bounding_functions} to deduce
\[
F(\alpha+1,\alpha+1;2;xt)=(1-xt)^{-2\alpha}F(1-\alpha,1-\alpha;2;xt)\leq\frac{(1-t)^{-2\alpha}}{\alpha^2\operatorname B(\alpha,\alpha)}.
\]
It follows that
\[
\begin{aligned}
|\partial_\beta^nF(\alpha,\alpha;\beta;x)|%
&\leq\frac{(-1)^n}{\operatorname B(\alpha,\alpha)}\int_0^1\log^n(1-t)\,(1-t)^{\beta-2\alpha-1}\,\mathrm dt+\mathds 1_{n=0}\\
&=\frac{n!}{\operatorname B(\alpha,\alpha)}\frac{1}{(\beta-2\alpha)^{n+1}}+\mathds 1_{n=0}.
\end{aligned}
\]
Combining these three results yields the desired form for the case $\beta>2\alpha$.
\end{subproof}


\item[$(2)$] $\beta\leq 2\alpha$.

\begin{subproof}
First observe that $\beta>0\land\beta\leq2\alpha\implies\alpha>0$ and so we have from Lemma \ref{lem:hyper2f1_bounding_functions}
\[
F(\alpha+1,\alpha+1;2;xt)\leq\frac{(1-xt)^{-2\alpha}}{\alpha^2\operatorname B(\alpha,\alpha)}\leq \frac{(1-x)^{\beta-2\alpha}(1-xt)^{-\beta}}{\alpha^2\operatorname B(\alpha,\alpha)}
\]
and
\[
|\partial_\beta^nF(\alpha,\alpha;\beta;x)|%
\leq\frac{(1-x)^{\beta-2\alpha}}{\operatorname B(\alpha,\alpha)}I_n+\mathds 1_{n=0},
\]
where
\[
I_n=\int_0^1(-1)^n\log^n(1-t)\,(1-t)^{\beta-1}(1-xt)^{-\beta}\,\mathrm dt.
\]
Performing integration by parts with $u=(-1)^n\log^n(1-t)$ and $\mathrm dv=(1-t)^{\beta-1}(1-xt)^{-\beta}\,\mathrm dt$ we use Lemma \ref{lem:indefinite_hypergeom_integral} to write
\begin{multline*}
I_n=-\frac{(-1)^n}{\beta}\log^n(1-t)(1-t)^\beta(1-xt)^{-\beta}\pFq{}{}{1,\beta}{1+\beta}{x\frac{1-t}{1-xt}}\bigg|_{t=0}^1\\
+\frac{n}{\beta}\int_0^1(-1)^{n-1}\log^{n-1}(1-t)(1-t)^{\beta-1}(1-xt)^{-\beta}\pFq{}{}{1,\beta}{1+\beta}{x\frac{1-t}{1-xt}}\,\mathrm dt.
\end{multline*}
Evaluating the limit term gives
\begin{equation}
\label{eq:uv_term_cases}
uv|_{t=0}^1=%
\begin{cases}
\frac{1}{\beta}F(1,\beta;1+\beta;x), &n=0\\
0, &n\in\Bbb N
\end{cases}
\end{equation}
with the latter case being due to $\log^n(1-t)(1-t)^\beta\to 0$ in both limits. Furthermore, calling in Lemma \ref{lem:balanced_hypergeometric_bound} we find
\[
\pFq{}{}{1,\beta}{1+\beta}{x\frac{1-t}{1-xt}}\leq 1-\beta\log\left(1-\frac{x(1-t)}{1-xt}\right)\leq 1-\beta\log(1-x)
\]
Hence, $I_n\leq\frac{n}{\beta}(1-\beta\log(1-x))I_{n-1}$ and
\[
I_n\leq\frac{n!}{\beta^n}(1-\beta\log(1-x))^nI_0.
\]
Note  that $I_0$ is given by the $n=0$ case of $(\ref{eq:uv_term_cases})$ and $I_0\leq\frac{1}{\beta}(1-\beta\log(1-x))$. These observations lead us to conclude
\[
|\partial_\beta^nF(\alpha,\alpha;\beta;x)|%
\leq\frac{(1-x)^{\beta-2\alpha}}{\operatorname B(\alpha,\alpha)}\frac{n!}{\beta^{n+1}}(1-\beta\log(1-x))^{n+1}+\mathds 1_{n=0},
\]
which is the final result for the $\beta<2\alpha$ case. Substituting $\beta\mapsto 2\alpha$ into the r.h.s.~of the inequality then gives the desired result for the remaining case of $\beta=2\alpha$. The proof is now complete.
\end{subproof}


\end{enumerate}
\end{proof}